\documentclass[11pt,reqno]{smfbook}
\usepackage[english,frenchb]{babel}
\usepackage[latin1]{inputenc}
\usepackage{smfthm,smfenum,amsmath,amssymb,mathdots}

\usepackage[all]{xy}

\usepackage{smfhref}

\usepackage{graphicx,epsfig,color}
\usepackage{pdfpages}

\NumberTheoremsIn{chapter}
\newtheorem{theoa}{Théorème}

\newtheorem*{theo*}{Théorème}
\newtheorem{corol}{Corollaire}
\newenvironment{nota}{\begin{enonce}[definition]{Notation}}{\end{enonce}}
\theoremstyle{remark}

\numberwithin{equation}{chapter}
\numberwithin{figure}{chapter}

\AtBeginDocument{\renewcommand{\labelitemi}{\textbullet}}

\DeclareMathOperator{\id}{id}
\DeclareMathOperator{\SL}{\sf SL}
\DeclareMathOperator{\PSL}{\sf PSL}
\DeclareMathOperator{\GL}{\sf GL}
\DeclareMathOperator{\PGL}{\sf PGL}
\DeclareMathOperator{\M}{\mathcal M}
\DeclareMathOperator{\Isom}{\sf Isom}
\DeclareMathOperator{\SO}{\sf Isom^+}
\DeclareMathOperator{\Aut}{\sf Aut}
\DeclareMathOperator{\End}{\sf End}
\DeclareMathOperator{\Diff}{\sf Diff}
\DeclareMathOperator{\Bir}{\sf Bir}
\DeclareMathOperator{\BD}{\sf BirDiff}
\DeclareMathOperator{\Homeo}{\sf Homeo}
\DeclareMathOperator{\dist}{\sf dist}
\DeclareMathOperator{\kob}{\sf kob}
\DeclareMathOperator{\aire}{\sf aire}
\DeclareMathOperator{\vol}{\sf vol}
\DeclareMathOperator{\mvol}{\sf mvol}
\DeclareMathOperator{\length}{\sf length}
\DeclareMathOperator{\longueur}{\sf long}

\DeclareMathOperator{\diam}{\sf diam}
\DeclareMathOperator{\rg}{\sf rang}
\DeclareMathOperator{\rk}{\sf rank}
\DeclareMathOperator{\h}{\sf h_{top}}
\DeclareMathOperator{\liap}{\sf \chi_{top}}
\DeclareMathOperator{\Div}{\sf Div}
\DeclareMathOperator{\NS}{\sf NS}
\DeclareMathOperator{\Pic}{\sf Pic}
\DeclareMathOperator{\Amp}{\sf Amp}
\DeclareMathOperator{\Nef}{\sf Nef}

\DeclareMathOperator{\Kah}{\sf Kah}
\DeclareMathOperator{\Fat}{\sf Fatou}

\DeclareMathOperator{\Supp}{\sf Supp}
\DeclareMathOperator{\ddc}{dd^c}
\DeclareMathOperator{\Ind}{\sf Ind}
\DeclareMathOperator{\ml}{\sf m\ell}
\DeclareMathOperator{\ma}{\sf ma}
\DeclareMathOperator{\kod}{\sf kod}
\DeclareMathOperator{\Vect}{\rm Vect}
\renewcommand\epsilon{\varepsilon}
\renewcommand\Im{\mathfrak{Im}}
\renewcommand\Re{\mathfrak{Re}}
\newcommand\N{\mathbf{N}}
\newcommand\Z{\mathbf{Z}}
\newcommand\Q{\mathbf{Q}}
\newcommand\R{\mathbf{R}}
\newcommand\C{\mathbf{C}}
\newcommand\K{\mathbf{K}}
\renewcommand\P{\mathbb{P}}
\newcommand\D{\mathbb{D}}
\renewcommand{\H}{\mathbb{H}}
\renewcommand{\S}{\mathbb{S}}
\newcommand\T{\mathbb{T}}
\newcommand\G{\mathbb{G}}
\renewcommand\b{\overline}
\renewcommand{\O}{\mathcal O}
\newcommand{\Rot}{\mathrm{Rot}}
\newcommand{\e}{\mathrm{e}}

\makeatletter
\def\@tvsp{\mathchoice{{}\mkern-4.5mu}{{}\mkern-4.5mu}{{}\mkern-2.5mu}{}}
\def\ltrivert{\left|\@tvsp\left|\@tvsp\left|}
\def\rtrivert{\right|\@tvsp\right|\@tvsp\right|}
\makeatother

\makeatletter
\let\original@addcontentsline\addcontentsline
\newcommand{\dummy@addcontentsline}[3]{}
\newcommand{\DesactivateToc}{\let\addcontentsline\dummy@addcontentsline}
\newcommand{\ActivateToc}{\let\addcontentsline\original@addcontentsline}
\makeatother

\begin{document}

\includepdf{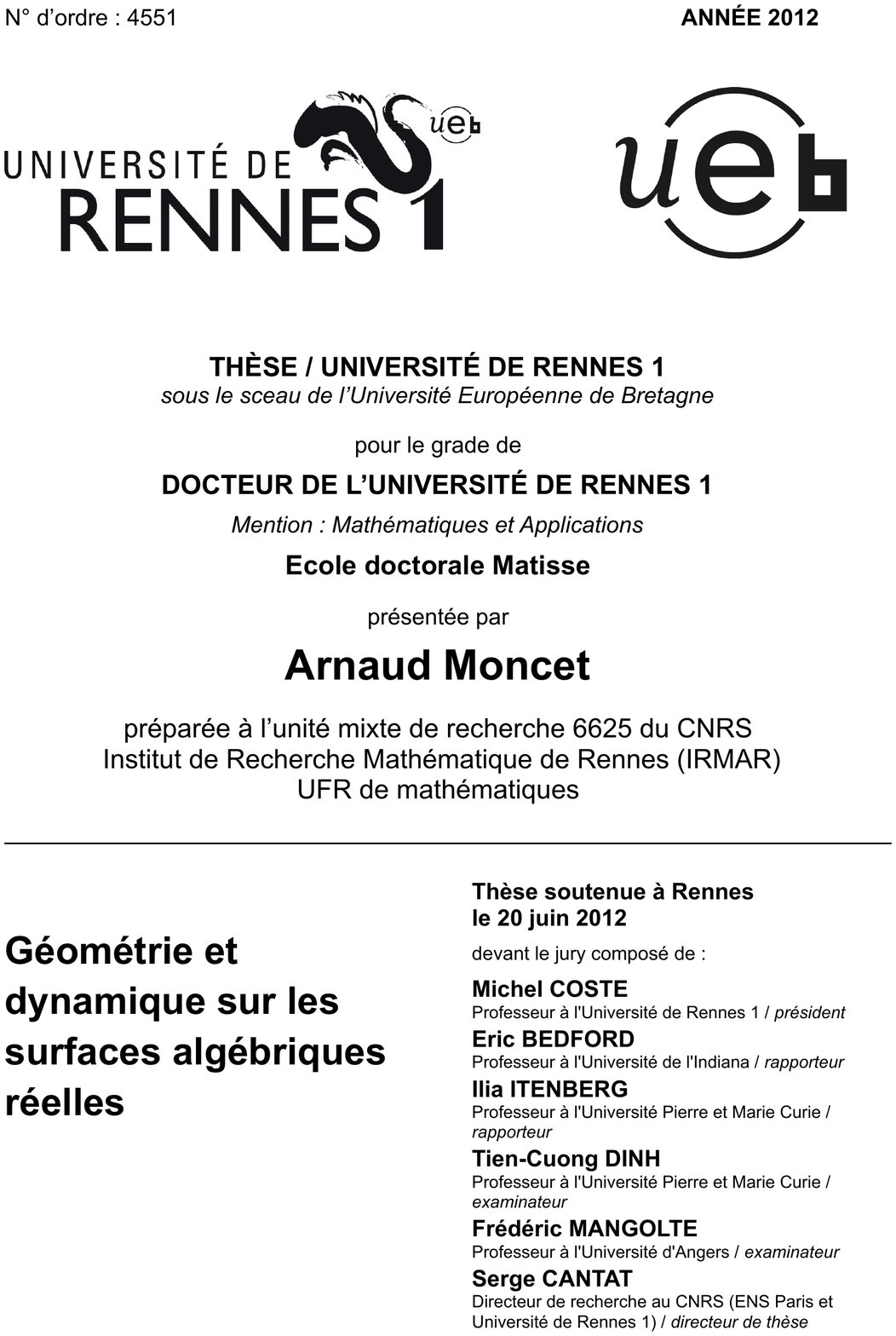}

\frontmatter

\title{Géométrie et dynamique sur les surfaces algébriques réelles}
\author{Arnaud Moncet}
\address{Université de Rennes 1\\
IRMAR\\
campus de Beaulieu\\
bâtiments 22 et 23\\
263 avenue du général Leclerc\\
CS 74205\\
35042 Rennes cedex
}
\email{arnaud.moncet@univ-rennes1.fr}
\date{20 juin 2012}

\begin{abstract}
%\addcontentsline{toc}{chapter}{Résumé}
Cette thèse s'intéresse aux automorphismes des surfaces algébriques réelles, c'est-à-dire les transformations polynomiales admettant un inverse polynomial. La question centrale est de savoir si leur restriction au lieu réel reflète toute la richesse de la dynamique complexe. Celle-ci est traitée sous deux aspects : celui de l'entropie topologique et celui de l'ensemble de Fatou.

Pour le premier point, on introduit une quantité purement géométrique, appelée concordance, qui ne dépend que de la surface. Puis on montre que le rapport des entropies réelle et complexe est relié à cette quantité. La concordance est calculée explicitement sur de nombreux exemples de surfaces, notamment les surfaces abéliennes qui sont traitées en détails, ainsi que certaines surfaces~K3.

Dans la seconde partie, on étudie l'ensemble de Fatou, qui correspond aux points complexes pour lesquels la dynamique est simple. On montre, grâce à des résultats antérieurs de Dinh et Sibony sur les courants positifs fermés, que celui-ci est hyperbolique au sens de Kobayashi, quitte à lui enlever certaines courbes fixées par (un itéré de) notre transformation. Cette propriété permet d'en déduire que ce lieu réel ne peut pas être entièrement contenu dans l'ensemble de Fatou, hormis quelques cas exceptionnels où la topologie du lieu réel est simple et la dynamique bien comprise. Ainsi la complexité de la dynamique est presque toujours observable sur les points réels.\end{abstract}

\begin{altabstract}
This thesis deals with automorphisms of real algebraic surfaces, which are polynomial transformations with a polynomial inverse. The main concern is whether their restriction to the real locus reflects all the richness of the complex dynamics. This question is declined in two directions: the topological entropy and the Fatou set.

For the first one, we introduce a purely geometric quantity depending only on the surface, and we call it concordance. Then we show that the ratio of real and complex entropies is linked to this quantity. The concordance is explicitely computed for many examples of surfaces, especially abelian surfaces which are broadly studied, as well as some K3 surfaces.

In the second part, we are interested in the Fatou set, which corresponds to complex points for which the dynamics is simple. Thanks to previous results of Dinh and Sibony about closed positive currents, we prove that this set is hyperbolic in the sense of Kobayashi, after possibly deleting some curves which are fixed by (an iterate of) our transformation. From this property we deduce that, except for some exceptional cases in which the topology of the real locus is simple and the dynamics well understood, this real locus cannot be entirely contained in the Fatou set. Thus the complexity of the dynamics is observable on real points in most cases.
\end{altabstract}


\maketitle

\setcounter{tocdepth}{2}
%\addcontentsline{toc}{chapter}{Table des matières}
\tableofcontents

%\DesactivateToc

\chapter*{Remerciements}

En premier lieu, je remercie Serge Cantat qui a encadré cette thèse de manière exceptionnelle. Tout ce qu'il m'a apporté durant ces quatre années, autant sur le plan professionnel que humain, est considérable et je ne saurai hélas trouver les mots justes pour lui exprimer toute la gratitude qu'il mérite. Je lui suis en particulier reconnaissant de m'avoir toujours soutenu et encouragé, notamment dans les moments les plus difficiles où il a toujours réussi à me relancer. Je le remercie également pour toutes les connaissances qu'il m'a transmises, ses brillantes idées qui m'ont si souvent permis d'avancer et les interrogations pertinentes qu'il a soulevées. Merci aussi pour sa patience à relire ce manuscrit au fur et à mesure de sa rédaction. 

Je tiens ensuite à remercier Eric Bedford et Ilia Itenberg pour avoir rapporté cette thèse ainsi que pour l'intérêt qu'ils ont porté à mes travaux. C'est un véritable honneur pour moi que mon manuscrit ait été relu par deux scientifiques de tel renom, dont les livres et articles ont été une réelle source d'inspiration pour ma recherche.

Merci aussi à Michel Coste, Tien-Cuong Dinh et Frédéric Mangolte d'avoir accepté de faire partie du jury de thèse. Le premier a d'abord été pour moi un excellent professeur lorsque j'étais étudiant, puis pendant ma thèse je l'ai souvent côtoyé avec grand plaisir dans les différents séminaires ; son livre coécrit avec Jacek Bochnak et Marie-Françoise Roy a été d'une aide précieuse pour m'aider à comprendre les rudiments de la géométrie algébrique réelle. Quant au second, bien que ne l'ayant jamais rencontré, ses articles en collaboration avec Nessim Sibony ont constitué le point de départ de la seconde partie de cette thèse ; c'est grâce à un de leurs théorèmes non publié que j'ai obtenu les résultats sur l'hyperbolicité de l'ensemble de Fatou, qui en constituent l'ingrédient central. Quant à Frédéric, je lui suis plus spécialement reconnaissant pour s'être intéressé très tôt à mes travaux de recherche, ainsi que pour toutes les discussions mathématiques enrichissantes que nous avons pu avoir ; merci aussi de m'avoir invité à deux reprises à Angers au cours de cette dernière année.

J'en profite pour remercier tous les autres chercheurs qui m'ont invité à parler dans des conférences ou des séminaires : Florent Malrieu, Laura DeMarco, Ronan Quarez, Goulwen Fichou, Jean-Yves Briend et Damien Gayet.

Je remercie également l'ensemble des personnels de l'IRMAR, de l'UFR mathématiques et de l'ENS Cachan antenne de Bretagne pour l'ambiance très agréable qu'ils contribuent à mettre au sein du laboratoire. Merci à tous les collègues avec lesquels j'ai pu avoir des discussions mathématiques intéressantes, ainsi qu'aux secrétaires et bibliothécaires qui par leur efficacité nous rendent la vie plus facile.

Mes remerciements les plus sincères vont à l'ensemble des personnes qui ont contribué à m'insuffler la passion des mathématiques : tout d'abord mes parents, puis mon oncle Michel qui a su m'intéresser dès mon plus jeune âge à des problèmes mathématiques divers, et naturellement tous les professeurs de mathématiques que j'ai eus durant ma scolarité, plus particulièrement MM. Kervigno, Roblet, et Pierre qui restent pour moi des modèles de pédagogie.

L'enseignement a naturellement représenté une part importante de mon travail  de doctorant. Je remercie Arnaud Debussche et Michel Pierre de m'avoir permis d'exercer cette activité dans le cadre exceptionnel du département de mathématiques de Ker Lann, et pour m'y avoir accueilli si chaleureusement. Merci aux étudiants qui ont assisté à mes cours pour l'attention et la motivation dont ils font preuve : c'est un réel plaisir que d'avoir affaire à des élèves aussi intéressés et intéressants. Merci également aux collègues avec qui j'ai partagé des travaux dirigés ou des jurys d'oraux blancs pour les nombreux échanges pédagogiques enrichissants.

La vie de doctorant aurait sans doute été beaucoup plus terne sans quelques moments de détente. Je remercie pour cela les autres doctorants, et plus particulièrement ceux avec qui j'ai partagé un bureau -- Yoann, François, Damian, Andrea, Adel -- ainsi que ceux du bureau 334 et ceux qui le squattent régulièrement pour la pause café du midi : Baron, Baptiste, Jean-Louis, Gaël, Cyrille, Jacques, Jean-Romain. Merci également aux anciens du séminaire de midi trente, ainsi qu'à tous les thésards avec qui j'ai eu le plaisir de partager chaque midi les mets succulents du RU.

Merci aussi à tous les autres camarades rennais : Maëlle, Tiffany, Chloé, Baptiste, Marine, Lucas, Céline ($\times 2$), Yacine, Nora, Emeline, Adrien, Olivier, Stéphane, Jacques, Richy, Vincent, Mat, Eric, Morgane, Fabien, Basile, Jean-Seb, Jeff, et c\ae tera. Enfin, merci à tous mes très bons amis qui hélas vivent en dehors de la capitale bretonne, mais avec qui j'ai passé d'excellents moments ces dernières années : Bruno, Fred, Emma, Hélène, Boris, Julien, Laure, Romain, Lohen, Catherine, Karine, Brice, Léo, Bibi, Pitrolls, Vincent, Leïla, Antoine ($\times 2$). Merci surtout à ceux d'entre eux qui sont venus de loin pour assister à ma soutenance et boire quelques coupes.

Je ne saurais terminer ces remerciements sans mentionner ma famille, qui a été très présente durant ces quatre ans, et ce malgré l'éloignement géographique. Merci en particulier à mon père Gérard et à mon frère Romain, ainsi qu'à Christine et Denis qui m'ont souvent accueilli chez eux lors de mes week-ends à Paris. Merci aussi à ma mamie pour sa générosité, à mes tantes, oncles, cousines et cousins (je ne vais pas tous les énumérer, la liste serait trop longue). Enfin, j'adresse une pensée émue à ma mère Roseline, qui nous a quittés bien trop tôt : j'aurais tant souhaité qu'elle soit là pour participer à cette étape importante de ma vie.

%\ActivateToc

\mainmatter

%%%%%%%%%%%%%%%%%%%%%%%%%%%%%%%%%%
% INTRODUCTION
%%%%%%%%%%%%%%%%%%%%%%%%%%%%%%%%%%

%%%%%%%%%%%%%%%%%%%%%%%%%%%%%%%%%%
% Chapitre
%%%%%%%%%%%%%%%%%%%%%%%%%%%%%%%%%%
% Introduction
%%%%%%%%%%%%%%%%%%%%%%%%%%%%%%%%%%

\chapter*{Introduction}

Cette thèse s'intéresse aux interactions entre géométrie réelle et géométrie complexe, notamment au niveau de la dynamique des automorphismes. Pour préciser un peu les choses, donnons-nous $X$ une variété algébrique réelle de dimension $d$. Une telle variété peut être vue ou bien comme l'ensemble des zéros d'un ensemble de polynômes à coefficients réels, ou bien comme une variété projective complexe munie d'une involution anti-holomorphe $\sigma$. Quelque soit le point de vue adopté, on a une notion de points réels et de points complexes : dans le premier cas, $X(\R)$ est l'ensemble des zéros réels, tandis que $X(\C)$ est l'ensemble des zéros complexes ; dans le second cas, $X(\R)$ est l'ensemble des points fixes de $\sigma$.\\

\textbf{On suppose dans ce qui suit que $X(\C)$ est une variété lisse et que $X(\R)$ est non vide.\\
}

Notons que $X(\C)$ est alors une variété analytique complexe compacte de dimension complexe $d$, donc de dimenion réelle $2d$, tandis que $X(\R)$ est une variété analytique réelle compacte de dimension $d$.

%%%%%%%%%%%%%%%%%%%%%%%%%%%%%%%%%%
% Automorphismes
%%%%%%%%%%%%%%%%%%%%%%%%%%%%%%%%%%

\section*{Groupes d'automorphismes}

Selon le point de vue choisi, un automorphisme $f$ de $X$ est ou bien une application birégulière de $X(\C)$ dans lui-même donnée par des équations à coefficients réels\footnote{La terminologie diffère ici de \cite{bcr} par exemple, pour lesquels les morphismes sont des applications régulières seulement sur $X(\R)$.}, ou bien une application biholomorphe de $X$ qui commute avec $\sigma$. Un tel automorphisme définit un difféomorphisme réel-analytique $f_\R$ sur $X(\R)$, ce qui donne une flèche injective
\begin{equation*}
\begin{split}
\Aut(X) &\hookrightarrow \Diff(X(\R))\\
f &\mapsto f_\R.
\end{split}
\end{equation*}
Dans une très large mesure, le but de cette thèse est d'essayer de comprendre l'image de ce morphisme. Avant d'explorer plusieurs questions relatives à ce sujet, étudions un exemple.

\begin{enonce*}[remark]{Exemple}[\cite{mazur}]
Soit $P(x_1,x_2,x_3)\in\R[x_1,x_2,x_3]$ un polynôme de degré $2$ par rapport à chacune des trois variables. On considère la surface algébrique $X$ donnée par l'annulation du polynôme $P$, que l'on projectivise dans $\P^1\times\P^1\times\P^1$. Autrement dit, si $\tilde P$ est le polynôme homogénéisé par rapport à chaque variable
\begin{equation*}
\tilde P(u_1,v_1,u_2,v_2,u_3,v_3) = (v_1v_2v_3)^2 \, P\left(\frac{u_1}{v_1},\frac{u_2}{v_2},\frac{u_3}{v_3}\right),
\end{equation*}
alors pour $\K=\R$ ou $\C$, la variété $X(\K)$ est donnée par
\begin{equation*}
X(\K) = \left\{
\left([u_i:v_i]\right)_{1\leq i\leq 3}\in{\P^1(\K)}^3 \,\big|\, \tilde P(u_1,v_1,u_2,v_2,u_3,v_3)=0
\right\}.
\end{equation*}
On dit que $X$ est une surface de degré $(2,2,2)$ dans $\P^1\times\P^1\times\P^1$. Une telle surface vient avec trois revêtements doubles $p_i:X\to\P^1\times\P^1$, qui consistent à oublier une des trois coordonnées $x_i$. Considérons les trois involutions $s_i$ de chacun de ces revêtements. D'après un théorème de Lan Wang \cite{wang} (voir aussi \cite{cantat-k3}), ces trois involutions engendrent un produit libre 
\begin{equation*}
\Gamma\simeq\Z/2\Z*\Z/2\Z*\Z/2\Z
\end{equation*}
dans $\Aut(X)$ dès que $X$ est lisse. On voit donc pour cet exemple que le groupe des automorphismes est gros, au sens où il contient un groupe libre. Cependant, l'image dans $\Diff(X(\R))$ peut être relativement mince. Il existe ainsi des exemples (voir \cite{moncet-m2}) pour lesquels il y a au plus un automorphisme par composante connexe de $\Diff(X(\R))$.\footnote{Plus exactement, on montre dans ce texte qu'il existe une surface réelle lisse $X$ de degré $(2,2,2)$ telle que :
\begin{itemize}
\item $X(\R)$ est une surface compacte orientable de genre $3$ ;
\item le groupe $\Aut(X)$ est égal au produit libre $\Gamma$ des trois involutions ;
\item $\Aut(X)$ s'injecte dans le mapping class group de la surface $X(\R)$, via le morphisme composé
$$
\Aut(X) \hookrightarrow \Diff(X(\R)) \twoheadrightarrow {\sf MCG}(X(\R)).
$$
\end{itemize}
}
\end{enonce*}

%%%%%%%%%%%%%%%%%%%%%%%%%%%%%%%%%%
% Dynamique
%%%%%%%%%%%%%%%%%%%%%%%%%%%%%%%%%%

\section*{Dynamique des automorphismes}

Intéressons-nous maintenant aux aspects dynamiques. \'Etant donné un auto\-morphisme $f$, on cherche à comparer les deux systèmes dynamiques
\begin{equation*}
f_\R:X(\R)\to X(\R)
\quad\text{et}\quad
f_\C:X(\C)\to X(\C).
\end{equation*}
Par restriction, on sait que la dynamique de $f_\C$ est plus complexe que celle de~$f_\R$. Les deux questions que l'on se pose naturellement sont les suivantes :
\begin{enumerate}
\item 
La dynamique dans le réel peut-elle être aussi riche que celle dans le complexe ?
\item
À l'inverse, se peut-il que la dynamique dans le réel soit très simple alors que celle dans le complexe est très compliquée ?
\end{enumerate}

Dans \cite{bedford-kim-maxent}, Bedford et Kim ont donné une réponse positive à la première interrogation dans le cadre des surfaces (voir plus loin). Dans cette thèse, on s'intéresse plutôt à la seconde problématique. Afin de lui donner une forme plus précise, nous utilisons des outils quantitatifs (entropie) et qualitatifs (ensemble de Julia et de Fatou) pour étudier ces dynamiques.

\begin{enonce*}[remark]{Remarque}
En dimension $1$, la dynamique d'un automorphisme est simple, car on est dans un des trois cas suivants, selon le genre $g$ de $X(\C)$ :
\begin{enumerate}
\item Si $g=0$, on a $X(\C)\simeq\P^1(\C)$, donc $\Aut(X(\C))\simeq\PGL_2(\C)$ est connexe. En particulier tout automorphisme est isotope à l'identité, et la dynamique est bien comprise (c'est celle d'une application linéaire).
\item Si $g=1$, alors $X(\C)$ est un tore $\C/\Lambda$, et le groupe $\Aut(X(\C))$ est égal, à indice fini près, au groupe des translations sur ce tore.
\item Si $g\geq 2$, le groupe $\Aut(X(\C))$ est fini.
\end{enumerate}
\end{enonce*}

Dans cette thèse, on se limite à la dimension $2$, pour laquelle on a déjà de nombreux exemples intéressants (cf. \cite{cantat-k3,mcmullen-k3-siegel,bedford-kim-degree,mcmullen-rat,bedford-kim-maxent,bedford-kim-rotation} entre autres). Voir aussi \cite{bedford-survey} ou \cite{cantat-panorama} pour des surveys récents sur la dynamique complexe des automorphismes de surfaces.

%%%%%%%%%%%%%%%%%%%%%%%%%%%%%%%%%%
% Entropie
%%%%%%%%%%%%%%%%%%%%%%%%%%%%%%%%%%

\section*{Entropie topologique}

On définit de la manière suivante (voir \cite{katok-hasselblatt}) l'\emph{entropie topologique} d'un système dynamique continu $T:E\to E$, où $(E,\dist)$ est un espace métrique compact. Pour $\epsilon>0$ et $n\in\N^*$, on dit qu'un ensemble fini $F\subset E$ décrit~$E$ à la précision $\epsilon$ et à l'ordre $n$ si
\begin{equation*}
\forall x\in E,~\exists y\in F,~\forall k\in\{0,\cdots,n-1\},~\dist(T^k(x),T^k(y))<\epsilon.
\end{equation*}
On note $N(\epsilon,n)$ le cardinal minimal d'un tel ensemble, puis on définit l'entropie topologique de $T$ par la formule :
\begin{equation*}
\h(T) \, = \, \lim_{\epsilon\to 0} \, \limsup_{n\to+\infty} \, \frac{1}{n}\log N(\epsilon,n) \,\geq 0.
\end{equation*}
C'est un invariant topologique (il ne dépend pas du choix de la métrique $\dist$) qui mesure la complexité de la dynamique de $T$ : plus $\h(T)$ est grand, plus la dynamique est chaotique, au sens où l'on a besoin de plus de conditions initiales pour pouvoir décrire toutes les orbites à une précision $\epsilon$ donnée. Un système d'entropie nulle aura en ce sens une dynamique relativement simple.\\

Lorsque $f:X\to X$ est un automorphisme d'une variété algébrique réelle, les entropies des deux systèmes dynamiques $f_\R$ et $f_\C$ sont reliées par
\begin{equation*}
0\leq \h(f_\R)\leq \h(f_\C),
\end{equation*}
car $X(\R)\subset X(\C)$. L'entropie complexe $\h(f_\C)$ peut être calculée plus simplement en terme d'action sur la cohomologie (voir \cite{yomdin} et \cite{gromov-entropie}) :

\begin{theo*}[Gromov, Yomdin]
Soit $f$ une transformation holomorphe d'une variété complexe compacte kählérienne $X$. On note $\lambda(f)$ le rayon spectral de l'application induite sur la cohomologie $f^*: H^*(X;\R)\to H^*(X;\R)$. On a alors l'égalité\footnote{Gromov a montré l'inégalité $\h(f)\leq\log(\lambda(f))$ dans des notes rédigées en 1977 (finalement publiées en 2003) ; Yomdin a montré l'autre inégalité dans le cadre plus général d'une transformation $\mathcal{C}^\infty$, en considérant les taux de croissance du volume des sous-variétés.}
\begin{equation*}
\h(f)=\log(\lambda(f)).
\end{equation*}
De plus, le rayon spectral $\lambda(f)$ est atteint sur le sous-espace $\bigoplus_{p=0}^d H^{p,p}(X;\R)$, où $d$ est la dimension complexe de $X$.
\end{theo*}

Dans le cas où $f$ est un automorphisme d'une surface, $\h(f_\C)$ est donc égal au logarithme $\lambda(f)$ du rayon spectral de $f^*$ restreint à $H^{1,1}(X(\C);\R)$. On se limite au cas où cette entropie est positive (strictement), ce qui correspond à des dynamiques non triviales. En effet, dans tous les autre cas, soit un itéré de $f$ est isotope à l'identité, soit $f$ préserve une fibration elliptique, d'après \cite{cantat-k3} et \cite{gizatullin} (voir aussi le chapitre~\ref{chap-auto}).

\begin{enonce*}[definition]{Définition}
Un automorphisme $f$ d'une surface complexe compacte kählérienne $X$ est dit \emph{de type loxodromique}\footnote{Voir le chapitre \ref{chap-auto}, ou \cite{cantat-k3}, pour la justification d'une telle terminologie.} (ou plus simplement \emph{loxodromique}) lorsque
\begin{equation*}
\lambda(f)>1,
\end{equation*}
ce qui équivaut à $\h(f_\C)>0$.
\end{enonce*}

Ces automorphismes ont été décrits par Cantat dans \cite{cantat-cras} et \cite{cantat-k3}. En particulier, il y montre le théorème suivant (la seconde partie est due à \cite{nagata}) :

\begin{theo*}[Cantat, Nagata]
Soit $X$ une surface complexe compacte, et soit $f$ un automorphisme de $X$ de type loxodromique. Alors la dimension de Kodaira $\kod(X)$ vaut $0$ ou $-\infty$. Plus précisément, on est dans un des deux cas suivants :
\begin{enumerate}
\item Si $\kod(X)=0$, alors il existe un morphisme birationnel $\pi:X\to X_0$ tel que la surface $X_0$ soit un tore, une surface K3 ou une surface d'Enriques, et $f$ induit un automorphisme $f_0$ sur $X_0$, qui est de type loxodromique avec~$\h(f_0)=\h(f)>0$\footnote{Autrement dit, $f$ est obtenu à partir d'un automorphisme loxodromique sur une surface minimale en éclatant certaines orbites finies, et cette surface minimale est un tore, une surface K3 ou une surface d'Enriques.}.
\item Si $\kod(X)=-\infty$, alors $X$ est une surface rationnelle non minimale, isomorphe à $\P^2$ éclaté en $n$ points, avec $n\geq 10$.
\end{enumerate}
\end{theo*}

\begin{enonce*}[remark]{Suite de l'exemple}
La formule d'adjonction et le théorème des sections hyperplanes de Lefschetz montrent qu'une surface lisse $X$ de degré $(2,2,2)$ est une surface K3. L'automorphisme $f$ défini par la composition des trois involutions $s_i$ (dans n'importe quel ordre) est de type loxodromique, et son entropie se calcule facilement grâce à l'action de $f^*$ sur le sous-groupe de~$H^{1,1}(X(\C);\R)$ engendré par les fibres des trois fibrations $\pi_i:(x_1,x_2,x_3)\in X\mapsto x_i\in\P^1$ (voir \cite{cantat-k3}, ou l'exemple \ref{ex-loxo} du chapitre \ref{chap-auto}) :
\begin{equation*}
\h(f_\C)=\log(9+4\sqrt{5})>0.
\end{equation*}
\end{enonce*}

Dans le chapitre \ref{chap-k3}, on introduit une famille ${(X^t)}_{t\in\R^*}$ de telles surfaces, chacune étant munie de l'automorphisme loxodromique $f^t:X^t\to X^t$ décrit ci-dessus. Pour cette famille d'automorphismes, l'entropie du système dynamique complexe est constante, tandis que celle du système dynamique réel converge vers $0$ lorsque $t\to 0$. Il existe donc des exemples d'automorphismes de surfaces pour lesquels le rapport
\begin{equation*}
\frac{\h(f_\R)}{\h(f_\C)}
\end{equation*}
est arbitrairement petit. La question de savoir si, pour les automorphismes de surfaces, ce rapport peut être nul est actuellement ouverte.

\begin{enonce*}[remark]{Remarque}
Dans \cite{bedford-kim-maxent}, Bedford et Kim ont répondu à la question inverse de savoir s'il existe un automorphisme $f$ d'une surface algébrique réelle~$X$ qui soit d'entropie maximale, c'est-à-dire tel que $\h(f_\R)=\h(f_\C)>0$. Ils répondent positivement à cette question, en donnant des exemples de tels automorphismes sur des surfaces géométriquement rationnelles, construites par éclatements succesifs de $\P^2$. Par contre, on ne connaît pas de tels exemples sur d'autres types de surfaces (K3 ou Enriques).
\end{enonce*}

%%%%%%%%%%%%%%%%%%%%%%%%%%%%%%%%%%
% Comparaison de volumes
%%%%%%%%%%%%%%%%%%%%%%%%%%%%%%%%%%

\section*{Comparaison de volumes}

Dans \cite{yomdin} (voir aussi \cite{gromov-aft-yomdin}), Yomdin relie l'entropie topologique avec la croissance des volumes des sous-variétés. Cette approche nous incite à comparer les longueurs des courbes réelles avec les aires des courbes complexifiées, plutôt que de comparer les entropies. Nous menons cette étude dans la partie \ref{part2}. 

Pour une courbe algébrique réelle $C$ sur~$X$, on note
\begin{align*}
&\mvol_\R(C) = \max\left\{ \longueur(C'(\R))\,|\,C'\text{ courbe telle que }[C]=[C'] \right\}\\
\text{et}\quad &\vol_\C(C) = \aire(C(\C)),\label{pasdemax}
\end{align*}
où $[C]$ désigne la classe d'homologie de $C(\C)$ dans $H_2(X(\C);\R)$. Les volumes sont ici calculés vis-à-vis d'une métrique de Kähler sur $X(\C)$, ce qui a pour conséquence que l'on n'a pas besoin de prendre de maximum dans la définition de $\vol_\C$ (l'aire d'une courbe ne dépend que de sa classe, par la formule de Wirtinger). On a toujours l'inégalité
$$
\mvol_\R(C)\leq C^{ste}\vol_\C(C),
$$
qui est une conséquence de la formule de Cauchy-Crofton (cf. chapitre \ref{chap-upper-bound}).

\begin{enonce*}[definition]{Définition}
On définit la \emph{concordance} $\alpha(X)$ comme la borne supérieure des exposants $\alpha\geq 0$ pour lesquels il existe $C^{ste}>0$ telle que
\begin{equation*}
\mvol_\R(C)\geq C^{ste}\vol_\C(C)^\alpha
\end{equation*}
pour toute courbe algébrique réelle $C$ dans un ensemble suffisamment large de courbes\footnote{Voir le chapitre \ref{chap-conc} pour une définition précise.}.
\end{enonce*}

La concordance $\alpha(X)$ est ainsi un nombre compris entre $0$ et $1$, qui ne dépend que de $X$ (il ne dépend pas du choix d'une métrique particulière sur~$X$).

On arrive à calculer cette concordance sur des exemples précis. Par exemple, lorsque le nombre de Picard $\rho(X)$\footnote{Il s'agit ici du nombre de Picard réel.} est égal à $1$, on a $\alpha(X)=1$. Ceci concerne par exemple le plan projectif ${X=\P^2}$, ou encore la quadrique projective ${x^2+y^2+z^2-t^2=0}$ dans $\P^3$. Plus généralement lorsque le cône nef\footnote{Là encore, il s'agit du cône nef réel.} est polyédral rationnel, et que ses rayons extrémaux sont engendrés par de \og vraies\fg~courbes réelles, au sens où ces courbes ont un lieu réel non vide qui n'est pas constitué uniquement de points isolés, alors $\alpha(X)=1$. Ceci s'applique lorsque $X$ est une surface de Del Pezzo (cf. chapitre~\ref{chap-conc}), par exemple~${X=\P^1\times\P^1}$. Les premiers exemples de surfaces pour lesquelles la concordance n'est pas $1$ sont donnés au chapitre \ref{chap-tores} :

\begin{theoa}\label{theo-conc-ab}
Soit $X$ une surface abélienne réelle, c'est-à-dire une surface algébrique réelle telle que $X(\C)$ est un tore. Le nombre de Picard $\rho(X)$ vaut $1$, $2$ ou $3$, et la concordance est donnée par :
\begin{enumerate}
\item Si $\rho(X)=1$, alors $\alpha(X)=1$.
\item Si $\rho(X)=2$, alors $\alpha(X)$ vaut $1$ ou $1/2$ selon que $X$ possède ou non une fibration elliptique réelle.
\item Si $\rho(X)=3$, alors $\alpha(X)=1/2$.
\end{enumerate}
De plus, la concordance vaut $1/2$ exactement pour les tores sur lesquels il existe un automorphisme loxodromique réel.
\end{theoa}

La concordance $1/2$ qui apparaît dans ce théorème est égale au rapport des entropies réelles et complexes des automorphismes loxodromiques sur les tores. Ce lien entre concordance et rapport des entropies est un phénomène plus général, qui est mis en relief au chapitre~\ref{chap-ent} :

\begin{theoa}\label{theo-conc-ent}
Soit $X$ une surface algébrique réelle.
\begin{enumerate}
\item On suppose qu'il existe un automorphisme $f$ de type loxodromique sur~$X$. On a alors la majoration
$$
\alpha(X)\leq\frac{\h(f_\R)}{\h(f_\C)}.
$$
\item Si de plus $\rho(X)=2$, alors
$$
\alpha(X)=\frac{\h(f_\R)}{\h(f_\C)}.
$$
\end{enumerate}
\end{theoa}

La première partie du théorème est liée au théorème de Yomdin sur la croissance des volumes \cite{yomdin}. La seconde se démontre grâce à l'approximation de l'entropie par des \og fers à cheval\fg, un théorème dû à Katok \cite{katok-hasselblatt}.

Au chapitre \ref{chap-k3}, on obtient comme corollaire de ce théorème et de l'exemple de famille ${(X^t,f^t)}_{t\in\R^*}$ cité plus haut :

\begin{corol}\label{coro-k3}
Il existe une famille $(X^t)_{t\in\R^*}$ de surfaces K3 réelles telles que
$$
\lim_{t\to 0}\alpha(X^t)=0.
$$
\end{corol}

Comme autre conséquence de la première partie du théorème \ref{theo-conc-ent}, on obtient au chapitre \ref{chap-non-dense} des informations sur l'image du morphisme
$$
\Aut(X)\to\Diff(X(\R)).
$$
Plus exactement, lorsque $X$ est de dimension de Kodaira nulle, les automorphismes préservent une forme d'aire canonique $\mu$ sur $X(\R)$, et on peut remplacer le groupe $\Diff(X(\R))$ par le sous-groupe $\Diff_\mu(X(\R))$ des difféomorphismes qui préservent l'aire. On a alors l'énoncé suivant (comparer avec \cite[Theorem 4, Proposition 6]{kollar-mangolte}) :

\begin{corol}\label{coro-non-dense}
Soit $X$ une surface algébrique réelle telle que $\alpha(X)>0$.
\begin{enumerate}
\item L'image de $\Aut(X)$ dans $\Diff(X(\R))$ n'est pas dense pour la topologie~$\mathcal{C}^\infty$ sur $\Diff(X(\R))$.
\item Si $\kod(X)=0$, l'image de $\Aut(X)$ dans $\Diff_\mu(X(\R))$ est également non dense, pour la topologie $\mathcal{C}^\infty$.
\item Si le groupe $\Aut(X)$ est discret pour sa topologie usuelle, alors son image dans $\Diff(X(\R))$ est aussi discrète pour la topologie $\mathcal{C}^1$.
\end{enumerate}
\end{corol}

La question de savoir s'il existe des surfaces avec $\alpha(X)=0$ reste ouverte. Cette question est liée à celle évoquée précédemment de savoir s'il existe un automorphisme loxodromique d'une surface $X$ d'entropie réelle nulle : une telle surface $X$ serait nécessairement de concordance nulle, d'après la majoration du théorème \ref{theo-conc-ent}.

%%%%%%%%%%%%%%%%%%%%%%%%%%%%%%%%%%
% Ensemble de Fatou
%%%%%%%%%%%%%%%%%%%%%%%%%%%%%%%%%%

\section*{Ensemble de Julia et ensemble de Fatou}

Par définition, l'\emph{ensemble de Fatou} de $f$, noté $\Fat(f)$, est le plus grand ouvert de $X$ sur lequel la famille ${(f^n)}_{n\in\Z}$ est une famille normale\footnote{Certains auteurs définissent l'ensemble de Fatou comme le domaine de normalité des itérés positifs de $f$. Si l'on note $\Fat^+(f)$ ce dernier ensemble, et $\Fat^-(f)=\Fat^+(f^{-1})$, on a la relation \begin{equation*}\Fat(f)=\Fat^+(f)\cap\Fat^-(f).\end{equation*}}. Autrement dit, un point $x\in X(\C)$ est dans l'ensemble de Fatou s'il possède un voisinage $U$ tel que toute suite à valeurs dans $\{f^n\,|\,n\in\Z\}$ possède une sous-suite qui converge uniformément sur les compacts de $U$. Le complémentaire de l'ensemble de Fatou est appelé \emph{ensemble de Julia}.

Intuitivement, l'ensemble de Julia correspond aux points autour desquels les orbites ont une forte dépendance aux conditions initiales : c'est donc le lieu où la dynamique est chaotique. \emph{A contrario}, l'ensemble de Fatou apparaît comme le lieu où la dynamique est plutôt simple. Un exemple évident d'ouvert inclus dans l'ensemble de Fatou est donné par les domaines sur lesquels $f$ est conjugué à une rotation sur un domaine circulaire de $\C^2$ : un tel ouvert est appelé \emph{domaine de rotation au sens fort}.

Dans la partie \ref{part3}, on cherche à répondre à la question suivante :

\begin{enonce*}[remark]{Question}
Existe-t-il un automorphisme loxodromique $f$ d'une surface algébrique réelle $X$ tel que $X(\R)\subset\Fat(f)$ ?
\end{enonce*}

Notons qu'un tel automorphisme serait d'entropie nulle sur $X(\R)$, et fournirait ainsi des réponses positives aux questions que l'on s'est déjà posées plus haut. On va plutôt essayer de montrer que la réponse est négative. Ceci fait l'objet de la partie \ref{part2}.

Pour cela, on étudie l'hyperbolicité de l'ensemble de Fatou, que l'on déduit essentiellement d'un théorème non publié de Dinh et Sibony sur les courants positifs fermés (cf. chapitre \ref{chap-hyp}).

\begin{theoa}\label{theoa-fatou}
Soit $f$ un automorphisme loxodromique d'une surface complexe compacte kählérienne $X$. Alors l'ensemble $\Fat(f)$, privé d'éventuelles courbes périodiques, est hyperbolique au sens de Kobayashi.\footnote{On a en fait un énoncé plus fort, cf. théorème \ref{fat-hyp}.}
\end{theoa}

Comme conséquence, on obtient dans le chapitre \ref{chap-rec} que toutes les composantes de Fatou qui possèdent des points récurrents sont des domaines de rotation au sens de \cite{bedford-kim-rotation}, ce qui transpose au cas des automorphismes des surfaces un résultat obtenu par Ueda pour les endomorphismes de $\P^n$ \cite[Theorem 3.1]{ueda}.

\begin{corol}\label{theoa-rot}
Soit $f$ un automorphisme loxodromique d'une surface complexe compacte kählérienne $X$, et soit $\Omega$ une composante connexe de l'ensemble de Fatou. On suppose qu'il existe $x$ et $y$ dans $\Omega$ tels que
\begin{equation*}
f^{n_k}(x) \mathop{\longrightarrow}\limits_{k\to+\infty} y
\end{equation*}
pour une suite $n_k\to\pm\infty$. Alors $\Omega$ est un domaine de rotation, au sens où il existe une suite $n_k\to\pm\infty$ telle que
\begin{equation*}
f^{n_k} \mathop{\longrightarrow}\limits_{k\to+\infty} \id_\Omega
\end{equation*}
uniformément sur les compacts de $\Omega$.
\end{corol}

On déduit de cette étude des restrictions topologiques pour le lieu réel des surfaces $X$ qui répondent positivement à la question ci-dessus.

\begin{theoa}\label{theoa-fat-reel}
Soit $f$ un automorphisme loxodromique d'une surface algébrique\footnote{Ce théorème est vrai dans le cas plus général où $X$ est une surface complexe compacte kählérienne, munie d'une structure réelle, cf. chapitre \ref{chap-str-reelle}.} réelle $X$. Si une composante connexe de $X(\R)$ est contenue dans $\Fat(f)$, alors celle-ci ne peut être qu'une sphère, un tore, un plan projectif ou une bouteille de Klein.
\end{theoa}

On obtient aussi une description simple de ce que doit être la dynamique sur une telle composante réelle (voir chapitre \ref{chap-fatou-reel}).

Dès que la topologie de $X(\R)$ devient compliquée, il est donc impossible d'obtenir un tel exemple. En revanche, il est légitime de se demander s'il existe des modèles algébriques réels de surfaces topologiques simples qui possèdent un automorphisme loxodromique pour lequel le lieu réel est contenu dans l'ensemble de Fatou. Je montre au chapitre \ref{chap-bir} qu'un tel exemple existe, si l'on s'autorise à prendre une application birationnelle au lieu d'un automorphisme.

\begin{theoa}\label{theoa-birdiff}
Il existe une application birationnelle réelle $f$ sur la surface~$X=\P^1\times\P^1$ telle que
\begin{enumerate}
\item $\lambda(f)>1$ ;\footnote{Ici, $\lambda(f)$ doit être défini comme $\lim_{n\to+\infty}\left\|{(f^n)}^*\right\|^{1/n}$, plutôt que comme le rayon spectral de $f^*$ qui serait égal à $\lim_{n\to+\infty}\left\|{(f^*)}^n\right\|^{1/n}$.}
\item le lieu réel $X(\R)\simeq\R^2/\Z^2$ est inclus dans  l'ensemble de Fatou.
\end{enumerate}
Plus précisément, $X(\R)$ est inclus dans un domaine de rotation au sens fort.
\end{theoa}

%%%%%%%%%%%%%%%%%%%%%%%%%%%%%%%%%%
% Plan
%%%%%%%%%%%%%%%%%%%%%%%%%%%%%%%%%%

\section*{Contenu de la thèse}

La première partie est constituée essentiellement de rappels de résultats éparpillés dans la littérature. Rien n'y est nouveau, excepté la proposition~\ref{courbe-entiere-aire-finie} et le théorème \ref{dilat+per}. Les chapitres~\ref{chap-courants} et \ref{chap-dilat} sont nécessaires seulement pour la partie~\ref{part3}, et peuvent donc n'être lus qu'après la partie \ref{part2}.

Le chapitre \ref{chap-str-reelle} fixe les principales notations et conventions, utilisées tout au long du texte, concernant les surfaces algébriques complexes et réelles, et plus généralement les surfaces kählériennes munies d'une structure réelle.

La chapitre \ref{chap-auto} s'intéresse aux automorphismes de ces surfaces. On y rappelle la classification homologique des automorphismes introduite dans \cite{cantat-k3}, qui justifie la terminologie de \emph{type loxodromique} pour les automorphismes d'entropie positive. Une attention particulière est aussi portée aux courbes périodiques d'un tel automorphisme.

Dans le chapitre \ref{chap-courants}, on passe en revue la théorie des courants positifs fermés sur une variété complexe. On y fait apparaître la notion de courant d'Ahlfors associé à une courbe entière.

Dans le chapitre \ref{chap-dilat}, on montre comment on peut associer à un automorphisme loxodromique $f$ des courants positifs fermés $T^+_f$ et $T^-_f$, avec potentiels locaux continus, qui sont étroitement liés à la dynamique de $f$. Lorsque~$f$ possède des courbes périodiques, on montre que les potentiels de $T^+_f$ et~$T^-_f$ sont constants le long de ces courbes, si l'on choisit un bon recouvrement (théorème \ref{dilat+per}). Cette dernière propriété s'avèrera utile au chapitre \ref{chap-hyp}.\\

La partie \ref{part2} s'intéresse à la comparaison de volumes évoquée plus haut. Ses résultats font l'objet d'un article \cite{volumes} à paraître dans \emph{International Mathematics Research Notices}, publié en ligne depuis septembre 2011. Cette partie est indépendante des chapitres \ref{chap-courants} et \ref{chap-dilat}.

Le chapitre \ref{chap-upper-bound} est consacré à la formule de Cauchy-Crofton dans le cadre de l'espace projectif. Cette formule implique la majoration, évoquée plus haut, de la longueur d'une courbe réelle par l'aire de la courbe complexe correspondante, à une constante multiplicative près.

Le chapitre \ref{chap-conc} est dédié à la définition et aux premières propriétés de la concordance. On y voit plusieurs exemples de surfaces ayant une concordance égale à un.

Le chapitre \ref{chap-ent} établit le lien entre cette concordance nouvellement définie et le rapport des entropies pour un automorphisme. On y démontre le théorème~\ref{theo-conc-ent}, en utilisant pour la première partie le théorème de Yomdin déjà mentionné sur la relation entre entropie et croissance des volumes, et pour la seconde un théorème de Katok sur l'approximation de l'entropie par celle d'un \og fer à cheval\fg.

Dans le chapitre \ref{chap-tores}, on décrit de manière complète la concordance pour les surfaces abéliennes réelles, et on démontre le théorème \ref{theo-conc-ab}. On y montre aussi que la concordance peut varier si l'on change de structure réelle sur la surface complexe sous-jacente (\textsection\ref{dep-str}, qui n'est pas dans \cite{volumes}).

Le chapitre \ref{chap-k3} est dédié aux surfaces K3 réelles. On commence par y donner une version améliorée, dans le cadre des surfaces K3, de la deuxième partie du théorème \ref{theo-conc-ent} (celle concernant les surfaces de nombre de Picard $2$) : on précise à quelles conditions les hypothèses de ce théorème sont satisfaites, et on montre que la concordance est un dans le cas contraire. Puis on décrit l'exemple de famille $(X^t,f^t)_{t\in\R^*}$ mentionné plus haut, ce qui donne le corollaire \ref{coro-k3}.

Enfin, le chapitre \ref{chap-non-dense} concerne la démonstration du corollaire \ref{coro-non-dense} sur la non-densité et la discrétude de $\Aut(X)$ dans $\Diff(X(\R))$.\\

Dans la troisième et dernière partie, on s'intéresse à l'ensemble de Fatou des automorphismes loxodromiques d'une surface kählérienne compacte, dans le but de répondre partiellement à la question posée plus haut.

Le chapitre~\ref{chap-hyp} contient la preuve du résultat principal de cette troisième partie, à savoir le théorème \ref{theoa-fatou} concernant l'hyperbolicité de l'ensemble de Fatou. On y trouve aussi l'énoncé et la preuve du théorème de Dinh et Sibony sur les courants, qui permet de montrer ce résultat. Les courbes périodiques sont un obstacle à pouvoir déduire directement l'hyperbolicité de ce dernier théorème ; on s'en tire néanmoins en se plaçant sur un modèle minimal singulier où l'on a contracté toutes ces courbes, et grâce au théorème \ref{dilat+per}.

On montre ensuite au chapitre \ref{chap-rec} que les composantes de Fatou récurrentes sont des domaines de rotation au sens de Bedford et Kim (corollaire \ref{theoa-rot}).

Au chapitre \ref{chap-fatou-reel}, on applique ce corollaire pour prouver le théorème~\ref{theoa-fat-reel} sur les restrictions topologiques à ce que $X(\R)$ soit inclus dans l'ensemble de Fatou. Une description plus précise de la dynamique est donnée dans ce cas.

Dans le dernier chapitre, on discute de l'optimalité des restrictions topologiques données par ce dernier théorème, et on donne un exemple où l'ensemble de Fatou contient un tore réel, mais pour une application birationnelle au lieu d'un automorphisme (théorème \ref{theoa-birdiff}).

%%%%%%%%%%%%%%%%%%%%%%%%%%%%%%%%%%
%%%%%%%%%%%%%%%%%%%%%%%%%%%%%%%%%%
% PARTIE I
%%%%%%%%%%%%%%%%%%%%%%%%%%%%%%%%%%%%%%%%%%%%%%%%%%%%%%%%%%%%%%%%%%%%
% PRELIMINAIRES
%%%%%%%%%%%%%%%%%%%%%%%%%%%%%%%%%%%%%%%%%%%%%%%%%%%%%%%%%%%%%%%%%%%%

\part{Préliminaires}\label{part1}

%%%%%%%%%%%%%%%%%%%%%%%%%%%%%%%%%%
% Chapitre
%%%%%%%%%%%%%%%%%%%%%%%%%%%%%%%%%%
% Structure réelle
%%%%%%%%%%%%%%%%%%%%%%%%%%%%%%%%%%

\chapter{Variétés algébriques complexes et réelles}\label{chap-str-reelle}

%%%%%%%%%%%%%%%%%%%%%%%%%%%%%%%%%%
% Variétés algébriques complexes
%%%%%%%%%%%%%%%%%%%%%%%%%%%%%%%%%%

\section{Variétés algébriques complexes}

Dans ce texte, on appelle \emph{variété algébrique complexe} $X$ une sous-variété algébrique irréductible et \textbf{lisse} d'un espace projectif complexe $\P^n_\C$. De manière équivalente\footnote{C'est le théorème de Chow (voir par exemple \cite{griffiths-harris}).}, $X$ est une sous-variété analytique complexe, connexe, compacte et non singulière de $\P^n(\C)$.

% Groupes de diviseurs

\subsection{Groupes de diviseurs}

Par définition, un \emph{diviseur} (de Weil) sur~$X$ est une combinaison linéaire finie, à coefficients entiers, de sous-variétés algébriques irréductibles de codimension $1$, éventuellement singulières. Un tel diviseur est \emph{effectif} lorsque les coefficients sont positifs. On note $(\Div(X),+)$ le groupe abélien des diviseurs.

Lorsque $f:X\dashrightarrow\C$ est une fonction rationnelle non nulle, le diviseur associé à $f$ est
\begin{equation}
{\sf div}(f)=\sum_{Z}{\sf ord}_Z(f)\,Z,
\end{equation}
où la somme est prise sur les sous-variétés algébriques irréductibles de codimension $1$ (seul un nombre fini de termes sont non nuls) : un tel diviseur est dit \emph{principal}. Le quotient de $\Div(X)$ par le sous-groupe des diviseurs principaux est ce que l'on appelle le \emph{groupe de Picard}, noté $\Pic(X)$. Deux diviseurs ayant même classe dans $\Pic(X)$ sont dit \emph{linéairement équivalents}.

Par ailleurs, l'équivalence entre diviseurs de Weil et diviseurs de Cartier permet d'associer à un diviseur $D$ un fibré en droites $\O_X(D)$, défini à isomorphisme près : ceci induit un isomorphisme entre $\Pic(X)$ et le groupe $H^1(X,\O_X^*)$ des classes de fibrés en droites. La suite exacte courte
\begin{equation}
1 \longrightarrow \Z \longrightarrow \O_X \overset{\exp(2\pi i\cdot)}{\longrightarrow} \O_X^* \longrightarrow 1
\end{equation}
induit une suite exacte longue
\begin{equation}
\cdots\to H^1(X;\Z)\to H^1(X,\O_X)\to H^1(X,\O_X^*)\to H^2(X;\Z)\to\cdots,
\end{equation}
qui permet de définir la (première) classe de Chern de $D$ dans $H^2(X;\Z)$, notée~$[D]$. Deux diviseurs qui ont même classe de Chern sont dits \emph{algébrique\-ment équivalents}.

\begin{rema}
Modulo les sous-groupes de torsion de $H^2(X;\Z)$ et $H_{2n-2}(X;\Z)$, la classe de Chern $[D]$ est Poincaré-duale à la classe fondamentale de $D$ dans $H_{2n-2}(X;\Z)$. L'équivalence algébrique est donc caractérisée par l'égalité des classes d'homologie.
\end{rema}

\textbf{Dans la suite, on considère toujours les groupes de (co)homologie modulo torsion.}\\

Le sous-groupe de $H^2(X;\Z)$ formé par les classes de Chern est appelé \emph{groupe de Néron-Severi} de $X$, et noté $\NS(X)$, ou parfois $\NS(X;\Z)$ lorsqu'on veut insiter sur le fait que les coefficients sont entiers. C'est un groupe abélien libre, dont le rang $\rho(X)$ est appelé \emph{nombre de Picard} de $X$. Le théorème des $(1,1)$-classes de Lefschetz (voir \cite[p. 163]{griffiths-harris}) donne l'égalité
\begin{equation}\label{eq-ns-cplx}
\NS(X)=H^2(X;\Z)\cap H^{1,1}(X;\R).
\end{equation}

\begin{rema}
Lorsque le premier nombre de Betti de $X$ est nul, le groupe $\Pic(X)$ s'identifie à $\NS(X)$. Dans le cas général, on a un morphisme surjectif
\begin{equation}
\Pic(X)\to\NS(X)\to 0
\end{equation}
dont le noyau, noté $\Pic^0(X)$, peut être muni d'une structure de variété abélienne complexe, duale à la variété d'Albanese.
\end{rema}

% Cas des surfaces : forme d'intersection

\subsection{Cas des surfaces : forme d'intersection}

Lorsque $X$ est de dimension $2$, le cup-produit des classes de cohomologie induit une forme quadratique non dégénérée sur $H^2(X;\R)$, à valeurs entières sur $H^2(X;\Z)$. Cette forme quadratique $q$ est appelée \emph{forme d'intersection}. L'intersection $D_1\cdot D_2$ entre deux diviseurs $D_1$ et $D_2$ est définie par celle de leurs classes de Chern, qui correspond, par dualité de Poincaré, à l'intersection des classes d'homologie\footnote{Lorsque $D_1$ et $D_2$ se coupent transversalement, il s'agit simplement de compter le nombre de points d'intersection.}. On notera avec un point toutes ces intersections.

Sur l'espace vectoriel complexe $H^2(X;\C)=H^2(X;\R)\otimes_\R\C$, la forme d'intersection se prolonge en une forme hermitienne, toujours appelée forme d'intersection, donnée par
\begin{equation}
q([\omega])=\int_X\omega\wedge\b\omega
\end{equation}
pour toute $2$-forme différentielle $\omega$. La décomposition de Hodge
\begin{equation}
H^2(X;\C) = H^{2,0}(X;\C) \oplus H^{1,1}(X;\C) \oplus H^{0,2}(X;\C)
\end{equation}
est alors orthogonale vis-à-vis de cette forme d'intersection. Celle-ci est définie-positive sur les sous-espaces $H^{2,0}(X;\C)$ et $H^{0,2}(X;\C)$, et de signature $(1,h^{1,1}(X)-1)$ sur $H^{1,1}(X;\C)$, d'après le théorème de l'indice de Hodge. À noter que les sous-espaces complexes $H^{2,0}(X;\C)\oplus H^{0,2}(X;\C)$ et $H^{1,1}(X;\C)$ sont en fait définis sur $\R$ : on note $(H^{2,0}\oplus H^{0,2})(X;\R)$ et $H^{1,1}(X;\R)$ les sous-espaces réels correspondant. La forme d'intersection y est respectivement de signature $(2h^{2,0}(X),0)$ et $(1,h^{1,1}(X)-1)$.

Comme $X$ est projective, il existe des diviseurs qui sont d'auto-intersection strictement positive, donc la forme d'intersection sur le sous-espace $\NS(X;\R):=\NS(X)\otimes\R$ de $H^{1,1}(X;\R)$ est de signature $(1,\rho(X)-1)$. Le cône de positivité $\{q>0\}$ dans $\NS(X;\R)$ possède donc deux composantes connexes : on note $\NS^+(X)$ celle qui contient des classes de diviseurs effectifs.

\begin{figure}[h]
\begin{center}
\scalebox{0.4}{\input{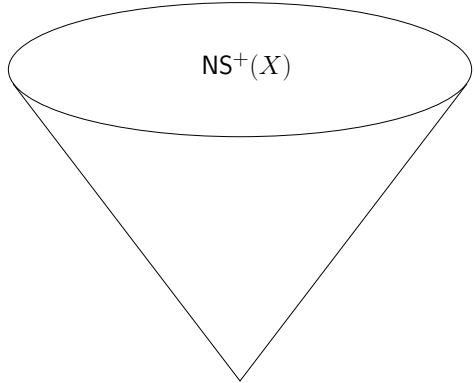}}
\end{center}
\caption{Le cône convexe ouvert $\NS^+(X)$.}
\end{figure}

% Autres notions de positivité

\subsection{Autres notions de positivité} (voir \cite{lazarsfeld})

Un diviseur $D$ est dit \emph{ample} lorsque l'application rationnelle
\begin{equation}
\begin{split}
\phi_{|kD|} : X & \dashrightarrow \,\P H^0(X,\O_X(kD))^*\\
x & \longmapsto \left\{s\in H^0(X,\O_X(kD))\,\big|\,s(x)=0\right\}
\end{split}
\end{equation}
définit un plongement pour un certain $k\in\N^*$ (on dit que $kD$ est \emph{très ample}). Cela équivaut à l'existence d'un plongement projectif de $X$ tel que le diviseur~$kD$ soit une intersection hyperplane. On note $\Amp(X)$ le cône convexe de $\NS(X;\R)$ engendré par les classes de diviseurs amples.

Un $\R$-diviseur\footnote{c'est-à-dire un élément de $\Div(X;\R):=\Div(X)\otimes\R$} $D$ est dit \emph{nef}\footnote{pour numériquement effectif, mais aussi pour \emph{numerically eventually free}} lorsque $D\cdot C\geq 0$ pour toute courbe algébrique $C$ sur $X$, où $D\cdot C$ désigne l'intersection entre la classe de Chern $[D]\in H^2(X;\Z)$ et la classe d'homologie de $C$ dans $H_2(X;\Z)$. C'est une notion plus faible que l'amplitude. Si l'on note $\Nef(X)$ le cône convexe formé par les classes de $\R$-diviseurs nef, un théorème de Kleiman donne alors :
\begin{equation}
\Nef(X)=\b{\Amp(X)}
\quad\text{et}\quad
\Amp(X)=\overset{\circ}{\Nef(X)}.
\end{equation}

Dans le cas où $X$ est une surface, les diviseurs amples sont d'auto-intersection strictement positive, ce qui entraîne l'inclusion
\begin{equation}
\Amp(X)\subset\NS^+(X).
\end{equation}

% Quelques théorèmes classiques sur les surfaces

\subsection{Quelques théorèmes classiques sur les surfaces}  (voir \cite{bpvdv})

On suppose à nouveau que $X$ est de dimension $2$. On note $K_X$ le diviseur canonique sur $X$ (défini à équivalence linéaire près), qui correspond au lieu des zéros moins le lieu des pôles d'une $2$-forme méromorphe sur $X$.

\begin{theo}[formule du genre]\label{theo-genre}
Soit $D$ une courbe irréductible sur~$X$. Son genre arithmétique, défini par $g(D)=h^1(D,\O_D)$, est donné par la formule
\begin{equation}
g(D)=1+\frac{1}{2}D\cdot(D+K_X).
\end{equation}
Dans le cas où $D$ est une courbe lisse, il est égal au genre de $D$ au sens topologique.
\end{theo}

Pour tout diviseur $D$ sur $X$, on note $\O_X(D)$ le faisceau inversible associé au diviseur~$D$, c'est-à-dire le faisceau des fonctions méromorphes $f$ telles que ${\sf div}(f)+D$ est effectif. Pour tout faisceau inversible $\mathcal{F}$, on définit la caractéristique d'Euler de $\mathcal{F}$ par la formule $\chi(\mathcal{F})=\sum_{k=0}^{2}(-1)^kh^k(X,\mathcal{F})$, où~${h^k(X,\mathcal{F})=\dim H^k(X,\mathcal{F})}$. La caractéristique d'Euler du faisceau structural~$\O_X$ est donnée par la formule de Noether :
\begin{equation}
\chi(\O_X)=\frac{\chi(X)+K_X^2}{12},
\end{equation}
où $\chi(X)=\sum_{k=0}^4 (-1)^k\dim H^k(X;\R)$ désigne la caractéristique d'Euler topologique de $X$.

\begin{theo}[formule de Riemann-Roch]
Pour tout diviseur $D$ sur $X$, on a
\begin{equation}
\chi(\O_X(D))=\chi(\O_X)+\frac{1}{2}D\cdot(D-K_X).
\end{equation}
\end{theo}

\begin{theo}[dualité de Serre]
Pour tout diviseur $D$ sur $X$, on a
\begin{equation}
h^2(X,\O_X(D))=h^0(X,\O_X(K_X-D)).
\end{equation}
\end{theo}

\begin{coro}\label{rr+serre}
Soit $D$ un diviseur sur $X$. On suppose que $K_X-D$ n'est linéairement équivalent à aucun diviseur effectif. Alors
\begin{equation}
h^0(X,\O_X(D))\geq \chi(\O_X)+\frac{1}{2}D\cdot(D-K_X).
\end{equation}
\end{coro}

\begin{proof}
On a $h^2(X,\O_X(D))=h^0(X,\O_X(K_X-D))=0$ par hypothèse, donc la formule de Riemann-Roch donne
\begin{equation}
h^0(X,\O_X(D))-h^1(X,\O_X(D))+0 = \chi(\O_X)+\frac{1}{2}D\cdot(D-K_X).
\end{equation}
On en déduit le résultat, car $h^1(X,\O_X(D))\geq 0$.
\end{proof}

%%%%%%%%%%%%%%%%%%%%%%%%%%%%%%%%%%
% Variétés algébriques réelles
%%%%%%%%%%%%%%%%%%%%%%%%%%%%%%%%%%

\section{Structure réelle sur une variété algébrique complexe} 

Le lecteur désireux de plus de détails pourra ici se référer à \cite{silhol}.

% Généralités

\subsection{Généralités}

Par définition, une \emph{variété algébrique réelle} $X_\R$ est la donnée :
\begin{enumerate}
\item d'une variété algébrique complexe $X_\C$ (au sens de la section précédente) ;
\item d'une \emph{structure réelle} sur $X_\C$, c'est-à-dire une involution anti-holomorphe $\sigma_X$ (notée plus simplement $\sigma$ la plupart du temps).
\end{enumerate}

\textbf{On supposera en outre que l'ensemble $X(\R)$ des points fixés par~$\sigma$ est non vide.} Cet ensemble forme alors une variété analytique réelle compacte, non nécessairement connexe, de dimension $\dim_\C(X)$.\\

Par définition, les morphismes entre variétés algébriques réelles sont équivariants vis-à-vis des structures réelles, et les sous-variétés algébriques de~$X_\R$ sont préservées par $\sigma$. Au niveau des notations, on distinguera la variété algébrique réelle $X_\R$ de l'ensemble des points réels $X(\R)$ sous-jacent, ainsi que la variété algébrique complexe $X_\C$ et l'ensemble des points complexes~$X(\C)$.

\begin{rema}
Lorsque $X_\C$ est le lieu des zéros, dans $\P^n_\C$, de polynômes à coefficients réels, le choix naturel pour $\sigma$ est la structure réelle induite par la conjugaison complexe des coordonnées dans $\P^n(\C)$. L'ensemble $X(\R)$ des points réels est alors l'intersection de $X(\C)$ avec $\P^n(\R)$. On dit dans ce cas que $X_\R$ est une sous-variété algébrique de $\P^n_\R$, qui est irréductible et géométriquement lisse (d'après les hypothèses faites sur $X_\C$). En fait, toute variété algébrique réelle est obtenue de cette manière (cf. proposition \ref{sous-var-reelle}).
\end{rema}

% Action de sigma sur les diviseurs et sur la cohomologie

\subsection{Action de $\sigma$ sur les diviseurs et sur la cohomologie}\label{action-sigma}

L'involution anti-holomorphe $\sigma$ agit naturellement :
\begin{enumerate}
\item sur l'ensemble des diviseurs $D$ de $X_\C$ : $D^\sigma$ est par définition l'image de $D$ par $\sigma$. En terme de diviseurs de Cartier, si $D$ a pour équation (localement)~${f=0}$, alors le diviseur $D^\sigma$ est donné par $\sigma_\C\circ f\circ\sigma_X$, où $\sigma_\C$ désigne la conjugaison dans $\C$. Par définition, un diviseur $D$ sur $X_\C$ est réel (on dit aussi que $D$ est un diviseur sur $X_\R$) si $D=D^\sigma$.
\item sur l'espace $H^2(X_\C;\R)$, grâce à l'action de $\sigma$ sur les $2$-formes différentielles. Cette action $\sigma^*$ préserve la structure entière $H^2(X_\C;\Z)$ et renverse la structure de Hodge, \emph{i.e.} $\sigma^*H^{p,q}(X;\C)=H^{q,p}(X;\C)$. En particulier, le sous-groupe~${\NS(X_\C;\Z)=H^2(X_\C;\Z)\cap H^{1,1}(X_\C;\R)}$ est préservé par~$\sigma^*$.
\end{enumerate}

En revanche, l'action de $\sigma$ sur les diviseurs ne correpsond pas à l'action sur les classes de Chern. Plus exactement, on a la formule suivante (voir \cite[\textsection I.4]{silhol}) :
\begin{equation}\label{pb-sign}
[D^\sigma] = -\sigma^*[D].
\end{equation}
Par conséquent, l'ensemble des classes de Chern de diviseurs réels est égal au sous-groupe
\begin{equation}\label{eq-ns-reel}
\NS(X_\R):=\left\{\theta\in\NS(X_\C)\,|\,\sigma^*\theta=-\theta\right\}.
\end{equation}
Ce groupe est appelé \emph{groupe de Néron-Severi} de $X_\R$, et son rang $\rho(X_\R)$ est le \emph{nombre de Picard} de $X_\R$.

% Positivité

\subsection{Positivité}

Les notions de positivité introduites précédemment sont préservées par l'action de $\sigma$ sur les diviseurs. On désigne par $\Amp(X_\R)$ et $\Nef(X_\R)$ l'intersection des cônes convexes $\Amp(X_\C)$ et $\Nef(X_\C)$ avec le sous-espace $\NS(X_\R;\R)\subset\NS(X_\C;\R)$. À noter que $\Amp(X_\R)$ est non vide, car si $D$ est un diviseur ample sur $X_\C$, alors $D+D^\sigma$ est un diviseur ample sur $X_\R$.

\begin{prop}\label{sous-var-reelle}
Toute variété algébrique réelle est isomorphe à une sous-variété algébrique de $\P^n_\R$.
\end{prop}

\begin{proof}[Esquisse de démonstration]
Soit $D$ un diviseur très ample sur $X_\R$. Le fait que $D^\sigma=D$ permet de munir le fibré $\O_X(D)$ d'une structure réelle naturelle. L'espace projectif complexe $\P^n_\C=\P H^0(X,\O_X(D))^*$ vient alors avec une conjugaison complexe $\sigma_{\P^n}$, et le plongement $\Phi_{|D|}:X_\C\hookrightarrow\P^n_\C$ est équivariant vis-à-vis de $\sigma_X$ et $\sigma_{\P^n}$. Il définit donc un isomorphisme entre $X_\R$ et son image, qui est une sous-variété algébrique de $\P^n_\R$, grâce au théorème de Chow.
\end{proof}

Dans le cas où $X$ est une surface, la restriction de la forme d'intersection au sous-espace $\NS(X_\R;\R)$ est de signature $(1,\rho(X_\R)-1)$, car un diviseur ample sur $X_\R$ est d'auto-intersection strictemement positive. L'intersection 
\begin{equation}
\NS^+(X_\R):=\NS^+(X_\C)\cap\NS(X_\R;\R)
\end{equation}
est donc un cône convexe ouvert dans $\NS(X_\R;\R)$, et on a le même dessin que pour $\NS^+(X_\C)$.

%%%%%%%%%%%%%%%%%%%%%%%%%%%%%%%%%%
% 
%%%%%%%%%%%%%%%%%%%%%%%%%%%%%%%%%%

\section{Variétés kählériennes réelles}

Toutes ces notions s'étendent au cas où $X_\C$ est une variété complexe kählérienne, \textbf{compacte} et connexe, munie d'une structure réelle $\sigma$ : on dit alors (comme par exemple dans \cite{mangolte-k3}) que $X_\R$ est une \emph{variété kählérienne réelle}. Il faut alors remplacer la notion de sous-variété algébrique complexe par celle de sous-variété analytique complexe fermée.

Cependant, lorsque $X_\C$ est une surface qui n'est pas algébrique, la forme d'intersection est négative (mais pas nécessairement définie-négative) sur $\NS(X_\R;\R)$ et $\NS(X_\C;\R)$, et les cônes $\NS^+(X_\R)$ et $\NS^+(X_\C)$ sont vides. Il est donc préférable d'étudier l'espace $H^{1,1}(X_\C;\R)$ tout entier plutôt que les groupes de Néron-Severi.\\

Par analogie avec l'espace $\NS(X_\R;\R)$ (cf. $(\ref{eq-ns-reel})$), on pose
\begin{equation}
H^{1,1}(X_\R;\R):=\{\theta\in H^{1,1}(X_\C;\R)\,|\,\sigma^*\theta=-\theta\}.\footnote{Attention, cela n'a rien à voir avec la cohomologie de $X(\R)$.}
\end{equation}
Avec cette notation, l'équation $(\ref{eq-ns-reel})$ se réécrit (comparer avec $(\ref{eq-ns-cplx})$) :
\begin{equation}
\NS(X_\R) = H^{1,1}(X_\R;\R)\cap H^2(X_\C;\Z).
\end{equation}

La notion correspondant au cône ample est celle du cône de Kähler, dont les éléments sont les classes des formes de Kähler dans $H^{1,1}(X_\C;\R)$, et que l'on note $\Kah(X_\C)$. On définit aussi le cône de Kähler réel 
\begin{equation}
\Kah(X_\R):=\Kah(X_\C)\cap H^{1,1}(X_\R;\R).
\end{equation}
Pour $\K=\R$ ou $\C$, le théorème de plongement de Kodaira donne l'égalité
\begin{equation}
\Amp(X_\K)=\Kah(X_\K)\cap\NS(X_\K;\R),
\end{equation}
mais dans le cas où $X$ n'est pas algébrique, cet ensemble est vide.

Lorsque $X$ est une surface, la forme d'intersection $q$ restreinte aux sous-espaces $H^{1,1}(X_\C;\R)$ et $H^{1,1}(X_\R;\R)$ est non dégénérée de signature $(1,*)$\footnote{Pour $H^{1,1}(X_\R;\R)$, cela provient du fait que si $\kappa$ est une forme de Kähler quelconque sur $X_\C$, alors $\kappa-\sigma^*\kappa$ est également une forme de Kähler, et sa classe est dans $H^{1,1}(X_\R;\R)$.}. On peut donc définir, par analogie avec $\NS^+(X_\C)\subset\NS(X_\C;\R)$ et $\NS^+(X_\R)\subset\NS(X_\R;\R)$, des cônes convexes ouverts 
\begin{equation}
H^+(X_\C)\subset H^{1,1}(X_\C;\R)
\quad\text{et}\quad
H^+(X_\R)\subset\NS(X_\R;\R),
\end{equation}
en prenant la composante connexe de $\{q>0\}$ qui contient les classes de formes de Kähler.

%%%%%%%%%%%%%%%%%%%%%%%%%%%%%%%%%%
% Chapitre
%%%%%%%%%%%%%%%%%%%%%%%%%%%%%%%%%%
% Automorphismes
%%%%%%%%%%%%%%%%%%%%%%%%%%%%%%%%%%

\chapter{Automorphismes d'une surface}\label{chap-auto}

%%%%%%%%%%%%%%%%%%%%%%%%%%%%%%%%%%
% Généralités
%%%%%%%%%%%%%%%%%%%%%%%%%%%%%%%%%%

\section{Généralités sur les automorphismes d'une variété complexe}

Un \emph{automorphisme} d'une variété complexe compacte kählérienne $X$ est une application biholomorphe $f:X\to X$. Lorsque $X$ est algébrique, le principe GAGA \cite{gaga} garantit que $f$ est un automorphisme si et seulement si $f$ est un morphisme birégulier de $X$ dans elle-même. On note $\Aut(X)$ le groupe des automorphismes de $X$. Un théorème de Bochner et Montgomery\footnote{Ce théorème est facile à démontrer dans le cas où $X$ est algébrique.} \cite{bm1,bm2} affirme que, muni de la topologie de la convergence uniforme, $\Aut(X)$ est un groupe de Lie complexe, d'algèbre de Lie $H^0(X,{\rm T}_X)$.

% Action sur les diviseurs

\subsection{Action des automorphismes sur les diviseurs}

Le groupe $\Aut(X)$ agit de deux manières sur l'ensemble $\Div(X)$ des diviseurs de $X$ :

\begin{enumerate}
\item Si $f\in\Aut(X)$ et $D\in\Div(X)$, on peut considérer l'image directe de $D$ par $f$, notée $f_*D$. Ceci définit une action à gauche.
\item On peut aussi considérer $D$ comme un diviseur de Cartier défini localement par des équations $h_i=0$. Le diviseur $f^*D$ est alors le diviseur donné par les équations $h_i\circ f=0$. Ceci définit une action à droite.
\end{enumerate}
Ces deux actions sont en fait inverses l'une de l'autre : on a
\begin{equation}
f^*D=f_*^{-1}D \quad \forall D\in\Div(X).
\end{equation}

L'action de $\Aut(X)$ sur $\Div(X)$ est compatible avec la loi de groupe abélien, ainsi qu'avec les relations d'équivalence linéaire et d'équivalence algébrique. Elle définit donc une action par automorphismes sur les groupes $\Pic(X)$ et $\NS(X)$.

% Action sur la cohomologie

\subsection{Action des automorphismes sur la cohomologie}

Le groupe $\Aut(X)$ agit linéairement (à droite) sur la cohomologie de De Rham $H^k(X;\R)$, en préservant à la fois la structure entière $H^k(X;\Z)$, et la structure de Hodge
\begin{equation}
H^k(X;\C)=\bigoplus_{p+q=k}H^{p,q}(X;\C).
\end{equation}
Pour tout $f\in\Aut(X)$, on note $f^*$ les différents automorphismes induits par cette action, et on note aussi $f_*=(f^{-1})^*$ (ce qui définit une action à gauche). À noter que l'action sur le sous-groupe $\NS(X)=H^{1,1}(X;\C)\cap H^2(X;\Z)$ coïncide avec celle définie au paragraphe précédent. Autrement dit, l'action au niveau des classes de Chern est compatible avec celle au niveau des diviseurs (notez la différence avec le \textsection \ref{action-sigma}).

% Cas des surfaces

\subsection{Cas des surfaces}

Lorsque $X$ est une surface, l'action de ${f\in\Aut(X)}$ sur la cohomologie de $X$ est dominée par celle sur la cohomologie de rang $2$, au niveau du rayon spectral. En effet :
\begin{enumerate}
\item L'action de $f$ sur $H^0(X;\Z)$ et sur $H^4(X;\Z)$ est triviale.
\item Si $\omega$ est une $1$-forme holomorphe non nulle telle que $f^*\omega=\lambda\,\omega$, alors~${\omega\wedge\b\omega}$ définit une classe non nulle dans $H^2(X;\R)$ qui est un vecteur propre associé à $|\lambda|^2$. Comme $H^1(X;\C)=H^{1,0}(X;\C)\oplus\b{H^{1,0}(X;\C)}$, les valeurs propres de $f^*$ sur $H^1(X;\C)$ sont inférieures, en module, à celles sur $H^2(X;\C)$.
\item Par dualité, l'action sur $H^3(X;\C)$ est dominée par celle sur $H^2(X;\C)$.
\end{enumerate}
Le rayon spectral de $f^*$ sur $H^*(X;\C)$ est donc atteint sur $H^2(X;\C)$.

De plus, la forme d'intersection $q$ sur $H^2(X;\Z)$ est préservée par l'action de $\Aut(X)$. On a donc une action par isométries de $f\in\Aut(X)$ sur chacun des sous-espaces supplémentaires $(H^{2,0}\oplus H^{0,2})(X;\R)$ et $H^{1,1}(X;\R)$. Comme le premier est défini-positif pour $q$, la restriction de ${f^*}$ à ce sous-espace appartient au groupe compact 
\begin{equation}
{\sf O}\left((H^{2,0}\oplus H^{0,2})(X;\R),q\right)\simeq {\sf O}\left(2h^{2,0}(X),\R\right).
\end{equation}
En particulier les valeurs propres complexes de $f^*$ sur $(H^{2,0}\oplus H^{0,2})(X;\R)$ sont de module~$1$.

\begin{rema}
Dans le cas où $X$ est une surface algébrique, $f^*$ est même d'ordre fini sur ce sous-espace $(H^{2,0}\oplus H^{0,2})(X;\R)$. En effet, $f^*$ préserve~$\NS(X)$, et donc préserve aussi son supplémentaire orthogonal dans~${H^2(X;\R)}$, que l'on note $E$. Comme
\begin{equation}
E=(H^{2,0}\oplus H^{0,2})(X;\R) \overset{\perp}\oplus \left(E\cap H^{1,1}(X;\R)\right),
\end{equation}
et que les deux sous-espaces de la somme directe sont à la fois stables par $f^*$ et définis pour la forme d'intersection, on en déduit que ${f^*}_{|E}$ est dans un sous-groupe compact de $\GL(E)$. Comme par ailleurs $E$ est un sous-espace rationnel de $H^2(X;\R)$, la transformation ${f^*}_{|E}$ est donné par une matrice à coefficients entiers : elle est donc d'ordre fini.
\end{rema}

Tout ceci montre que l'action de $\Aut(X)$ sur la cohomologie est dominée par l'action sur le sous-espace $H^{1,1}(X;\R)$.

\begin{defi}
Pour $f\in\Aut(X)$, on appelle \emph{degré dynamique} de $f$, et on note $\lambda(f)$, le rayon spectral de $f^*$ sur la cohomologie de $X$ ou, ce qui est la même chose, sur $H^{1,1}(X;\R)$.
\end{defi}

Dans la suite, on s'intéresse donc à l'action de $f$ sur $H^{1,1}(X;\R)$, qui induit une isométrie d'un espace hyperbolique. On commence par rappeler la classification des isométries d'un tel espace.

%%%%%%%%%%%%%%%%%%%%%%%%%%%%%%%%%%
% Isométries d'un espace hyperbolique
%%%%%%%%%%%%%%%%%%%%%%%%%%%%%%%%%%

\section{Isométries d'un espace hyperbolique}

Pour plus de détails, on renvoie à \cite{benedetti}. Soit $E$ un espace vectoriel réel de dimension $n+1\geq 2$, muni d'une forme quadratique $q$ de signature $(1,n)$. Pour fixer les idées, on peut prendre la forme quadratique standard sur $\R^{n+1}$ :
\begin{equation}
q(x)=x_0^2-x_1^2-\cdots-x_n^2.\footnote{On parle dans ce cas de l'espace de Minkowski.}
\end{equation}
Notons que les sous-espaces totalement isotropes maximaux sont de dimension~$1$ pour une telle forme quadratique.

La quadrique $\{q=1\}$ est un hyperboloïde à deux nappes. On note $\H^n$ l'une d'entre elles. Le plan tangent à $\H^n$ en l'un de ses points $v$ a pour direction l'orthogonal $v^\perp$, et hérite donc d'une forme quadratique définie négative par restriction de $q$ à $v^\perp$. En prenant l'opposé de ces formes quadratiques, on définit ainsi une métrique riemanienne sur $\H^n$, de courbure négative constante égale à $-1$. C'est la métrique dont la distance associée est donnée par la formule
\begin{equation}
\dist(v_1,v_2)={\rm argch}(v_1\cdot v_2),
\end{equation}
où $(v_1,v_2)\mapsto v_1\cdot v_2$ désigne la forme bilinéaire associée à $q$. Il s'agit de l'un des modèles de l'espace hyperbolique de dimension $n$. Topologiquement, $\H^n$ s'identifie à la projection de $\H^n$ dans $\P E$, qui est homéomorphe à une boule ouverte, dont le bord $\partial\H^n$ s'identifie avec la projection du cône isotrope ${\{q=0\}}$ dans $\P E$.

Les transformations linéaires de $E$ qui préservent $q$ fixent l'hyperboloïde $\{q=1\}=\H^n\cup-\H^n$. Le sous-groupe d'indice $2$ de ${\sf O}(E,q)$ formé par les éléments qui préservent chacune des deux nappes s'identifie avec le groupe des isométries de l'espace hyperbolique $\H^n$, que l'on note $\Isom(\H^n)$.

Une isométrie $\phi\in\Isom(\H^n)$ induit une transformation continue de la boule fermée $\b{\H^n}=\H^n\cup\partial\H^n$, et le théorème de Brouwer implique l'existence d'au moins un point fixe $v\in\b{\H^n}$, qui correspond à une direction propre pour $\phi$ dans le cône ${\{q\geq 0\}}$. On distingue trois types d'isométries selon la nature des points fixes. Commençons par énoncer un résultat préliminaire bien connu sur les valeurs propres de~$\phi$ :

\begin{lemm}\label{vp-reelle}
Soit $\phi\in{\sf O}(E,q)$. 
\begin{enumerate}
\item
Si $\lambda$ et $\mu$ sont deux valeurs propres réelles telles que $\lambda\mu\neq 1$, alors les sous-espaces propres $\ker(\phi-\lambda\id)$ et $\ker(\phi-\mu\id)$ sont orthogonaux.
\item
Toute valeur propre complexe $\lambda$ de module différent de $1$ est une valeur propre réelle, et le sous-espace propre associé est une droite contenue dans le cône isotrope $\{q=0\}$. 
\end{enumerate}
\end{lemm}

\begin{proof}
\begin{enumerate}
\item
Soient $v$ et $w$ deux vecteurs propres associés respectivement aux valeurs propres $\lambda$ et $\mu$. On a alors
\begin{equation}
v\cdot w = \phi(v)\cdot\phi(w) = \lambda\mu\,(v\cdot w),
\end{equation}
ce qui montre que $v$ et $w$ sont orthogonaux, car $\lambda\mu\neq 1$.
\item
On complexifie l'espace vectioriel $E$ en $E_\C=E\otimes_\R\C$, et on prolonge $q$ en une forme hermitienne sur $E_\C$,  toujours notée $q$. L'endomorphisme $\phi$ induit une application $\C$-linéaire $\phi\otimes_\R\id_\C$ sur $E_\C$, toujours notée~$\phi$, qui préserve la forme hermitienne $q$.

Soit $v\in E_\C$ un vecteur propre complexe pour la valeur propre $\lambda$. On a les égalités
\begin{gather}
q(v)=q(\phi(v))=q(\lambda v)=|\lambda|^2 q(v)\label{vp-iso}\\
v\cdot \b v = \phi(v)\cdot \phi(\b v) = \lambda v \cdot \b{\lambda v}  = \lambda^2 (v\cdot\b v).
\end{gather}
Comme $|\lambda|\neq 1$, cela implique $q(v)=0$ et $v\cdot \b v=0$. On en déduit que $q$ est nulle sur le sous-espace complexe engendré par $v$ et $\b v$. Or les sous-espaces totalement isotropes maximaux sont de dimension $1$, donc $\rg_\C(v,\b v)=1$. Ainsi $w:=\frac{v+\b v}{2}$ est un vecteur propre réel pour la valeur propre $\lambda$, ce qui montre que $\lambda\in\R$.

Par ailleurs, le sous-espace propre associé à $\lambda$ est un sous-espace totalement isotrope, en vertu de $(\ref{vp-iso})$. Il est donc de dimension $1$.
\end{enumerate}
\end{proof}

Le lemme \ref{vp-reelle} implique la trichotomie suivante pour les points fixes de $\phi$ sur~$\b{\H^n}$ :

\begin{prop}
Soit $\phi\in\Isom(\H^n)$. On a trois cas disjoints :
\begin{enumerate}
\item \textbf{\emph{Type elliptique :}}
$\phi$ possède un point fixe $v$ dans $\H^n$. L'application linéaire $\phi$ est semi-simple\footnote{On rappelle qu'un endomorphisme est semi-simple si sa matrice est $\C$-diagonalisable. Ceci équivaut au fait que tout sous-espace stable admet un supplémentaire stable (cf. \cite{bourbaki-algebre}).}, et ses valeurs propres complexes sont toutes de module $1$.
\item \textbf{\emph{Type parabolique :}}
$\phi$ ne possède pas de point fixe dans $\H^n$, mais possède un unique point fixe sur $\partial\H^n$. Un tel point fixe correspond alors à une droite~$D$ de points fixes pour $\phi$ dans le cône isotrope $\{q=0\}$. L'application linéaire $\phi$ n'est pas semi-simple, et ses valeurs propres complexes sont toutes de module $1$.
\item \textbf{\emph{Type loxodromique :}}
$\phi$ possède exactement deux points fixes sur~$\partial\H^n$. Ces points fixes correspondent à deux droites propres $D^+$ et $D^-$ dans le cône isotrope, associées à des valeurs propres $\lambda$ et~$\lambda^{-1}$, avec $\lambda>1$. L'application linéaire $\phi$ est semi-simple, avec $\lambda$ et $\lambda^{-1}$ valeurs propres simples et toutes les autres valeurs propres complexes de module $1$.
\end{enumerate}
\end{prop}

\begin{proof}
\begin{enumerate}
\item
Soit $v$ un point fixe de $\phi$ dans $\H^n$. La forme d'intersection est définie-négative sur l'orthogonal de $\R v$, donc $\phi_{|v^\perp}$ est semi-simple avec des valeurs propres de module $1$. Comme $E=\R v\oplus v^\perp$ et $\phi(v)=v$, on en déduit la même propriété pour $\phi$.
\item
Soit $D=\R v$ l'unique droite propre dans le cône isotrope. Comme $\phi$ préserve $\H^n$, la valeur propre associée est positive. Si celle-ci est différente de $1$, le fait que $\det(\phi)^2=1$ ($\phi$ préserve une forme bilinéaire symétrique non dégénérée) implique l'existence d'une deuxième valeur propre de module différent de~$1$, donc (d'après le lemme \ref{vp-reelle}) une deuxième droite propre dans le cône isotrope : ceci contredit l'unicité. La droite $D$ est donc une droite de points fixes.

Supposons que $\phi$ soit semi-simple. L'hyperplan $D^\perp$ possède alors un supplémentaire stable $\R w$, avec $\phi(w)=\lambda w$. Comme $v\cdot w=\phi(v)\cdot\phi(w)=\lambda\,v\cdot w\neq 0$, on a $\lambda=1$ et $\phi(w)=w$. Une combinaison linéaire de $v$ et $w$ est alors un point fixe dans $\H^n$ : contradiction.
\item
Supposons que l'on ne soit dans aucun des deux premiers cas. Alors~$\phi$ possède au moins deux droites propres $D^+$ et $D^-$ dans le cône isotrope, associées à des valeurs propres positives $\lambda$ et $\mu$, avec par exemple $\mu\leq \lambda$. On ne peut pas avoir $\lambda=\mu=1$, car sinon on aurait un point fixe dans $\H^n$ par combinaison linéaire. Or $D^+$ et $D^-$ ne sont pas orthogonales (sinon $D^+\oplus D^-$ serait un plan isotrope). On en déduit d'après le lemme \ref{vp-reelle} que $\mu=1/\lambda$, et en particulier~${\lambda>1}$. Par ailleurs, il ne peut pas y avoir de troisième droite propre dans le cône isotrope, car d'après le même raisonnement, celle-ci serait associée à la fois à la valeur propre $1/\lambda$ et $1/\mu=\lambda$.

La forme d'intersection est définie-négative sur le supplémentaire orthogonal~$S$ de $D^+\oplus D^-$, donc $\phi_{|S}$ est semi-simple et ses valeurs propres complexes sont de module $1$. Comme $E=D^+\oplus D^-\oplus S$, le résultat s'en suit.
\end{enumerate}
\end{proof}

\begin{rema}\label{croissance}
Le type de $\phi$ est caractérisé par la croissance des normes~$\|\phi^k\|$, pour un choix quelconque de norme sur $\End(E)$ :
\begin{enumerate}
\item $\phi$ est elliptique si et seulement si $\|\phi^k\|$ est bornée.
\item $\phi$ est parabolique si et seulement si $\|\phi^k\|$ a une croissance polynomiale (non bornée).
\item $\phi$ est loxodromique si et seulement si $\|\phi^k\|$ croît exponentiellement. Plus précisément, $\|\phi^k\|$ croît comme $\lambda^k$, où $\lambda$ est l'unique valeur propre supérieure à $1$ de $\phi$.
\end{enumerate}
Le rayon spectral $\lim_{k\to+\infty}\|\phi^k\|^{1/k}$ vaut $1$ dans les deux premiers cas, et $\lambda>1$ dans le dernier cas. Remarquons aussi que $\phi$ et $\phi^{-1}$ ont même type et même rayon spectral.
\end{rema}

Dans le cas parabolique, la croissance est en fait quadratique, comme le montre la proposition suivante\footnote{On trouve dans l'appendice de \cite{diller-favre} une preuve de ce fait dans le cas particulier où~$\phi$ est induite par un automorphisme d'une surface complexe (voir le paragraphe suivant), et préserve donc une structure entière. Dans le cas général, c'est aussi bien connu, mais difficile à \og localiser\fg~dans la littérature.} :

\begin{prop}\label{car-D}
Soit $\phi$ une isométrie de type parabolique de $\H^n$, et soit~$D$ la droite fixe de $\phi$ dans le cône isotrope. Il existe un sous-espace vectoriel~$F$ de dimension $3$, stable par $\phi$ et contenant $D$, tel que :
\begin{enumerate}
\item $E=F\oplus F^\perp$ ;
\item $q_{|F^\perp}$ est définie-négative, donc $\phi_{|F^\perp}$ est semi-simple avec des valeurs propres de module $1$ ;
\item $q_{|F}$ a pour signature $(1,2)$ ;
\item il existe une base $(v_0,v_1,v_2)$ de $F$, avec $v_0\in D$, $v_1\in D^\perp$ et $v_2\in\H^n$, dans laquelle la matrice de $\phi_{|F}$ est un bloc de Jordan de taille $2$
\begin{equation}
\begin{pmatrix}1&1&0\\0&1&1\\0&0&1\end{pmatrix}.
\end{equation}
\global\edef\savecount{\arabic{enumi}}
\end{enumerate}
En particulier, on obtient les deux faits suivants :
\begin{enumerate}
\setcounter{enumi}{\savecount}
\item La droite $D$ admet la caractérisation suivante, qui est indépendante de la forme quadratique $q$ :
\begin{equation}
D=\ker(\phi-\id)\cap{\rm im}(\phi-\id).
\end{equation}
\item La croissance de $\|\phi^k\|$ est quadratique.
\end{enumerate}
\end{prop}

\begin{proof}
Le polynôme caractéristique de $\phi$ s'écrit $(x-1)^mQ$, où~${Q\in\R[x]}$ n'admet pas $1$ pour racine. Le théorème de décomposition des noyaux donne la décomposition en sous-espaces stables $E=E_1\oplus E_Q$, où~${E_1=\ker\left((\phi-\id)^m\right)}$ et $E_Q=\ker\left(Q(\phi)\right)$. Remarquons que $D\subset E_1$. De plus, on a l'inclusion $E_Q\subset D^\perp$ : en effet, si $v\in D$ et $w\in E_Q$, on a
\begin{equation}
0 = v \cdot Q(\phi)(w) = Q(\phi)(v) \cdot w = Q(1) \, (v \cdot w),
\end{equation}
donc $v\cdot w=0$. On en déduit que $E_1\nsubseteq D^\perp$, donc la forme quadratique $q_{|E_1}$ est non dégénérée et a pour signature $(1,\dim(E_1)-1)$. Dans la suite, on se retreint à ce sous-espace caractéristique $E_1$, sur lequel la seule valeur propre complexe de $\phi$ est $1$.

Soit $v_2\in \H^n\cap E_1$, et soit $v_1=\phi(v_2)-v_2$. Ce vecteur $v_1$ est dans l'orthogonal de la droite $D$ dans $E_1$, que l'on note $D^\perp$, et il est non nul, car sinon $\phi$ serait de type elliptique.

Soit $\b\phi$ l'endomorphisme induit par $\phi$ sur $D^\perp/D$. Cet endomorphisme préserve la forme quadratique induite, qui est définie-négative (car $q_{|D^\perp}$ est négative de noyau $D$). Ainsi $\b\phi$ est semi-simple, avec pour seule valeur propre complexe~$1$ (car $D^\perp\subset E_1$ par hypothèse), donc $\b\phi=\id$. On en déduit qu'il existe un vecteur~$v_0\in D$ tel que~${\phi(v_1)=v_1+v_0}$.

Supposons que $v_0=0$. Le plan $P=\Vect(v_1,v_2)$ est alors stable par $\phi$. Or ce plan est de signature $(1,1)$, donc il contient deux droites isotropes, qui sont préservées ou échangées par $\phi$. Ainsi,~$\phi_{|P}$ admet une matrice de la forme
\begin{equation}
\begin{pmatrix}\lambda&0\\0&\mu\end{pmatrix}
\quad\text{ou}\quad
\begin{pmatrix}0&\lambda\\\mu&0\end{pmatrix},
\end{equation}
et comme de plus sa seule valeur propre complexe est $1$ (car $P\subset E_1$), on a~${\phi_{|P}=\id_P}$. Ceci contredit le fait que~${\phi(v_2)\neq v_2}$.

On en déduit que $v_0\in D\backslash\{0\}$. Les vecteurs $v_0$, $v_1$ et $v_2$ forment donc une base d'un sous-espace stable~$F\subset E_1$ de dimension $3$, sur lequel $\phi$ a pour matrice
\begin{equation}
\begin{pmatrix}1&1&0\\0&1&1\\0&0&1\end{pmatrix}.
\end{equation}
En restriction à~$F$,~$q$ est non dégénérée et a pour signature $(1,2)$. L'orthogonal de $F$ dans $E$ (et pas dans $E_1$) est alors un supplémentaire stable, sur lequel~$q$ est définie-négative. Ceci achève la démonstration.
\end{proof}

La proposition suivante est triviale, mais nous servira à de nombreuses reprises.

\begin{prop}\label{meme-type}
Soit $F$ un sous-espace de dimension $m+1$ de $E$ sur lequel la restriction de $q$ est de signature $(1,m)$, et soit $\phi$ une isométrie de $\H^n$ qui préserve le sous-espace $F$. Alors $\phi$ définit par restriction une isométrie $\phi_F$ de l'espace hyperbolique $\H^m=F\cap\H^n$, qui est de même type que $\phi$.
\end{prop}

\begin{proof}
Il suffit de remarquer que la suite $\|\phi^k\|$ est bornée sur le supplémentaire orthogonal $F^\perp$ de $F$, car $q$ y est définie-négative. La croissance de $\|\phi^k\|$ ne dépend donc que de la restriction à $F$.
\end{proof}

Dans les cas parabolique et loxodromique, on en déduit que les droites propres dans le cône isotrope sont contenues dans le sous-espace $F$.

%%%%%%%%%%%%%%%%%%%%%%%%%%%%%%%%%%
% Classification homologique des automorphismes
%%%%%%%%%%%%%%%%%%%%%%%%%%%%%%%%%%

\section{Classification homologique des automorphismes}

Soit $X$ une surface complexe compacte kählérienne. On applique les résultats de la section précédente en prenant pour $E$ l'espace $H^{1,1}(X;\R)$ muni de la forme d'intersection. On note $\H_X$ l'espace hyperbolique de dimension $h^{1,1}(X)-1$ obtenu en prenant celle des deux nappes de l'hyperboloïde~${\{q=1\}}$ qui est contenue dans $H^+(X)$ (\emph{i.e.} qui contient une classe de Kähler). Le groupe $\Aut(X)$ agit par isométries sur $H^{1,1}(X;\R)$, en préservant le cône de Kähler donc $H^+(X)$ : ses éléments induisent ainsi des isométries de l'espace hyperbolique $\H_X$.

\begin{defi}[\cite{cantat-k3}]
Soit $f\in\Aut(X)$. On dit que $f$ est un automorphisme de type \emph{elliptique}, \emph{parabolique} ou \emph{loxodromique}\footnote{On pourrait dire aussi \emph{de type hyperbolique}.} selon que $f^*$ définit une isométrie de type elliptique, parabolique ou loxodromique de $\H_X$.
\end{defi}

Le type de $f$ est aussi le type de $f_*=(f^*)^{-1}$, en vertu de la remarque \ref{croissance}. Le degré dynamique $\lambda(f)$ est le rayon spectral de $f^*$ ou $f_*$ agissant sur l'un des espaces $H^{1,1}(X;\R)$, $H^2(X;\R)$ ou $H^*(X;\R)$. Il est égal à~$1$ lorsque~$f$ est de type elliptique ou parabolique, et strictement supérieur à $1$ lorsque $f$ est de type loxodromique.

% Type elliptique

\subsection{Type elliptique}

Soit $f$ un automorphisme de type elliptique de~$X$, et soit $\theta\in\H_X$ un vecteur propre de $f^*$. Notons $F$ l'orthogonal du vecteur $\theta$ dans $H^{1,1}(X;\R)$. Puisque $f^*$ préserve à la fois la structure entière $H^2(X;\Z)$ et la décomposition orthogonale
\begin{equation}
H^2(X;\R)=\R \theta\oplus F \oplus (H^{2,0}\oplus H^{0,2})(X;\R),
\end{equation}
où la forme d'intersection est définie sur chacun des trois sous-espaces, $f^*$ est alors d'ordre fini. Un théorème de Lieberman et Fujiki permet d'en déduire :

\begin{theo}[\cite{lieberman, fujiki, cantat-these}]
Soit $f$ un automorphisme d'une surface complexe compacte kählérienne $X$. Alors $f$ est de type elliptique si et seulement s'il existe un itéré de $f$ qui est isotope à l'identité\footnote{En fait, $f$ est dans la composante connexe de l'identité du groupe de Lie $\Aut(X)$.} sur~$X$.
\end{theo}

% Type parabolique

\subsection{Type parabolique}

Soit $f$ un automorphisme de type parabolique de $X$. D'après la proposition \ref{car-D}, l'unique droite propre $D$ de $f^*$ contenue dans le cône isotrope est donnée par
\begin{equation}
D=\ker(f^*-\id)\cap{\rm im}(f^*-\id).
\end{equation}
Cette équation est vraie \emph{a priori} dans $H^{1,1}(X;\R)$, mais aussi dans $H^2(X;\R)$, car $f^*$ est semi-simple sur l'espace $(H^{2,0}\oplus H^{0,2})(X;\R)$. Comme $f^*-\id$ préserve la structure entière, on en déduit que la droite $D$ intersecte non trivialement~${H^2(X;\Z)}$.

Il existe donc un vecteur $\theta\in\NS(X;\Z)$ tel que $\theta^2=0$ et $f^*\theta=\theta$. On peut alors montrer qu'un multiple de $\theta$ est la classe d'une unique fibration elliptique~${\pi:X\to B}$ qui est préservée par $f$, et on en déduit la caractérisation suivante :

\begin{theo}[\cite{gizatullin,cantat-these}]\label{theo-gizatullin}
Soit $f$ un automorphisme d'une surface complexe compacte kählérienne $X$. Alors $f$ est de type parabolique si et seulement si $f$ préserve une fibration elliptique et aucun itéré de $f$ n'est isotope à l'identité.
\end{theo}

\begin{exem}\label{ex-para}
On reprend l'exemple vu dans l'introduction (voir aussi \cite{cantat-k3}). Soit $X$ une surface lisse de degré $(2,2,2)$ dans $\P^1\times\P^1\times\P^1$, donnée par l'annulation d'un polynôme~$P(x_1,x_2,x_3)$ de degré $2$ en chaque variable. On note $s_1$,~$s_2$ et $s_3$ les trois involutions canoniques de~$X$. Par exemple, $s_1$ est l'involution~${(x_1,x_2,x_3)\mapsto(x'_1,x_2,x_3)}$, où $x_1$ et $x'_1$ sont les deux racines du polynôme~${P(x,x_2,x_3)\in\C[x]}$. La composée $f=s_1\circ s_2$ définit un automorphisme de type parabolique sur $X$, qui préserve la fibration elliptique
\begin{equation}
\begin{split}
\pi_3:X & \rightarrow \P^1\\
x & \mapsto x_3.
\end{split}
\end{equation}

Pour montrer que $f$ est effectivement de type parabolique, notons $F_i$ la classe dans $\NS(X)$ de la fibre de $\pi_i:x\mapsto x_i$ (pour $i\in\{1,2,3\}$). Le sous-espace engendré $\NS^{\sf gen}(X)$ est stable par $f^*$, car
\begin{equation}
s_i^*F_j = F_j
\quad\text{et}\quad
s_i^*F_i=-F_i+2F_j+2F_k
\end{equation}
pour $\{i,j,k\}=\{1,2,3\}$. Sur ce sous-espace, la forme d'intersection a pour matrice
\begin{equation}
\begin{pmatrix}0&2&2\\2&0&2\\2&2&0\end{pmatrix}
\end{equation}
dans la base $(F_1,F_2,F_3)$, donc sa signature est $(1,2)$. Le type de $f$ ne dépend donc que de la restriction de $f^*$ à $\NS^{\sf gen}(X)$, en vertu de la proposition \ref{meme-type}. Or la matrice de $f^*$ sur ce sous-espace s'écrit
\begin{equation}
\begin{pmatrix}1&2&0\\0&-1&0\\0&2&1\end{pmatrix}
\begin{pmatrix}-1&0&0\\2&1&0\\2&0&1\end{pmatrix}
=\begin{pmatrix}3&2&0\\-2&-1&0\\6&2&1\end{pmatrix},
\end{equation}
de polynôme caractéristique $(x-1)^3$. Comme cette matrice n'est pas diagonalisable, $f^*$ est de type parabolique.
\end{exem}

% Type loxodromique

\subsection{Type loxodromique}

Il apparaît donc que les seuls automorphismes ayant une dynamique riche sont ceux de type loxodromique. En effet, on a le théorème de Gromov et Yomdin :

\begin{theo}[\cite{gromov-entropie,yomdin}]\label{theo-gyc}
Soit $f$ un automorphisme d'une surface complexe compacte kählérienne $X$. Alors $f$ est de type loxodromique si et seulement si $f$ est d'entropie positive, et on a l'égalité
\begin{equation}\label{ent-deg}
\h(f)=\log\left(\lambda(f)\right).
\end{equation}
\end{theo}

\begin{rema}
Lorsque $f$ est un automorphisme d'une surface complexe compacte $X$ qui n'est pas kählérienne, on a nécessairement $\h(f)=0$ d'après \cite{cantat-cras} (voir aussi \cite[p. 26]{cantat-these} pour une démonstration plus détaillée).
\end{rema}

\begin{exem}\label{ex-loxo}
Reprenons l'exemple \ref{ex-para}, avec cette fois-ci $f=s_1\circ s_2\circ s_3$. La matrice de $f^*$ dans la base $(F_1,F_2,F_3)$ de $\NS^{\sf gen}(X)$ s'écrit
\begin{equation}
\begin{pmatrix}1&0&2\\0&1&2\\0&0&-1\end{pmatrix}
\begin{pmatrix}1&2&0\\0&-1&0\\0&2&1\end{pmatrix}
\begin{pmatrix}-1&0&0\\2&1&0\\2&0&1\end{pmatrix}
=\begin{pmatrix}15&6&2\\10&3&2\\-6&-2&-1\end{pmatrix}.
\end{equation}
Son polynôme caractéristique est $x^3-17x^2-17x+1$, et ses valeurs propres sont $-1$, $\lambda=9+4\sqrt{5}$ et $1/\lambda=9-4\sqrt{5}$. Ainsi $f$ est de type loxodromique, et son entropie vaut $\h(f)=\log(9+4\sqrt{5})$.
\end{exem}

\begin{rema}\label{lehmer}
Lorsque $f$ est un automorphisme de type loxodromique, on a vu que $f^*$ admet $\lambda(f)$ et $\lambda(f)^{-1}$ comme valeurs propres simples, et toutes les autres valeurs propres sont de module $1$. Comme $f^*$ provient d'une transformation $\Z$-linéaire de $H^2(X;\Z)$, son polynôme caractéristique est à coefficients entiers, et le nombre $\lambda(f)$ est également un entier algébrique. On voit ainsi que $\lambda(f)$ est un entier quadratique ou un \emph{nombre de Salem}\footnote{Par définition, un nombre de Salem est un entier algébrique $\lambda>1$ dont les conjugués sont $\lambda^{-1}$ et des nombres complexes de module $1$.}. Parmi les nombres de Salem, le nombre de Lehmer $\lambda_{10}\approx 1,17628081$, défini comme la plus grande racine du polynôme
\begin{equation}
x^{10}+x^9-x^7-x^6-x^5-x^4-x^3+x+1,
\end{equation}
est le plus petit connu. Un théorème de McMullen \cite{mcmullen-rat} affirme que $\lambda_{10}$ est le minimum\footnote{Voir aussi \cite{bedford-kim-degree} et \cite{mcmullen-k3-ent,mcmullen-k3-proj} pour des exemples où ce minimum est atteint, le premier pour une surface rationnelle, le deuxième pour une surface K3.} de tous les degrés dynamiques d'automorphismes loxodromiques (sur toutes les surfaces kählériennes compactes possibles).
\end{rema}

\begin{nota}
Fixons une forme de Kähler $\kappa$ sur $X$. Lorsque $f$ est un automorphisme loxodromique de $X$, on note $\theta^+_f$ et $\theta^-_f$ les vecteurs propres de $f^*$ sur $H^{1,1}(X;\R)$ définis par les conditions
\begin{equation}
\left\{
\begin{split}
&f^*\theta^+_f=\lambda(f)\,\theta^+_f\\
&\theta^+_f\cdot[\kappa]=1
\end{split}
\right.
\quad\text{et}\quad
\left\{
\begin{split}
&f^*\theta^-_f=\lambda(f)^{-1}\,\theta^-_f\\
&\theta^-_f\cdot[\kappa]=1
\end{split}
\right.
\end{equation}
On note aussi
\begin{equation}
D^+_f=\R\theta^+_f
\quad\text{et}\quad
D^-_f=\R\theta^-_f
\end{equation}
les directions propres associées, qui elles ne dépendent pas de $\kappa$.
\end{nota}

\begin{prop}\label{loxo-kahler}
Soit $f$ un automorphisme loxodromique d'une surface kählérienne compacte $X$. Les deux droites $D^+_f$ et $D^-_f$ de $f^*$ supportent des rayons extrémaux dans l'adhérence du cône de Kähler $\Kah(X)$. Plus précisément, pour toute classe de Kähler $\theta\in H^{1,1}(X;\R)$, les limites
\begin{equation}
\lim_{n\to+\infty}\frac{1}{\lambda(f)^{n}}\,{f^*}^n\theta
\quad\text{et}\quad
\lim_{n\to+\infty}\frac{1}{\lambda(f)^{n}}\,f_*^n\theta
\end{equation}
sont des vecteurs non nuls de $D^+_f$ et $D^-_f$ respectivement.
\end{prop}

\begin{proof}
La classe de Kähler $\theta$ se décompose en 
\begin{equation}
\theta=\alpha^+\theta^+_f + \alpha^-\theta^-_f + \theta^\perp,
\end{equation}
où $\theta^\perp\in{\rm Vect}(\theta^+_f,\theta^-_f)^\perp$ et $\alpha^+\alpha^-\neq 0$. On a ${f^*}^n\theta^-_f=\lambda(f)^{-n}\,\theta^-_f\to 0$, et ${f^*}^n\theta^\perp$ est bornée car $q$ est définie-négative sur ${\rm Vect}(\theta^+_f,\theta^-_f)^\perp$. On en déduit que
\begin{equation}
\lim_{n\to+\infty}\frac{1}{\lambda(f)^n}\,{f^*}^n\theta
= \lim_{n\to+\infty} \frac{1}{\lambda(f)^n}\,{f^*}^n(\alpha^+\theta^+_f)=\alpha^+\theta^+_f.
\end{equation}
Le résultat pour la limite de $\frac{1}{\lambda(f)^{n}}f_*^n\theta$ est obtenu de manière similaire, en remplaçant $f$ par $f^{-1}$.

Le cône de Kähler étant invariant par $f^*$, on obtient des vecteurs non nuls de $D^+_f$ et $D^-_f$ comme limites d'éléments de $\Kah(X)$. Le fait que les rayons engendrés par ces vecteurs soient extrémaux dans $\b{\Kah(X)}$ provient du fait qu'ils sont d'auto-intersection nulle, et $\b{\Kah(X)}\subset\{q\geq 0\}$.
\end{proof}

% Cas des surfaces algébriques

\subsection{Cas des surfaces algébriques}

Dans le cas où $X$ est algébrique, la forme d'intersection est de signature $(1,\rho(X)-1)$ sur $\NS(X;\R)$. On dispose donc d'un second espace hyperbolique $\H_X^{alg}=\NS^+(X)\cap\{q=1\}$, sur lequel $\Aut(X)$ agit aussi par isométries. D'après la proposition \ref{meme-type}, le type d'un automorphisme $f$ est aussi le type de $f^*$ en tant qu'isométrie de $\H_X^{alg}$. Dans le cas où $f$ est de type loxodromique, on obtient donc :

\begin{prop}
Soit $X$ une surface algébrique complexe, et soit $f$ un automorphisme loxodromique de $X$. Les directions propres $D^+_f$ et $D^-_f$ sont dans le sous-espace $\NS(X;\R)$ de $H^{1,1}(X;\R)$, et supportent des rayons extrémaux dans le cône convexe $\Nef(X)$.
\end{prop}

%%%%%%%%%%%%%%%%%%%%%%%%%%%%%%%%%%
% Classification d'Enriques-Kodaira et automorphismes
%%%%%%%%%%%%%%%%%%%%%%%%%%%%%%%%%%

\section{Classification d'Enriques-Kodaira et automorphismes}

Les puissances tensorielles du fibré en droites canonique $K_X$ définissent des applications rationnelles
\begin{equation}
\phi_k:X\dashrightarrow\P H^0(X,K_X^{\otimes k})^*.
\end{equation}
On définit la \emph{dimension de Kodaira} de $X$ par la formule
\begin{equation}
\kod(X)=\max_{k\in\N^*}\dim\left(\phi_k(X)\right),
\end{equation}
avec la convention que $\dim\left(\phi_k(X)\right)=-\infty$ lorsque $H^0(X,K_X^{\otimes k})=\{0\}$. On a donc
\begin{equation}
\kod(X)\in\{-\infty,0,1,\cdots,\dim(X)\}.
\end{equation}

Pour les surfaces kählériennes compactes, la classification suivante est due à Enriques et Kodaira (voir \cite[\textsection 6]{bpvdv}) :
\begin{enumerate}
\item $\kod(X)=-\infty$ : $X$ est une surface rationnelle ou une surface réglée.
\item $\kod(X)=0$ : le modèle minimal de $X$ est un tore, une surface K3, une surface hyper-elliptique ou une surface d'Enriques.
\item $\kod(X)=1$ : $X$ est une surface elliptique d'un type particulier (la fibration elliptique est donnée par une puissance du fibré canonique $K_X$).
\item $\kod(X)=2$ : $X$ est ce que l'on appelle une surface de type général.
\end{enumerate}
De plus, on a unicité du modèle minimal lorsque $\kod(X)\geq 0$.

Nous reviendrons plus en détails sur certains de ces types de surfaces, et notamment sur ceux qui apparaissent dans le théorème suivant :

\begin{theo}[\cite{cantat-cras,nagata}]\label{theo-cantat-loxo}
Soit $X$ une surface kählérienne compacte. On suppose qu'il existe un automorphisme $f$ de type loxodromique sur $X$. Alors $\kod(X)=-\infty$ ou $0$. De plus :
\begin{enumerate}
\item Dans le cas où $\kod(X)=-\infty$, $X$ est une surface rationnelle isomorphe à $\P^2_\C$ éclaté en au moins 10 points. 
\item Dans le cas où $\kod(X)=0$, $f$ induit un automorphisme loxodromique sur le modèle minimal de $X$, qui est un tore, une surface K3 ou une surface d'Enriques.
\end{enumerate}
\end{theo}

\begin{proof}[Idée de la démonstration]
Pour une démonstration plus complète, on pourra se reporter à  \cite{cantat-these}. Si $\kod(X)=2$, le groupe $\Aut(X)$ est fini d'après \cite{matsumura}. Si $\kod(X)=1$, la fibration elliptique d'Iitaka $\phi_k:X\to B$ (pour un certain~${k\in\N^*}$) est préservée par tout automorphisme $f$, ce qui implique que $f$ est de type parabolique ou elliptique. Le même argument montre aussi que les automorphismes des surfaces hyper-elliptiques et des surfaces réglées sont d'entropie nulle, car ils préservent une fibration elliptique ou rationnelle.

Pour les surfaces rationnelles, le fait qu'elles soient isomorphes à $\P^2$ éclatées en au moins $10$ est dû à \cite{nagata}.

Lorsque $X_0$ est une surface minimale de dimension de Kodaira nulle, ses applications birationnelles coïncident avec ses automorphismes d'après \cite[p. 180]{shafarevich}. Ainsi $f$ induit un automorphisme sur le modèle minimal $X_0$ lorsque~${\kod(X)=0}$.
\end{proof}

Des exemples sont connus pour chacun des cas énumérés dans le théorème \ref{theo-cantat-loxo} :
\begin{enumerate}
\item Lorsque $X=\C^2/\Lambda$ est un tore, un automorphisme $f$ de $X$ provient, à translation près, d'une transformation linéaire $F$ de $\C^2$. Le rayon spectral de $f^*$ sur $H^{1,1}(X;\R)$ est alors le carré de celui de $F$, et il existe de nombreux exemples où celui-ci est plus grand que $1$. Par exemple, lorsque le réseau $\Lambda$ se décompose en $\Gamma\times\Gamma$, la matrice $\begin{pmatrix}2&1\\1&1\end{pmatrix}$ induit un automorphisme $f$ sur $\C^2/\Lambda$ avec $\lambda(f)=\left(\frac{3+\sqrt{5}}{2}\right)^2$. On verra d'autres exemples au chapitre \ref{chap-tores}.
\item Un automorphisme loxodromique d'un tore donne par passage au quotient un automorphisme loxodromique sur la surface de Kummer associée, qui est une surface K3. Un exemple moins trivial est celui donné dans \cite{cantat-k3}, qui a déjà été décrit en \ref{ex-loxo}. D'autres exemples avaient été étudiés auparavant dans \cite{fano,severi} et \cite{wehler}.
\item Une surface d'Enriques est le quotient d'une surface K3 par une involution holomorphe agissant sans point fixe, et les automorphismes proviennent d'automorphismes sur cette surface K3.
\item Des exemples sur les surfaces rationnelles ont été étudiés dans  \cite{mcmullen-rat} et \cite{bedford-kim-degree,bedford-kim-maxent,bedford-kim-fam,bedford-kim-rotation} entre autres.
\end{enumerate}

\begin{rema}\label{forme-volume}
Soit $X$ une surface kählérienne compacte de dimension de Kodaira nulle. On sait que son modèle minimal $X'$ admet un revêtement étale fini $\pi:Y\to X'$, où~$Y$ est un tore ou une surface K3. Comme le fibré canonique~$K_Y$ est trivial,~$Y$ admet une $2$-forme holomorphe $\Omega$ qui ne s'annule pas, et $\Omega\wedge\b\Omega$ définit alors une forme volume sur $Y$, qui est uniquement déterminée à une constante près. On peut par exemple choisir la normalisation $\int_Y\Omega\wedge\b\Omega=1$. Cette forme volume est invariante par les automorphismes du revêtement, et passe donc au quotient en une forme volume \og canonique\fg~sur~$X$, de volume total $1$. Celle-ci est préservée par tous les automorphismes de $X$.

De même, lorsque $X$ (et donc $Y$) possède une structure réelle, $|\Omega|$ détermine une forme d'aire sur $Y(\R)$, qui passe au quotient en une forme d'aire canonique~$\mu_X$ sur $X(\R)$, qui est préservée par tous les automorphismes de $X_\R$.
\end{rema}

%%%%%%%%%%%%%%%%%%%%%%%%%%%%%%%%%%
% Courbes périodiques
%%%%%%%%%%%%%%%%%%%%%%%%%%%%%%%%%%

\section{Courbes périodiques}\label{sec-per}

Soit $f$ un automorphisme d'une surface kählérienne compacte $X$. Une courbe complexe compacte $C$ est dite \emph{périodique} lorsque cette courbe est invariante par un itéré de $f$.

\begin{prop}\label{prop-per}
Soit $f$ un automorphisme loxodromique d'une surface kählérienne compacte $X$. Une courbe irréductible $C$ est périodique si et seulement si sa classe $[C]$ dans $H^{1,1}(X;\R)$ appartient à l'orthogonal du plan ${\Pi=\Vect(\theta^+_f,\theta^-_f)}$. En particulier, toute courbe périodique est d'auto-intersection négative.
\end{prop}

\begin{proof}
Soit $C$ une courbe (irréductible ou non) invariante par $f^k$. Puisque~$f^*$ est une isométrie pour la forme d'intersection, les sous-espaces propres de~${f^*}^k$ sont orthogonaux deux à deux, et on en déduit que $[C]\in\Pi^\perp$. La forme d'intersection est définie-négative sur $\Pi^\perp$ (car de signature $(1,1)$ sur~$\Pi$), d'où $C^2<0$.

Réciproquement, soit $C$ une courbe irréductible telle que $[C]\in\Pi^\perp$. Comme~$q$ est définie-négative sur $\Pi^\perp$, il n'y a qu'un nombre fini de points entiers $\theta\in\Pi^\perp\cap H^2(X;\Z)$ tels que $q(\theta)=[C]^2$. En particulier, il existe un entier $k>0$ tel que ${f^*}^k[C]=[C]$. Puisque $C^2<0$ et $C$ est irréductible, la classe $[C]$ est réduite à la seule courbe $C$, et donc $f^{-k}(C)=C$.
\end{proof}

En utilisant le critère de Grauert-Mumford (voir \cite[\textsection III (2.1) p.~91]{bpvdv}), on peut alors contracter les courbes périodiques, et on obtient le résultat suivant :

\begin{theo}[\cite{cantat-k3,djs,kawaguchi,cantat-invariant}]\label{theo-contr}
Soit $f$ un automorphisme loxodromique d'une surface kählérienne compacte~$X$.
\begin{enumerate}
\item Il n'y a qu'un nombre fini de courbes irréductibles périodiques, et celles-ci sont de genre\footnote{Cf. théorème \ref{theo-genre} pour la définition du genre pour une courbe non lisse.} $0$ ou $1$. Plus précisément, quitte à prendre un itéré de $f$ et à contracter un nombre fini de $(-1)$-courbes\footnote{Par définition, une $(-1)$-courbe est une courbe rationnelle lisse (irréductible) d'auto-intersection $-1$.} périodiques, chaque courbe périodique connexe est
\begin{itemize}
\item soit un arbre de courbes rationnelles lisses ;
\item soit une courbe de genre $1$ qui est de l'un des types suivants :
\begin{itemize}
\item une courbe elliptique lisse,
\item une courbe rationnelle avec une singularité nodale ou cuspidale,
\item une union de deux courbes rationnelles qui se coupent en deux points (confondus ou non),
\item trois courbes rationnelles qui s'intersectent en un unique point
\item ou un cycle de $n\geq 3$ courbes rationnelles lisses.
\end{itemize}
\end{itemize}

\item Il existe un morphisme birationnel $\pi:X\to X_0$, où $X_0$ est une surface normale éventuellement singulière, qui contracte les courbes périodiques connexes de~$f$. Par conséquent, $f$ induit sur $X_0$ un automorphisme ${f_0:X_0\to X_0}$ de type loxodromique.
\end{enumerate}
\end{theo}

\begin{rema} Si $X$ est un tore, il ne peut pas y avoir de courbe périodique, car $X$ n'admet pas de courbe d'auto-intersection négative. Lorsque~$X$ est une surface K3 ou une surface d'Enriques, la formule du genre permet de montrer que les courbes périodiques irréductibles sont des courbes rationnelles lisses d'auto-intersection $-2$. En revanche, les courbes périodiques peuvent être de genre $1$ sur une surface rationnelle (voir \cite{cantat-invariant} pour un exemple).
\end{rema}

%%%%%%%%%%%%%%%%%%%%%%%%%%%%%%%%%%
% Automorphismes des surfaces réelles
%%%%%%%%%%%%%%%%%%%%%%%%%%%%%%%%%%

\section{Automorphismes des surfaces réelles}

À partir de maintenant, $X_\R$ est une surface kählérienne réelle (compacte par hypothèse). Un automorphisme de $X_\R$ (ou \emph{automorphisme réel} de $X$) est par définition un automorphisme de $X$ qui commute avec $\sigma$\footnote{Il est demandé ici que les automorphismes soient définis sur $X(\C)$, ce qui diffère de la terminologie utilisée par certains auteurs, comme par exemple \cite{kollar-mangolte}. Dans ce dernier texte, un automorphisme de $X_\R$ est ce que nous appelons dans la suite un difféomorphisme birationnel (voir chapitre \ref{chap-bir}).}. Dans le cas où $X_\R$ est une sous-variété algébrique de $\P^n_\R$, les automorphismes de $X_\R$ correspondent aux automorphismes de $X$ qui sont définis par des formules polynomiales à coefficients réels. On note $\Aut(X_\R)$ le sous-groupe de Lie réel de $\Aut(X_\C)$ formé par ces automorphismes.

\begin{nota}
Lorsque $f\in\Aut(X_\R)$, on note $f_\R$ (resp. $f_\C$) le difféomorphisme analytique induit sur $X(\R)$ (resp. sur $X(\C)$).
\end{nota}

Par restriction, on a une action du groupe $\Aut(X_\C)$ sur les diviseurs et la cohomologie de $X_\C$. Le fait que les automorphismes réels commutent avec $\sigma$ implique que cette action commute avec l'action de $\sigma$. Par conséquent, l'action de $\Aut(X_\R)$ préserve :
\begin{itemize}
\item le sous-groupe $\Div(X_\R)$ de $\Div(X_\C)$ ;
\item le sous-groupe $\Pic(X_\R)$ de $\Pic(X_\C)$ ;
\item le sous-groupe $\NS(X_\R)$ de $\NS(X_\C)$, ainsi que les cônes convexes $\NS^+(X_\R)$ et $\Amp(X_\R)$ ;
\item le sous-espace $H^{1,1}(X_\R;\R)$ de $H^{1,1}(X_\C;\R)$, ainsi que les cônes convexes $H^+(X_\R)$ et $\Kah(X_\R)$.\\
\end{itemize}

Là encore, la proposition \ref{meme-type} implique que l'on peut regarder l'action de $\Aut(X_\R)$ sur le sous-espace $H^{1,1}(X_\R;\R)$, ou même sur $\NS(X_\R;\R)$ dans le cas où $X$ est algébrique, pour connaître le type de $f_\C:X_\C\to X_\C$ (on dit plus simplement le type de $f$). Dans le cas loxodromique, cela donne le résultat suivant :

\begin{prop}
Soit $X_\R$ une surface kählérienne réelle, et soit $f$ un automorphisme de type loxodromique de $X_\R$.\begin{enumerate}
\item Les directions propres $D^+_f$ et $D^-_f$ sont dans le sous-espace $H^{1,1}(X_\R;\R)$, et supportent des rayons extrémaux dans l'adhérence du cône de Kähler réel $\Kah(X_\R)$.
\item Si de plus $X$ est algébrique, alors les droites $D^+_f$ et $D^-_f$ sont dans le sous-espace $\NS(X_\R;\R)$, et supportent des rayons extrémaux dans le cône convexe $\Nef(X_\R)$.
\end{enumerate}
\end{prop}

En combinant la proposition \ref{meme-type} avec le théorème \ref{theo-gizatullin}, on en déduit aussi la caractérisation suivante des automorphismes réels de type parabolique :

\begin{theo}
Soit $f$ un automorphisme d'une surface kählérienne réelle $X_\R$. Alors $f$ est de type parabolique si et seulement si $f$ préserve une fibration elliptique réelle et aucun itéré de $f_\C$ n'est isotope à l'identité sur~$X(\C)$.
\end{theo}

\begin{proof}[Esquisse de démonstration]
Seul le sens direct est non trivial. Supposons que $f_\C$ soit de type parabolique. Il préserve alors une fibration elliptique, dont la fibre générique $F$ est une courbe lisse de genre $1$ telle que~${F^2=0}$. La classe de $F$ engendre l'unique droite fixe $D$ de $f^*$ dans le cône isotrope, donc celle-ci est dans $\NS(X_\R)$. Autrement dit, on a $[F^\sigma]=[F]$. Le système linéaire~${|F+F^\sigma|}$ est alors un système linéaire réel non trivial, car~${H^0(X,\O_X(F))}$ s'injecte dans~${H^0(X,\O_X(F+F^\sigma))}$. Sa partie mobile est sans point base, car~${(F+F^\sigma)^2=0}$. La factorisation de Stein du morphisme induit par ce système linéaire est alors une fibration elliptique réelle qui est préservée par~$f$. Voir aussi \cite[p. 5]{kollar-surfaces} pour une démonstration similaire.
\end{proof}

\begin{rema}
Le difféomorphisme $f_\R$ peut être isotope à l'identité sur~$X(\R)$ sans que $f_\C$ soit isotope à l'identité sur $X(\C)$. Par exemple, soit~$X$ la surface réelle lisse de degré $(2,2,2)$ dans $\P^1_\R\times\P^1_\R\times\P^1_\R$, donnée par l'annulation du polynôme
\begin{equation}
P(x_1,x_2,x_3)=(x_1^2+1)(x_2^2+1)(x_3^2+1)+tx_1x_2x_3-2,
\end{equation}
où $t$ est un paramètre réel non nul. Topologiquement $X(\R)$ est une sphère, donc tout automorphisme de $X_\R$ est isotope à l'identité sur $X(\R)$. Cependant, si $f$ est l'automorphisme réel donné par l'un des deux exemples \ref{ex-para} ou \ref{ex-loxo},~$f$ est de type parabolique ou loxodromique, donc n'est pas isotope à l'identité sur $X(\C)$.
\end{rema}

%%%%%%%%%%%%%%%%%%%%%%%%%%%%%%%%%%
% Chapitre
%%%%%%%%%%%%%%%%%%%%%%%%%%%%%%%%%%
% Courants
%%%%%%%%%%%%%%%%%%%%%%%%%%%%%%%%%%

\chapter{Courants positifs fermés}\label{chap-courants}

On suit ici \cite{sibony-survey}, auquel on pourra se référer pour plus de détails.

%%%%%%%%%%%%%%%%%%%%%%%%%%%%%%%%%%
% Généralités
%%%%%%%%%%%%%%%%%%%%%%%%%%%%%%%%%%

\section{Généralités sur les courants}

\begin{defi}
Soit $M$ une variété $\mathcal{C}^\infty$ de dimension $m$. Un \emph{courant de degré $k$} (ou plus simplement un $k$-courant) sur $M$ est une forme linéaire continue sur l'espace $\mathcal{D}^{m-k}(M)$ des formes différentielles $\mathcal{C}^\infty$ de degré $m-k$ à support compact. On note $\mathcal{D}^{m-k}(M)'$ l'espace des $k$-courants, que l'on munit de la topologie de la convergence faible.
\end{defi}

\begin{rema}
L'espace des $0$-courants s'identifie à l'espace des distributions, via un choix de forme volume sur $M$.
\end{rema}

\begin{exem}
Toute $k$-forme différentielle $\omega$ sur $M$ définit un $k$-courant, toujours noté $\omega$, par la formule
\begin{equation}
\langle \omega,\theta\rangle=\int_M \omega\wedge\theta
\end{equation}
pour toute $\theta\in\mathcal{D}^{m-k}(M)$.
\end{exem}

\begin{exem}
Soit $N$ une sous-variété de dimension $m-k$ de $M$. On définit le \emph{courant d'intégration sur $N$}, noté $\{N\}$, par la formule :
\begin{equation}
\langle\{N\},\theta\rangle=\int_N\theta_{|N}
\end{equation}
pour toute $\theta\in\mathcal{D}^{m-k}(M)$.
\end{exem}

\begin{defi}
Pour tout courant $T$ sur une variété $M$, on définit la \emph{restriction de $T$ à un ouvert $U$} par la formule
\begin{equation}
\langle T_{|U},\theta \rangle = \langle T,\theta \rangle
\end{equation}
pour toute forme différentielle $\theta$ à support compact dans $U$. Le \emph{support} de $T$, noté $\Supp(T)$, est alors le plus grand fermé $F\subset M$ tel que $T_{|M\backslash F}=0$.
\end{defi}

\begin{defi}
Soit $T$ un $k$-courant sur une variété $M$. On définit un $(k+1)$ courant ${\rm d}T$ par la formule
\begin{equation}
\langle{\rm d}T,\theta\rangle=(-1)^{k+1}\langle T,{\rm d}\theta\rangle
\end{equation}
pour toute $\theta\in\mathcal{D}^{m-k-1}(M)$. On dit que $T$ est \emph{fermé} lorsque ${\rm d}T=0$, et \emph{exact} lorsqu'il existe $S\in\mathcal{D}^{m-k+1}(M)'$ tel que $T={\rm d}S$.
\end{defi}

\begin{rema}
Lorsque $\omega$ est une $k$-forme différentielle de classe $\mathcal{C}^1$, la différentielle au sens des courants coïncide avec celle au sens des formes, grâce au théorème de Stokes. En particulier, $\omega$ est fermé en tant que courant si et seulement si $\omega$ est fermée en tant que forme différentielle.
\end{rema}

\begin{rema}
Un courant d'intégration sur une sous-variété fermée $N$ est fermé, toujours d'après le théorème de Stokes. Plus généralement, le courant $\{N\}$ est fermé si et seulement si $\partial N$ est de codimension $\geq 2$ dans $N$.
\end{rema}

Notons $H^k_{\mathcal{D}'}(M;\R)$ l'espace vectoriel quotient
\begin{equation}
H^k_{\mathcal{D}'}(M;\R)
=
\frac{
\{T\in\mathcal{D}^k(M)'\,|\,\text{$T$ est fermé}\}
}{
\{T\in\mathcal{D}^k(M)'\,|\,\text{$T$ est exact}\}
}.
\end{equation}
L'espace de cohomologie de De Rham $H^k(M;\R)$ se plonge naturellement dans $H^k_\mathcal{D'}(M;\R)$, et il se trouve que ce plongement est un isomorphisme (voir \cite[p. 382]{griffiths-harris}). On peut donc considérer la classe d'un $k$-courant fermé $T$, modulo les courants exacts, comme un élément de $H^k(M;\R)$. On notera $[T]$ une telle classe.

%%%%%%%%%%%%%%%%%%%%%%%%%%%%%%%%%%
% Courants sur une variété complexe
%%%%%%%%%%%%%%%%%%%%%%%%%%%%%%%%%%

\section{Courants sur une variété complexe}

\begin{defi}
Soit $X$ une variété complexe de dimension $n$. Un \emph{courant de type $(p,q)$} (ou plus simplement un $(p,q)$-courant) est une forme $\C$-linéaire continue sur l'espace $\mathcal{D}^{n-p,n-q}(X)$ des formes différentielles $\mathcal{C}^\infty$ de type $(n-p,n-q)$ à support compact. On note $\mathcal{D}^{n-p,n-q}(X)'$ l'espace des $(p,q)$-courants, que l'on munit de la topologie de la convergence faible.
\end{defi}

Pour tout entier $k\leq 2n$, on a la décomposition en somme directe de sous-espaces fermés :
\begin{equation}
\mathcal{D}^{2n-k}(X)\otimes\C=\bigoplus_{p+q=k}\mathcal{D}^{n-p,n-q}(X).
\end{equation}
Si l'on identifie l'espace $\mathcal{D}^{n-p,n-q}(X)'$ au sous-espace de $\mathcal{D}^{2n-k}(X)'\otimes\C$ formé des formes linéaires qui s'annulent sur $\bigoplus_{(p',q')\neq(p,q)}\mathcal{D}^{n-p',n-q'}(X)$ (ce que l'on fera toujours), on a alors, par dualité, la décomposition en somme directe :
\begin{equation}
\mathcal{D}^{2n-k}(X)'\otimes\C=\bigoplus_{p+q=k}\mathcal{D}^{n-p,n-q}(X)'.
\end{equation}
Avec cette écriture, un courant de type $(p,q)$ est un courant (complexe) de degré $p+q$ particulier. Par exemple, le $(p+q)$-courant défini par une forme de type $(p,q)$ est un courant de type $(p,q)$.

Au niveau de la cohomologie, cette décomposition en somme directe induit la décomposition de Hodge usuelle lorsque $X$ est une variété kählérienne compacte. Autrement dit, la classe d'un courant de type $(p,q)$ est dans $H^{p,q}(X;\C)$.

De même que l'on a défini un opérateur ${\rm d}$ sur les courants, on définit les opérateurs $\partial$ et $\b\partial$. On définit également un opérateur
\begin{equation}
{\rm d^c}=\frac{1}{2i\pi}(\partial-\b\partial),
\end{equation}
de telle sorte que l'on a
\begin{equation}
\ddc=\frac{i}{\pi}\partial\b\partial
\end{equation}
et
\begin{equation}
\langle\ddc T,\theta\rangle=\langle T,\ddc\theta\rangle
\end{equation}
pour tout courant $T\in\mathcal{D}^{2n-k}(X)'$ et pour toute forme $\theta\in\mathcal{D}^{2n-k}(X)$.

Comme l'opérateur ${\rm d}$, mais contrairement aux opérateurs $\partial$ et $\b\partial$, ${\rm d^c}$ est un opérateur réel, au sens où $\b{{\rm d^c}T}={\rm d^c}\b T$. En particulier, si $T$ est une courant réel de type $(p,p)$, alors $\ddc T$ est un courant réel de type $(p+1,p+1)$. Le lemme suivant est une variation de \cite[p. 149]{griffiths-harris} :

\begin{lemm}[dit \og du $\ddc$ global\fg]\label{ddc-glob}
Soit $T$ un courant réel de type $(p,p)$ sur une variété kählérienne compacte $X$. On suppose que le courant $T$ est exact pour l'un des opérateurs ${\rm d}$, $\partial$ ou $\b\partial$. Alors il existe un courant réel $S$ de type $(p-1,p-1)$ tel que
\begin{equation}
T=\ddc S.
\end{equation}
De plus, si $T$ est donné par une forme différentielle $\mathcal{C}^\infty$, alors $S$ également.
\end{lemm}

%%%%%%%%%%%%%%%%%%%%%%%%%%%%%%%%%%
% Courants positifs
%%%%%%%%%%%%%%%%%%%%%%%%%%%%%%%%%%

\section{Courants positifs}

\begin{defi}
Soit $T$ un courant de type $(p,p)$ sur une variété complexe $X$. On dit que $T$ est \emph{positif} lorsque
\begin{equation}
\langle T,i\alpha_1\wedge\b\alpha_1\wedge\cdots\wedge i\alpha_{n-p}\wedge\b\alpha_{n-p}\rangle\geq 0
\end{equation}
pour tout $(n-p)$-uplet de formes $(\alpha_1,\cdots,\alpha_{n-p})\in(\mathcal{D}^{1,0}(X))^{n-p}$.
\end{defi}

En particulier, un courant positif est réel.

\begin{exem}
Soit $Y$ une sous-variété complexe de $X$ de dimension $n-p$. Alors le courant d'intégration $\{Y\}$ est un courant de type $(p,p)$ qui est positif, et fermé lorsque $Y$ est fermée.

Plus généralement, lorsque $Y$ est une sous-variété analytique complexe qui n'est pas lisse, on définit le courant d'intégration $\{Y\}$ comme le courant d'intégration sur la partie lisse de $Y$. C'est un courant positif, qui est fermé lorsque $Y$ est fermée (car l'ensemble des points singuliers de $Y$ est de codimension réelle au moins $2$ dans $Y$).
\end{exem}

\begin{prop}[{\cite[p.386]{griffiths-harris}}]\label{ordre0}
Soit $T$ un $(p,p)$-courant positif sur une variété complexe $X$ de dimension $n$. Alors $T$ s'étend en une unique forme linéaire continue $\tilde T$ sur l'espace des $(n-p,n-p)$-formes \emph{continues} à support compact sur $X$.\footnote{En termes plus savants, on dit que le courant $S$ est d'ordre $0$.}
\end{prop}

Dans la suite, lorsqu'on parle de \emph{courants positifs fermés}, il est sous-entendu que ceux-ci sont de type $(1,1)$.

% Masse

\subsection{Masse}

Soit $X$ une variété complexe compacte de dimension $n$, que l'on munit d'une métrique hermitienne $h$. Cette métrique induit une norme sur l'espace $\mathcal{D}^{n-p,n-p}(X)$. Lorsque $T$ est un courant positif de type $(p,p)$, la norme de~$T$ en tant que forme linéaire sur $\mathcal{D}^{n-p,n-p}(X)$ est appelée \emph{masse} du courant~$T$ par rapport à la métrique $h$, que l'on note $\|T\|_h$. Si l'on note $\kappa$ la $(1,1)$-forme~$\Im(h)$ (non nécessairement fermée), on a la formule
\begin{equation}
{\|T\|}_h=\left\langle T,\frac{\kappa^{n-p}}{(n-p)!}\right\rangle.
\end{equation}
Lorsque $\kappa$ est une forme de Kähler, on note plutôt $\|T\|_\kappa$ la masse par rapport à la métrique kählérienne, et
l'égalité précédente se traduit en terme d'intersections de classes :
\begin{equation}
{\|T\|}_\kappa=\frac{1}{(n-p)!}[T]\cdot[\kappa]^{n-p}.
\end{equation}

Comme la norme induite par la métrique $h$ sur $\mathcal{D}^{n-p,n-p}(X)$ est continue par rapport à la topologie de $\mathcal{D}^{n-p,n-p}(X)$, on en déduit :

\begin{prop}
Un ensemble de courants positifs de masses bornées est relativement compact (pour la topologie faible sur l'espace des courants).
\end{prop}

% Potentiel

\subsection{Potentiel}

Une fonction pluri-sous-harmonique sur une variété complexe $X$ est une fonction $u:X\to\R\cup\{-\infty\}$ telle que
\begin{enumerate}
\item[(i)] $u$ est semi-continue supérieurement ;
\item[(ii)] pour tout disque holomorphe $\varphi:\D\to X$, la fonction $v=u\circ\varphi$ est sous-harmonique, ce qui signifie que
\begin{equation}
v(z)\leq\frac{1}{2\pi}\int_0^{2\pi}v(z+r{\rm e}^{i\theta})\,{\rm d}\theta
\end{equation}
dès que le disque de rayon $r$ centré en $z$ est inclus dans $\D$ (cela n'exclut pas le cas où $v$ est identiquement égale à $-\infty$).
\end{enumerate}
Pour les propriétés de base des fonctions sous-harmoniques et pluri-sous-harmoniques, je renvoie le lecteur à \cite{gunning}. Par exemple, on montre que toute fonction pluri-sous-harmonique $u$ est localement intégrable (sauf si elle est identiquement égale à $-\infty$). Ceci permet de définir le $(1,1)$-courant $\ddc u$, qui est positif fermé. Localement,  tous les courants positifs fermés sont de cette forme :

\begin{lemm}[dit \og du $\ddc$ local\fg~{\cite[p. 387]{griffiths-harris}}]\label{ddc-loc}
Soit $T$ un courant positif fermé sur une variété complexe $X$. Alors localement, on peut écrire
\begin{equation}
T=\ddc u,
\end{equation}
où $u$ est une fonction pluri-sous-harmonique, unique à addition près d'une fonction pluri-harmonique. De plus, si $T$ est donné par une forme différentielle $\mathcal{C}^\infty$, alors $u$ est une fonction $\mathcal{C}^\infty$.
\end{lemm}

Cela signifie que pour tout point $x\in X$, il existe un voisinage ouvert $U$ et une fonction pluri-sous-harmonique $u$ sur $U$ telle que $T_{|U}=\ddc u$, où le courant $T_{|U}$ est défini par la restriction de la forme linéaire $T$ au sous-espace fermé $\mathcal{D}^{1,1}(U)\subset\mathcal{D}^{1,1}(X)$.

\begin{defi}
Une fonction $u$ comme dans le lemme \ref{ddc-loc} est appelée \emph{potentiel local} du courant $T$.
\end{defi}

\begin{rema}\label{regularite-potentiel}
Comme la différence entre deux potentiels locaux est une fonction pluri-harmonique, donc $\mathcal{C}^\infty$, les potentiels locaux de $T$ ont tous la même régularité. On peut ainsi parler sans ambiguïté de la régularité des potentiels (locaux) d'un courant positif fermé : régularité $\mathcal{C}^0$, $\alpha$-Hölder, $\mathcal{C}^\infty$, etc.
\end{rema}

\begin{prop}
Soit $T$ un courant positif fermé sur une variété complexe $X$. On suppose que $T$ est donné sur un ouvert $U$ par un potentiel local $u$. Alors
\begin{equation}
\left(X\backslash\Supp(T)\right)\cap U = \{x\in X\,|\,\text{$u$ est pluri-harmonique dans un voisinage de $x$}\}.
\end{equation}
\end{prop}

%%%%%%%%%%%%%%%%%%%%%%%%%%%%%%%%%%
% Opérations sur les courants
%%%%%%%%%%%%%%%%%%%%%%%%%%%%%%%%%%

\section{Opérations sur les courants}

% Image directe

\subsection{Image directe}

Soit $f:M\to N$ une application lisse entre deux variétés $M$ et $N$ de dimensions respectives $m$ et $n$, et soit $T$ un $k$-courant sur $M$. On suppose que l'application $f$ est propre en restriction au support de $T$ (ceci est le cas par exemple lorsque $M$ est compacte). On peut alors définir l'\emph{image directe $f_*T$ du courant $T$ par $f$} par la formule :
\begin{equation}
\langle f_*T,\theta\rangle=\langle T,f^*\theta\rangle
\end{equation}
pour toute forme $\theta\in\mathcal{D}^{m-k}(N)$. Il s'agit donc d'un courant de degré $k+n-m$ sur $N$.

% Image réciproque

\subsection{Image réciproque}

Soit $f:X\to Y$ une application holomorphe entre deux variétés complexes $X$ et $Y$, et soit $T$ un courant positif fermé de type $(1,1)$ sur $Y$. On peut définir un courant positif fermé $f^*T$ sur $X$, dit \emph{image réciproque de $T$ par $f$}, de la manière suivante : si $u$ est un potentiel local de $T$ sur un ouvert $U\subset Y$, on pose $f^*T_{|f^{-1}(U)}=\ddc u\circ f$. Cette définition ne dépend pas du potentiel choisi, car si $u'$ est un autre potentiel, $(u-u')\circ f$ est pluri-harmonique, et donc $\ddc (u-u')\circ f=0$. Comme $X$ est recouvert par de tels ouverts, on définit ensuite $f^*T$ globalement grâce à une partition de l'unité subordonnée à un tel recouvrement.

L'image réciproque d'un courant est compatible avec l'opération pull-back sur la cohomologie, c'est-à-dire que pour tout courant positif fermé $T$, on a
\begin{equation}
f^*[T]=[f^*T].
\end{equation}

Dans le cas où $f:X\to X$ est un automorphisme d'une variété complexe compacte, on peut prendre l'image directe ou l'image réciproque d'un courant par $f$ ou par $f^{-1}$, et on a la formule :
\begin{equation}
f^*T=(f^{-1})_*T
\end{equation}
pour tout courant positif fermé $T$.

% Produit extérieur

\subsection{Produit extérieur}

Soient $T$ et $S$ des courants positifs fermés sur une variété complexe $X$, de types respectifs $(1,1)$ et $(p,p)$. On suppose que les potentiels locaux de $T$ sont continus. Sous ces hypothèses, nous allons définir, selon \cite{bedford-taylor}, un courant $T\wedge S$ de type $(p+1,p+1)$, dit \emph{produit extérieur} (ou \emph{intersection}) \emph{de $T$ et $S$}.

Soit $u$ un potentiel continu de $T$ sur un ouvert $U\subset X$, et soit $\tilde S$ le prolongement de $S$ à l'espace des $(n-p,n-p)$-formes \emph{continues} à support compact (cf. proposition \ref{ordre0}). On peut donc définir un courant $uS$ sur $U$ par la formule
\begin{equation}
\langle uS,\theta\rangle=\langle\tilde S,u\theta\rangle
\end{equation}
pour tout $\theta\in\mathcal{D}^{n-p,n-p}(U)$. Puis on pose
\begin{equation}
(T\wedge S)_{|U}=\ddc uS.
\end{equation}
Comme $S$ est fermé, cette définition ne dépend pas du potentiel $u$ choisi. On définit ensuite $T\wedge S$ globalement grâce à une partition de l'unité.

Cette définition prolonge celle du produit extérieur des formes différentielles. Lorsque $S$ est un courant de type $(n-1,n-1)$, le produit extérieur de $T$ et $S$ est une mesure de masse totale $[T]\cdot[S]$.

%%%%%%%%%%%%%%%%%%%%%%%%%%%%%%%%%%
% Nombres de Lelong
%%%%%%%%%%%%%%%%%%%%%%%%%%%%%%%%%%

\section{Nombres de Lelong}

% Définition

\subsection{Définition et exemples}

Je suis ici \cite{griffiths-harris}. Soient $X$ une variété complexe de dimension $n$, $T$ un courant de type $(p,p)$ positif fermé sur $X$, et $x$ un point de $X$. Nous allons associer à chaque point $x\in X$ un nombre positif $\nu(T,x)$, dit \emph{nombre de Lelong de $T$ en $x$}, qui vaut $0$ lorsque le courant $T$ est lisse en $x$ (c'est-à-dire qu'en restriction à un voisinage de $x$, $T$ est donné par une $(n-p,n-p)$-forme lisse), et qui vaut $1$ lorsque $T$ est un courant d'intégration sur une sous-variété lisse passant par $x$.

Dans un premier temps, définissons les nombres de Lelong pour un courant positif fermé $T$ sur un ouvert $\Omega$ de $\C^n$ ($T$ de type $(p,p)$). On note $\omega=\frac{i}{2}\partial\b\partial\|z\|^2$ la forme de Kähler standard sur $\C^n$. Fixons un point $x\in\Omega$. Pour $0<r\leq\dist(x,\partial\Omega)$, on définit
\begin{equation}
\nu(T,x,r)=\frac{1}{(\pi r^2)^{n-p}}\langle T,\chi_{B(x,r)}\omega^{n-p}\rangle,
\end{equation}
où $\chi_{B(x,r)}$ désigne la fonction caractéristique de la boule $B(x,r)$ de centre $x$ et de rayon $r$ dans $\C^n$. Ici, le nombre $\langle T,\chi_{B(x,r)}\rangle$ est défini grâce à une limite croissante de fonctions $\mathcal{C}^\infty$ à support compact dans $B(x,r)$ qui convergent, dans $L^1_{\rm loc}$, vers $\chi_{B(x,r)}$.

\begin{theo}[Lelong (voir {\cite[p. 390]{griffiths-harris}})]\label{lemm-lelong}
La fonction $r\mapsto\nu(T,x,r)$ est une fonction positive et croissante sur l'intervalle $]0,\dist(x,\partial\Omega)]$.
\end{theo}

Cette fonction possède donc une limite positive en $0$, que l'on appelle \emph{nombre de Lelong de $T$ en $x$} :
\begin{equation}
\nu(T,x)=\lim_{r\to 0}\nu(T,x,r).
\end{equation}

Lorsque $T$ est un courant positif fermé sur $X$, tout point $x\in X$ possède une carte holomorphe $\psi:U\to\C^n$ centrée en $x$, et on définit $\nu(T,x)=\nu(\psi_*T,0)$. Ce nombre ne dépend pas du choix de la carte, car les nombres de Lelong sont invariants par isomorphisme.

\begin{exem}
Si le courant $T$ est donné par une forme à coefficients dans $ L^1_{\rm loc}$, alors $\nu(T,x)=0$ pour tout $x\in X$, par convergence dominée.
\end{exem}

\begin{exem}
Lorsque $T$ est le courant d'intégration sur un sous-espace vectoriel complexe de codimension $p$ dans $\C^n$, on voit, grâce à la formule de Wirtinger, que
\begin{equation}
\frac{1}{(n-p)!}\langle T,\chi_{B(0,r)}\omega^{n-p}\rangle = \frac{(\pi r^2)^{n-p}}{(n-p)!}
\end{equation}
(c'est le volume de la boule de rayon $r$ dans $\C^{n-p}$). Par conséquent, on a $\nu(T,0)=\nu(T,0,r)=1$ pour tout $r>0$.

De manière plus générale, si $T$ est le courant d'intégration sur une sous-variété analytique complexe fermée $Y\subset X$, le nombre de Lelong $\nu(T,x)$ est égal à la multiplicité de $Y$ en $x$.
\end{exem}

% Théorème de Lelong

\subsection{Minoration de l'aire d'une courbe holomorphe et comparaison aire--diamètre}

Le résultat suivant est un cas particulier du théorème de Lelong \ref{lemm-lelong} (voir aussi \cite[appendice]{briend-duval} pour un énoncé similaire) :

\begin{coro}\label{coro-lelong}
Soit $X$ une variété complexe compacte munie d'une métrique riemannienne $g$. Il existe $r_0>0$ et $c>0$ tels que pour toute courbe complexe $C\subset X$, pour tout $x\in C$ et pour tout $r\leq r_0$, on ait :
\begin{equation}
\dist_g(x,\partial C)\geq r\implies\aire_g(C)\geq cr^2.
\end{equation}
\end{coro}

Dans la démonstration, on utilise un type particulier de recouvrement, que l'on réutilisera à de nombreuses reprises :

\begin{defi}\label{rec-compact}
Soit $X$ une variété complexe compacte de dimension~$n$. Un \emph{recouvrement relativement compact} de $X$ est une famille de triplets $(U^\alpha,V^\alpha,\psi^\alpha)_{\alpha\in A}$, avec $A$ fini, telle que :
\begin{enumerate}
\item
$U^\alpha$ et $V^\alpha$ sont des ouverts de $X$ ;
\item
$X=\bigcup_{\alpha\in A} U^\alpha$ ;
\item
$U^\alpha\Subset V^\alpha$, \emph{i.e.} $U^\alpha$ est  relativement compact dans $V^\alpha$ ;
\item
$\psi^\alpha:V^\alpha\to\C^n$ est une carte holomorphe.
\end{enumerate}
\end{defi}

\begin{proof}[Démonstration du corollaire \ref{coro-lelong}]
Fixons un recouvrement relativement compact $(U^\alpha,V^\alpha,\psi^\alpha)_\alpha$ de $X$. Soit $r_0>0$ tel que chaque boule $B_g(x,r_0)$ de rayon $r_0$ soit incluse dans un des ouverts $U^\alpha$ du recouvrement. Comme $U^\alpha$ est relativement compact dans $V^\alpha$, il existe une constante $\delta>0$ telle que l'on ait, pour tout $\alpha$ et tout vecteur $v$ tangent à $U^\alpha$ :
\begin{equation}\label{rapport-normes}
\frac{1}{\delta}\|\psi^\alpha_*v\|
\leq {\|v\|}_g
\leq \delta\|\psi^\alpha_*v\|.
\end{equation}

Soient $C$ une courbe analytique dans $X$, et $x$ un point de $C$. On suppose qu'il existe $r\in\,]0,r_0]$ tel que $\dist_g(x,\partial C)\geq r$. Soit $\alpha$ tel que $B(x,r)\subset U^\alpha$. On pose
\begin{equation}
x'=\psi^\alpha(x), \quad C'=\psi^\alpha(C\cap U^\alpha) \quad\text{et}\quad r'=r/\delta,
\end{equation}
de telle sorte que $\psi^\alpha(B_g(x,r))$ contienne la boule $B(x',r')$ de $\C^n$, et donc
\begin{equation}
\partial C'\cap B(x',r')=\emptyset.
\end{equation}

On définit sur $B(x',r')\subset\C^n$ le courant d'intégration 
\begin{equation}
T=\{C'\cap B(x',r')\}.
\end{equation}
Comme $\partial C'\cap B(x',r')=\emptyset$, il s'agit d'un courant positif fermé, donc on peut lui appliquer le théorème de Lelong \ref{lemm-lelong} : la fonction $t\mapsto\nu(T,0,t)$ est décroissante sur $]0,r']$, et converge en $0$ vers $\nu(T,0)=1$. On en déduit que
\begin{equation}
\frac{\aire(C'\cap B(x',r'))}{\pi r'^2}=\nu(T,0,r')\geq 1,
\end{equation}
et donc, grâce à l'inégalité $(\ref{rapport-normes})$,
\begin{equation}
\aire_g(C)\geq\frac{1}{\delta^2}\aire(C')\geq\frac{\pi r'^2}{\delta^2}=\frac{\pi}{\delta^4}r^2.
\end{equation}
On a ainsi la minoration voulue, avec $c=\pi/\delta^4$.
\end{proof}

\begin{rema}
Lorsque $X=\C^n$ munie de sa métrique hermitienne standard, on peut choisir $r_0=+\infty$ et $c=\pi$.
\end{rema}

Comme conséquence du corollaire \ref{coro-lelong}, on a le théorème suivant sur la comparaison aire--dimamètre (voir \cite[appendice]{briend-duval}) :

\begin{theo}[Briend -- Duval]\label{aire-diam}
Soit $X$ une variété complexe compacte (ou plus généralement un espace analytique complexe), munie d'une métrique riemannienne $g$. Pour tous $\epsilon>0$ et $a\in\,]0,1[$, il existe $\eta>0$ tel que pour tout disque holomorphe $\psi:\D\to X$, on ait :
\begin{equation}
\aire_g(\psi(\D))<\eta \implies \diam_g(\psi(\D_a))<\epsilon,
\end{equation}
où $\D_a$ désigne le disque de rayon $a$ dans $\C$.
\end{theo}

% Estimées volumiques

\subsection{Estimées volumiques}

Le théorème suivant est une généralisa\-tion de \cite[théorème 3.1]{kiselman} :

\begin{theo}[Zeriahi {\cite[4.2]{zeriahi}}]\label{theo-zeriahi}
Soit $\Omega$ un ouvert de $\C^n$, et soit $\mathcal{U}$ un ensemble de fonctions pluri-sous-harmoniques sur $\Omega$. On suppose qu'il existe un compact $K\subset\Omega$ et un réel $\theta\in\R_+^*$ tels que
\begin{equation}\label{hyp-zeriahi}
\sup_{x\in K}\,\lim_{r\to 0}\,\sup_{u\in\mathcal{U}}\,\nu(\ddc u,x,r)\leq \theta<+\infty.
\end{equation}
Alors pour tout ouvert $W$ contenant $K$ et relativement compact dans $\Omega$, il existe des constantes $C_1>0$ et $C_2>0$ telles que
\begin{equation}
\log\left(\vol\{x\in K\,|\,u(x)<-t\}\right)
\leq C_1\int_W|u|\,{\rm dvol}+C_2-\frac{t}{\theta}
\end{equation}
pour tous $u\in\mathcal{U}$ et $t\in\R_+^*$.
\end{theo}

%\subsection{Contrôle des nombres de Lelong en fonction de la masse}

L'hypothèse $(\ref{hyp-zeriahi})$ du théorème de Zeriahi est vérifiée lorsque l'ensemble $\mathcal{U}$ est donné par des potentiels locaux de courants positifs fermés définis sur une variété kählérienne compacte, et que les masses de ces courants sont bornées :

\begin{prop}\label{controle-masse}
Soit $X$ une variété complexe compacte, munie d'une métrique hermitienne $h$. Il existe une constante $\theta>0$ telle que pour tout $(p,p)$-courant positif fermé $T$ et pour tout $x\in X$, on ait la majoration
\begin{equation}
\nu(T,x)\leq \theta{\|T\|}_h.
\end{equation}
Plus précisément, si l'on fixe un recouvrement relativement compact $(U^\alpha,V^\alpha,\psi^\alpha)_\alpha$ de $X$, alors il existe des compacts $K^\alpha\subset U^\alpha$ qui recouvrent $X$, et il existe un réel $r_0>0$, tels que pour tout $x\in K^\alpha$, on a 
\begin{align}
&B(\psi^\alpha(x),r_0)\subset \psi^\alpha(U^\alpha)\\
\text{et}\quad
&\nu(\psi^\alpha_*T,\psi^\alpha(x),r)\leq\theta{\|T\|}_h
\end{align}
pour tout $r\leq r_0$ et pour tout $(p,p)$-courant positif fermé $T$.
En particulier, si~$\mathcal{T}$ est un ensemble de courants positifs fermés de masse $\leq 1$, alors
\begin{equation}
\max_{\alpha}\,\sup_{x\in K^\alpha}\,\lim_{r\to 0}\,\sup_{T\in\mathcal{T}}\,\nu(\psi^\alpha_*T,\psi^\alpha(x),r)\leq\theta.
\end{equation}
\end{prop}

\begin{proof}
Comme $U^\alpha\Subset V^\alpha$, il existe $\delta>0$ tel que
\begin{equation}
\forall\alpha,\, \forall v\in{\rm T}U^\alpha, \quad \frac{1}{\delta}\|\psi^\alpha_*v\|
\leq {\|v\|}_h
\leq \delta\|\psi^\alpha_*v\|.
\end{equation}
En termes de $(1,1)$-formes positives, cela se traduit par
\begin{equation}
\frac{1}{\delta^2}\,\psi^\alpha_*\omega \leq \kappa \leq \delta^2\,\psi^\alpha_*\omega,
\end{equation}
où $\kappa=\Im(h)$ et $\omega$ désigne la forme de Kähler standard sur $\C^n$ ($n=\dim(X)$).

Soit $\rho_0>0$ tel que pour tout point $x\in X$, il existe $\alpha$ tel que $B_h(x,\rho_0)\subset U^\alpha$. Pour un tel $\alpha$, on a
\begin{equation}
B(\psi^\alpha(x),\rho_0/\delta)\subset\psi^\alpha(B_h(x,\rho_0))\subset\psi^\alpha(U^\alpha).
\end{equation}
On pose $r_0=\rho_0/\delta$ et
\begin{equation}
K^\alpha=\{x\in U^\alpha\,|\,B(\psi^\alpha(x),r_0)\subset\psi^\alpha(U^\alpha)\}.
\end{equation}
Par construction, ce sont des ensembles compacts qui recouvrent $X$.

Soit $x\in K^\alpha$. Pour tout $r\leq r_0$, on obtient :
\begin{align}
\nu(\psi^\alpha_*T,\psi_*^\alpha(x),r)
&\leq
\nu(\psi^\alpha_*T,\psi_*^\alpha(x),r_0)\\
&=\frac{1}{(\pi r_0^2)^{n-p}}\langle \psi^\alpha_*T,\chi_{B(\psi^\alpha(x),r_0)}\omega^{n-p}\rangle,\\
&\leq \left(\frac{\delta^2}{\pi r_0^2}\right)^{n-p}\langle T,\kappa^{n-p}\rangle\\
&= (n-p)!\left(\frac{\delta^2}{\pi r_0^2}\right)^{n-p} {\|T\|}_h.
\end{align}
On a donc la majoration voulue avec $\theta=(n-p)!\left(\delta^2/\pi r_0^2\right)^{n-p}$. On en déduit :
\begin{equation}
\nu(T,x)
=\nu(\psi^\alpha_*T,\psi^\alpha(x))
\leq\nu(\psi^\alpha_*T,\psi_\alpha(x),r)
\leq \theta {\|T\|}_h.
\end{equation}
\end{proof}

%%%%%%%%%%%%%%%%%%%%%%%%%%%%%%%%%%
% Ahlfors
%%%%%%%%%%%%%%%%%%%%%%%%%%%%%%%%%%

\section{Courants d'Ahlfors}

Soit $X$ une variété complexe de dimension $n$, munie d'une métrique hermitienne $h$. Soit $\varphi:\C\to X$ une courbe entière (non nécessairement injective). Pour $r>0$, on pose
\begin{align}
{\sf a}_\varphi(r)&=\int_{\D_r}{\|\varphi'\|}_h^2\,{\rm dvol}=\int_0^r\int_0^{2\pi}{\|\varphi'(t\e^{i\theta})\|}_h^2\,t\,{\rm d}\theta\,{\rm d}t\\
\text{et}\quad
\ell_\varphi(r)&=\int_{\partial\D_r}{\|\varphi'\|}_h\,{\rm d}\sigma_r=\int_0^{2\pi}{\|\varphi'(r\e^{i\theta})\|}_h\,r\,{\rm d}\theta,
\end{align}
où $\D_r$ désigne le disque de rayon $r$ dans $\C$ (centré en $0$), ${\rm dvol}$ la forme volume standard sur $\C$ et $\sigma_r$ la mesure de Lebesgue sur le cercle $\partial\D_r$. On pose aussi
\begin{align}
\ma_\varphi(r)&=\int_0^r {\sf a}_\varphi(t)\frac{{\rm d}t}{t}\\
\text{et}\quad
\ml_\varphi(r)&=\int_0^r \ell_\varphi(t)\frac{{\rm d}t}{t}
\end{align}

On note $S_{\varphi,r}$ le $(n-1,n-1)$-courant positif défini par :
\begin{equation}
\langle S_{\varphi,r},\theta\rangle = \frac{1}{\ma_\varphi(r)}\int_0^r \left[\int_{\D_t}\varphi^*\theta\right]\frac{{\rm d}t}{t}.
\end{equation}
Lorsque $\varphi$ est injectif, on a
\begin{align}
{\sf a}_\varphi(r)&=\aire(\varphi(\D_r)),\\
\ell_\varphi(r)&=\longueur(\varphi(\partial\D_r))\\
\text{et}\quad
S_{\varphi,r}&=\int_0^r\{\varphi(\D_t)\}\frac{{\rm d}t}{t} \left/ \int_0^r\aire(\varphi(\D_t))\frac{{\rm d}t}{t}\right..
\end{align}

\begin{rema}
Si l'on note $\kappa$ la $(1,1)$-forme positive associée à $h$, on a la formule
\begin{equation}
{\sf a}_\varphi(r)=\int_{\D_r}\varphi^*\kappa.
\end{equation}
En particulier, les courants $S_{\varphi,r}$ sont tous de masse $1$.
\end{rema}

\begin{lemm}[Ahlfors]\label{lemm-ahlfors}
Soit $X$ une variété complexe \emph{compacte}, munie d'une métrique hermitienne $h$. Pour toute courbe entière $\varphi:\C\to X$, on a :
\begin{equation}\label{lim-ahlfors}
\liminf_{r\to+\infty} \frac{\ml_\varphi(r)}{\ma_\varphi(r)}=0.
\end{equation}
Par conséquent, il existe des suites $r_n\to+\infty$ telles que les courants $S_{\varphi,r_n}$ convergent vers un courant positif fermé $S$. Par définition, une telle limite est un \emph{courant d'Ahlfors}\footnote{Certains auteurs les appellent courants de Nevanlinna, et réservent l'appellation courants d'Ahlfors aux limites fermées des courants non moyennés $\{\varphi(\D_r)\}/{\sf a}_\varphi(r)$.} associé à $\varphi$.
\end{lemm}

\begin{proof} Je reproduis ici la preuve de \cite{brunella}. Par Cauchy-Schwarz, on a, pour tout $r>0$, l'\emph{inégalité d'Ahlfors} :
\begin{equation}\label{ineg-ahlfors}
\begin{split}
\ell_\varphi(r)^2
&=
\left(\int_0^{2\pi}{\|\varphi'(r\e^{i\theta})\|}_h\,r\,{\rm d}\theta\right)^2\\
&\leq
2\pi r\left(\int_0^{2\pi}{\|\varphi'(r\e^{i\theta})\|}_h^2\,r\,{\rm d}\theta\right)
=
2\pi r{\sf a}_\varphi'(r).
\end{split}
\end{equation}
On en déduit, à nouveau par Cauchy-Schwarz,
\begin{align}
\ml_\varphi(r)-\ml_\varphi(1)
&\leq \int_1^r (2\pi {\sf a}_\varphi'(t))^{1/2}\frac{{\rm d}t}{t^{1/2}} \\
&\leq \left( \int_1^r 2\pi {\sf a}_\varphi'(t)\,{\rm d}t \, \int_1^r \frac{{\rm d} t}{t} \right)^{1/2} \\
&\leq \left( 2\pi\log(r) {\sf a}_\varphi(r) \right)^{1/2} \\
&= \left( 2\pi r\log(r) \ma_\varphi'(r) \right)^{1/2}.
\end{align}

Pour $\epsilon>0$ et $R\geq 1$, on pose $B(\epsilon,R)=\{r\geq R\,|\ml_\varphi(r)-\ml_\varphi(1)\geq \epsilon\ma_\varphi(r)\}$. Lorsque $r\in B(\epsilon,R)$, on a
\begin{equation}
\ma_\varphi(r)^2
\leq
\frac{2\pi}{\epsilon^2}\,r\log(r)\ma_\varphi'(r),
\end{equation}
et on en déduit que
\begin{equation}
\int_{B(\epsilon,R)}\frac{{\rm d}r}{r\log(r)}
\leq
\frac{2\pi}{\epsilon^2}\int_{B(\epsilon,R)}\frac{\ma_\varphi'(r)}{\ma_\varphi(r)^2}\,{\rm d}r\\
\leq
\frac{2\pi}{\epsilon^2\ma_\varphi(R)}
<+\infty.
\end{equation}
Comme $\int_R^{+\infty}\frac{{\rm d}r}{r\log(r)}=+\infty$, il existe donc $r\geq R$ tel que $r\notin B(\epsilon,R)$, \emph{i.e.}
\begin{equation}
\frac{\ml_\varphi(r)}{\ma_\varphi(r)}<\epsilon.
\end{equation}
Ceci montre l'égalité $(\ref{lim-ahlfors})$ pour la limite inférieure, car $\ma_\varphi(r)\to+\infty$.

Soit $r_n\to+\infty$ une suite telle que
\begin{equation}
\lim_{n\to+\infty}\frac{\ml_\varphi(r_n)}{\ma_\varphi(r_n)}=0.
\end{equation}
La suite de courants $\left(S_{\varphi,r_n}\right)_n$ est de masse constante égale à $1$, donc quitte à en extraire une sous-suite, on peut supposer qu'elle converge faiblement vers un courant positif $S$. Montrons que ce courant est fermé. En effet, pour toute $1$-forme $\theta$ sur $X$, on a, grâce au théorème de Stokes,
\begin{align}
\langle {\rm d}S,\theta\rangle
&= -\lim_{n\to+\infty}\langle S_{\varphi,r_n},{\rm d}\theta\rangle \\
&= \lim_{n\to+\infty}\frac{-1}{{\sf a}_\varphi(r_n)} \int_0^{r_n} \left[\int_{\D_{t}}\varphi^*{\rm d}\theta\right] \frac{{\rm d}t}{t} \\
&= \lim_{n\to+\infty}\frac{-1}{{\sf a}_\varphi(r_n)} \int_0^{r_n} \left[\int_{\partial\D_{t}}\varphi^*\theta\right] \frac{{\rm d}t}{t}.
\end{align}
Pour $x\in X$, on note $\ltrivert \theta(x) \rtrivert_h$ la norme d'opérateur de $\theta(x):{\rm T}_xX\to\R$ vis-à-vis de la métrique kählérienne sur ${\rm T}X$, et on note $\ltrivert\theta\rtrivert_{h}=\max_{x\in X}\ltrivert\theta(x)\rtrivert_h<+\infty$ (par compacité de $X$). Avec ces notations, on a alors
\begin{equation}
\left|\int_{\partial\D_{t}}\varphi^*\theta\right|
\leq
\int_{\partial\D_t} \ltrivert\theta(\varphi(z))\rtrivert_h {\|\varphi'\|}_h \,{\rm d}\sigma_t(z)
\leq
\ltrivert\theta\rtrivert_h \, \ell_\varphi(t),
\end{equation}
d'où
\begin{align}
\left|\frac{1}{{\sf a}_\varphi(r_n)}\int_{\partial\D_{r_n}}\varphi^*\theta\right|
&\leq
\ltrivert\theta\rtrivert_h \frac{\ml_\varphi(r_n)}{\ma_\varphi(r_n)}
\quad\underset{n\to+\infty}{\longrightarrow}\quad 0,
\end{align}
et donc $\langle{\rm d}S,\theta\rangle=0$, ce qui montre que le courant $S$ est fermé.
\end{proof}

À partir de maintenant, $X$ désigne une surface. Dans ce cas, la formule de Jensen (cf. \cite[p. 12]{demailly-gazette}) permet de démontrer (voir \cite{brunella}) :

\begin{theo}[Nevanlinna]\label{ahlfors-nef}
Soit $S$ un courant d'Ahlfors associé à une courbe entière $\varphi:\C\to X$, où $X$ est une surface complexe compacte. On suppose que l'image de $\varphi$ n'est contenue dans aucune courbe complexe compacte. Alors $[S]$ est une classe nef, ce qui signifie
\begin{equation}
[S]\cdot[C] \geq 0
\end{equation}
pour toute courbe complexe $C$. En particulier, $[S]^2\geq 0$.
\end{theo}

Lorsque la fonction ${\sf a}_\varphi$ est bornée, on montre en revanche que l'image de $\varphi$ est contenue dans une courbe rationnelle, ce qui était déjà connu dans le cas où $X$ est une surface projective \cite{demailly-gazette}, mais pas dans le cas général.

\begin{prop}\label{courbe-entiere-aire-finie}
Soit $X$ une variété complexe munie d'une métrique hermitienne $h$, et soit $\varphi:\C\to X$ une courbe entière telle que
\begin{equation}
\lim_{r\to+\infty}{\sf a}_\varphi(r)<+\infty.
\end{equation}
Alors $\varphi$ se prolonge par continuité en une application holomorphe
\begin{equation}
\tilde\varphi:\P^1(\C)\to X.
\end{equation}
\end{prop}

\begin{proof}
L'inégalité d'Ahlfors $(\ref{ineg-ahlfors})$ implique qu'il existe une suite $R_n\to+\infty$ telle que 
\begin{equation}
\ell_\varphi(R_n)\to 0.
\end{equation}
En effet, dans le cas contraire, il existerait $\epsilon>0$ et $R>0$ tel que $\ell_\varphi(r)>\epsilon$ pour $r>R$, puis par l'inégalité d'Ahlfors
\begin{equation}
+\infty
=
\frac{\epsilon^2}{2\pi}\int_{R}^{+\infty}\frac{{\rm d}r}{r}
\leq
\int_{R}^{+\infty}{\sf a}_\varphi'(r)
\leq
\lim_{r\to+\infty}{\sf a}_\varphi(r)
< +\infty.
\end{equation}

Pour $r>0$, posons $E_r=\varphi(\C\backslash\D_r)$, où $\D_r$ désigne le disque de rayon $r$. Les ensembles $E_r$ sont décroissants avec $r$, et on pose
\begin{equation}
\delta=\lim_{r\to+\infty}\diam_h(E_{r}).
\end{equation}
Il suffit de montrer que $\delta=0$. En effet : pour tout suite $z_n\to\infty$, la suite $\varphi(z_n)$ sera alors de Cauchy, donc convergente dans $X$, et toutes ces suites auront la même limite $L$, que l'on prend comme valeur de $\tilde\varphi$ en $\infty$.

Supposons donc que $\delta>0$. Pour $r<r'$, on note $\mathcal{C}_{r,r'}$ la couronne
\begin{equation}
\mathcal{C}_{r,r'}=\left\{z\in\C\,\big|\,r<|z|<r'\right\}.
\end{equation}
À $r$ fixé, on a
\begin{equation}
\diam_h\varphi(\mathcal{C}_{r,r'})\mathop{\longrightarrow}\limits_{r'\to+\infty}\diam_h\varphi(\C\backslash\D_r)\geq\delta,
\end{equation}
donc pour $r'$ assez grand, ce diamètre est supérieur à $2\delta/3$. Il existe donc $(r_n)$ et $(r'_n)$ deux suites extraites de $(R_n)$ telles que :
\begin{enumerate}
\item $r_n<r'_n<r_{n+1}$ ;
\item $\ell_\varphi(r_n)\leq \frac{\delta}{6}$ et $\ell_\varphi(r'_n)\leq \frac{\delta}{6}$ ;
\item $\diam_h(A_n)\geq\frac{2\delta}{3}$, où $A_n=\varphi(\mathcal{C}_{r_n,r'_n})$.
\end{enumerate}

\begin{enonce*}{Affirmation}
Il existe $x_n\in A_n$ tel que
\begin{equation}
\dist_h(x_n,\partial A_n)\geq \frac{\delta}{12}.
\end{equation}
\end{enonce*}

Dans le cas contraire, on aurait $A_n=B_n\cup B'_n$, avec 
\begin{align}
B_n
&=
\left\{x\in A_n\,|\,\dist_h(x,\varphi(\partial\D_{r_n}))<\frac{\delta}{12}\right\}\\
\text{et}\quad B'_n
&=
\left\{x\in A_n\,|\,\dist_h(x,\varphi(\partial\D_{r'_n}))<\frac{\delta}{12}\right\}.
\end{align}
Comme
$\max\left\{\diam_h(\varphi(\partial\D_{r_n})),\diam_h(\varphi(\partial\D_{r'_n}))\right\}
\leq\frac{\delta}{12}$
(d'après la condition $(2)$), cela impliquerait :
\begin{equation}
\max\left\{\diam_h(B_n),\diam_h(B'_n)\right\}
\leq \frac{\delta}{12}+\frac{\delta}{12}+\frac{\delta}{12}
= \frac{\delta}{4}.\end{equation}
Or par connexité de $A_n$, $B_n\cap B'_n\neq\emptyset$, et donc
\begin{equation}
\diam_h(A_n)\leq\diam_h(B_n)+\diam_h(B'_n)\leq\frac{\delta}{2},
\end{equation}
ce qui contredit la condition (3) et termine la preuve de l'affirmation.\\

Appliquons alors le corollaire \ref{coro-lelong} du théorème de Lelong, avec les courbes $A_n$ et les boules $B(x_n,\delta/12)$ : il existe $\eta>0$, indépendant de $n$, tel que
\begin{equation}
\aire_h(A_n)\geq \eta
\end{equation}
pour tout $n$. Cela implique :
\begin{equation}
\aire_h(\varphi(\C))\geq \sum_{n=0}^{+\infty} \aire_h(A_n) = +\infty,
\end{equation}
ce qui contredit l'hypothèse, et achève ainsi la preuve de la proposition \ref{courbe-entiere-aire-finie}.
\end{proof}

%%%%%%%%%%%%%%%%%%%%%%%%%%%%%%%%%%
% Singularités
%%%%%%%%%%%%%%%%%%%%%%%%%%%%%%%%%%

%\section{Généralisation aux espaces analytiques complexes}

%Dans la suite, on considère une surface kählérienne compacte $X$, munie d'un automorphisme loxodromique. On va avoir besoin de considérer des courants non seulement sur $X$, mais aussi sur la surface $X_0$ où l'on a contracté les courbes périodiques (cf. théorème \ref{theo-contr}). Or cette surface $X_0$ peut avoir des singularités : il faut donc étendre la notion de courant sur les espaces analytiques complexes. Ceci est fait dans \cite{demailly-ma}. Nous en résumons ci-dessous les grandes lignes.

%Par définition un espace analytique complexe $X$ est une \og variété\fg~complexe éventuellemnt singulière, dont un atlas est formé d'ouverts isomorphes au lieu d'annulation d'un nombre fini de fonctions holomorphes sur un ouvert~$\Omega$ de~$\C^N$. Une forme différentielle est définie localement comme la restriction d'une forme différentielle sur $\Omega$, et l'espace des formes différentielles est muni de la topologie la moins fine qui rend continus les morphismes de restriction. Par définition, un courant sur $X$ est une forme linéaire continue vis-à-vis de cette topologie. Les opérateurs ${\rm d}$, $\partial$ et $\b\partial$ sont définis comme dans le cas lisse, ainsi que la notion de courant positif.

%\textbf{La plupart des résultats de ce chapitre s'étendent aux espaces analyiques complexes.}
%%%%%%%%%%%%%%%%%%%%%%%%%%%%%%%%%%
% Chapter
%%%%%%%%%%%%%%%%%%%%%%%%%%%%%%%%%%
% Courant dilaté
%%%%%%%%%%%%%%%%%%%%%%%%%%%%%%%%%%

\chapter[Courants dilatés et contractés]{Courants dilatés et contractés par un automorphisme}\label{chap-dilat}

Le but de ce chapitre est de donner une démonstration complète du théorème suivant :

\begin{theo}[Cantat \cite{cantat-k3}, Dinh -- Sibony \cite{ds-green,ds-jalg}]\label{theo-cds}
Soit $X$ une surface complexe compacte kählérienne de forme de Kähler $\kappa$, et soit $f:X\to X$ un automorphisme loxodromique de $X$. Il existe un unique courant positif fermé $T^+_f$ tel que $[T^+_f]=\theta^+_f$\footnote{On rappelle que $\theta^+_f$ l'unique vecteur de $H^{1,1}(X;\R)$ qui vérifie $f^*\theta^+_f=\lambda(f)\,\theta^+_f$ et $\theta^+_f\cdot[\kappa]=1$.}. En particulier, $T^+_f$ est l'unique courant positif fermé de masse $1$ tel que
\begin{equation}
f^*T^+_f=\lambda(f)\, T^+_f.
\end{equation}

De plus, les potentiels locaux de $T^+_f$ sont $\alpha$-höldériens (donc continus)\footnote{Cette régularité ne dépend pas du choix des potentiels locaux, cf. remarque \ref{regularite-potentiel}.} pour tout $\alpha<\h(f)/\liap(f)=\log(\lambda(f))<\liap(f)$, où $\liap(g)$ désigne l'exposant de Liapunov topologique $\liap(g)=\lim_{n\rightarrow+\infty}\frac{1}{n}\log\ltrivert{\rm d}g^n\rtrivert_{\kappa,\infty}$ (qui ne dépend pas du choix de la métrique $\kappa$).
\end{theo}

En appliquant le même résultat à $f^{-1}$, on obtient l'existence et l'unicité d'un courant positif fermé $T^-_f=T^+_{f^{-1}}$, de classe $\theta^-_f$ et de masse $1$, tel que
\begin{equation}
f^*T^-_f=\frac{1}{\lambda(f)}\,T^-_f,
\end{equation}
avec potentiels höldériens également. On dit que le courant $T^+_f$ est \emph{dilaté} par $f$, alors que $T^-_f$ est \emph{contracté} par $f$. 

Dans le dernier paragraphe, on montre que si $f$ possède des courbes périodiques, alors les potentiels de $T^+_f$ et $T^-_f$ sont constants le long de ces courbes, quitte à réduire les ouverts sur lesquels ils sont définis. Cette propriété s'avèrera essentielle au chapitre \ref{chap-hyp} lors de la démonstration du théorème \ref{fat-hyp}.

La plupart du temps, on notera simplement $T^+$ et $T^-$ au lieu de $T^+_f$ et $T^-_f$. Dans la suite du chapitre, on note aussi $\lambda=\lambda(f)$.

%%%%%%%%%%%%%%%%%%%%%%%%%%%%%%%%%%
% Existence du courant dilaté
%%%%%%%%%%%%%%%%%%%%%%%%%%%%%%%%%%

\section{Existence du courant dilaté et continuité du potentiel}

Appelons $r$ la dimension de $H^{1,1}(X;\R)$, et fixons $r$ formes de Kähler $\omega_1,\cdots,\omega_r$ dont les classes forment une base de $H^{1,1}(X;\R)$. Notons $A$ la matrice de $\frac{1}{\lambda}f^*$ dans cette base, et $\underline\omega$ le vecteur colonne
\begin{equation}
\underline{\omega}=\begin{pmatrix}\omega_1\\\vdots\\\omega_r\end{pmatrix}.
\end{equation}
Comme les composantes du vecteur $\frac{1}{\lambda}f^*\underline{\omega}-A\underline{\omega}$ sont des formes exactes, le lemme du $\ddc$ global \ref{ddc-glob} donne l'existence de fonctions $\mathcal{C}^\infty$ $u_1,\cdots,u_r$ sur $X$ telles que, si l'on note $\underline{u}$ le vecteur $\,^t(u_1,\cdots,u_r)$,
\begin{equation}
\frac{1}{\lambda}f^*\underline{\omega}=A\underline{\omega}+\ddc \underline{u}.\end{equation}
On a ainsi, pour tout $n\in\N$ :
\begin{equation}
\frac{1}{\lambda^n}f^{*n}\underline{\omega}=A^n\underline{\omega}+\ddc \underbrace{\sum_{k=0}^{n-1}\frac{1}{\lambda^k}A^{n-k-1}\underline{u}\circ f^k}_{\underline{v}^n}.
\end{equation}

Comme toutes les valeurs propres de $A$ autres que $1$ sont de module $<1$, la suite $A^n$ converge vers la matrice d'une projection sur $\R\theta^+_f$ (c'est la projection de direction $(\theta^-_f)^\perp$). Ainsi, le premier membre de la somme converge vers un vecteur $\underline{\eta}$, dont les composantes $\eta_i$ sont des $(1,1)$ formes fermées.

Par ailleurs, la suite $(\underline{v}^n)_{n\in\N}$ est de Cauchy pour la topologie de la convergence uniforme. En effet, fixons $\epsilon>0$, et notons $N$ un entier positif tel que
\begin{equation}
\max\left\{
\sum_{k=N}^{+\infty}\frac{1}{\lambda^k}
~;~
\sup_{(n,p)\in\N^2,\,n\geq N}\ltrivert A^n-A^{n+p}\rtrivert
\right\}
\leq\epsilon,
\end{equation}
où $\ltrivert\cdot\rtrivert$ désigne la norme subordonnée à la norme infinie sur $\R^r$. Posons également
\begin{equation}
M=\sup_{n\in\N}\ltrivert A^n \rtrivert <+\infty.
\end{equation}
Pour une fonction continue $\underline v:X\to\R^n$, on note 
\begin{equation}
{\|\underline v\|}_\infty=\max_{x\in X}\|\underline v(x)\|,
\end{equation}
où $\|\underline v(x)\|$ désigne la norme infinie du vecteur $\underline v(x)\in\R^r$. Lorsque $n\geq 2N$ et $p\geq 0$, on a alors :
\begin{align}
\begin{split}
{\|\underline{v}^{n}-\underline{v}^{n+p}\|}_\infty\;
&\leq\;
\sum_{k=0}^{n-N+1}\frac{1}{\lambda^k}\ltrivert A^{n-k-1}-A^{n-k-1+p}\rtrivert{\|\underline{u}\|}_\infty\\
&\quad\quad\quad
+ \sum_{k=n-N}^{n-1}\frac{1}{\lambda^k}\ltrivert A^{n-k-1}\rtrivert  {\|\underline{u}\|}_\infty\\
&\quad\quad\quad\quad
+ \sum_{k=n-N}^{n+p-1}\frac{1}{\lambda^k}\ltrivert A^{n+p-k-1}\rtrivert {\|\underline{u}\|}_\infty\end{split}\\
&\leq\;
\sum_{k=0}^{n-N+1}\frac{\epsilon{\|\underline{u}\|}_\infty}{\lambda^k}
+
2\sum_{k=n-N}^{+\infty}\frac{M{\|\underline{u}\|}_\infty}{\lambda^k}
\\
&<\;
\epsilon{\|\underline{u}\|}_\infty\left(\sum_{k=0}^{+\infty}\frac{1}{\lambda^k}+2M\right).
\end{align}
On en déduit que $(\underline{v}^n)_{n\in\N}$ converge uniformément vers une fonction continue 
$\underline{v}^\infty=\,^t(v_1^\infty,\cdots,v_r^\infty)$.

Ainsi, les courants positifs fermés $\frac{1}{\lambda^n}f^{*n}\omega_i$ convergent faiblement vers les courants $T_i=\eta_i+\ddc v_i^\infty$. En particulier, on a $f^*T_i=\lambda T_i$. Ces courants $T_i$ sont positifs fermés, comme limites de courants positifs fermés. Par ailleurs, les potentiels locaux de $T_i$ sont continus, comme sommes de la fonction continue $v_i^\infty$ et des potentiels locaux de $\eta_i$, qui sont lisses. On a ainsi montré l'existence de courants positifs fermés à potentiels continus (de classe $\theta^+_f$ quitte à les multiplier par un nombre positif) qui sont dilatés par $f$.\\

Bien que cela ne soit pas nécessaire dans la suite du texte, nous allons montrer que les fonctions $v_i^\infty$ sont en fait höldériennes, ce qui montrera que les potentiels de $T_i$ le sont également (voir aussi \cite{briend-these}, \cite{ds-green} et \cite{cantat-leborgne} par exemple). Pour cela, soit $\alpha>0$ tel que
\begin{equation}
\alpha<\frac{\log(\lambda)}{\liap(f)}=\frac{\h(f)}{\liap(f)}\leq 1.
\end{equation}

Posons
\begin{equation}
\epsilon=\frac{\log(\lambda)-\alpha\liap(f)}{2\alpha}>0.
\end{equation}
Par définition de $\liap(f)$, il existe $n_0\in\N^*$ tel que pour tout $n\geq n_0$, on ait :
\begin{equation}
\frac{\log\ltrivert{\rm d}f^n\rtrivert_{\kappa,\infty}}{n}\leq\liap(f)+\epsilon,
\end{equation}
ce qui se réécrit
\begin{equation}
\frac{{\left(\ltrivert{\rm d}f^n\rtrivert_{\kappa,\infty}\right)}^\alpha}{\lambda^n}\leq \mu^n,
\end{equation}
où l'on a posé
\begin{equation}
\mu=\exp\left(\frac{\alpha\liap(f)-\log(\lambda)}{2}\right)<1.
\end{equation}

Les fonctions $u_i$ sont $\mathcal{C}^1$, donc $\alpha$-höldériennes. Il existe alors $C>0$ tel que
\begin{equation}
|u_i(x)-u_i(y)|\leq C\dist_\kappa(x,y)^\alpha
\end{equation}
pour tout $i$ et pour tous $x$ et $y$ dans $X$.

Soient $n\in\N$ et $x, y\in X$. On a alors :
\begin{align}
{\|\underline{v}^n(x)-\underline{v}^n(y)\|}_\infty
&\leq
\sum_{k=0}^{n-1}\frac{\ltrivert A^{n-k-1}\rtrivert {\|\underline{u}\circ f^k(x)-\underline{u}\circ f^k(y)\|}_\infty}{\lambda^k}\\
&\leq
\sum_{k=0}^{n-1}\frac{MC\dist_\kappa(f^k(x),f^k(y))^\alpha}{\lambda^k}\\
&\leq
MC\sum_{k=0}^{n-1}\frac{\left(\ltrivert{\rm d}f^k\rtrivert_{\kappa,\infty}\dist_\kappa(x,y)\right)^\alpha}{\lambda^k} \\
&\leq
\underbrace{
MC
\left(
\sum_{k=0}^{n_0-1}\frac{{\left(\ltrivert{\rm d}f^k\rtrivert_{\kappa,\infty}\right)}^\alpha}{\lambda^k}
+
\sum_{k=n_0}^{n-1}\mu^k
\right)
}
_{C'<+\infty}
\dist_\kappa(x,y)^\alpha.
\end{align}
Par passage à la limite, on obtient donc
\begin{equation}
{\|\underline{v}^\infty(x)-\underline{v}^\infty(y)\|}_\infty\leq C'\dist_\kappa(x,y)^\alpha,
\end{equation}
ce qui montre que les fonctions $v_i^\infty$ sont $\alpha$-höldériennes.

%%%%%%%%%%%%%%%%%%%%%%%%%%%%%%%%%%
% Unicité du courant dilaté
%%%%%%%%%%%%%%%%%%%%%%%%%%%%%%%%%%

\section{Unicité du courant dilaté}

Nous venons de voir qu'il existe un courant positif fermé $T^+$ tel que $[T^+]=\theta^+_f$, $f^*T^+=\lambda T^+$, et qui a des potentiels locaux continus. Supposons qu'il existe un courant positif fermé $T\neq T^+$ tel que $[T]=[T^+]$. D'après le lemme du $\ddc$ global \ref{ddc-glob}, il existe une distribution $g$ telle que
\begin{equation}
T=T^+ +\ddc g.
\end{equation}
Localement, cette distribution est donnée par la différence entre un potentiel pluri-sous-harmonique de $T$ et un potentiel continu de $T^+$. On voit donc que $g$ est en fait une fonction semi-continue supérieurement. Par conséquent, cette fonction atteint son maximum, et on va supposer que
\begin{equation}
\max_{x\in X}g(x)=0.
\end{equation}

Pour tout $n\in\N$, on pose $T_n=\lambda^nf_*^nT$ et $g_n=\lambda^ng\circ f^{-n}$. Puisque $\lambda^nf_*T^+=T^+$, on a
\begin{equation}
T_n=T^++\ddc g_n.
\end{equation}
De plus,
\begin{equation}\label{maxg_n}
\max_{x\in X}g_n(x)=0.
\end{equation}

Fixons un recouvrement relativement compact $(U^\alpha,V^\alpha,\psi^\alpha)_{\alpha\in A}$ de $X$ (cf. définition \ref{rec-compact}) tel que les courants $T^+$ et $T_n$ admettent des potentiels $u_+^\alpha$ et $u_n^\alpha$ sur chaque ouvert $V^\alpha$, avec
\begin{equation}\label{g_n}
{g_n}_{|V^\alpha}=u_n^\alpha-u_+^\alpha
\end{equation}
pour tous $n$ et $\alpha$. Comme les fonctions $u_+^\alpha$ sont bornées sur $U^\alpha$ (car $\b U^\alpha\subset V^\alpha$ est compact et $u_+^\alpha$ est continue), on peut également supposer que les potentiels $u_+^\alpha$ sont négatifs sur $U^\alpha$, et donc
\begin{equation}
{u_n^\alpha}_{|U^\alpha}\leq{g_n}_{|U^\alpha}\leq 0
\end{equation}
pour tout $n\in\N$.

\begin{lemm}
Il existe une suite strictement croissante $(n_k)_{k\in\N}\in\N^\N$ telle que pour tout $\alpha\in A$, la suite $(u_{n_k}^\alpha)_{k\in\N}$ converge dans $L^1_{\rm loc}(U^\alpha)$.
\end{lemm}

\begin{proof}
On applique \cite[théorème 3.2.12]{hormander} aux fonctions $u_n^\alpha$, qui sont uniformément majorées par $0$ sur $U^\alpha$ : pour tout $\alpha\in A$,
\begin{enumerate}
\item
ou bien $(u_n^\alpha)_{n\in\N}$ converge vers $-\infty$ localement uniformément sur $U^\alpha$,
\item
ou bien il existe une sous-suite $\left(u_{n_k}\right)_{k\in\N}$ qui converge dans $L^1_{\rm loc}(U^\alpha)$.
\end{enumerate}
Si l'on est dans le premier cas pour un certain $\alpha$, alors on est dans le même cas pour tous les $\beta\in A$ (car $u_n^\alpha-u_n^\beta=u_+^\alpha-u_+^\beta$ ne dépend pas de $n$), et donc d'après $(\ref{g_n})$ la suite $g_n$ converge uniformément vers $-\infty$ sur $X$, ce qui contredit $(\ref{maxg_n})$.

On est donc dans le second cas pour tous les $\alpha\in A$. On commence par prendre une première suite extraite $(u^{\alpha_1}_{n_k})_{k\in\N}$ qui converge dans $L^1_{\rm loc}(U^{\alpha_1})$. Les suites $(u_{n_k}^\alpha)_{k\in\N}$ vérifient les mêmes propriétés que les suites de départ, donc on peut faire une deuxième extraction de manière à avoir des suites qui convergent dans $U^{\alpha_1}$ et $U^{\alpha_2}$, etc.
\end{proof}

D'après la proposition \ref{controle-masse}, il existe des compacts $K^\alpha\subset U^\alpha$ qui recouvrent $X$, et il existe $\theta>0$, tels que
\begin{equation}
\max_{\alpha}\,\sup_{x\in K^\alpha}\,\lim_{r\to 0}\,\sup_{n\in\N}\,\nu(\psi^\alpha_*T_n,\psi^\alpha(x),r)\leq\theta.
\end{equation}
Fixons des voisinages ouverts $W^\alpha$ de $K^\alpha$, avec $W^\alpha\Subset U^\alpha$. D'après le théorème de Zeriahi \ref{theo-zeriahi}, il existe des constantes $C_1>0$ et $C_2>0$ telles que
\begin{equation}
\log\left(\vol_\kappa\{x\in K^\alpha\,|\,u_n^\alpha(x)<-t\}\right)
\leq
C_1\int_{W^\alpha} |u_n^\alpha|\,{\rm dvol_\kappa}+C_2-\frac{t}{\theta}
\end{equation}
pour tous $\alpha\in A$, $n\in\N$ et $t\in\R_+^*$. Comme les suites $(u_{n_k}^\alpha)_{k\in\N}$ convergent dans $L^1_{\rm loc}(U^\alpha)$, donc dans $L^1(W^\alpha)$, il existe en fait une constante $C>0$ telle que
\begin{equation}
\log\left(\vol_\kappa\{x\in K^\alpha\,|\,u_{n_k}^\alpha(x)<-t\}\right)
\leq
C-\frac{t}{\theta}
\end{equation}
pour tous $\alpha\in A$, $k\in\N$ et $t\in\R_+^*$.

Comme on a supposé $T\neq T^+$, alors la fonction $g$ n'est pas presque partout nulle, et il existe $t>0$ tels que
\begin{equation}
\epsilon:=\vol_\kappa\{x\in X\,|\,g(x)< -t\} >0.
\end{equation}
En utilisant que $g_n\circ f^n=\lambda^n g$, on obtient, pour tout $k\in\N$ :
\begin{align}
\epsilon\;
&=\;
\vol_\kappa\left(f^{-n_k}\{x\in X\,|\,g_{n_k}\leq -\lambda^{n_k} t\}\right)\\
&\leq\;
\ltrivert{\rm d}f^{-{n_k}}\rtrivert_{\kappa,\infty}\vol_\kappa\{x\in X\,|\,g_{n_k}(x)\leq-\lambda^{n_k}t\}\\
&\leq\;
\left(\ltrivert{\rm d}f^{-1}\rtrivert_{\kappa,\infty}\right)^{n_k}\sum_\alpha\vol_\kappa\{x\in K^\alpha\,|\,u_{n_k}^\alpha(x)<-\lambda^{n_k}t\}\\
&\leq\;
{\rm card}(A)\,\exp\left(C-\lambda^{n_k}t+n_k\log\ltrivert{\rm d}f^{-1}\rtrivert_{\kappa,\infty}\right).
\end{align}
Comme le membre de droite tend vers $0$ lorsque $n\to+\infty$, on a obtenu une contradiction. Ceci achève la preuve du théorème \ref{theo-cds}.
\qed

%%%%%%%%%%%%%%%%%%%%%%%%%%%%%%%%%%
% Avec courbes périodiques
%%%%%%%%%%%%%%%%%%%%%%%%%%%%%%%%%%

\section{Courants dilatés et courbes périodiques}\label{section-dilat+per}

\begin{theo}\label{dilat+per}
Soit $f$ un automorphisme loxodromique d'une surface kählérienne compacte $X$. Il existe un recouvrement de $X$ par des ouverts $V^\alpha$ tels que
\begin{enumerate}
\item $T^+_{|V^\alpha}=\ddc u^\alpha$, où $u^\alpha$ est une fonction continue sur $V^\alpha$ ;
\item le potentiel $u^\alpha$ est constant sur ${\sf CP}(f)\cap V^\alpha$, où ${\sf CP}(f)$ désigne l'union des courbes périodiques de $f$.
\end{enumerate}
\end{theo}

Si l'on note $\pi:X\to X_0$ le morphisme birationnel de contraction des courbes périodiques (cf. théorème \ref{theo-contr}), les potentiels $u^\alpha$ sont alors constants sur les fibres de $\pi$, et induisent donc des fonctions continues $u^\alpha_0$ sur $V_0^\alpha:=\pi(V^\alpha)$.

\begin{rema}
On a la même propriété pour le courant $T^-_f$, en remplaçant $f$ par $f^{-1}$.
\end{rema}

\begin{proof}
D'après le premier point du théorème \ref{theo-contr}, il existe un morphisme birationnel ${\eta:X\to X_1}$, avec $X_1$ lisse, tel que $f$ (quitte à le remplacer par un itéré, ce qui ne change pas le courant $T^+_f$) induit un automorphisme~${f_1:X_1\to X_1}$, pour lequel chaque courbe périodique connexe est
\begin{itemize}
\item soit un arbre de courbes rationnelles lisses ;
\item soit une courbe (réductible) de genre $1$ qui est de l'un des types suivants :
\begin{itemize}
\item une courbe elliptique lisse,
\item une courbe rationnelle avec une singularité nodale ou cuspidale,
\item une union de deux courbes rationnelles qui se coupent en deux points (confondus ou non),
\item trois courbes rationnelles qui s'intersectent en un unique point
\item ou un cycle de $n\geq 3$ courbes rationnelles lisses.
\end{itemize}
\end{itemize}

Nous affirmons qu'il suffit de montrer le théorème sur la surface $X_1$, pour l'automorphisme $f_1$. En effet, l'unicité du courant dilaté donne $T^+_f=\eta^*T^+_{f_1}$, à une constante multiplicative près. Si $u_1^\alpha$ est un potentiel local pour $T^+_{f_1}$ sur~$V_1^\alpha$, alors $u^\alpha:=u_1^\alpha\circ\eta$ sera un potentiel pour~$T^+_f$ sur~${V^\alpha:=\eta^{-1}(V_1^\alpha)}$. Si l'on montre que $u^\alpha_1$ est constant sur ${\sf CP}(f_1)\cap V_1^\alpha$, alors~$u^\alpha$ sera constant sur~${\eta^{-1}({\sf CP}(f_1)\cap V_1^\alpha)={\sf CP}(f)\cap V^\alpha}$.

On suppose donc dans la suite que chaque courbe périodique connexe est de l'un des types décrits plus haut. Quitte à faire un nombre fini d'éclatements $f$-équivariants (au plus trois), on peut alors supposer que chaque courbe périodique connexe est
\begin{itemize}
\item une courbe elliptique lisse,
\item un arbre de courbes rationnelles lisses
\item ou un cycle d'au moins trois courbes rationnelles lisses.
\end{itemize}
En effet, si l'on note $\beta:\tilde X\to X$ une telle suite d'éclatements, et $\tilde f$ l'automorphisme induit sur $\tilde X$, le courant $T^+_{\tilde f}$ a pour potentiels localement $u^\alpha\circ\beta$, si bien qu'il suffit de montrer la propriété $(2)$ sur $\tilde X$.

Dans la suite, on suppose donc, ces modifications étant faites, que pour tout point $x\in X$ sur une courbe périodique, l'union $E$ des courbes périodiques irréductibles contenant $x$ est de la forme suivante :
\begin{itemize}
\item $E$ est une courbe elliptique lisse,
\item $E$ est une courbe rationnelle lisse
\item ou $E$ est l'union de deux courbes rationnelles qui se coupent transversalement en $x$ (et uniquement en $x$).
\end{itemize}
Nous allons montrer qu'un tel point $x$ possède un voisinage comme souhaité, en montrant que $E$ possède un voisinage ouvert $V$ sur lequel $T^+_f$ admet un potentiel, qui est continu et constant sur $E$. Il suffira ensuite de réduire la taille de ce voisinage de telle sorte qu'il n'intersecte pas les autres courbes périodiques.

Pour construire un tel voisinage~$V$, on considère le morphisme birationnel~${\pi_E:X\to X_E}$ qui contracte $E$ sur un point~$y$ (a priori $X_E$ est singulière en $y$), et on prend pour $V$ la préimage par~$\pi_E$ d'une petite boule autour de $y$. Lorsque cette boule est suffisamment petite, on a
\begin{equation}\label{km}
H^1(V,\O)\simeq H^1(E,\O_E)
\end{equation}
par \cite[proposition 4.45]{kollar-mori}. Par ailleurs, cet espace $H^1(E,\O_E)$ est nul lorsque $E$ est constitué d'une ou deux courbes rationnelles lisses, et isomorphe à~$\C$ lorsque $E$ est une courbe elliptique lisse (voir par exemple \cite{bpvdv}).

Dans la suite, on note $T$ le courant $T^+_f$ restreint à $V$. Ce courant est localement défini par des potentiels continus $u^\alpha$ sur un recouvrement~${(V^\alpha)}_\alpha$ de~$V$ :
\begin{equation}
T_{|V^\alpha}=\ddc u^\alpha,
\end{equation}
où le potentiel $u^\alpha$ est uniquement déterminé modulo l'addition d'une fonction pluri-harmonique. Notons $\mathcal{C}^0$ le faisceau des fonctions continues, et~$\mathcal{PH}$ celui des fonctions pluri-harmoniques. Alors $T$ peut être vu comme un élément de $H^0(V,\mathcal{C}^0/\mathcal{PH})$. La suite exacte courte de faisceaux
\begin{equation}
\xymatrix{
0 \ar[r] & \mathcal{PH} \ar[r] & \mathcal{C}^0 \ar[r] & \mathcal{C}^0/\mathcal{PH} \ar[r] & 0
}
\end{equation}
fournit la suite exacte longue
\begin{equation}\label{suite-ex}
\xymatrix{
0 \ar[r] & H^0(V,\mathcal{PH}) \ar[r] & H^0(V,\mathcal{C}^0) \ar[r] & H^0(V,\mathcal{C}^0/\mathcal{PH}) \ar@/_/[lld]^{c_1} \\
& H^1(V,\mathcal{PH}) \ar@{.>}[r] &
}
\end{equation}
où $c_1$ est le morphisme connectant. Le but est de montrer que
\begin{equation}
T\in H^0(V,\mathcal{C}^0/\mathcal{PH})
\end{equation}
provient d'un élément de $H^0(V,\mathcal{C}^0)$, ce qui équivaut à $c_1(T)=0$.

Pour cela on considère la suite exacte courte de faisceaux
\begin{equation}
\xymatrix{
0 \ar[r] & \R \ar[r]^-{j} & \O \ar[r]^-{\Re} & \mathcal{PH} \ar[r] & 0\\
},
\end{equation}
où $j$ désigne la multiplication par $2i \pi$. Le fait que cette suite soit exacte correspond au fait que toute fonction pluri-harmonique est localement partie réelle d'une fonction holomorphe, unique à l'addition près d'une constante imaginaire pure. On obtient la suite exacte longue
\begin{equation}
\xymatrix{
0 \ar[r] & H^0(V,\R) \ar[r]^-{j_0} & H^0(V,\O) \ar[r]^-{\Re_0} & H^0(V,\mathcal{PH}) \ar@/_/[lld]^{\delta_1}\\
& H^1(V,\R) \ar[r]^-{j_1} & H^1(V,\O) \ar[r]^-{\Re_1} & H^1(V,\mathcal{PH}) \ar@/_/[lld]^{\delta_2}\\
& H^2(V,\R) \ar@{.>}[r] & 
}
\end{equation}

\begin{lemm}
Le morphisme $\Re_0$ est surjectif, quitte à réduire la taille de~$V$.
\end{lemm}

\begin{proof}
On peut supposer que $V$ est recouvert par des ouverts $V^\alpha$ tels que pour tous $\alpha$ et $\beta$ :
\begin{itemize}
\item $V^\alpha$ est simplement connexe ;
\item $E\cap V^\alpha$ est non vide et connexe ;
\item Si $V^\alpha\cap V^\beta\neq\emptyset$, alors $E$ rencontre cet ouvert.
\end{itemize}

Soit $u$ une fonction pluri-harmonique sur $V$. Cette fonction est harmonique sur $E$, donc constante, et on peut supposer que cette constante est nulle. Comme $V^\alpha$ est simplement connexe, il existe une fonction $h^\alpha\in\O(V^\alpha)$ telle que $u_{|V^\alpha}=\Re(h^\alpha)$. Sur $E\cap V^\alpha$, $h^\alpha$ est alors une fonction holomorphe de partie réelle nulle, donc une constante imaginaire pure, que l'on peut choisir nulle. Sur les intersections $V^\alpha\cap V^\beta$ qui sont non vides, la fonction $h^\alpha-h^\beta$ est une fonction pluri-harmonique, de partie réelle $u-u=0$ : c'est donc une constante imaginaire pure, qui est nulle car $h^\alpha=h^\beta=0$ sur $E\cap V^\alpha\cap V^\beta\neq\emptyset$. On peut donc recoller les fonctions $h^\alpha$ en une fonction holomorphe $h$ définie globalement sur~$V$, et telle que $u=\Re(h)$.
\end{proof}

On en déduit que $\delta_1$ est le morphisme nul, et que $j_1$ est injectif. Lorsque $E$ est une courbe elliptique lisse, on a \begin{itemize}
\item $H^1(V,\O)\simeq H^1(E,\O_E)\simeq\C$ par $(\ref{km})$, \item $H^1(V;\R)\simeq H^1(E,\R)\simeq\R^2$ car $V$ est contractible sur $E$.
\end{itemize}
Comme $j_1$ est injectif, on en déduit qu'il est aussi surjectif, car il s'agit d'une application linéaire entre deux espaces vectoriels réels de même dimension. Dans le cas où $E$ est formé d'une ou deux courbes rationnelles, $H^1(V,\O)=0$, si bien que $j_1$ est surjectif dans tous les cas.

Ceci montre que $\Re_1$ est le morphisme nul, donc $\delta_2$ est injectif. Pour montrer que $c_1(T)\in H^1(V,\mathcal{PH})$ est nul, il suffit donc de montrer $\delta_2\circ c_1(T)=0$.

\begin{lemm}
L'image $\delta_2\circ c_1(T)\in H^2(V,\R)$ correspond à la classe de cohomologie du courant $T$ via l'isomorphisme de De Rham.
\end{lemm}
\begin{proof}
Commençons par voir comment est construit $\delta_2\circ c_1(T)$. Le $0$-cocycle $T\in H^0(V,\mathcal{C}^0/\mathcal{PH})$ est donné localement par des fonctions continues~$u^\alpha\in\mathcal{C}^0(U^\alpha)$, avec $u^{\alpha\beta}:=u^\beta-u^\alpha$ pluri-harmonique sur les intersections~${U^{\alpha\beta}=U^\alpha\cap U^\beta}$. Ces fonctions $u^{\alpha\beta}$ définissent un $1$-cocycle qui correspond à $c_1(T)\in H^1(V,\mathcal{PH})$, modulo un cobord.

On peut supposer, quitte à raffiner le recouvrement, que les intersections~$U^{\alpha\beta}$ sont simplement connexes, et que les tri-intersections $U^{\alpha\beta\gamma}:=U^\alpha\cap U^\beta\cap U^\gamma$ sont connexes. Ainsi $u^{\alpha\beta}\in\mathcal{PH}(U^{\alpha\beta})$ est la partie réelle d'une fonction holomorphe $h^{\alpha\beta}$ sur $U^{\alpha\beta}$. Sur les tri-intersections $U^{\alpha\beta\gamma}$, la fonction $h^{\beta\gamma}-h^{\alpha\gamma}+h^{\alpha\beta}$ est holomorphe de partie réelle nulle, donc il existe une constante réelle $C^{\alpha\beta\gamma}$ telle que 
\begin{equation}
h^{\beta\gamma}-h^{\alpha\gamma}+h^{\alpha\beta} = 2i\pi\, C^{\alpha\beta\gamma}.
\end{equation}
Par construction, $(C^{\alpha\beta\gamma})_{\alpha\beta\gamma}$ est le $2$-cocycle qui correspond à $\delta_2\circ c_1(T)$.

Voyons maintenant comment ce $2$-cocycle est relié à la classe de cohomologie de De Rham du courant $T$. La première étape dans l'isomorphisme de De Rham est la correspondance
\begin{equation}
H^2_{\rm{DR}}(V) \simeq H^1(V,\mathcal{Z}^1),
\end{equation}
où $\mathcal{Z}^1$ désigne le faisceau des $1$-formes différentielles fermées. Comme $T$ est donné localement par la $1$-forme exate $\ddc u^\alpha$, le $1$-cocycle correspondant à~$[T]$ dans~$H^1(V,\mathcal{Z}^1)$ est donné par 
\begin{equation}
{\rm d^c} \,u^\beta - {\rm d^c}\, u^\alpha = {\rm d^c}\,u^{\alpha\beta}\in\mathcal{Z}^1(U^{\alpha\beta}).
\end{equation}
Or $u^{\alpha\beta}=\Re(h^{\alpha\beta})$, donc
\begin{align}
{\rm d^c}\,u^{\alpha\beta} &= \frac{1}{2i\pi}\,(\partial - \b\partial)\, u^{\alpha\beta}\\
&= \frac{1}{4i\pi} \left( \partial \,h^{\alpha\beta} - \b\partial \, \b h^{\alpha\beta} \right)\\
&= {\rm d}\left( \frac{1}{2\pi} \, \Im(h^{\alpha\beta}) \right).
\end{align}
Ainsi, le $2$-cocycle correspondant dans l'isomorphisme de De Rham
\begin{equation}
H^1(V,\mathcal{Z}^1) \simeq H^2(V,\R)
\end{equation}
est donné par
\begin{equation}
\frac{1}{2\pi}\, \Im(h^{\beta\gamma} - h^{\alpha\gamma} +  h^{\alpha\beta}) 
= C^{\alpha\beta\gamma}.
\end{equation}
Le lemme est donc démontré.
\end{proof}

Pour terminer la démonstration du théorème \ref{dilat+per}, il s'agit maintenant de montrer que $[T]\in H^2(V,\R)$ est nul. Notons $\phi_{[T]}$ la forme linéaire sur $H_2(V,\R)$ qui correspond à $[T]$ par dualité. Par construction du voisinage $V$, $H_2(V,\R)$ est engendré par les classes d'homologie $[E_i]$, où les $E_i$ sont les composantes de la courbe $E$ (il y en a une ou deux). Comme les $E_i$ sont des courbes périodiques, la proposition \ref{prop-per} implique
\begin{equation} 
\phi_{[T]}([E_i]) = [T^+_f] \cdot [E_i] = 0.
\end{equation}
Ceci montre que la forme linéaire $\phi_{[T]}$ est nulle, et donc $[T]=0$.

Récapitulons. On a montré que $\delta_2\circ c_1(T)=0$, ce qui implique par injectivité que $c_1(T)=0$. Par exactitude de la suite $(\ref{suite-ex})$, ceci implique que $T$ est donné par un potentiel continu $u$ globalement défini sur $V$. De plus, $\ddc u_{|E_i}$ définit une $(1,1)$-forme positive sur $E_i$, qui est d'intégrale nulle car
\begin{equation}
\int_{E_i} \ddc u = [T^+_f]\cdot [E_i] = 0.
\end{equation}
On en déduit que $\ddc u_{|E_i}=0$, et $u_{|E_i}$ est une fonction harmonique, donc constante. Comme $E=\cup_i E_i$ est connexe, $u$ est constante sur $E$.
\end{proof}

%%%%%%%%%%%%%%%%%%%%%%%%%%%%%%%%%%
%%%%%%%%%%%%%%%%%%%%%%%%%%%%%%%%%%
% PARTIE II
%%%%%%%%%%%%%%%%%%%%%%%%%%%%%%%%%%%%%%%%%%%%%%%%%%%%%%%%%%%%%%%%%%%%
% VOLUMES
%%%%%%%%%%%%%%%%%%%%%%%%%%%%%%%%%%%%%%%%%%%%%%%%%%%%%%%%%%%%%%%%%%%%

\part{Entropie et comparaison de volumes}\label{part2}

\selectlanguage{english}

%%%%%%%%%%%%%%%%%%%%%%%%%%%%%%%%%%%%%%%%%%%%%%%%%%%%%%%%%%%%%%%%%%%%
% An upper bound for the volume of a real algebraic subvariety
%%%%%%%%%%%%%%%%%%%%%%%%%%%%%%%%%%%%%%%%%%%%%%%%%%%%%%%%%%%%%%%%%%%%

\chapter[Cauchy-Crofton Formula]{Cauchy-Crofton Formula and an Upper Bound for Real Volumes}\label{chap-upper-bound}

Let $X$ be an arbitrary $d$-dimensional real algebraic variety\footnote{Recall that with our conventions, $X$ is projective, smooth, irreducible, and has non empty real locus $X(\R)$.}. We fix a Riemannian metric on the complex manifold $X(\C)$. When $Y$ is a $k$-dimensional real algebraic subvariety of $X$, we denote by
\begin{itemize}
\item $\vol_\R(Y)$ the $k$-volume of $Y(\R)$,
\item $\vol_\C(Y)$ the $2k$-volume of $Y(\C)$,
\end{itemize}
both with respect to the Riemannian metric. The goal of what follow is to compare $\vol_\R(Y)$ and $\vol_\C(Y)$ for arbitrary subvarieties. In this chapter, we show that $\vol_\R(Y)$ is bounded from above by $\vol_\C(Y)$, up to a positive constant. We begin with the case of projective spaces, where we have a Cauchy-Crofton formula. This easily implies the case of a projective variety.

%%%%%%%%%%%%%%%%%%%%%%%%%%%%%%%%%%
% A Cauchy-Crofton formula
%%%%%%%%%%%%%%%%%%%%%%%%%%%%%%%%%%

\section{Cauchy-Crofton formula}

What is described here can be found in the manuscript \cite{christol}, except for the proof of Lemma \ref{deg-2}.\\

There is a classical way, in integral geometry (see \cite{santalo}), to compute the volume of a $k$-dimensional submanifold $N$ of $\P^d(\R)$, just by taking the mean of the number of intersections between $N$ and $k$-codimensional projective subspaces. In order to make it work, we choose both a metric on $\P^d(\R)$, and a probability measure on the Grassmannian $\G(d-k,d)$ (i.e., the real algebraic variety of $(d-k)$-dimensional projective subspaces of $\P^d(\R)$), which are invariant under the action of the orthogonal group ${\rm O}(d+1)$. Namely, we set the Fubini--Study metric on $\P^d(\R)$, and the probability $\mu_{d-k,d}$ on $\G(d-k,d)$ induced by the Haar measure on ${\rm O}(d+1)$ (the Grassmannian is homogeneous with respect to this group). Now we can state the formula.

\begin{theo}[Cauchy--Crofton formula]
Let $N$ be a $k$-dimensional submanifold of $\P^d(\R)$. With respect to the Fubini--Study metric,
\begin{equation}\label{formule-crofton}
\vol(N)=\vol(\P^k(\R))\int_{\Pi\in\G(d-k,d)}\sharp(N\cap \Pi)\,{\rm d}\mu_{d-k,d}(\Pi).
\end{equation}
\end{theo}

\begin{proof}[Sketch of proof]
It is enough to check the formula when $N$ is a $k$-simplex, and then to approach an arbitrary submanifold by such simplices. A similar formula in the Euclidean context can be found in \cite[p. 245 (14.70)]{santalo}.
\end{proof}

%%%%%%%%%%%%%%%%%%%%%%%%%%%%%%%%%%
% Upper bound
%%%%%%%%%%%%%%%%%%%%%%%%%%%%%%%%%%

\section{Un upper bound for real volumes}

% Projective space

\subsection{Case of projective space}

We set the Fubiny-Study metric on~$\P^d(\C)$. Now Wirtinger equality (cf. \cite[p. 31]{griffiths-harris}) implies, for any $k$-dimensional subvariety $Y$,
\begin{equation}
\vol_\C(Y)=\deg(Y)\vol_\C(\P^k),
\end{equation}
where $\deg(Y)$ is the degree of $Y$ as a subvariety of $\P^d$, \emph{i.e.} the number of intersections with a general $(d-k)$-plane. Combining this with Cauchy-Crofton formula, we get:

\begin{coro}\label{1.11}
Let $Y$ be a real $k$-dimensional algebraic subvariety of $\P^d_{\R}$. With respect to the Fubini--Study metric, 
\begin{equation}\label{vol-reel}
\vol_\R(Y)\leq\deg(Y)\vol_\R(\P^k)=\frac{\vol_\R(\P^k)}{\vol_\C(\P^k)}\vol_\C(Y),
\end{equation}
with equality if and only if $Y$ is a union of $\deg(Y)$ real projective subspaces.
\end{coro}

\begin{proof}[Proof of Corollary \ref{1.11}]
Observe that $\sharp(Y(\R)\cap\Pi)\leq\deg(Y)$ for all $\Pi$ in~$\G(d-k,d)$. Thus we get Inequality (\ref{vol-reel})  using the Cauchy--Crofton formula.

The equality is obviously achieved when $Y$ is the union of $\deg(Y)$ projective subspaces. Now we prove that this last condition is necessary.

\begin{lemm}\label{deg-2}
Let $Y$ be a geometrically irreducible real $k$-dimensional algebraic subvariety of $\P^d_{\R}$. If $\deg(Y)>1$, then there exists a real projective $k$-codimensional subspace $\Pi$ such that the number of real points of $Y\cap\Pi$, counted with multiplicities, is no more than $\deg(Y)-2$.
\end{lemm}

\begin{proof}
Observe that the assumptions imply $0<k<d$. By Bertini's theorem (see \cite[3.3.1]{lazarsfeld}), there exists a real projective subspace $L$ of dimension~${d-k+1\geq 2}$, such that the curve $C=Y\cap L$ is irreducible over $\C$ and 
\begin{equation}
\deg(C)=\deg(Y)>1.
\end{equation}

First we suppose that there is no hyperplane of $L$ containing the curve $C$. We choose two distinct complex conjugate points $P$ and $\b P$ on the curve $C(\C)$, and a real hyperplane $\Pi$ of $L$ such that $\Pi(\C)$ contains these two points. As~$C$ is irreducible and not contained in $\Pi$, the intersection $C\cap\Pi=Y\cap \Pi$ is a finite number of points, including the complex points $P$ and $\b P$. The number of complex points of this intersection, counted with multiplicities, is exactly $\deg(Y)$, and at least two complex points are not real. The result follows.

Otherwise, let $L'\varsubsetneq L$ be the minimal projective subspace that contains the curve~$C$. As $\deg(C)>1$, it follows that $\dim(L')\geq 2$. By the first step, we can choose a hyperplane $\Pi'\subset L'$ such that
\begin{equation}
\sharp(C\cap\Pi')(\R)\leq\deg(C)-2.
\end{equation}
Then we take any hyperplane~$\Pi$ of $L$ containing $\Pi'$ and not $L'$, and we are done.
\end{proof}

Let us go back to the proof of the case of equality. Let $Y$ be a real $k$-dimensional subvariety of $\P^d$ that is not the union of real projective subspaces. We may assume that $Y$ is irreducible over $\R$. If it is not geometrically irreducible, then~${Y=Z\cup\sigma(Z)}$, with $Z$ a complex subvariety that is not real, hence ${\vol_\R(Y)=0 <\deg(Y)\vol_\R(\P^k)}$. Otherwise we can deduce from Lemma \ref{deg-2} that there exists a real hyperplane $\Pi$ such that $\sharp (Y\cap\Pi)(\R)\leq\deg(Y)-2$ (with multiplicities). This inequality remains satisfied on a neighborhood of $\Pi$ in the Grassmannian ${\G(d-k,d)}$. Such a neighborhood has a positive probability for $\mu_{d-k,k}$, so the Cauchy--Crofton formula implies~${\vol_\R(Y)<\deg(Y)\vol_\R(\P^k)}$.
\end{proof}

% General case

\subsection{General case}\label{metric-comparable}

Let $X$ be an arbitrary $d$-dimensional real algebraic variety. Since it is projective, we can put a Riemannian metric that is induced by a Fubini--Study metric on some $\P^n$ in which it is embedded. Corollary \ref{1.11} provides positive constants $C_k>0$ such that $\vol_\R(Y)\leq C_k\vol_\C(Y)$ for any $k$-dimensional subvariety $Y$. Since $X(\C)$ is compact, two arbitrary Riemannian metric are comparable and we get a similar inequality for any Riemannian metric on $X$. This yields the following proposition.

\begin{prop}\label{mvolr-leq}
Let $X$ be a real algebraic variety, equipped with an arbitrary Riemannian metric. For any $k\in\N^*$ there exists a constant $C_k>0$, depending on the choice of the metric, such that
\begin{equation}\label{ineg-gen}
\vol_\R(Y)\leq C_k\vol_\C(Y)
\end{equation}
for all $k$-dimensional subvarieties $Y$ of $X$.
\end{prop}

%%%%%%%%%%%%%%%%%%%%%%%%%%%%%%%%%%%%%%%%%%%%%%%%%%%%%%%%%%%%%%%%%%%%
% FIRST PROPERTIES
%%%%%%%%%%%%%%%%%%%%%%%%%%%%%%%%%%%%%%%%%%%%%%%%%%%%%%%%%%%%%%%%%%%%

\chapter{Concordance of a Real Algebraic Variety}\label{chap-conc}

Let $X$ be a real algebraic variety, whose dimension is $d$. We choose a Kähler\footnote{The reason for this choice will be made obvious in what follows.} metric on $X(\C)$ (this is a particular case of a Riemannian metric), whose Kähler $2$-form is denoted by $\kappa$. We are interested in volumes of real subvarieties $Y$ with respect to this metric. We restrict ourselves to codimension $1$ subvarieties, which carry sufficient information on the geometry of $X$. We rather use the language of divisors instead of that of algebraic $(d-1)$-dimensional subvarieties, and we will talk about the (real or complex) volume of an effective divisor $D$.

In the preceding chapter, we have shown that $\vol_\R(D)$ is bounded from above by $\vol_\C(D)$ (up to a constant). In what follows, we are interested in finding a lower bound for $\vol_\R(D)$ which depends only on $\vol_\C(D)$. The first step is to assure positivity of $\vol_\R(D)$\footnote{Let us recall that our conventions imply $X(\R)\neq\emptyset$.}.

%%%%%%%%%%%%%%%%%%%%%%%%%%%%%%%%%%
% Positivity of volumes
%%%%%%%%%%%%%%%%%%%%%%%%%%%%%%%%%%

\section{Positivity of volumes} 

% Complex volumes

\subsection{Complex volumes}

Let $D$ be an effective divisor on $X_\C$. Since we have chosen a Kähler metric, the volume of $D(\C)$ is given by Wirtinger formula :
\begin{equation}\label{formule-vol}
\vol_\C(D)=\frac{1}{(d-1)!}\int_{D(\C)}\kappa^{d-1}=\frac{1}{(d-1)!}[\kappa^{d-1}]\cdot[D]>0,
\end{equation}
where $[\kappa^{d-1}]\cdot[D]$ denotes the cup product between the De Rham cohomology class of $\kappa^{d-1}$ in $H^{2d-2}(X_\C;\R)$ and the Chern class of $D$ in $H^2(X_\C;\R)$. In particular, $\vol_\C(D)$ only depends on this Chern class $[D]\in\NS(X_\C)$.

\begin{prop}\label{min-unif-volc}
There exists a positive constant $K$ such that
\begin{equation}
\vol_\C(D)\geq K
\end{equation}
for all effective divisors $D\neq 0$.
\end{prop}

\begin{proof}
As all Riemannian metrics are equivalent, it is enough to show the inequality when the metric is the Fubini--Study metric on $\P^n\supset X$. In this case the volume of~$D(\C)$ is  proportional to the degree of $D$ as a subvariety of $\P^n$, which is a positive integer. Thus we get the lower bound with $K=\vol_\C(\P^{d-1})>0$.
\end{proof}

% Real volumes

\subsection{Real volumes}\label{pos-mvolr}

Let $D$ be a real effective divisor on $X$. Although the volume of~$D(\C)$ is always positive, it may happen that $\vol_\R(D)=0$ for some divisors~$D$. For instance, on $X=\P^d_\R$, for any even degree $\delta$ the divisor $D_\delta$ given by the equation $\sum_{j=0}^dZ_j^\delta=0$ has an empty real locus, hence $\vol_\R(D_\delta)=0$. Yet this divisor is algebraically (and even linearly) equivalent to $D'_\delta$ given by $\sum_{j=1}^dZ_j^\delta=Z_0^\delta$, and we have $\vol_\R(D'_\delta)>0$.

\begin{defi}
We denote by $\mathcal{V}(D)$ the family of effective divisors $D'$ on $X_\R$ such that $[D]=[D']$ in $\NS(X_\R)$\footnote{For a subvariety $Y$ of arbitrary dimension $k$, we would rather define $\mathcal{V}(Y)$ as the family of real algebraic subvarieties which have the same homology class in $H_{2k}(X_\C;\Z)$. This definition is equivalent for divisors, by Poincaré duality.}. Then  we set
\begin{equation}
\mvol_\R(D)=\max_{D'\in\mathcal{V}(D)}\vol_\R(D).
\end{equation}
\end{defi}

\begin{rema}
Since $\vol_\C(D)$ only depends on $[D]$, there is no need to define $\mvol_\C(D)$.
\end{rema}

In the preceding example, we have $\vol_\R(D_\delta)=0$, but $\mvol_\R(D_\delta)\geq\vol_\R(D'_\delta)>0$. In the following, we are interested in the positivity of $\mvol_\R(D)$ instead of that of $\vol_\R(D)$.

\begin{rema}
Since
\begin{equation}
\mathcal{V}(D_1+D_2)\supset\mathcal{V}(D_1)+\mathcal{V}(D_2),
\end{equation}
the function $\mvol_\R$ is superadditive on the set of real effective divisors. In particular
\begin{equation}
\mvol_\R(kD)\geq k\mvol_\R(D) \quad \forall k\in\N^*.
\end{equation}
\end{rema}

\begin{prop}\label{pinceau}
Let $D$ be a real effective divisor such that the linear system $|D|$ contains a pencil, that is, $h^0(X,\O_X(D))\geq 2$. Then 
\begin{equation}\mvol_\R(D)>0.\end{equation}
\end{prop}

\begin{proof}
Let $(D_\lambda)_{\lambda\in\P^1(\C)}$ be a real pencil in $|D|$ (in this context, real means $D_{\b\lambda}=\sigma(D_\lambda)$). By Bertini's theorem \cite[p.137]{griffiths-harris} there is a finite set $S\subset\P^1(\C)$ such that, for all $\lambda\notin S$, the subvariety $D_\lambda(\C)$ is smooth away from the base locus~$B$ of the pencil $(D_\lambda)_{\lambda\in\P^1(\C)}$. Let $P\in X(\R)\backslash \left(B\cup\bigcup_{\lambda\in S} D_\lambda\right)$. Then there exists $\lambda$ in~$\P^1(\R)\backslash S$ such that the point $P$ is on (the support of) the divisor $D_\lambda$. As~$D_\lambda$ is smooth at $P$, the real locus $D_\lambda(\R)$ contains an arc around $P$, and thus  $\mvol_\R(D)\geq\vol_\R(D_\lambda)>0$.
\end{proof}

\begin{coro}\label{very-ample}
Let $D$ be a very ample divisor on $X_\R$. Then \begin{equation}\mvol_\R(D)>0.\end{equation}
\end{coro}

By contrast, it may happen that $\mvol_\R(D)=0$ for some effective divisors that are not ample, as shown in the two following examples.

\begin{exem}\label{blowup}
Let $X$ be the variety obtained by blowing up $\P^d_\R$ at two (distinct) complex conjugate points, and let $E$ be the exceptional fiber of the blow-up. For any $k\in\N^*$, we have $\mathcal{V}(kE)=\{kE\}$, thus 
\begin{equation}
\mvol_\R(kE)=0,
\end{equation}
for~$E(\R)$ is empty. Nevertheless, observe that $[E]$ is not in the closure of the cone $\NS^+(X_\R)$, since its self-intersection is negative.
\end{exem}

\begin{exem}\label{quartique}
Let $C$ be a real smooth quartic in $\P^2_{\R}$ such that $C(\R)$ is empty, for instance the one given by the equation
\begin{equation}
Z_0^4+Z_1^4+Z_2^4=0.
\end{equation}
Take $8$ pairs of complex conjugate points $(P_i,\b P_i)$ on $C$ in such a way that the linear class of
\begin{equation}
\sum_i(P_i+\b P_i) -\O_{\P^2}(4)_{|C}
\end{equation}
is not a torsion point of ${\rm Pic}^0(C)$. Let $\pi:X\to\P^2$ be the blow-up morphism above these $16$ points (defined over $\R$) and let $C'$ be the strict transform of~$C$ in~$X$. For all divisors $D$ in $\mathcal{V}(kC')$ the curve $\pi_*D$ has degree $4k$ and passes through the~$16$ blown-up points with multiplicity at least $k$. Then the choice of the points~$P_i$ implies that $\pi_*D=kC$, hence $D=kC'$. So we see that 
\begin{equation}
\mvol_\R(kC')=0
\end{equation}
for all~$k\in\N^*$. Here~$C'$ is a nef divisor that is not ample, as an irreducible divisor with self-intersection~$0$ (see \cite[\textsection 1.4]{lazarsfeld}).
\end{exem}

%%%%%%%%%%%%%%%%%%%%%%%%%%%%%%%%%%
% Concordance
%%%%%%%%%%%%%%%%%%%%%%%%%%%%%%%%%%

\section{Definition and first properties of concordance}\label{section-concordance}

We would like to know for which non-negative exponents $\alpha$ we can write inequalities such as $\mvol_\R(D)\geq C\vol_\C(D)^\alpha$, with $C>0$ independent of $D$. Since such an inequality is possible only when $\mvol_\R(D)>0$, the preceding paragraph suggests us to restrict to ample divisors.

\begin{defi}\label{defi-concordance}
Let $\mathcal{A}(X)$ be the set of non-negative exponents $\alpha$ for which there exist $C>0$ and $q\in\N^*$ such that
\begin{equation}\label{eq-alpha}
\mvol_\R(D)\geq C\vol_\C(D)^\alpha
\end{equation}
for all real ample divisors $D$ whose Chern classes are $q$-divisible. The upper bound of $\mathcal{A}(X)$ is the \emph{concordance} of $X$, and is denoted by $\alpha(X)$. We say the concordance is \emph{achieved} when $\alpha(X)$ is contained in $\mathcal{A}(X)$.
\end{defi}

\begin{rema}
The set $\mathcal{A}(X)$, and thus the concordance $\alpha(X)$, only depend on $X$, and not on the choice of a particular metric (by the same argument as in \textsection\ref{metric-comparable}).
\end{rema}

\begin{prop}
The set $\mathcal{A}(X)$ is an interval of the form $[0,\alpha(X)]$ or $[0,\alpha(X))$, depending on whether the concordance is achieved or not.
\end{prop}

\begin{proof}
Let $\alpha\in\mathcal{A}(X)$, and take a real number $\beta$ such that $0\leq\beta<\alpha$. We show that $\beta\in\mathcal{A}(X)$.

By Proposition \ref{min-unif-volc}, there exists $K>0$ such that $\mvol_\C(D)\geq K$ for all real effective divisors $D$. When $[D]$ is $q$-divisible, we then have
\begin{equation}
\mvol_\R(D)\geq C\vol_\C(D)^\alpha \geq CK^{\alpha-\beta}\vol_\C(D)^{\beta},
\end{equation}
and so $\beta$ is in $\mathcal{A}(X)$ too.
\end{proof}

\begin{prop}
The concordance~$\alpha(X)$ is in the interval $[0,1]$.
\end{prop}

\begin{proof}
Let $\alpha\in\mathcal{A}(X)$. By Proposition \ref{mvolr-leq}, there exists a positive constant $C'>0$ such that~$\mvol_\R(D)\leq C'\vol_\C(D)$ for all real ample divisors $D$. When $[D]$ is also $q$-divisible, we obtain, for all $k\in\N^*$,
\begin{equation}
C\vol_\C(kD)^\alpha\leq\mvol_\R(kD)\leq C'\vol_\C(kD).
\end{equation}
If $\alpha>1$, this contradicts $\lim_{k\to+\infty}\vol_\C(kD)=+\infty$.
\end{proof}

%%%%%%%%%%%%%%%%%%%%%%%%%%%%%%%%%%
% Examples with concordance 1
%%%%%%%%%%%%%%%%%%%%%%%%%%%%%%%%%%

\section{Examples of varieties with concordance $1$}\label{section-alpha=1}

% rho=1

\subsection{Picard number $1$}

We can apply the corollary of Cauchy-Crofton formula  to compute the concordance of the projective space.

\begin{prop}
The concordance of the $d$-dimensional projective space~$\P^d_\R$ is equal to $1$, and it is achieved.
\end{prop}

\begin{proof}
We take the Fubini--Study metric on $\P^d(\C)$, and we set
\begin{equation}
C=\frac{\vol_\R(\P^{d-1})}{\vol_\C(\P^{d-1})}.
\end{equation}
Corollary \ref{1.11} implies that $\mvol_\R(D)=C\vol_\C(D)$ for any effective divisor on~$\P^d_\R$. This shows that $1$ in contained in $\mathcal{A}(X)$.
\end{proof}
 
More generally, the concordance is $1$ when the Picard number of $X_\R$ is $1$.

\begin{prop}\label{rho=1}
Let $X$ be a real algebraic variety such that $\rho(X_\R)=1$. Then $\alpha(X)=1$, and the concordance is achieved.
\end{prop}

\begin{proof}
Let $D_0$ be a very ample divisor on $X_\R$. Its Chern class $[D_0]$ spans a finite index subgroup of $\NS(X_\R;\Z)$ : let $q$ denote this index. If $D$ is a real effective divisor on $X_\R$ such that $[D]$ is $q$-divisible, we can write $[D]=k[D_0]$ for some positive integer $k$. Thus, by setting 
\begin{equation}
C=\mvol_\R(D_0)/\vol_\C(D_0),
\end{equation}
we have
\begin{equation}
\mvol_\R(D)=\mvol_\R(kD_0)\geq k\mvol_\R(D_0)=C k\vol_\C(D_0) = C\vol_\C(D).
\end{equation}
This shows that $1$ belongs to the interval $\mathcal{A}(X)$.
\end{proof}

\begin{coro}
If $X$ is a real algebraic curve, then $\alpha(X)=1$ and is achieved.
\end{coro}

Thus, we see that nontrivial cases (those with $\alpha(X)<1$) can only occur when both the dimension and the Picard number are at least $2$. In the following, we focus on the case of surfaces, which already include many intersting examples (cf. chapter \ref{chap-tores} for examples with $\alpha(X)<1$).

% Cone nef simple

\subsection{Rational polyhedral nef cone and Del Pezzo surfaces}

Proposition \ref{rho=1} is a particular case of the following one.

\begin{prop}\label{conenef}
Let $X$ be a real algebraic variety. Assume that the cone $\Nef(X_\R)$ is polyhedral, with extremal rays spanned by classes $[D_j]$ such that $\mvol_\R(D_j)>0$. Then the concordance $\alpha(X)$ is $1$, and it is achieved.
\end{prop}

\begin{proof}
Define 
\begin{equation}
C=\min_j\left(\mvol_\R(D_j)/\vol_\C(D_j)\right)>0.
\end{equation}
The classes $[D_j]$ span a finite index subgroup of $\NS(X_\R;\Z)$. Denote by $q$ this index. Since $\Amp(X_\R)\subset\Nef(X_\R)$, every real ample divisor $D$ with $[D]$ $q$-divisible is equivalent to a divisor of the form~$\sum_j k_jD_j$, where the $k_j$'s are nonnegative integers. Hence
\begin{equation}
\mvol_\R(D) \geq \sum_jk_j\mvol_\R(D_j)
\geq C\sum_jk_j\vol_\C(D_j)
= C\vol_\C(D).
\end{equation}
We can then conclude that $1$ is contained in $\mathcal{A}(X)$.
\end{proof}

\begin{coro}\label{DP}
Let $X$ be a real Del Pezzo surface. The concordance of~$X$ is $1$, and it is achieved.
\end{coro}

\begin{proof}
By definition, a surface is Del Pezzo when its anticanonical divisor~$-K_X$ is ample. The cone of curves is then rational polyhedral by the cone theorem (see \cite[1.5.33, 1.5.34]{lazarsfeld}, or \cite{kollar-surfaces} for a statement over $\R$). Thus its dual cone $\Nef(X_\R)$ is also rational polyhedral. 

\begin{lemm}\label{nef}
Let $D$ be a nef divisor on a Del Pezzo surface $X$, which is not numerically trivial. Then the linear system $|D|$ contains a pencil.
\end{lemm}

\begin{proof}
The proof is a simple application of Riemann--Roch formula and Serre duality (cf. corollary \ref{rr+serre}):
\begin{equation}
h^0(X,\O_X(D))\geq\chi(\O_X)+\frac{1}{2}(-K_X\cdot D+D^2).
\end{equation}
Here, $\chi(\O_X)=1$ since $X$ is rational. Moreover, we have 
\begin{equation}
-K_X\cdot D>0 \quad \text{and} \quad D^2\geq 0,
\end{equation}
since $-K_X$ is ample and $D$ is nef. This shows that
\begin{equation}
h^0(X,\O_X(D))>1.
\end{equation}
\end{proof}

Let us go back to the proof of Corollary \ref{DP}. Since the extremal rays of the convex cone $\Nef(X_\R)$ are spanned by classes $[D_j]$ such that $|D_j|$ contains a pencil, Proposition \ref{pinceau} implies that $\mvol_\R(D_j)>0$, and thus we are in the context of Proposition~\ref{conenef}.
\end{proof}

%%%%%%%%%%%%%%%%%%%%%%%%%%%%%%%%%%%%%%%%%%%%%%%%%%%%%%%%%%%%%%%%%%%%
% CONCORDANCE & ENTROPY
%%%%%%%%%%%%%%%%%%%%%%%%%%%%%%%%%%%%%%%%%%%%%%%%%%%%%%%%%%%%%%%%%%%%

\chapter{Concordance and Entropy of Automorphisms}\label{chap-ent}

From now on $X$ is a real algebraic surface equipped with a K\"ahler metric, whose K\"ahler form is denoted by $\kappa$.

%%%%%%%%%%%%%%%%%%%%%%%%%%%%%%%%%%
% Complex volume of the iterates of a divisor
%%%%%%%%%%%%%%%%%%%%%%%%%%%%%%%%%%

\section{Complex volumes}

\begin{theo}\label{iteres-mvolc}
Let $f$ be an automorphism of a complex algebraic surface~$X$. For all ample divisors $D$ we have
\begin{equation}\label{ent-c}
\lim_{n\to+\infty}\frac{1}{n}\log\left(\vol_\C(f_*^nD)\right)=\h(f_\C)=\log(\lambda(f)).
\end{equation}
\end{theo}

\begin{proof}
Wirtinger's equality gives (cf. (\ref{formule-vol})) 
\begin{equation}
\vol_\C(f_*^nD)=f_*^n[D]\cdot[\kappa].
\end{equation}

If the spectral radius $\lambda(f)$ of $f_*$ is $1$, then the last sequence has at most a polynomial growth (actually it is bounded or quadratic, cf. chapter \ref{chap-auto}). Hence
\begin{equation}
\lim_{n\to+\infty}\frac{1}{n}\log\left(\vol_\C(f_*^nD)\right)=0=\log(\lambda(f)).
\end{equation}

If $\lambda(f)>1$, since $[D]$ is in the ample cone, the sequence $\left(\frac{f_*^n[D]}{\lambda(f)^n}\right)_{n\in\N}$ converges to $\alpha\,\theta^-_f$, with $\alpha>0$ (cf. Proposition \ref{loxo-kahler}). In particular
\begin{equation}\label{volc-borne}
\lim_{n\to+\infty}\frac{\vol_\C(f_*^nD)}{\lambda(f)^n}=\alpha,
\end{equation}
and then
\begin{equation}
\lim_{n\to+\infty}\frac{1}{n}\log\left(\vol_\C(f_*^nD)\right)=\log(\lambda(f)).
\end{equation}
The equality $\log(\lambda(f))=\h(f)$ is given by Gromov--Yomdin theorem \ref{theo-gyc}.
\end{proof}

\begin{rema}
For varieties with arbitrary dimension $d$, the formula
\begin{equation}
\lim_{n\to+\infty}\frac{1}{n}\log\left(\vol_\C(f_*^nD)\right)=\log(\lambda(f))
\end{equation}
still holds. But this is not necessarily equal to the entropy $\h(f)$, which is the logarithm of the spectral radius on the whole cohomology, \emph{a priori} distinct from the spectral radius~$\lambda(f)$ on $H^2(X_\C;\R)$.
\end{rema}

%%%%%%%%%%%%%%%%%%%%%%%%%%%%%%%%%%
% An upper bound for real volume of the iterates of a divisor
%%%%%%%%%%%%%%%%%%%%%%%%%%%%%%%%%%

\section{An upper bound for real volumes}

\begin{theo}\label{iteres-mvolr}
Let $f$ be a real automorphism of a real algebraic surface $X$. For all ample real divisors $D$ we have
\begin{equation}\label{ent-r}
\limsup_{n\to+\infty}\frac{1}{n}\log\left(\mvol_\R(f_*^nD)\right)\leq\h(f_\R).
\end{equation}
\end{theo}

The proof of this result relies on \cite[Theorem 1.4]{yomdin}, which gives a lower bound for entropy in terms of volume growth. It is here stated in the particular case of dimension $1$ submanifolds.

\begin{theo}[Yomdin]\label{yomdin}
Let $M$ be a compact Riemannian manifold,
\begin{equation}
{g:M\rightarrow M}
\end{equation}
be a differentiable map and $\gamma:[0,1]\to M$ be an arc, each of class~$\mathcal{C}^r$, with~${r\geq 1}$. Then \begin{equation}
\limsup_{n\rightarrow+\infty}\frac{1}{n}\log\left(\length(g^n\circ\gamma)\right)\leq\h(g)+\frac{2}{r}\liap(g),
\end{equation}
where $\liap(g)$ denotes the \emph{topological Liapunov exponent}
\begin{equation}
\liap(g)=\lim_{n\rightarrow+\infty}\frac{1}{n}\log\ltrivert{\rm d}g^n\rtrivert_\infty.\footnote{The notation $\ltrivert{\rm d}g\rtrivert_\infty$ stands for $\max_{x\in M}\ltrivert{\rm d}g(x)\rtrivert$, the norm being taken with respect to the Riemannian metric. Nevertheless the number $\liap(g)$ does not depend on the choice of the Riemannian metric.}
\end{equation}
In particular when the regularity is $\mathcal{C}^\infty$, then
\begin{equation}
\limsup_{n\rightarrow+\infty}\frac{1}{n}\log\left(\length(g^n\circ\gamma)\right)\leq\h(g).
\end{equation}
\end{theo}

Looking carefully at the proof in Yomdin \cite{yomdin}, one sees that this result can be improved to the case when we consider a family of $\mathcal{C}^r$-arcs $(\gamma_j)_j$ whose derivatives are uniformly bounded to the order $r$, that is, there is a positive number $K$ such that~$\|\gamma_j^{(k)}(t)\|\leq K$ for all $j$, $t\in [0,1]$ and $k\leq r$. Under these assumptions we have
\begin{equation}\label{yomdin-famille}
\limsup_{n\rightarrow+\infty}\frac{1}{n}\log\left(\max_j\left\{\length(g^n\circ\gamma_j)\right\}\right)\leq\h(g)+\frac{2}{r}\liap(g).
\end{equation}

We also use the following lemma:

\begin{lemm}[{\cite[Algebraic Lemma 3.3]{gromov-aft-yomdin}}]\label{lemme-gromov}
Let $Y$ be the intersection of an algebraic affine curve in~$\R^d$ with $[-1,1]^d$. For any $r\in\N^*$ there exist at most $m_0$ arcs of class $\mathcal{C}^r$
\begin{equation}
\gamma_j:[0,1]\to Y,
\end{equation}
where~$m_0$ is an integer depending only on~$d$,~$r$, and $\deg(Y)$, such that
\begin{enumerate}
\item $Y=\bigcup_j \gamma_j([0,1])$;
\item $\|\gamma_j^{(k)}(t)\|\leq 1$ for all $j$, $t\in[0,1]$ and $k\leq r$;
\item all $\gamma_j$'s are analytic diffeomorphisms from $(0,1)$ to their images;
\item the images of the $\gamma_{j}$'s can only meet on their boundaries.
\end{enumerate}
\end{lemm}

\begin{proof}[Proof of Theorem \ref{iteres-mvolr}]
Inequality (\ref{ent-r}) does not depend on the choice of a particular metric on $X$, so we can consider an embedding $X\subset\P^d_\R$ and take the metric induced by Fubini--Study on $X$. The projective space $\P^d(\R)$ is covered by the $(d+1)$ cubes~$Q_k$, for $k\in\{0,\cdots,d\}$, given in homogeneous coordinates by
\begin{equation}
|Z_k|=\max_j|Z_j|.
\end{equation}
Each of these $Q_k$ is located in the affine chart~${U_k=\{Z_k\neq 0\}\simeq\R^d}$, and in this chart $Q_k$ is identified with $[-1,1]^d$.

The degree of $D$ as a subvariety of $\P^d$ only depends on the Chern class~$[D]$. Therefore we can apply Lemma \ref{lemme-gromov} to any divisor $D'\in\mathcal{V}(D)$, intersected with one  of the $Q_k$'s: any real locus of $D'\in\mathcal{V}(D)$ is covered by at most $m_1$ arcs~$\gamma_{D',j}$ of class $\mathcal{C}^r$, the integer $m_1=(d+1)m_0$ being independent of $D'$, such that
\begin{equation}
\|\gamma_{D',j}^{(k)}\|_\infty\leq K
\end{equation}
for all $k\leq r$, where $r$ is a fixed positive integer and~$K$ a positive constant (which comes from the comparison of Euclidean and Fubini--Study metrics on~${[-1,1]^d}$). Now we apply (\ref{yomdin-famille}) to obtain
\begin{equation}
\begin{split}
\limsup_{n\to+\infty}\frac{1}{n}\log\left(\mvol_\R(f_*^nD)\right)
&\leq
\limsup_{n\to+\infty}\frac{1}{n}\log\left(m_1 \max_{D',j}\left\{\length(f_\R^n\circ\gamma_{D',j})\right\}\right)\\
&\leq
\h(f_\R)+\frac{2}{r}\liap(f_\R).
\end{split}
\end{equation}
Since the regularity of both $X(\R)$ and $f_\R$ is $\mathcal{C}^\infty$, we may take the limit as $r$ goes to $+\infty$ and get the desired inequality.
\end{proof}

\begin{rema}
Yomdin's theorem (as well as its version in family) and Gromov's lemma still hold for arbitrary dimensional submanifolds. Therefore the proof of Theorem \ref{iteres-mvolr} can be adapted when $X$ is a variety with higher dimension.
\end{rema}

%%%%%%%%%%%%%%%%%%%%%%%%%%%%%%%%%%
% Upper bound for concordance
%%%%%%%%%%%%%%%%%%%%%%%%%%%%%%%%%%

\section{Link between the concordance and the ratio of entropies}

\begin{theo}\label{entropie}
Let $X$ be a real algebraic surface and let $f$ be a real loxodromic type\footnote{Recall that \emph{loxodromic type} means $\lambda(f)>1$, or equivalently $\h(f_\C)>0$.} automorphism of $X$. Then 
\begin{equation}\label{dist-ent}
\alpha(X)\leq\frac{\h(f_\R)}{\h(f_\C)}.
\end{equation}
\end{theo}

\begin{proof}
Let $\alpha$ be an exponent in the interval $\mathcal{A}(X)$. This means that there are~${q\in\N^*}$ and~$C>0$ such that 
\begin{equation}
\mvol_\R(D)\geq C\vol_\C(D)^\alpha
\end{equation}
for all real ample divisors $D$ with~$[D]$~$q$-divisible. For such a divisor, $f_*^n[D]$ is also $q$-divisible for all $n\in\N$, and by Theorems \ref{iteres-mvolc} and \ref{iteres-mvolr} we get
\begin{equation}
\begin{split}
\h(f_\R)
&\geq
\limsup_{n\to+\infty}\frac{1}{n}\log\mvol_\R(f_*^nD)\\
&\geq
\limsup_{n\rightarrow+\infty}\frac{1}{n}\left(\log C+\alpha\log\vol_\C(f_*^nD)\right)\\
&=
\alpha\h(f_\C).
\end{split}
\end{equation}
Then we take the limit as $\alpha\to\alpha(X)$ and we obtain (\ref{dist-ent}).
\end{proof}

We have seen in Remark \ref{lehmer} that all loxodromic type automorphisms $f$ of~$X$ satisfy
\begin{equation}
\h(f_\C)\geq \log(\lambda_{10}),
\end{equation}
where $\lambda_{10}>1$ denotes the Lehmer number. Combined with Theorem \ref{entropie}, this gives the following corollary.

\begin{coro}\label{coro-lehmer}
Let $f$ be a real automorphism of a real algebraic surface~$X$. If $\h(f_\R)>0$, then
\begin{equation}
\h(f_\R)\geq \alpha(X)\log(\lambda_{10}).
\end{equation}
\end{coro}

This universal lower bound is interesting when $\alpha(X)>0$, and enables us (cf. Chapter \ref{chap-non-dense}) to show nondensity results for $\Aut(X_\R)$ in the group of diffeomorphisms of $X(\R)$ (or in those diffeomorphisms which preserve the area when $\kod(X)=0$).

\begin{rema}
If there exists a real parabolic type automorphism ${f:X\to X}$, we know that, for each ample divisor $D$, $\vol_\C(f_*^nD)$ grows quadratically (cf. Proposition \ref{car-D} in Cahpter \ref{chap-auto}), and we can reasonably think that $\mvol_\R(f_*^nD)$ grows at most linearly (see what happens for a torus for instance). So we conjecture that 
$
\alpha(X)\leq\frac{1}{2}
$
in this case. 
\end{rema}

%%%%%%%%%%%%%%%%%%%%%%%%%%%%%%%%%%
% Lower bound for real volumes
%%%%%%%%%%%%%%%%%%%%%%%%%%%%%%%%%%

\section{A lower bound for real volumes}

\begin{defi}\label{veryample}
Let $M$ be a differentiable surface. A family $\Gamma$ of curves on~ $M$ is said to be \emph{very ample} if for all $P\in M$ and for all directions $\mathcal{D}\subset {\rm T\!}_xM$, there is a curve $\gamma\in\Gamma$ on which $P$ is a regular point and whose tangent direction at $P$ is $\mathcal{D}$.
\end{defi}

\begin{exem}\label{ex-tres-ample}
Let $X$ be a real algebraic surface and $D$ be a very ample real divisor on $X$. Then the family $\mathcal{V}(D)$, as a family of curves on $X(\R)$, is a very ample family in the sense of Definition \ref{veryample}.
\end{exem}

\begin{theo}\label{horseshoe}
Let $M$ be a compact Riemannian surface,
\begin{equation}
g:M\rightarrow M
\end{equation}
be a diffeomorphism of class $\mathcal{C}^{1+\epsilon}$ (with $\epsilon>0$) with positive entropy, and $\Gamma$ be a very ample family of curves on $M$. Then for all $\lambda<\exp(\h(g))$, there exist a curve~${\gamma\in\Gamma}$ and a constant $C>0$ such that
\begin{equation}\label{long-exp}
\length(g^n(\gamma))\geq C\lambda^n \quad \forall n\in\N.
\end{equation}
\end{theo}

In other words, we have the following inequality:
\begin{equation}\label{type-newhouse}
\sup_{\gamma\in\Gamma}\left\{\liminf_{n\rightarrow+\infty}\frac{1}{n}\log\left(\length(g^n(\gamma))\right)\right\}\geq\h(g).
\end{equation}
This has to be compared with a similar result due to Newhouse \cite{newhouse}, who considers manifolds of arbitrary dimension and noninvertible maps, but obtains the inequality (\ref{type-newhouse}) with a limit superior instead of a limit inferior (assumptions on the family $\Gamma$ are also lightly different). On the other hand, the lower bound (\ref{type-newhouse}) is optimal when $M$ and $g$ are $\mathcal{C}^\infty$, by Yomdin's Theorem~\ref{yomdin}.

\begin{coro}\label{minorer-mvolr}
Let $f$ be a real automorphism of a real algebraic surface~$X$.
For all $\lambda<\exp(\h(f_\R))$ and all very ample real divisors $D$ on $X$, there exists $C>0$ such that
\begin{equation}
\mvol_\R(f_*^nD)\geq C\lambda^n \quad \forall n\in\N.
\end{equation}
\end{coro}

The proof of Theorem \ref{horseshoe} relies on a result due to Katok \cite[Theorem S.5.9 p. 698]{katok-hasselblatt}, which asserts that the entropy of surface diffeomorphisms is well approximated by horseshoes. For definition and properties of horseshoes, we refer to \cite[\textsection 6.5]{katok-hasselblatt}.

\begin{theo}[Katok]\label{katok}
Let $M$ be a compact surface, and
\begin{equation}
g:M\rightarrow M
\end{equation}
be a diffeomorphism of class $\mathcal{C}^{1+\epsilon}$ (with $\epsilon>0$) with positive entropy. For any~$\eta>0$, there exists a horseshoe $\Lambda$ for some positive iterate $g^k$ of $g$ such that
\begin{equation}\label{h-approx}
\h(g)\leq\frac{1}{k}\h(g^k_{|\Lambda})+\eta.
\end{equation}
\end{theo}

\begin{proof}[Proof of Theorem \ref{horseshoe}]
Fix real numbers $\lambda$ and $\eta$ such that
\begin{equation}
1<\lambda<\exp(\h(g))
\quad \text{and} \quad
0<\eta\leq\h(g)-\log(\lambda).
\end{equation}
Let $\Lambda$ be a horseshoe for $G=g^k$ satisfying~(\ref{h-approx}), and let~$\Delta\supset\Lambda$ be a ``rectangle'' corresponding to this horseshoe, in such a way that
\begin{equation}
{\Lambda=\bigcap_{j\in\Z}G^j(\Delta)}.
\end{equation}
The set $G(\Delta)\cap\Delta$ has $q$ connected components $\Delta_1,\cdots,\Delta_q$, which are ``subrectangles'' crossing entirely $\Delta$ downward (see Figure \ref{ex-horseshoe}). The restriction~$G_{|\Lambda}$ is topologically conjugate to the full-shift on $q$ symbols, by the conjugacy map
\begin{equation*}
\begin{split}
\{1,\cdots,q\}^\Z
&\longrightarrow
\Lambda\\
(\omega_j)_{j\in\Z}
&\longmapsto
\bigcap_{j\in\Z}G^j(\Delta_{\omega_j}).
\end{split}
\end{equation*}
In particular $\h(G_{|\Lambda})=\log(q)$. We denote by $L$ the distance between the upper and lower side of $\Delta$.

\begin{figure}[h]
\begin{center}
\scalebox{0.5}{\input{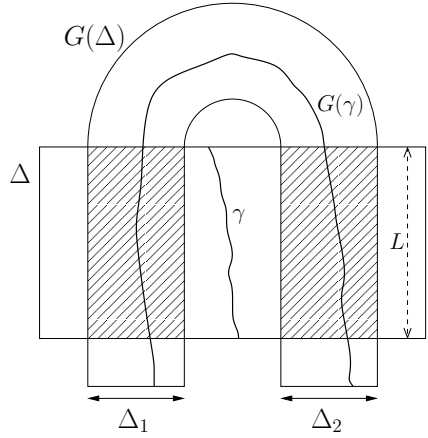}}
\caption{An example of horseshoe, here with $q=2$.}
\label{ex-horseshoe}
\end{center}
\end{figure}

\begin{lemm}\label{q^n}
Let $\gamma\subset\Delta$ be an arc crossing the rectangle $\Delta$ downward. Then
\begin{equation}
\length(G^n(\gamma))\geq q^nL \quad \forall n\in\N.
\end{equation}
\end{lemm}

\begin{proof}
It is enough to remark that the arc $G^n(\gamma)$ contains $q^n$ subarcs crossing $\Delta$ downward (see Figure \ref{ex-horseshoe} for $n=1$). This can be seen by induction on~$n$.
\end{proof}

Now fix a point $P\in\Lambda$, which we can write
\begin{equation}
P=\bigcap_{j\in\Z}G^j(\Delta_{\omega_j}).
\end{equation}
Let $\gamma\in\Gamma$ be a curve that goes through $P$ transversally to the stable variety~$W^s(P)$ (the horizontal one). For any sequence $(\epsilon_j)_{j\in\N}\in\{1,\cdots,q\}^\N$, we set (see Figure \ref{ex-horseshoe2})
\begin{equation}
R_{\epsilon_1,\cdots,\epsilon_n}=\bigcap_{j=0}^{n}G^{-j}(\Delta_{\epsilon_{j}}).
\end{equation}

\begin{figure}[h]
\begin{center}
\scalebox{0.7}{\input{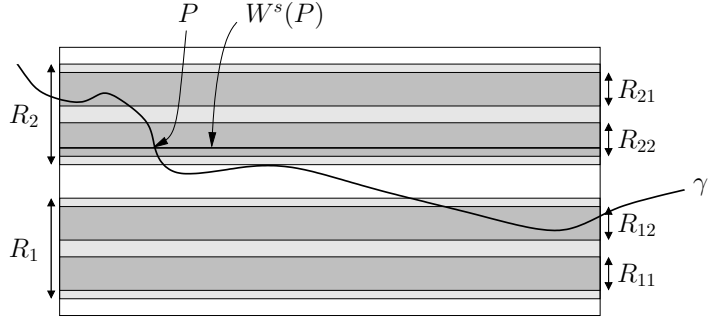}}
\caption{The rectangles $R_{\epsilon_1}$ and $R_{\epsilon_1,\epsilon_2}$ for the horseshoe of Figure \ref{ex-horseshoe}.}
\label{ex-horseshoe2}
\end{center}
\end{figure}

The sequence $(R_{\epsilon_1,\cdots,\epsilon_n})_{n\in\N}$ is a decreasing sequence of nested rectangles that converge to the curve
\begin{equation}
\bigcap_{j\in\N}G^{-j}(\Delta_{\epsilon_{j}}).
\end{equation}
If $(\epsilon_n)_{n\in\N}=(\omega_{-n})_{n\in\N}$ this curve is the stable variety $W^s(P)$ (intersected with~$\Delta$). Since $\gamma$ is transverse to it, there exist an integer $n_0$ and a subarc~${\gamma'\subset\gamma}$ such that $\gamma'$ crosses the rectangle $R_{\omega_0,\cdots,\omega_{-n_0}}$ downward. (On Figure \ref{ex-horseshoe2}, we may choose $\gamma'\subset R_{22}$.) Hence the arc ${G^{n_0}(\gamma')\subset G^{n_0}(\gamma)}$ satisfies the assumptions of Lemma \ref{q^n}, and thus
\begin{equation}
\length(G^{n_0+n}(\gamma))\geq q^nL \quad \forall n\in\N.
\end{equation}
So if we set
\begin{equation}
C'=\min\left\{\frac{L}{q^{n_0}},\left(\frac{\length(G^n(\gamma))}{q^n}\right)_{0\leq n\leq n_0-1}\right\},
\end{equation}
then
\begin{equation}
\begin{split}
\length(g^{nk}(\gamma))
&\geq C'q^n\\
&= C'\exp(n\h(g^k_{|\Lambda}))\\
&\geq C'\exp(nk(\h(g)-\eta))\\
&\geq C'\lambda^{nk}.
\end{split}
\end{equation}
Since
\begin{equation}
\length(g^n(\gamma))\leq\ltrivert{\rm d}g^{-1}\rtrivert_\infty\length(g^{n+1}(\gamma)),
\end{equation}
we get Inequality (\ref{long-exp}) by Euclidean division by $k$, where we have set
\begin{equation}
{C=C'(\lambda\ltrivert{\rm d}g^{-1}\rtrivert_\infty)^{-k}>0}.
\end{equation}
\end{proof}

%%%%%%%%%%%%%%%%%%%%%%%%%%%%%%%%%%
% Concordance when rho=2
%%%%%%%%%%%%%%%%%%%%%%%%%%%%%%%%%%

\section{Concordance for surfaces with Picard number $2$}

\begin{theo}\label{formule-exacte}
Let $X$ be a real algebraic surface with $\rho(X_\R)=2$. Assume that there exists a real loxodromic type automorphism $f$ on $X$. Then the inequality of Theorem \ref{entropie} is an equality :
\begin{equation}\label{exacte}
\alpha(X)=\frac{\h(f_\R)}{\h(f_\C)}.
\end{equation}
\end{theo}

\begin{rema}
The assumptions of the theorem imply that the surface $X$ is either a torus, a K3 surface, or an Enriques surface. Indeed, as seen in Chapter \ref{chap-auto}, its minimal model is either one of these three types of surfaces, or a rational surface. But if $X$ is not minimal or if $X$ is rational, then the class of the canonical divisor~$K_X$ would be nontrivial in $\NS(X_\R;\R)$. Since this class is preserved by $f_*$, this map would have $1$ as an eigenvalue. This is impossible, because $\NS(X_\R;\R)$ has dimension $2$ and the spectral radius of $f_*$ must be greater than $1$.
\end{rema}

\begin{proof}[Proof of Theorem \ref{formule-exacte}]
By Theorem \ref{entropie}, it is enough to prove that any nonnegative exponent
\begin{equation}
\alpha<\frac{\h(f_\R)}{\h(f_\C)}
\end{equation}
belongs to $\mathcal{A}(X)$. This is obvious when $\h(f_\R)=0$, so we suppose that $f_\R$ has positive entropy, and we fix such an exponent $\alpha$.

\begin{lemm}\label{une orbite}
Let $D$ be a very ample real divisor on $X$. There exists a constant $C > 0$ such that
\begin{equation}\label{eq-une-orbite}
\mvol_\R(f_*^nD) \geq C \vol_\C(f_*^nD)^\alpha \quad \forall n\in\Z.
\end{equation}
\end{lemm}

\begin{proof}
Since
\begin{equation}
\lambda(f)^\alpha=\exp(\alpha\h(f_\C))<\exp(\h(f_\R)),
\end{equation}
there exists, by Corollary~\ref{minorer-mvolr}, a positive number $C_\R$ such that
\begin{equation}
\mvol_\R(f_*^nD)\geq C_\R\lambda(f)^{n\alpha} \quad \forall n\in\N.
\end{equation}
On the other hand there is a positive number $C_\C$ such that (cf. (\ref{volc-borne}))
\begin{equation}
\vol_\C(f_*^nD)\leq C_\C\lambda(f)^n \quad \forall n\in\N.
\end{equation}
It follows that 
\begin{equation}
\mvol_\R(f_*^nD) \geq C^+ \vol_\C(f_*^nD)^\alpha \quad\forall n\in\N,
\end{equation}
where we have set $C^+=C_\R/C_\C^\alpha$.

Applying the same argument to $f^{-1}$, there exists a positive number $C^-$ such that 
\begin{equation}
{\mvol_\R(f_*^{-n}D) \geq C^- \vol_\C(f_*^{-n}D)^\alpha} \quad \forall n\in\N.
\end{equation}
Hence we obtain (\ref{eq-une-orbite}), with~${C=\min(C^+,C^-)}$.
\end{proof}

\begin{lemm}\label{nb fini d'orbites}
There are finitely many real ample divisors $D_1,\cdots,D_r$ on~$X$ such that any real ample divisor $D$ on $X$ is algebraically equivalent to one of the form
\begin{equation}
\sum_{k=1}^{s} f_*^nD_{j_k},
\end{equation}
with $n\in\Z$ and $j_k\in\{1,\cdots,r\}$.
\end{lemm}

\begin{proof}
Snce $\NS(X_\R;\R)$ has dimension $2$ and the isotropic eigenvectors $\theta^+_f$ and $\theta^-_f$ belong to $\b{\Amp(X_\R)}$, then we have the equality
\begin{equation}
\Amp(X_\R) = \NS^+(X_\R),
\end{equation}
and this open convex cone is just the open quadrant delimited by the half-lines~$\R^+\theta^+_f$ and $\R^+\theta^-_f$. 
The integer points in this cone correspond to classes of real ample divisors. Let $\theta_1$ be such a point that we choose to be primitive, and let $\theta_2 = f_* \theta_1$ (observe that $\theta_2$ is also primitive). Denote by $\mathcal{D}$ the closed convex cone of $\NS(X_\R;\R)$ bordered by half-lines $\R^+\theta_1$ and $\R^+\theta_2$. By construction $\mathcal{D} \backslash \{0\}$ is a fundamental domain for the action of $f_*$ on $\Amp(X_\R)$ (see Figure~\ref{dom-fond}).

\begin{figure}[h]
\begin{center}
\scalebox{0.4}{\input{fig3.pdftex_t}}
\caption{The fundamental domain $\mathcal{D}\backslash\{0\}$.}\label{dom-fond}
\end{center}
\end{figure}

Denote by $\theta_3,\theta_4,\cdots,\theta_r$ the entire points inside the parallelogram whose vertices are $0$, $\theta_1$, $\theta_1+\theta_2$ and $\theta_2$. Any point in $\mathcal{D}$ can be expressed uniquely as 
\begin{equation}
k_1\theta_1 + k_2\theta_2 + \theta_j \quad \text{or} \quad k_1\theta_1+k_2\theta_2,
\end{equation}
with $(k_1,k_2)\in\N^2$ and $j \in \{3,\cdots,r\}$. For all real ample divisors~$D$, there is~$n\in\Z$ such that $f_*^{-n}[D] \in \mathcal{D}$, so we are done by setting $D_1,\cdots,D_r$ real ample divisors whose classes are $\theta_1,\cdots,\theta_r$.
\end{proof}

We go back to the proof of Theorem \ref{formule-exacte}. Let $q$ be a positive integer such that the divisors $D'_1=qD_1,\cdots,D'_r=qD_r$ are all very ample. By Lemma~\ref{une orbite}, there exists a positive number $C$ such that
\begin{equation}
\mvol_\R(f_*^nD'_j) \geq C \vol_\C(f_*^nD'_j)^\alpha \quad \forall j \in \{1,\cdots,r\},~\forall n \in \Z.
\end{equation}
Let $D$ be a real ample divisor whose Chern class is $q$-divisible. There are $n \in \Z$ and $j_1,\cdots,j_s \in \{1,\cdots,r\}$ such that
\begin{equation}
{[D] = \sum_{k=1}^sf_*^n[D'_{j_k}]}.
\end{equation}
Then
\begin{align}
\mvol_\R(D) & \geq \sum_k\mvol_\R(f_*^nD'_{j_k})\\
& \geq C \sum_k\vol_\C(f_*^nD'_{j_k})^\alpha\label{ligne1}\\
& \geq C \left(\sum_k\vol_\C(f_*^nD'_{j_k})\right)^\alpha\label{ligne2}\\
& = C \vol_\C(D)^\alpha.
\end{align}
From (\ref{ligne1}) to (\ref{ligne2}), we have used the following special case of Minkowski inequality:
\begin{equation}\label{minkowski}
\left(\sum_{k=1}^s|x_k|\right)^\alpha\leq\sum_{k=1}^s|x_k|^\alpha \quad \forall\alpha\in(0,1].
\end{equation}

Hence we see that $\alpha$ belongs to $\mathcal{A}(X)$, and Theorem \ref{formule-exacte} is proved.
\end{proof}

\begin{rema}
We do not know if concordance is achieved in Theorem~\ref{formule-exacte}.
\end{rema}

%%%%%%%%%%%%%%%%%%%%%%%%%%%%%%%%%%%%%%%%%%%%%%%%%%%%%%%%%%%%%%%%%%%%
% ABELIAN SURFACES
%%%%%%%%%%%%%%%%%%%%%%%%%%%%%%%%%%%%%%%%%%%%%%%%%%%%%%%%%%%%%%%%%%%%

\chapter{Concordance for Abelian Surfaces}\label{chap-tores}

The aim of this chapter is to prove the following theorem, which describes exhaustively the concordance for real abelian surfaces.

\begin{theo}\label{thm-surf-ab}
Let $X$ be a real abelian surface. We have the following alternative:
\begin{enumerate}
\item $\rho(X_\R)=1$ and $\alpha(X)=1$.
\item $\rho(X_\R)=2$ and
\begin{enumerate}
\item if the intersection form represents $0$ on $\NS(X_\R;\Z)$, then ${\alpha(X)=1}$;
\item otherwise, $\alpha(X)=1/2$.
\end{enumerate}
\item $\rho(X_\R)=3$ and $\alpha(X)=1/2$.
\end{enumerate}
The concordance is achieved in all cases. It equals $1/2$ if and only if $X$ admits real loxodromic type automorphisms.
\end{theo}

Of course the case $\rho(X_\R)=1$ is trivial (cf. Proposition \ref{rho=1}), so we just have to prove the last two cases.

%%%%%%%%%%%%%%%%%%%%%%%%%%%%%%%%%%
% Preliminaries
%%%%%%%%%%%%%%%%%%%%%%%%%%%%%%%%%%

\section{Generalities on abelian surfaces}

A {\it real abelian variety} $X$ is a real algebraic variety whose underlying complex manifold $X(\C)$ is a complex torus $\C^g/\Lambda$. We say \emph{real elliptic curve} when~${g=1}$, and \emph{real abelian surface} when $g=2$. As~$X(\R)\neq\emptyset$, we can assume that the antiholomorphic involution~$\sigma_X$ comes from the complex conjugation on $\C^g$, and the lattice $\Lambda$ has the form
\begin{equation}\label{reseau}
\Lambda=\Z^g\oplus\tau\Z^g,
\end{equation}
where $\tau\in\M_g(\C)$ is such that 
\begin{equation}
\Im(\tau)\in\GL_g(\R) \quad \text{and} \quad 2\Re(\tau)=\begin{pmatrix}I_r&0\\0&0\end{pmatrix},
\end{equation}
the integer~$r$ being characterized by the fact that $X(\R)$ has $2^{g-r}$ connected components (cf. \cite[\textsection IV]{silhol}).

\begin{nota}[Elliptic curves]
A real elliptic curve has the form
\begin{equation}
E=\C/(\Z\oplus\tau\Z),
\end{equation}
where $\tau=iy$ or $\tau=1/2+iy$, with $y>0$. We denote by $E_y$ and $E'_y$ these elliptic curves, and by $\Gamma_y$ and $\Gamma'_y$ the corresponding lattices in $\C$. Namely :
\begin{equation}
\left\{\begin{split}
\Gamma_y &=\Z\oplus iy\Z \\
E_y &= \C/\Gamma_y
\end{split}\right.
\quad\text{and}\quad
\left\{\begin{split}
\Gamma'_y &= \Z\oplus (1/2+iy)\Z \\
E'_y &= \C/\Gamma'_y
\end{split}\right.
\end{equation}
Note that $E_y(\R)$ has two connected components, whereas $E'_y(\R)$ has only one connected component.
\end{nota}

A (real) {\it homomorphism} between two real abelian varieties is a holomorphic map
\begin{equation}
f:\C^g/\Lambda\longrightarrow X'=\C^{g'}/\Lambda'
\end{equation} which is compatible with the real structures (that is, $\sigma_{X'}\circ f=f\circ\sigma_X$) and which respects the abelian group structures (this is equivalent to $f(0)=0$). Such a map lifts to a unique $\C$-linear map~${F:\C^g\to\C^{g'}}$ such that $F(\Lambda)\subset\Lambda'$, whose matrix has integer coefficients (for ${F(\Z^g)\subset\Lambda'\cap\R^{g'}=\Z^{g'}}$). We also talk about \emph{endomorphisms}, \emph{isomorphisms} and \emph{automorphisms} of real abelian varieties. Observe that in this context, \textbf{automorphisms are asked to preserve the origin}.

% Isogenies

\subsection{Isogenies and Picard numbers}

A real {\it isogeny} between two real abelian varieties of same dimension is a homomorphism of real abelian varieties that is surjective, which means that its matrix has maximal rank. Two real abelian varieties are said to be {\it isogenous} when there exists an isogeny from one to the other (this is an equivalence relation; cf. \cite[1.2.6]{birkenhake-lange}).

\begin{exem}\label{ey-e'y}
The multiplication by $2$ in $\C$ sends the lattice $\Gamma_y$ into $\Gamma'_y$, and conversely. Thus we have the following isogenies of real elliptic curves:
\begin{equation}
E_y \overset{\times 2}\longrightarrow E'_y
\quad\text{and}\quad
E'_y \overset{\times 2}\longrightarrow E_y.
\end{equation}
Thus the two types of real elliptic curves $E_y$ and $E'_y$ are in the same class of isogeny.
\end{exem}

\begin{rema}
The real Picard number does not change by isogeny. Indeed, any isogeny $f:X\rightarrow X'$ gives rise to a homomorphism 
\begin{equation}
{f^*:\NS(X'_\R;\Z)\rightarrow\NS(X_\R;\Z)}
\end{equation}
which is injective, hence $\rho(X'_\R)\leq\rho(X_\R)$; the other inequality follows by the symmetry of the isogeny relation.
\end{rema}

\begin{prop}\label{isog}
Any real abelian surface $X$ is isogenous to $\C^2/\Lambda$, with
\begin{equation}\label{reseau2}
\Lambda=\Z^2\oplus iS\Z^2,
\end{equation}
where $S=\left(S_{ij}\right)\in\M_2(\R)$ is a symmetric positive definite matrix.

When $X$ is on this form, the Néron--Severi group $\NS(X_\C)$ identifies itself with $\R$-bilinear forms $\omega$ whose matrix in the $\R$-basis $(e_1,e_2,iSe_1,iSe_2)$ of $\C^2$ writes down
\begin{equation}
\Omega=\begin{pmatrix}A&B\\-{}^t\!B&C\end{pmatrix},
\end{equation}
with $A,B,C\in\M_4(\Z)$, $A$ and $C$ antisymmetric, and satisfying both conditions
\begin{equation}
C=\det(S)A\quad\text{and}\quad SB={}^t\!BS.
\end{equation}
Elements of $\NS(X_\R)$ are those which satisfy the supplementary condition
\begin{equation}
A=C=0.
\end{equation}
Their respective ranks are given by
\begin{align}
\rho(X_\C)&=6-\rk_\Q(S_{11},S_{12},S_{22})-\rk_\Q(1,\det(S))&&\in\{1,2,3,4\},\\
\rho(X_\R)&=4-\rk_\Q(S_{11},S_{12},S_{22})&&\in\{1,2,3\}.\label{nbpicard-ab}
\end{align}
\end{prop}

\begin{proof}
The existence of a real polarization on $X=\C^g/\Lambda$ (see \cite[\textsection IV.3]{silhol}) implies that the lattice $\Lambda$ can be set on the form
\begin{equation}
D\Z^2\oplus\tau\Z^2,
\end{equation}
the matrix~$D$ being diagonal with integer coefficients, and the matrix $\tau$ being symmetric, with $S=\Im(\tau)$ positive definite and $2\Re(\tau)$ an integer matrix. Hence the dilation by $2$ in $\C^2$ gives rise to a real isogeny
\begin{equation}
\C^2 / \Lambda \overset{\times 2}\longrightarrow \C^2 / (\Z^2 \oplus iS\Z^2).
\end{equation}

The description of $\NS(X_\C)$ comes from \cite[\textsection 1, 3.4]{birkenhake-lange}. Elements $\omega$ of~$\NS(X_\R)$ must moreover satisfy $\sigma^*\omega=-\omega$ (cf. \cite[\textsection IV (3.4)]{silhol}), where~$\sigma$ denotes the complex conjugation on $\C^2$. The matrix of $\sigma^*$ in the prescripted $\R$-basis is
\begin{equation}
\Sigma=\begin{pmatrix}I_2&0\\0&-I_2\end{pmatrix},
\end{equation}
so this condition becomes ${}^t\Sigma\Omega\Sigma=-\Sigma$, hence
\begin{equation}
\begin{pmatrix}A&-B\\{}^t\!B&C\end{pmatrix} = \begin{pmatrix}-A&-B\\{}^t\!B&-C\end{pmatrix},
\end{equation}
which gives $A=C=0$. The assertion on the Picard numbers follows from an easy calculation.
\end{proof}

\begin{rema}
The difference $\rho(X_\C)-\rho(X_\R)$ between complex and real Picard number is $1$ or $0$, depending on whether $\det(S)\in\Q$ or not. The exceptional value $\rho(X_\C)=4$ is achieved exactly when $X$ is isogenous to the square of an elliptic curve with complex multiplication (cf. \cite[\textsection 2 7.1]{birkenhake-lange}).
\end{rema}

The interest of looking at real abelian surfaces up to isogeny resides in the following proposition.

\begin{prop}\label{inv-iso}
Let $X$ and $X'$ be two isogenous real abelian varieties. Then ${\mathcal{A}(X)=\mathcal{A}(X')}$, and consequently $\alpha(X)=\alpha(X')$.
\end{prop}

\begin{proof}
Since the isogeny relation is symmetric, it is enough to show the inclusion~$\mathcal{A}(X)\subset\mathcal{A}(X')$. Let $f: X' \rightarrow X$ be an isogeny. For $\K=\R$ or $\C$, denote by~$f_\K$ the induced map from $X'(\K)$ to $X(\K)$. We take an arbitrary K\"ahler metric on $X$, and then we take its pullback on $X'$, so that $f$ is locally an isometry for the respective metrics.

Fix any $\alpha\in\mathcal{A}(X)$. There exist $C>0$ and $q\in\N^*$ such that any real ample divisor $D$ on $X$, whose Chern class is $q$-divisible, satisfies 
\begin{equation}
\mvol_\R(D)\geq C\vol_\C(D)^\alpha.
\end{equation}
Since $f^*:\NS(X_\R;\Z)\rightarrow\NS(X'_\R;\Z)$ is an injective homomorphism, its image has finite index $n$. Let $D'$ be a real ample divisor on $X'$ whose class is $nq$-divisible. Then there is a real ample divisor $D$ on $X$ with $[D']=[f^*D]$, and furthermore $[D]$ is $q$-divisible.

Any point on the curve $D(\R)$ has exactly $\deg(f_\R)$ preimages, hence 
\begin{equation}
\vol_\R(f^*D)=\deg(f_\R)\vol_\R(D)
\end{equation}
by the choice of the metrics. Since $f^*$ realizes a bijective map between $\mathcal{V}(D)$ and $\mathcal{V}(D')$, we deduce, taking the upper bound on $\mathcal{V}(D)$, that 
\begin{equation}
\mvol_\R(D')=\deg(f_\R)\mvol_\R(D).
\end{equation}

By the same argument, we also have 
\begin{equation}
\vol_\C(D')=\deg(f_\C)\vol_\C(D)
\end{equation}
So if we set~$C'=C\deg(f_\R)/{\deg(f_\C)}^\alpha$, we obtain 
\begin{equation}
\mvol_\R(D') \geq C'\vol_\C(D')^\alpha.
\end{equation}
This shows that the exponent $\alpha$ is contained in $\mathcal{A}(X')$.
\end{proof}

% Entropy

\subsection{Entropy of automorphisms}

The following fact is very specific to tori, for which automorphisms come from linear maps.

\begin{prop}\label{alpha-leq-1/2}
Let $f$ be an automorphism of a real abelian surface $X$. Then
\begin{equation}
\h(f_\C)=2\h(f_\R).
\end{equation}
Accordingly, $\alpha(X)\leq 1/2$ as soon as $X$ admits real loxodromic type automorphisms.
\end{prop}

\begin{proof}
We lift the automorphism $f$ to a $\C$-linear map $F:\C^2\to\C^2$ whose matrix is in $\SL_2(\Z)$ (replacing $f$ by $f^2$ if necessary). If $F$ has spectral radius~$1$, then it is obvious that $\h(f_\R)=\h(f_\C)=0$. Otherwise, $F$ has two distinct eigenvalues $\lambda$ and~$\lambda^{-1}$, with $|\lambda|>1$. As a $\R$-linear map of $\C^2$, $F$ has eigenvalues~${(\lambda,\lambda,\lambda^{-1},\lambda^{-1})}$ (with multiplicities), thus $\h(f_\C)=2\log|\lambda|$ (see, for instance, \cite[2.6.4]{brin-stuck}). Restricted to $\R^2$, $F$ has eigenvalues $(\lambda,\lambda^{-1})$, hence $\h(f_\R)=\log(|\lambda|)$.

The last part is a consequence of Theorem \ref{entropie}.
\end{proof}

\begin{rema}
\begin{enumerate}
\item
If $f$ is a real loxodromic automorphism of a real abelian surface, then $\lambda(f)$ is an algebraic integer of degree at most $\rho(X_\R)\leq 3$. Since its minimal polynomial has $\lambda(f)$ and $\lambda(f)^{-1}$ as roots, and all the other roots (if they exist) have modulus one, then this polynomial has degree $2$. The minimal quadratic integer being $\lambda_2=\frac{3+\sqrt{5}}{2}$, we have
\begin{equation}
\h(f_\C)\geq \log(\lambda_2).
\end{equation}
This lower bound is reached : for instance, take $X=\C^2/\Z[j]^2$, and $f$ given by the matrix $\begin{pmatrix}2&1\\1&1\end{pmatrix}$.
\item
Now if $f$ is a complex loxodromic automorphism of an abelian surface, the degree of $\lambda(f)$ is at most $\rho(X_\C)\leq 4$. Thus $\lambda(f)$ is either a quadratic integer or a degree four Salem number, and the minimal possible value is the largest root  $\lambda_4$ of the polynomial $x^4-x^3-x^2-x+1$, so we get
\begin{equation}
\h(f_\C)\geq \log(\lambda_4).
\end{equation}
Again this lower bound is reached, on the same surface $X=\C^2/\Z[j]^2$ : for instance, take $f$ given by the matrix $
\begin{pmatrix}j&-1\\j&j^2\end{pmatrix}$ (see also \cite{mcmullen-k3-ent}). This gives a minimum which is strictly less than the one obtained for real automorphisms.
\end{enumerate}
\end{rema}

%%%%%%%%%%%%%%%%%%%%%%%%%%%%%%%%%%
% Picard number 2
%%%%%%%%%%%%%%%%%%%%%%%%%%%%%%%%%%

\section{Abelian surfaces with Picard number $2$}

% Hyperbolic rank 2 lattices

\subsection{Hyperbolic rank $2$ lattices}\label{section-reseaux}

By definition, a \emph{lattice} is a free abelian group $L$ of finite rank, equipped with a nondegenerate symmetric bilinear form $\varphi$ taking integral values. We say the lattice is \emph{hyperbolic} when the signature of the induced quadratic form on $L\otimes\R$ is $(1,{\rm rank}(L)-1)$. The determinant of the matrix of $\varphi$ in a base of $L$ is the same for all bases. Its absolute value is a positive integer, called the \emph{discriminant} of the lattice.

Let $(L,\varphi)$ be a rank 2 hyperbolic lattice. There are exactly two isotropic lines in $L\otimes\R$. The discriminant $\delta$ is a perfect square if and only if the quadratic form associated to $\varphi$ represents $0$, which means that there exists a nonzero isotropic point in $L$, or to say it otherwise both isotropic lines in~$L\otimes\R$ are rational.

Suppose that $\delta$ is no perfect square. Then the study of Pell--Fermat equation implies the existence of a \emph{hyperbolic isometry} of $L$, that is, an isometry whose spectral radius is greater than $1$. Such an isometry spans a finite index subgroup of the isometries of $L$. To be more precise, the group $\SO(L,\varphi)$ of direct isometries of $L$ (those with determinant $+1$) is an abelian group isomorphic to $\Z\times\Z/2\Z$, and any infinite order element in $\SO(L,\varphi)$ is hyperbolic.

Conversely if $\delta$ is a perfect square, there is no hyperbolic isometry, and the isometry group is finite. More precisely,  $\SO(L,\varphi)=\{\id,-\id\}\simeq\Z/2\Z$.

\begin{exem}
Let $X$ be a real algebraic surface with $\rho(X_\R)=2$. Then the group~$\NS(X_\R;\Z)$, equipped with the intersection form, is a rank 2 hyperbolic lattice.
\end{exem}

% with elliptic fibrations

\subsection{When $X$ has a real elliptic fibration}

Let $X$ and $B$ be two complex algebraic varieties. An \emph{elliptic fibration} on $X$ is a holomorphic map~${\pi:X\to B}$ that is proper and surjective, and such that the generic fiber is an elliptic curve. When the varieties $X$, $B$ and the morphism $\pi$ are defined over $\R$, we say the elliptic fibration is \emph{real}.

\begin{prop}
Let $X$ be a real abelian surface with $\rho(X_\R)=2$. The following are equivalent:
\begin{enumerate}
\item
the intersection form on $\NS(X_\R;\Z)$ represents $0$;
\item
there exists a real elliptic fibration on $X$;
\item
$X$ is isogenous to the product of two elliptic curves $E_{y_1}$ and $E_{y_2}$.
\end{enumerate}
In this case, the concordance of $X$ equals $1$, and it is achieved.
\end{prop}

\begin{rema}
The elliptic curves $E_{y_1}$ and $E_{y_2}$ cannot be isogenous, for otherwise the Picard number would be $3$.
\end{rema}

\begin{proof}
$(1)\Rightarrow (2)$: Let $\theta$ a nonzero primitive point in $\NS(X_\R;\Z)$ such that~${\theta^2=0}$. After changing $\theta$ into $-\theta$ if necessary, there exists a real effective and irreducible divisor $D$ whose class is $\theta$ (here we use the fact that $\Nef(X_\R)$ is the closure of~$\NS^+(X_\R)$; cf. \cite[1.5.17]{lazarsfeld}). By the genus formula, the arithmetic genus of~$D$ is $1$. Since an abelian surface does not have any rational curve, $D$ must be a real elliptic curve. We may suppose that~$D$ goes through~$0$ (if not, we translate it and obtain an algebraically equivalent divisor), and thus it is a real subtorus. Now the canonical projection
\begin{equation}
{\pi:X\to X/D}
\end{equation}
is a real elliptic fibration.

$(2)\Rightarrow (3)$ follows from the Poincar\'e reducibility theorem (see \cite[\textsection VI 8.1]{debarre}).

$(3)\Rightarrow (1)$: Let $f:X\to E_{y_1}\times E_{y_2}$ be an isogeny. The effective divisor~$D$ given by~$f^*(E_{y_1}\times\{0\})$ has self-intersection $0$, so the intersection form represents $0$.

As the nef cone of $E_{y_1}\times E_{y_2}$ is spanned by $[E_{y_1}\times\{0\}]$ and $[\{0\}\times E_{y_2}]$, Proposition \ref{conenef} implies that the interval $\mathcal{A}(E_{y_1}\times E_{y_2})$ is equal to $[0,1]$. By invariance under isogeny, we also have $\mathcal{A}(X)=[0,1]$.
\end{proof}

% no elliptic fibration

\subsection{When $X$ has no real elliptic fibration}

\begin{theo}\label{delta}
Let $X$ be a real abelian surface with $\rho(X_\R)=2$. Assume that the intersection form on $\NS(X_\R;\Z)$ does not represent $0$. Then
\begin{enumerate}
\item
there exists a real loxodromic type automorphism on $X$;
\item
the concordance of $X$ equals $1/2$ and it is achieved.
\end{enumerate}
\end{theo}

We use the following result (see, for instance, \cite{X}):

\begin{theo}[Torelli theorem for tori]\label{torelli-tori}
Let $X$ be a real abelian surface and let $\phi$ be a Hodge isometry\footnote{By definition, a Hodge isometry of $H^2(X_\C;\Z)$ is a bijective linear map which preserves the intersection form and the Hodge decomposition.} of $H^2(X_\C;\Z)$ that preserves the ample cone and has determinant $+1$. Then there exists an automorphism $f$ of $X_\C$, unique up to a sign, such that $f_*=\phi$. If moreover $\phi$ commutes with the involution $\sigma_X^*$, then~$f$ or $f^2$ is a \emph{real} automorphism.
\end{theo}

\begin{rema}
To prove the last part, it is enough to remark that 
\begin{equation}
\sigma_X\circ f\circ\sigma_X=\pm f^{-1}
\end{equation}
by the uniqueness part.
\end{rema}

\begin{lemm}\label{lemme-alg-com}
Let $L$ be an free abelian group of finite rank, $L'$ be a finite index subgroup of $L$ and $\phi'$ be an automorphism of $L'$. Then some positive iterate $\phi'^k$ extends to an automorphism $\phi$ on $L$.
\end{lemm}

\begin{proof}
Denote by $q$ the exponent of the group $L/L'$, so that  $qL\subset L'$. As~$\phi'$ projects to an automorphism of $L'/qL'$ that has finite order $k$, it follows that~${\phi'^k(qL)\subset qL}$. Let $\mu_q:L\to qL$ be the isomorphism defined by $\theta\mapsto q\theta$. Then the automorphism~${\phi=\mu_q^{-1}\circ{\phi'^k}_{|qL}\circ\mu_q}$ satisfies the desired property.
\end{proof}

\begin{proof}[Proof of Theorem \ref{delta}]
Since the intersection form does not represent $0$, there exists a hyperbolic isometry $\phi_1$ of $L_1=\NS(X_\R;\Z)$ (cf. \textsection \ref{section-reseaux}). Replacing $\phi_1$ by~$\phi_1^2$ if necessary, we may suppose that $\phi_1$ preserves the cone $\Amp(X_\R)$ and that $\det(\phi_1)=1$. Denote by $L_2$ the orthogonal of $L_1$ in $L=H^2(X_\C;\Z)$, and by~$L'$ the direct sum $L_1\oplus L_2$. The subgroup $L'$ has finite index in $L$, and so, by Lemma~\ref{lemme-alg-com},~${\phi_1^k\oplus \id_{L_2}}$ extends to an automorphism $\phi$ on $H^2(X_\C;\Z)$ for some~$k\in\N^*$. It is clear by construction that $\phi$ satisfies the assumptions of Theorem \ref{torelli-tori}. Thus there exists a real automorphism $f$ on $X$ such that~$f_*=\phi^2$. Its dynamical degree is greater than $1$, as a power of the spectral radius of $\phi_1$.

The equality $\alpha(X)=1/2$ follows from Theorem \ref{formule-exacte} and Proposition \ref{alpha-leq-1/2}. In order to show that the concordance is achieved, we replace $\alpha$ by
\begin{equation}
\frac{1}{2}=\frac{\h(f_\R)}{\h(f_\C)}
\end{equation}
inside the proof of Theorem \ref{formule-exacte} by using the following lemma, which improves the inequality of Corollary \ref{minorer-mvolr}.
\end{proof}

\begin{lemm}
Let $f$ be a real automorphism of a real abelian surface $X$. Assume that $\lambda=\exp(\h(f_\R))>1$. Then for all very ample real divisors $D$ on $X$, there exists $C>0$ such that 
\begin{equation}
\mvol_\R(f_*^nD) \geq C \lambda^n \quad \forall n\in\N.
\end{equation}
\end{lemm}

\begin{proof}
The automorphism $f$ lifts to a linear self-map $F$ of $\R^2$. Replacing~$f$ with~$f^2$ if necessary, $F$ has eigenvalues $\lambda$ and $\lambda^{-1}$. We choose a scalar product on $\R^2$ such that the eigenlines $\mathcal{D}^+$ and $\mathcal{D}^-$, respectively, associated to $\lambda$ and~$\lambda^{-1}$, are orthogonal. Then we take on $X(\R)$ the Riemannian metric induced by this scalar product.

If necessary, we change $D$ (by translation) into an algebraically equivalent divisor containing the origin as a smooth point. Then the curve $D(\R)$ contains a smooth simple arc $\gamma$ through $0$. Let $\tilde\gamma$ be the lift of $\gamma$ to $\R^2$ containing the origin, and let $p:\R^2\to\R^2$ be the projection on $\mathcal{D}^+$ with direction $\mathcal{D}^-$. Then 
\begin{equation}
\begin{split}
\mvol(f_*^nD)
&\geq
\length(f^n(\gamma))\\
&=
\length(F^n(\tilde\gamma))\\
&\geq
\length(p\circ F^n(\tilde\gamma))\\
&=
\lambda^n\length(p(\tilde\gamma)).
\end{split}
\end{equation}
Observe that $\length(p(\tilde\gamma))>0$. Indeed, if it were zero, $\tilde\gamma$ would be contained in~$\mathcal{D}^-$, hence the real analytic curve $D(\R)$ would contain the projection of~$\mathcal{D}^-$ on the torus $\R^2/\Z^2$. But then $D(\R)$ would be Zariski-dense, since the line $\mathcal{D}^-$ is irrational : this is impossible. Thus we get the result with $C=\length(p(\tilde\gamma))$.
\end{proof}

%%%%%%%%%%%%%%%%%%%%%%%%%%%%%%%%%%
% Picard number 3
%%%%%%%%%%%%%%%%%%%%%%%%%%%%%%%%%%

\section{Abelian surfaces with Picard number $3$}

\begin{lemm}
Let $X$ be a real abelian surface with $\rho(X_\R)=3$. There exists a real elliptic curve $E$ such that $X$ is isogenous to $E\times E$.
\end{lemm}

\begin{proof}
Changing $X$ by isogeny if necessary (cf. Proposition \ref{isog}), $X$ has the form $\C^2 / (\Z^2\oplus iS\Z^2)$, with $S=\left(S_{ij}\right)\in\M_2(\R)$ symmetric and positive definite. As $\rho(X_\R)=3$, it follows that $\rk_\Q(S_{11},S_{12},S_{22})=1$, so there exists~$m\in\N^*$ such that $mS_{12}$ and $mS_{22}$ are in $\Z S_{11}$ ($S_{11}\neq 0$, since ${\det(S)=S_{11}S_{22}-S_{12}^2>0}$). Then the dilation by $m$ in $\C^2$ gives rise to an isogeny from $X$ to $E_{S_{11}}\times E_{S_{11}}$.
\end{proof}

As a consequence, it is enough to restrict ourselves to the case 
\begin{equation}
X=E\times E,
\end{equation}
where~$E$ is a real elliptic curve. We can also suppose (see example \ref{ey-e'y}) that
\begin{equation}
E=E'_y
\end{equation}
with~${y>0}$. By making this choice, $X(\R)=E'_y(\R)\times E'_y(\R)$ has only one connected component, are the real curves will be more easy to study.

Observe that the group $\GL_2(\Z)$ acts on $X$ and gives many examples of real loxodromic type automorphisms. A consequence from this fact and Proposition \ref{alpha-leq-1/2} is that we already have the inequality
\begin{equation}\label{leq1/2}
\alpha(X)\leq 1/2.
\end{equation}

In order to compute volumes, we choose the standard Euclidean metric on the torus $X=\C^2/(\Gamma'_y)^2$, whose K\"ahler form is given by 
\begin{equation}
\kappa= \frac{i}{2} \left( dz_1 \wedge d\b{z_1} + dz_2 \wedge d\b{z_2} \right).
\end{equation}
For this metric, we have $\vol_\R(E)=1$ and $\vol_\C(E)=y$. In the remaining part of this section, we follow \cite{christol}.

\begin{defi}
A \emph{rational line} on $X$ is the projection of a line of $\C^2$ given by an equation $az_1=bz_2$ with $(a,b)\in\Z^2$. The number $a/b\in\Q\cup\{\infty\}$ is the \emph{slope} of this rational line.
\end{defi}

\begin{exem}
The curves $H$, $V$, and $\Delta$, which are respectively the horizontal, the vertical, and the diagonal of $E\times E$, are rational lines with slopes by $0$, $\infty$, and~$1$, respectively. Proposition \ref{isog} implies that their classes form a base of $\NS(X_\R;\Z)$ (this follows from an easy computation).
\end{exem}

\begin{lemm}\label{vol-rat}
Let $D$ be a rational line on $X$. Then 
\begin{equation}
\vol_\R(D)=C\vol_\C(D)^{1/2},
\end{equation}
with $C=y^{-1/2}=\vol_\R(E)/\vol_\C(E)^{1/2}$.
\end{lemm}

\begin{proof}
Let $a/b$ be the slope of $D$, with coprime integers $a$ and $b$. We compute the length of $D(\R)$ by the Pythagorean theorem:
\begin{equation}
\vol_\R(D)=\sqrt{a^2+b^2}.
\end{equation}
On the other hand, it is clear, from the form of $\kappa$, that
\begin{equation}
\vol_\C(D)=\vol_\C(E)(D\cdot H+D\cdot V).
\end{equation}
We easily check that $D\cdot H=a^2$ and $D\cdot V=b^2$, hence
\begin{equation}
\vol_\C(D)=y(a^2+b^2).
\end{equation}
\end{proof}

The group $\SL_2(\Z)$ acts by automorphisms on $X$, thus by isometries on~$\NS(X_\R;\R)$. This action preserves the ample cone $\Amp(X_\R)$, which here is the same as $\NS^+(X_\R)$ (see \cite[\textsection 1.5.B]{lazarsfeld}).

If we identify the disk $\D=\P(\Amp(X_\R))$ with the Poincar\'e half-plane $\H$, by the unique isometry matching the class of a rational line in $\partial\D$ with the inverse of its slope in $\partial\H=\R\cup\{\infty\}$, then the induced action of $\PSL_2(\Z)$ on~$\D$ corresponds to the standard action by homographies on $\H$. Accordingly we see that the triangle ${\T \subset \D}$, whose vertices are $\P[H]$, $\P[V]$, and $\P[\Delta]$, contains some fundamental domain for the action of $\PSL_2(\Z)$ on $\D$ (see Figure~\ref{triangle}).\\

\begin{figure}[h]
\begin{center}
\scalebox{0.8}{\input{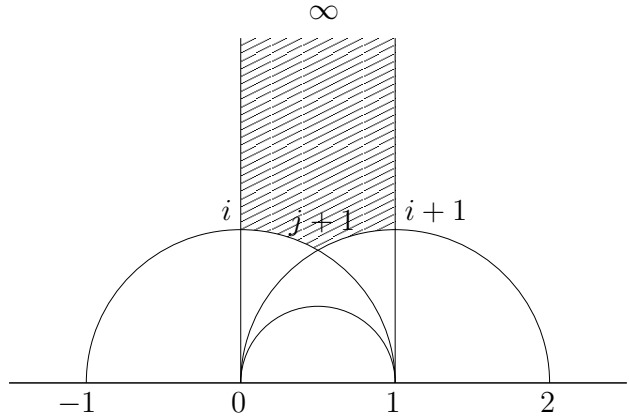}}
\caption{A fundamental domain for the action of $\PSL_2(\Z)$ on~$\H$. This domain is contained in the triangle corresponding to $\T$, whose vertices are $\infty$, $0$ and $1$.}\label{triangle}
\end{center}
\end{figure}

Let $D$ be a real ample divisor on $X$. There exists some $f\in\SL_2(\Z)$ such that~$\P(f_*^{-1} [D])\in\T$. So there are nonnegative numbers $k_1$, $k_2$ and $k_3$ such that
\begin{equation}
f_*^{-1} [D] = k_1 [H] + k_2 [V] + k_3 [\Delta].
\end{equation}
The numbers $k_1$, $k_2$, and $k_3$ are actually integers, for $([H],[V],[\Delta])$ is a base of~$\NS(X_\R;\Z)$. Hence the divisor $D$ is algebraically equivalent to
\begin{equation}
k_1D_1+k_2D_2+k_3D_3,
\end{equation}
where $D_1=f(H)$, $D_2=f(V)$, and $D_3=f(\Delta)$ are rational lines. Then
\begin{align}
\mvol_\R(D) & \geq \sum_j k_j\vol_\R(D_j)\\
& = C \sum_j k_j\vol_\C(D_j)^{1/2}
&&\text{by Lemma \ref{vol-rat}}\\
& \geq C \left(\sum_j {k_j}^2\vol_\C(D_j)\right)^{1/2}
&&\text{by Minkowski inequality $(\ref{minkowski})$}\\
& \geq C \left(\sum_j {k_j}\vol_\C(D_j)\right)^{1/2}\\
& = C \vol_\C(D)^{1/2}.
\end{align}

We deduce that $1/2\in\mathcal{A}(X)$. Thus the concordance is $1/2$ and it is achieved. This ends the proof of Theorem \ref{thm-surf-ab}.

%%%%%%%%%%%%%%%%%%%%%%%%%%%%%%%%%%
% The concordance depends on the real structure
%%%%%%%%%%%%%%%%%%%%%%%%%%%%%%%%%%

\section{The concordance depends on the real structure}\label{dep-str}

The aim of this paragraph is to show that there is a complex abelian surface $X_\C$ which admits two different real structures $\sigma_1$ and $\sigma_2$, and thus two distinct real abelian surfaces $X_1$ and $X_2$, such that $\alpha(X_1)\neq\alpha(X_2)$.\\

Fix a positive irrational number $\pi$ such that $\pi^2\notin\Q$. We consider the complex abelian surface
\begin{equation}
X_\C=E_\pi\times E_\pi,
\end{equation}
with the two following real structures which define two real abelian surfaces~$X_1$ and $X_2$:
\begin{enumerate}
\item $\sigma_1$ is the standard complex conjugation: $\sigma_1(z_1,z_2)=(\b{z_1},\b{z_2})$;
\item $\sigma_2$ is given by $\sigma_2(z_1,z_2)=(\b{z_1},-\b{z_2})$.
\end{enumerate}

Since $\pi$ is irrational, $E_\pi$ has no complex multiplication, and the Picard number of $X_\C$ is $3$.

We have seen in the last section that
\begin{equation}
\rho(X_1)=3
\quad\text{and}\quad
\alpha(X_1)=1/2.
\end{equation}

Now consider the complex isomorphism
\begin{equation}
\begin{split}
X_\C=E_\pi\times E_\pi &\overset{\psi}\longrightarrow X'_\C=E_\pi\times 
E_{\pi^{-1}}\\
(z_1,z_2) &\longmapsto \left(z_1,i\pi^{-1}z_2\right).
\end{split}
\end{equation}
Then $\psi\circ\sigma_2\circ\psi^{-1}$ is the standard complex conjugation on $X'_\C$, thus $X_2$ is isomorphic to the real abelian surface $X'_\R=E_\pi\times 
E_{\pi^{-1}}$ (with the standard real structure). Since we have chosen $\pi^2$ irrational, the real elliptic curves~$E_\pi$ and~$E_{\pi^{-1}}$ are not isogenous over $\R$ (although they are isomorphic over $\C$, via the complex multiplication by $i\pi^{-1}$), so $X'_\R$ has Picard number $2$ and concordance $1$. It follows that
\begin{equation}
\rho(X_2)=2
\quad\text{and}\quad
\alpha(X_2)=1.
\end{equation}

%%%%%%%%%%%%%%%%%%%%%%%%%%%%%%%%%%%%%%%%%%%%%%%%%%%%%%%%%%%%%%%%%%%%
% K3 SURFACES
%%%%%%%%%%%%%%%%%%%%%%%%%%%%%%%%%%%%%%%%%%%%%%%%%%%%%%%%%%%%%%%%%%%%

\chapter{Concordance for K3 Surfaces}\label{chap-k3}

% Subsection 5.1 - Preliminaries

\section{Generalities on K3 surfaces}

We recall that a K3 surface is a compact complex surface $X$ such that 
\begin{equation}
{H^1(X_\C;\Z)=0}
\end{equation}
and the canonical divisor $K_X$ is trivial.  Such a surface is necessarily Kähler. A real K3 surface is a K3 surface equipped with a real structure. See \cite[\textsection 8.4]{dik} for a review of the possible topological types for the $X(\R)$.

\textbf{In this chapter only, all K3 surfaces are supposed to be algebraic.}

% 5.1.1 - exceptional curves

\subsection{Exceptional curves}

On a real K3 surface, a complex irreducible curve $C$ with negative self-intersection must have self-intersection $-2$, by the genus formula. By contrast, when $C$ is a real curve that is irreducible over~$\R$ and has negative self-intersection, we can also have $C^2=-4$. Indeed the curve~$C$ can also have the form~${E+E^\sigma}$, where $E$ is a complex $(-2)$-curve with $E\cdot E^\sigma=0$.

By extension we call a real effective divisor \emph{exceptional} if it has self-intersection $-2$ or it has the form $E+E^\sigma$, where $E$ is a complex curve with $E^2=-2$ and~${E\cdot E^\sigma=0}$ (this implies $(E+E^\sigma)^2=-4$). We denote by~${\Delta\subset\NS(X_\R;\Z)}$ the set of classes of exceptional curves. By description of the K\"ahler cone (cf. \cite[\textsection VIII (3.9)]{bpvdv}),
\begin{equation}
\Amp(X_\R)=\{\theta\in\NS^+(X_\R)\,|\,\theta\cdot d>0\quad\forall d\in\Delta\}.
\end{equation}
This cone coincides with one of the chambers of $\NS^+(X_\R)\backslash\bigcup_{d\in\Delta}d^\perp$. As a special case, we see that the lack of exceptional curves is equivalent to the equality~$\Amp(X_\R)=\NS^+(X_\R)$.

% 5.1.2 - Torelli

\subsection{Torelli theorem} Let us recall the following result, which describes automorphisms of K3 surfaces (see {\cite[\textsection VIII (11.1) \& (11.4)]{bpvdv}}, {\cite[\textsection VIII (1.7)]{silhol}}).

\begin{theo}[real Torelli theorem]\label{torelli}
Let $X$ be a real $K3$ surface and let $\phi$ be a Hodge isometry of $H^2(X_\C;\Z)$ that preserves the ample cone. Then there exists a unique automorphism $f$ of $X_\C$ such that $f_*=\phi$. If moreover~$\phi$ commutes with the involution $\sigma_X^*$, then $f$ is a \emph{real} automorphism.
\end{theo}

\begin{rema}\label{kernel}
As a consequence, we see that the kernel of the representation 
\begin{equation}
\Aut(X_\R)\to \Isom(\NS(X_\R;\Z))
\end{equation}
is finite, where $\Isom(\NS(X_\R;\Z))$ denotes the group of isometries of $\NS(X_\R;\Z)$. Indeed, the space $H^2(X_\C;\R)$ decomposes into the orthogonal direct sum~${V_1\oplus V_2 \oplus V_3}$, where~$V_1=\NS(X_\R;\R)$, $V_2$ stands for the orthogonal of $V_1$ in $H^{1,1}(X_\C;\R)$, and~$V_3=(H^{0,2}\oplus H^{2,0})(X_\C;\R)$. The intersection form is negative definite on $V_2$ and positive definite on $V_3$. If $f\in\Aut(X_\R)$ is in the kernel of the representation, then the induced map $f_*$ on $H^2(X_\C;\R)$ preserves the intersection form, so it is contained in the compact subgroup~${\{\id_{V_1}\}\oplus\Isom(V_2)\oplus\Isom(V_3)}$. Since the matrix of $f_*$ must also have integer coefficients in a base of $H^2(X_\C;\Z)$, there are finitely many possibilities for $f_*$, and thus for $f$ by uniqueness in the Torelli theorem. See also \cite[D.1.4]{dik} for a proof independant of the Torelli theorem.
\end{rema}

% Subsection 5.2 - Picard number 2

\section{Picard number $2$}

When $\rho(X_\R)=2$, the nef cone $\Nef(X_\R)$ has exactly two extremal rays. We say this cone is \emph{rational} if both rays are rational, that is, if they contain an element of $\NS(X_\R;\Z)\backslash\{0\}$.

\begin{theo}\label{k3-rho=2}
Let $X$ be a real K3 surface with $\rho(X_\R)=2$.
\begin{enumerate}
\item
If the intersection form on $\NS(X_\R;\Z)$ represents $0$, or if there are exceptional curves on $X$, then the group $\Aut(X_\R)$ is finite, the cone $\Nef(X_\R)$ is rational and  $\alpha(X)=1$, the concordance being achieved.
\item
Otherwise $X$ admits a real loxodromic type automorphism $f$. Such an automorphism spans a finite index subgroup of $\Aut(X_\R)$, and 
\begin{equation}
\alpha(X)=\frac{\h(f_\R)}{\h(f_\C)}.
\end{equation}
\end{enumerate}
\end{theo}

\begin{rema}
The intersection form on $\NS(X_\R;\Z)$ represents $0$ if and only if there exists some real elliptic fibration (see \cite[Corollary 3 in \textsection 3]{safarevic-torelli}).
\end{rema}

\begin{proof}
\emph{First case: the intersection form represents $0$.} Then the group $\Isom(\NS(X_\R;\Z))$ is finite (cf. \textsection \ref{section-reseaux}), and so is the kernel of the representation $\Aut(X_\R)\to\Isom(\NS(X_\R;\Z))$, hence $\Aut(X_\R)$ must be finite. The extremal rays of $\Nef(X_\R)$ are either isotropic half-lines, or orthogonal to the class of an exceptional curve: they are rational in both cases.

\emph{Second case: the intersection form does not represent $0$ and there are exceptional curves.} We show that $X$ has many exceptional curves. To be more precise, let~$d\in\Delta$ be the class of such a curve, and let $\phi'$ be a hyperbolic isometry of $\NS(X_\R;\Z)$. Replacing $\phi'$ by a positive iterate if necessary, the isometry~${\phi'\oplus\id_{\NS(X_\R;\Z)^\perp}}$ extends to an isometry $\phi$ on $\NS(X_\C;\Z)$ (cf. Lemma \ref{lemme-alg-com}). Note that we cannot apply the Torelli theorem here, for $\phi$ does not preserve the ample cone (even if we suppose that it preserves $\NS^+(X_\R)$). Nevertheless, we show the following lemma.

\begin{lemm}
For all $n\in\Z$, $\pm\phi^n(d)$ is the class of an exceptional curve.
\end{lemm}

\begin{proof}
If $d^2=-2$, then $\phi^n(d)^2=-2$. The Riemann--Roch formula shows that ${h^0(X,\O_X(\phi^n(d)))+h^0(X,\O_X(-\phi^n(d)))\geq 2}$, hence $\phi^n(d)$ or $-\phi^n(d)$ is the class of an effective divisor.

Otherwise, $d=e-\sigma^*e$ (the minus sign comes from (\ref{pb-sign})) with $e^2=-2$ and ${e\cdot\sigma^*e=0}$, where $e$ is the class of a complex effective divisor. From the same argument, it comes that $\phi^n(e)$ or $-\phi^n(e)$ is also the class of a complex effective divisor. Thus $\pm\phi^n(d)=\pm(\phi^n(e)-\sigma^*\phi^n(e))$ is the class of a real effective divisor.
\end{proof}

Observe that when $n$ goes to $\pm\infty$, the lines $\R\phi^n(d)$ converge to the two isotropic lines. Consequently, there are exceptional curves whose classes are arbitrarily close to both isotropic directions. The cone $\Amp(X_\R)$ is one of the chambers of $\NS^+(X_\R)\backslash\bigcup_{d\in\Delta}d^\perp$, so both its extremal rays must be orthogonal to classes of exceptional curves, hence they are rational. Since the subgroup of $\Isom(\NS(X_\R;\Z))$ whose elements fix or exchange these extremal rays is finite, the group $\Aut(X_\R)$ is also finite by Remark \ref{kernel}.

In the first two cases, let $\R^+[D]$ be an extremal ray of the cone $\Nef(X_\R)$. By Riemann--Roch, $h^0(X,\O_X(D))\geq 2$, thus $\mvol_\R(D)>0$ by Proposition~\ref{pinceau}. The assertion about $\alpha(X)$ follows, using Proposition \ref{conenef}.

\emph{Third case: the intersection form does not represent $0$ and there is no exceptional curve.} Let $\phi'$ be a hyperbolic isometry of $\NS(X_\R;\Z)$ that preserves the cone~$\Amp(X_\R)=\NS^+(X_\R)$. By the same argument as in the proof of Theorem~\ref{delta}, some iterate $\phi'^k$ extends to a Hodge isometry $\phi$ of~$H^2(X_\C;\Z)$ that commutes with~$\sigma^*$. Then by the Torelli theorem, there exists on $X$ a real automorphism $f$ such that $f_*=\phi$. The dynamical degree $\lambda(f)$ is greater than $1$ as a power of the spectral radius of $\phi'$. Now the representation $\Aut(X_\R)\to\Isom(\NS(X_\R;\Z))$ has finite kernel, and the subgroup generated by $f_*$ has finite index in $\Isom(\NS(X_\R;\Z))$~(cf.~\textsection\ref{section-reseaux}). It follows that $f$ spans a finite index subgroup of $\Aut(X_\R)$. The concordance formula is a consequence of Theorem \ref{formule-exacte}.
\end{proof}

\begin{exem}[\cite{wehler}]\label{wehler}
Let $Y$ be the $3$-dimensional flag variety
\begin{equation}
Y=\left\{(P,L)\in\P^2_\C\times{(\P^2_\C)}^*\,|\,P\in L\right\}.
\end{equation}
Let $X$ be a smooth hypersurface of $Y$ such that the projections 
\begin{equation}
\pi_1:X_\C\to\P^2_\C \quad\text{and}\quad \pi_2:X_\C\to{(\P^2_\C)}^*
\end{equation}
are ramified $2$-coverings. Then $X$ is a K3 surface, called~a Wehler surface. Furthermore, generic Wehler surfaces have a rank $2$ N\'eron--Severi group, spanned by generic fibers of the two coverings. The automorphism group is then isomorphic to a free product $\Z/2\Z * \Z/2\Z$, the generators being the involutions~$s_1$ and $s_2$ of the coverings $\pi_1$ and $\pi_2$, respectively. Furthermore, the automorphism~$f=s_1\circ s_2$ has loxodromic type, his dynamical degree being easily computed from the action on the $2$-dimensional subspace of $\NS(X;\R)$:
\begin{equation}
\lambda(f)=7+4\sqrt{3}.
\end{equation}

If moreover the surface $X$ is defined over $\R$, then ${\NS(X_\R;\Z)=\NS(X_\C;\Z)}$, and the automorphism $f$ is real. So, by Theorem \ref{formule-exacte},
\begin{equation}
\alpha(X)=\frac{\h(f_\R)}{\h(f_\C)}=\frac{\h(f_\R)}{\log(7+4\sqrt{3})}.
\end{equation}
\end{exem}

% Subsection 5.3 - Deformation

\section{Deformation of K3 surfaces in $\P^1\times\P^1\times\P^1$}\label{deformation}

The following example was first described by McMullen \cite{mcmullen-k3-siegel}. Fix a nonzero real number $t$. Let $X^t$ be the hypersurface of $\P^1(\C)^3$ defined in its affine chart $\C^3$ by
\begin{equation}
(z_1^2+1)(z_2^2+1)(z_3^2+1)+tz_1z_2z_3=2.
\end{equation}
It is a smooth surface of tridegree $(2,2,2)$, hence a K3 surface \cite{mazur}, here defined over~$\R$. We have three double (ramified) coverings ${\pi_j^t:X^t\to\P^1\times\P^1}$ (with ${j\in\{1,2,3\}}$) that consist in forgetting the $j$-th coordinate. These three coverings give rise to three involutions $s_j^t$ on $X^t$ that span a subgroup of~$\Aut(X_\R)$, which is a free product $\Z/2\Z*\Z/2\Z*\Z/2\Z$ \cite{wang}. Let $f^t$ be the automorphism of $X^t$ obtained by composing these three involutions. We have seen in chapter \ref{chap-auto} (cf. Example \ref{ex-loxo}) that $f^t$ is loxodromic and
\begin{equation}
\h(f^t_\C)=\log(9+4\sqrt{5}).
\end{equation}

For parameter $t=0$, the complex surface $X^0(\C)$ is not smooth, for there are~$12$ singular points $(\infty,\pm i,\pm i)$, $(\pm i,\infty,\pm i)$, and $(\pm i,\pm i,\infty)$. However, these points are not real, so the real surface $X^t(\R)$ remains smooth at $t=0$. Restricted to~$X^0(\R)$, the birational map $f^0$  is an order $2$ diffeomorphism given by the formula $f^0(x_1,x_2,x_3)=(-x_1,-x_2,-x_3)$. Consequently 
\begin{equation}
\h(f^0_\R)=0.
\end{equation}

Let $\mathcal{X}$ be the submanifold of $\P^1(\R)^3\times\R$ defined by 
\begin{equation}
\mathcal{X}=\{(x,t)\,|\,x\in X^t(\R)\}.
\end{equation}
The projection $p:\mathcal{X}\to\R$ is a locally trivial bundle whose fibers are the real surfaces~$X^t(\R)$. Thus there is an open neighborhood $I_\epsilon=(-\epsilon,\epsilon)$ around $0$, and an injective local diffeomorphism 
\begin{equation}
\psi:X^0(\R)\times I_\epsilon\to \mathcal{X}
\end{equation}
such that $p\circ\psi$ is the natural projection on the second coordinate. For all~${t\in I_\epsilon}$ the map $\psi$ induces a diffeomorphism from $X^0(\R)$ to $X^t(\R)$, which enables us to conjugate ${f^t_\R:X^t(\R)\to X^t(\R)}$ to a diffeomorphism~${g^t:X^0(\R)\to X^0(\R)}$. This family of diffeomorphisms on $X^0(\R)$ is a continuous family for the $\mathcal{C}^\infty$-topology. As the map $g^0=f^0_\R$ has entropy $0$, it follows that 
\begin{equation}
{\lim_{t\to 0}\h(g^t)=0},
\end{equation}
by continuity of the topological entropy on~${\rm Diff}^\infty(X^0(\R))$ (see \cite{yomdin} or \cite{newhouse-continuity} for the upper semicontinuity, and \cite[Corollary S.5.13]{katok-hasselblatt} for the lower semicontinuity). Since the entropy does not change by conjugacy, we also have
\begin{equation}
{\lim_{t \to 0}\h(f^t_\R)=0}.
\end{equation}

On the other hand, Theorem~\ref{entropie} gives the inequality
\begin{equation}
{\alpha(X^t)\leq \h(f^t_\R)/\h(f^t_\C)=\h(f^t_\R)/\log(9+4\sqrt{5})}
\end{equation}
for $t\neq 0$, and so we get
\begin{equation}
\lim_{\substack{t\to 0\\ t\neq 0}} \alpha(X^t)=0.
\end{equation}

To sum up, we have found a family $(X^t,f^t)_{t\in \R^*}$ of real loxodromic type automorphisms~$f^t$ on real K3 surfaces $X^t\subset(\P^1)^3$, such that
\begin{enumerate}
\item
$\h(f^t_\C)$ is a positive constant;
\item
as $t$ goes to $0$, $X^t(\R)$ degenerates in a smooth surface, and $f^t_\R$ in a zero-entropy diffeomorphism;
\item
$\lim_{t\to 0}\alpha(X^t)=0$.
\end{enumerate}

So we have just proved the following theorem.

\begin{theo}\label{alpha-petit}
For any $\eta>0$ there exists a real K3 surface in $(\P^1)^3$ such that $\alpha(X)<\eta$.
\end{theo}

%%%%%%%%%%%%%%%%%%%%%%%%%%%%%%%%%%%%%%%%%%%%%%%%%%%%%%%%%%%%%%%%%%%%
% NON-DENSITY OF AUTOMORPHISMS
%%%%%%%%%%%%%%%%%%%%%%%%%%%%%%%%%%%%%%%%%%%%%%%%%%%%%%%%%%%%%%%%%%%%

\chapter{Automorphisms of $X_\R$ and Diffeomorphisms of $X(\R)$}\label{chap-non-dense}

Let $X$ be a real algebraic surface. For any $r\in\N\cup\{\infty\}$ we denote by~$\Diff^r(X(\R))$ the group of $\mathcal{C}^r$-diffeomorphisms of the surface $X(\R)$, together with its $\mathcal{C}^r$-topology (when $r=0$, $\Diff^0(X(\R))=\Homeo(X(\R))$ stands for homeomorphisms). The group $\Aut(X_\R)$ identifies with a subgroup of~$\Diff^r(X(\R))$. We would like to know how this subgroup sits into the whole group of diffeomorphisms.

%%%%%%%%%%%%%%%%%%%%%%%%%%%%%%%%%%
% Nondensity of Automorphisms in $\Diff(X(\R))$
%%%%%%%%%%%%%%%%%%%%%%%%%%%%%%%%%%

\section{Nondensity of Automorphisms in $\Diff(X(\R))$}

When $\Aut(X_\R)$ does not have any positive entropy element, it obviously cannot be dense in $\Diff^\infty(X(\R))$. Indeed, there always exist positive entropy diffeomorphisms on $X(\R)$, and these diffeomorphisms cannot be approached by any automorphism, by the continuity of the entropy on $\Diff^\infty(X(\R))$ (cf. \cite{katok-hasselblatt}).

Now suppose that $\alpha(X)>0$. We have seen in Corollary \ref{coro-lehmer} that an automorphism $f$ whose entropy on $X(\R)$ is not zero satisfies
\begin{equation}
\h(f_\R)\geq\alpha(X)\log(\lambda_{10})>0.
\end{equation}
Thus any diffeomorphism $g$ of $X(\R)$ whose entropy satisfies
\begin{equation}
0<\h(g)<\alpha(X)\log(\lambda_{10})
\end{equation}
cannot be approached (in the $\mathcal{C}^\infty$-topology) by any automorphism. Since there always exists such diffeomorphisms, this gives the following proposition.

\begin{prop}\label{non-dense}
Let $X$ be a real algebraic surface such that $\alpha(X)>0$. Then the image of the group $\Aut(X_\R)$ in $\Diff^\infty(X(\R))$ is not dense.
\end{prop}

When the Kodaira dimension is $0$, $X(\R)$ is naturally equipped with a canonical area form $\mu_X$, which comes from an everywhere nonzero holomorphic $2$-form on some finite covering space of the minimal model of $X$ (cf. Remark~\ref{forme-volume}). The image of $\Aut(X_\R)$ is contained in the subgroup $\Diff_{\mu_X}^\infty(X(\R))$ of $\Diff^\infty(X(\R))$ whose elements preserve this area form. So in this case, the non-density occurs for trivial reasons, as pointed out in \cite{kollar-mangolte}. Nevertheless, the same proof gives this more precise result.

\begin{enonce*}{Proposition \ref*{non-dense} bis}
Let $X$ be a real algebraic surface of Kodaira dimension $0$ such that $\alpha(X)>0$. Then the image of the group $\Aut(X_\R)$ in~$\Diff^\infty_{\mu_X}(X(\R))$ is not dense.
\end{enonce*}

\begin{rema}
In \cite{kollar-mangolte}, Koll\'ar and Mangolte established the nondensity of the image of $\Aut(X_\R)$ in $\Homeo(X(\R))$ as soon as $X(\R)$ has the topology of a connected orientable surface with genus $\geq 2$. By contrast, they proved, for surfaces birational to $\P^2_\R$, the density in $\Diff^\infty(X(\R))$ of the group of \emph{birational} transformations with imaginary indeterminacy points.
\end{rema}

%%%%%%%%%%%%%%%%%%%%%%%%%%%%%%%%%%
% Discreteness of Automorphisms in Diff(X(R))
%%%%%%%%%%%%%%%%%%%%%%%%%%%%%%%%%%

\section{Discreteness of Automorphisms in $\Diff(X(\R))$}

We can be more precise when the Lie group $\Aut(X_{\C})$ is discrete, that is, when the connected component $\Aut(X_{\C})^0$ of the identity is reduced to a single point. For instance, this is the case for K3 and Enriques surfaces, but not for tori (for which $\Aut(X_{\C})^0$ consists in all translations).

\begin{prop}
Let $X$ be a real algebraic surface. Assume that ${\alpha(X)>0}$ and ${\Aut(X_{\C})^0=\{\id_X\}}$. Then the image of the group $\Aut(X_{\R})$ in~$\Diff^1(X(\R))$ is a discrete subgroup.
\end{prop}

\begin{proof}
Fix $\alpha>0$ such that $\alpha\in\mathcal{A}(X)$. There are positive numbers $q$ and~$C$ such that any real ample divisor $D$ whose class is $q$-divisible satisfies 
\begin{equation}\label{truc}
\mvol_\R(D)\geq C\vol_\C(D)^\alpha.
\end{equation}
Let $D_0$ be such a divisor and let $M>1$ be such that  $CM^\alpha>\mvol_\R(D_0)$ (in particular, $\vol_\C(D_0)<M$).

\begin{lemm}\label{lemm-discret}
The set $\Gamma=\{f\in\Aut(X_\R)\,|\,\vol_\C(f_*D_0)\leq M\}$ is finite.
\end{lemm}

\begin{proof}[Proof of Lemma \ref{lemm-discret}]
Denote by $\Theta\subset\NS(X_\R;\Z)$ the set of classes of ample divisors~$D$ that satisfy $\vol_\C(D)\leq M$. This set is finite because such classes are in the compact set 
\begin{equation}
\{\theta\in\Nef(X_\R)\,|\,\theta\cdot[\kappa]\leq M\},
\end{equation}
where $\kappa$ denotes the K\"ahler form on $X$.

By \cite[2.2]{lieberman} or \cite[4.8]{fujiki}, the subgroup
\begin{equation}
\{f\in\Aut(X_\R)\,|\,f_*[D_0]=[D_0]\}
\end{equation}
has finitely many connected components, so in our case it is finite. It follows that the set
\begin{equation}
\Gamma=\{f\in\Aut(X_\R)\,|\,f_*[D_0]\in\Theta\}
\end{equation}
is finite.
\end{proof}

As $\Gamma$ is finite and
\begin{equation}
\frac{CM^\alpha}{\mvol_\R(D_0)}>1=\ltrivert{\rm d}(\id_{X(\R)})\rtrivert_\infty,
\end{equation}
we can find a neighborhood $U$ of $\id_{X(\R)}$ in $\Diff^1(X(\R))$ such that
\begin{enumerate}
\item $U\cap\Gamma=\{\id_{X(\R)}\}$;
\item for all $g\in U$, $\displaystyle \ltrivert{\rm d}g\rtrivert_\infty<\frac{CM^\alpha}{\mvol_\R(D_0)}$.
\end{enumerate}
Let $f$ be a real automorphism of $X$ such that the restricted map 
\begin{equation}
f_\R:X(\R)\to X(\R)
\end{equation}
is in $U$. Since the length of a curve is at most multiplied by $\ltrivert{\rm d}f_\R\rtrivert_\infty$ when we take its image by $f$, we obtain
\begin{equation}
\mvol_\R(f_*D_0) \leq \ltrivert{\rm d}f_\R\rtrivert_\infty \mvol_\R(D_0) < CM^\alpha.
\end{equation}
On the other hand, 
\begin{equation}
\mvol_\R(f_*D_0)\geq C\vol_\C(f_*D_0)^\alpha,
\end{equation}
so we get
\begin{equation}
\vol_\C(f_*D_0)<M,
\end{equation}
hence $f\in\Gamma$. Then by hypothesis on $U$, we get $f_\R=\id_{X(\R)}$. This implies that~$\Aut(X_\R)$ is a discrete subgroup of $\Diff^1(X(\R))$.
\end{proof}

\selectlanguage{frenchb}

%%%%%%%%%%%%%%%%%%%%%%%%%%%%%%%%%%
%%%%%%%%%%%%%%%%%%%%%%%%%%%%%%%%%%
% PARTIE III
%%%%%%%%%%%%%%%%%%%%%%%%%%%%%%%%%%%%%%%%%%%%%%%%%%%%%%%%%%%%%%%%%%%%
% FATOU
%%%%%%%%%%%%%%%%%%%%%%%%%%%%%%%%%%%%%%%%%%%%%%%%%%%%%%%%%%%%%%%%%%%%

\part{Ensemble de Fatou des automorphismes}\label{part3}

%%%%%%%%%%%%%%%%%%%%%%%%%%%%%%%%%%
% Chapitre
%%%%%%%%%%%%%%%%%%%%%%%%%%%%%%%%%%
% Hyperbolicité + Fatou
%%%%%%%%%%%%%%%%%%%%%%%%%%%%%%%%%%

\chapter{Hyperbolicité et ensemble de Fatou}\label{chap-hyp}

Dans ce chapitre, on montre que l'ensemble de Fatou d'un automorphisme~$f$ de type loxodromique, sur une surface kählérienne compacte, est hyperbolique au sens de Kobayashi, modulo d'éventuelles courbes périodiques (voir le théorème \ref{fat-hyp} pour un énoncé précis). Ce résultat s'appuie sur un théorème de Dinh et Sibony (théorème \ref{thm-ds}) concernant les courants positifs fermés à potentiels continus, que l'on applique aux courants $T^+_f$ et $T^-_f$ définis au chapitre \ref{chap-dilat}.

Je commence par faire quelques rappels sur la pseudo-distance de Kobayashi et l'hyperbolicité (\textsection\ref{rap-hyp}). Je donne ensuite dans la section \ref{sec-ds} une démonstration du théorème de Dinh et Sibony sur les courants. Enfin dans le paragraphe \ref{sec-fat-hyp}, j'énonce et démontre le théorème principal sur l'hyperbolicité de l'ensemble de Fatou. Le cas où $f$ possède des courbes périodiques nécessite une adaptation de la preuve du théorème \ref{thm-ds} qui utilise le fait, démontré au~\textsection\ref{section-dilat+per}, que les potentiels des courants $T^+_f$ et $T^-_f$ induisent des fonctions continues sur la surface $X_0$ où l'on a contracté les courbes périodiques.

%%%%%%%%%%%%%%%%%%%%%%%%%%%%%%%%%%
% Variétés hyperboliques
%%%%%%%%%%%%%%%%%%%%%%%%%%%%%%%%%%

\section{Généralités sur les espaces hyperboliques}\label{rap-hyp}

Dans ce chapitre, on a besoin de considérer des \og variétés\fg~complexes qui peuvent admettre des singularités (c'est la cas de la surface $X_0$ donnée par le théorème \ref{theo-contr}). On parle dans ce cas d'espace analytique complexe plutôt que de variété complexe. Par définition, un atlas d'un tel espace est formé d'ouverts isomorphes au lieu d'annulation d'un nombre fini de fonctions holomorphes sur un ouvert~$\Omega$ de~$\C^N$.

% Distance de Kobayashi

\subsection{Pseudo-distance de Kobayashi}

Dans \cite{kobayashi}, l'auteur définit une pseudo-distance sur un espace analytique complexe $X$ de la manière suivante. Soient $x$ et $y$ deux points de $X$. Une \emph{chaîne de disques} reliant ces points est, par définition, une famille $\Psi=\left(\psi_k:\D\to X\right)_{k=1,\dotsc,n}$ de disques holomorphes avec des points marqués $a_k,b_k\in\D$, tels que
\begin{itemize}
\item $\psi_k(b_k)=\psi_{k+1}(a_{k+1})$ ;
\item $\psi_1(a_1)=x$ et $\psi_n(b_n)=y$.
\end{itemize}
La longueur d'une telle chaîne est alors définie par
\begin{equation}
\longueur(\Psi)=\sum_{k=1}^n \dist_\D(a_k,b_k),
\end{equation}
où $\dist_\D$ désigne la distance de Poincaré sur le disque unité. On définit ensuite la \emph{pseudo-distance de Kobayashi} $\kob_X$ par
\begin{equation}
\kob_X(x,y)=\inf\left\{\longueur(\Psi)\,\big|\,\text{$\Psi$ chaîne de disques reliant $x$ à $y$}\right\}.
\end{equation}
La fonction $\kob_X:X\to\R^+\cup\{+\infty\}$ vérifie :
\begin{itemize}
\item $\kob_X(x,x)=0$ ;
\item symétrie : $\kob_X(x,y)=\kob_X(y,x)$ ;
\item inégalité triangulaire : $\kob_X(x,z)\leq\kob_X(x,y)+\kob_X(y,z)$ ;
\item $\kob_X(x,y)=+\infty$ exactement lorsque $x$ et $y$ ne sont pas dans la même composante connexe de $X$.
\end{itemize}
Autrement dit, en restriction à chaque composante connexe, $\kob_X$ vérifie tous les axiomes d'une distance, sauf éventuellement la séparation des points (d'où le nom de pseudo-distance).

\begin{defi}
On dit que $X$ est \emph{hyperbolique} lorsque $\kob_X$ est une distance en restriction à chaque composante connexe de $X$, \emph{i.e.} lorsque
\begin{equation}
\forall x,y\in X,\quad x\neq y \Rightarrow \kob_X(x,y)>0.
\end{equation}
\end{defi}

\begin{exem}
Lorsque $X$ est le disque unité $\D$, $\kob_\D$ coïncide avec la distance de Poincaré $\dist_\D$, et $\D$ est ainsi une variété hyperbolique. En revanche, sur $X=\C$, la pseudo-distance $\kob_\C$ est identiquement nulle, donc $\C$ n'est pas hyperbolique.
\end{exem}

La proposition suivante est immédiate d'après la définition :

\begin{prop}[Décroissance de la pseudo-distance de Kobayashi]
Soit $f:X\to Y$ une application holomorphe entre deux espaces analytiques complexes. Pour tout couple $(x,y)\in X^2$, on a
\begin{equation}
\kob_Y(f(x),f(y)) \leq \kob_X (x,y).
\end{equation}
\end{prop}

\begin{coro}
Si $X$ est une variété hyperbolique, elle ne contient aucune courbe entière (non constante) $\varphi:\C\to X$.
\end{coro}

\begin{coro}
Tout ouvert d'une variété hyperbolique est hyperbolique.
\end{coro}

\begin{coro}
Si $f:X\to X$ est un automorphisme d'une variété complexe, alors $f$ est une isométrie pour la pseudo-distance $\kob_X$.
\end{coro}

En revanche, le résultat suivant est non trivial (voir \cite{lang-hyperbolic}) :

\begin{theo}[Barth]
Soit $X$ un espace analytique complexe connexe. Si $X$ est hyperbolique, alors la distance $\kob_X$ induit la topologie usuelle sur $X$.
\end{theo}
 
% Brody

\subsection{Reparamétrisation de Brody}

Fixons une métrique hermitienne $h$ sur $X$.

\begin{defi}
Une \emph{courbe de Brody} est une courbe entière (non constante) $\varphi:\C\to X$ \og à vitesse majorée \fg, c'est-à-dire telle que la fonction ${\|\varphi'\|}_h$ est majorée sur $\C$.
\end{defi}

Notons que si $X$ est compact, cette définition ne dépend pas du choix de la métrique $h$. On introduit aussi la définition suivante (voir \cite[\textsection II]{lang-hyperbolic}) :

\begin{defi}
Soit $X$ un espace analytique complexe muni d'une métrique hermitienne $h$, et soit $Y$ un sous-ensemble de $X$. On dit que $Y$ est \emph{hyperboliquement plongé} dans $X$ s'il existe une fonction continue $\ell:X\to\R^{+*}$ telle que pour toute application holomorphe $\psi:\D\to Y$\footnote{C'est-à-dire une application holomorphe $\psi:\D\to X$ telle que $\psi(\D)\subset Y$.}, on ait
\begin{equation}
{\|\psi'(0)\|}_h\leq\ell(\psi(0)).
\end{equation}
\end{defi}

Cette définition ne dépend pas du choix de $h$. Lorsque $Y$ est relativement compact, on peut remplacer la fonction $\ell$ par une constante strictement positive. Un espace $X$ est hyperbolique si et seulement s'il est hyperboliquement plongé dans lui-même, et on a l'implication suivante :
\begin{equation}
Y\text{ hyperboliquement plongé dans }X \implies Y\text{ hyperbolique.}
\end{equation}

\begin{theo}[Brody]\label{theo-brody}
Soit $Y$ un sous-ensemble relativement compact d'un espace analytique complexe $X$. On suppose que $\b Y$ ne contient aucune courbe de Brody. Alors $Y$ est hyperboliquement plongé dans $X$ (donc $Y$ est hyperbolique).
\end{theo}

\begin{coro}
Un espace analytique complexe compact $X$ est hyperbolique si et seulement s'il ne contient aucune courbe de Brody.
\end{coro}

\begin{proof}[Démonstration du théorème de Brody](voir aussi \cite[\textsection III]{lang-hyperbolic})
Supposons que $Y$ ne soit pas hyperboliquement plongé dans $X$. Il existe alors une suite $\psi_n:\D\to  Y$ d'applications holomorphes telle que ${\|\psi_n'(0)\|}_h\geq n$. La reparamétrisation de Brody consiste à restreindre les applications $\psi_n$ à des disques plus petits, puis de les reparamétrer afin d'obtenir des applications $\varphi_n$ définies sur des disques de plus en plus grands avec contrôle des dérivées, de telle sorte que l'on puisse prendre une limite $\varphi:\C\to \b Y$ qui soit une courbe de Brody.

Pour $\psi:\D\to X$ et $z\in\D$, on note ${\ltrivert{\rm d}\psi(z)\rtrivert}_{\D,h}$ la norme d'opérateur de la différentielle vis-à-vis de la métrique de Poincaré ${\|\partial/\partial z\|}_{{\rm T}_z\D}=1/{(1-|z|^2)}$ sur ${\rm T}_z\D$ et de la métrique $h$ sur ${\rm T}_{\psi(z)}X$. On a donc :
\begin{equation}
\ltrivert{\rm d}\psi(z)\rtrivert_{\D,h} = {\left\|{\rm d}\psi(z)\cdot(1-|z|^2)\frac{\partial}{\partial z}\right\|}_h = (1-|z|^2){\lVert\psi'(z)\rVert}_h.
\end{equation}

Fixons un réel $0<r<1$. On note
\begin{equation}\begin{split}
\psi_{n,r}:\D&\to X\\
z&\mapsto \psi_n(rz).
\end{split}
\end{equation}
Pour tout $z\in\D$, on a :
\begin{equation}
\ltrivert{\rm d}\psi_{n,r}(z)\rtrivert_{\D,h}=\ltrivert{\rm d}\psi_n(rz)\rtrivert_{\D,h} \frac{{\left\|r\frac{\partial}{\partial z}\right\|}_{{\rm T}_{rz}\D}}{{\left\|\frac{\partial}{\partial z}\right\|}_{{\rm T}_z\D}}=\frac{r(1-|z|^2)}{1-|rz|^2}\ltrivert{\rm d}\psi_n(rz)\rtrivert_{\D,h}.
\end{equation}
En particulier, $\lim_{|z|\to 1}\ltrivert{\rm d}\psi_{n,r}(z)\rtrivert_{\D,h}=0$, donc la norme de la différentielle de $\psi_{n,r}$ atteint son maximum, noté $R_n$, en un point $z_n\in\D$. Notons que
\begin{equation}
R_n=\ltrivert{\rm d}\psi_{n,r}(z_n)\rtrivert_{\D,h}\geq\ltrivert{\rm d}\psi_{n,r}(0)\rtrivert_{\D,h}=r{\|\psi_n'(0)\|}_h\geq nr\underset{n\to+\infty}{\longrightarrow} +\infty.
\end{equation}

Pour $R>0$, on note $\D_R$ le disque de rayon $R$ (centré en $0$) dans $\C$, que l'on munit de la métrique ${\|\partial/\partial z\|}_{{\rm T}_z\D_R}=1/R(1-|z/R|^2)$. Soit $g_n:\D_{R_n}\to\D$ un isomorphisme tel que $g_n(0)=z_n$. On pose alors 
\begin{equation}
\varphi_n=\psi_{n,r}\circ g_n : \D_{R_n}\to Y.
\end{equation}
Comme $g_n$ est une isométrie entre $\D_{R_n}$ et $\D$, on a, pour tout $z\in\D_{R_n}$, 
\begin{align}
{\left\|\varphi_n'(z)\right\|}_h
&=
\frac{\ltrivert{\rm d}\varphi_n(z)\rtrivert_{\D_{R_n},h}}{R_n\left(1-\frac{|z|^2}{{R_n}^2}\right)}=\frac{\ltrivert{\rm d}\psi_{n,r}(g_n(z))\rtrivert_{\D,h}}{R_n\left(1-\frac{|z|^2}{{R_n}^2}\right)}\leq \frac{1}{1-\frac{|z|^2}{{R_n}^2}}\label{maj-der}\\
\text{et}\quad{\|\varphi_n'(0)\|}_h
&=
\frac{{\left\|{\rm d}\psi_{n,r}(z_n)\right\|}_\D}{R_n}=1.
\end{align}

Pour tout compact $K\subset\C$, les applications $\varphi_n$ sont définies sur $K$ pour $n$ assez grand, et la majoration (\ref{maj-der}) implique que les normes des dérivées $\varphi_n'$ sont uniformément majorées sur $K$. Par compacité de $X$, il existe alors une sous-suite qui converge uniformément sur $K$ vers une application à valeurs dans $\b Y$. Par un procédé d'extraction diagonale, on peut donc trouver une sous-suite $\left(\varphi_{n_k}\right)_{k\in\N}$ qui converge uniformément sur tout compact vers une fonction holomorphe $\varphi:\C\to\b Y$. De plus, on a pour tout $z\in\C$ :
\begin{equation}
{\|\varphi'(z)\|}_h=\lim_{k\to+\infty}{\|\varphi'_{n_k}(z)\|}_h\leq\lim_{k\to+\infty}\frac{1}{1-\frac{|z|^2}{{R_{n_k}}^2}}=1.
\end{equation}
Enfin, $\varphi$ est non constante car
\begin{equation}
\|\varphi'(0)\|=\lim_{k\to+\infty}\|\varphi'_{n_k}(0)\|=1.
\end{equation}
On a ainsi construit une courbe de Brody $\varphi:\C\to X$ à valeurs dans $\b Y$, ce qui contredit l'hypothèse.
\end{proof}

%%%%%%%%%%%%%%%%%%%%%%%%%%%%%%%%%%
% Dinh-Sibony
%%%%%%%%%%%%%%%%%%%%%%%%%%%%%%%%%%

\section{Un théorème de Dinh et Sibony}\label{sec-ds}

Cette section a pour but de donner la démonstration d'un théorème non publié de Dinh et Sibony (voir la première version de \cite{ds-green} sur \url{http://arxiv.org/abs/math/0311322v1}), que l'on utilise dans la section suivante pour montrer l'hyperbolicité de l'ensemble de Fatou.

\begin{theo}[Dinh -- Sibony]\label{thm-ds}
Soit $(X,\kappa)$ une variété kählérienne compacte de dimension $d$, et soit $T$ un $(1,1)$-courant positif fermé à potentiels continus. On suppose qu'il existe une courbe de Brody~${\varphi:\C\to X}$ telle que~${\varphi^*T=0}$. Alors tout courant d'Ahlfors $S$ associé à $\varphi$ vérifie~${T\wedge S=0}$.
\end{theo}

\begin{proof}

% Premier cas : aire bornée

\emph{\textbf{Premier cas : la fonction ${\sf a}_\varphi$ est bornée.}}

Dans ce cas, la proposition \ref{courbe-entiere-aire-finie} montre que l'image de $\varphi$ est contenue dans une courbe rationnelle $C=\b{\varphi(\C)}$. Si $S$ est un courant d'Ahlfors associé à $\varphi$, on a alors
\begin{equation}
S=\frac{1}{\aire_\kappa(C)}\{C\}.
\end{equation}
Pour montrer que $T\wedge S=0$, on prend $\theta$ une fonction $\mathcal{C}^\infty$ à support compact  dans un ouvert $U$ sur lequel $T=\ddc u$, avec $u$ continue. Le fait que $\varphi^*T=0$ se traduit par $\ddc(u\circ\varphi)=0$, et alors
\begin{align}
\langle T\wedge S,\theta \rangle
&= \langle S,u\ddc\theta\rangle\\
&= \frac{1}{\aire_\kappa(C)}\int_{\C} (u\circ\varphi) \,\ddc (\theta\circ\varphi)\\
&= \frac{1}{\aire_\kappa(C)}\int_{\C} (\theta\circ\varphi)\,\ddc (u\circ\varphi)\\
&= 0.
\end{align}
Comme $X$ est recouvert par de tels ouverts $U$, on a bien $T\wedge S=0$.\\

% Deuxième cas : aire non bornée

\emph{\textbf{Deuxième cas : la fonction ${\sf a}_\varphi$ est non bornée.}}

Fixons $\left(U^\alpha,V^\alpha,\psi^\alpha\right)_\alpha$ un recouvrement relativement compact de $X$ (cf. définition \ref{rec-compact}) tel que sur $V^\alpha$, on ait $T=\ddc u^\alpha$ avec $u^\alpha$ continue. Pour~$r>0$, on introduit la notation suivante :
\begin{align}
\D_r^\alpha &= \D_r\cap\varphi^{-1}(\b U^\alpha)=\left\{z\in\C\,\big|\,|z|<r\text{ et }\varphi(z)\in\b U^\alpha\right\}.
\end{align}

Soit $S=\lim S_{\varphi,r_n}$ un courant d'Ahlfors associé à $\varphi$. Nous allons montrer que $T\wedge S=0$, ce qu'il suffit de montrer en restriction à chaque ouvert~$U^\alpha$. Pour cela, soit $\theta$ une fonction $\mathcal{C}^\infty$ à support compact dans $U^\alpha$. On a :
\begin{equation}
\langle T\wedge S,\theta\rangle=\langle S,u^\alpha\ddc \theta\rangle=\lim_{n\to+\infty}I_n,
\end{equation}
où
\begin{equation}\label{223}
I_n=\langle S_{\varphi,n},u^\alpha\ddc\theta\rangle=\frac{1}{\ma_\varphi(r_n)}\int_0^{r_n}\frac{{\rm d}t}{t}\int_{\D_t^\alpha}(u^\alpha\circ\varphi) \ddc(\theta\circ\varphi).
\end{equation}
Pour alléger les notations, on note $\tilde u=u^\alpha\circ\varphi$ et $\tilde\theta=\theta\circ\varphi$. Puisque $\varphi^*T=0$, on a
\begin{equation}
\ddc \tilde u=0.
\end{equation}
En particulier, la fonction $\tilde u$ est harmonique, donc $\mathcal{C}^\infty$. On a l'égalité
\begin{align}
{\rm d}(\tilde\theta\,{\rm d^c} \tilde u)+{\rm d^c}(\tilde u\,{\rm d}\tilde\theta)
&= {\rm d}\tilde\theta\wedge{\rm d^c}\tilde u
+ \tilde\theta\ddc\!\tilde u
+ {\rm d}^c\tilde u\wedge{\rm d}\tilde\theta
+\tilde u{\rm d^cd}\tilde\theta\\
&= -\tilde u\ddc\!\tilde\theta,
\end{align}
qui donne, en remplaçant dans $(\ref{223})$ :
\begin{align}
I_n
&=
\frac{1}{\ma_\varphi(r_n)}\int_0^{r_n}\frac{{\rm d}t}{t}\int_{\D^\alpha_{t}}\tilde u\ddc\!\tilde\theta\\
&=
\frac{-1}{\ma_\varphi(r_n)}\int_0^{r_n}\frac{{\rm d}t}{t}\int_{\D^\alpha_{t}}\left({\rm d}(\tilde\theta\,{\rm d^c}\tilde u)+{\rm d^c}(\tilde u\,{\rm d}\tilde\theta)\right)\\
&=
\underbrace{\frac{-1}{\ma_\varphi(r_n)}\int_0^{r_n}\frac{{\rm d}t}{t}\int_{\partial\D^\alpha_{t}}\tilde\theta\,{\rm d^c}\tilde u}_{I_n^1}
+
\underbrace{\frac{-1}{\ma_\varphi(r_n)}\int_0^{r_n}\frac{{\rm d}t}{t}\int_{\D^\alpha_{t}}{\rm d^c}(\tilde u\,{\rm d}\tilde\theta)}_{I_n^2}.
\end{align}

Il suffit donc de majorer les deux intégrales $I_n^1$ et $I_n^2$. Commençons par la deuxième. Si $\tilde u\,{\rm d}\tilde\theta$ s'écrit en coordonnées $v_1\,{\rm d}x+v_2\,{\rm d}y$, on a
\begin{equation}
{\rm d^c}(\tilde u\,{\rm d}\tilde\theta)=\frac{-1}{2\pi}\left(
\frac{\partial v_1}{\partial x}+\frac{\partial v_2}{\partial y}
\right)
{\rm d}x\wedge{\rm d}y,
\end{equation}
puis par le théorème de Stokes
\begin{equation}
\int_{\D^\alpha_{r_n}}{\rm d^c}(\tilde u\,{\rm d}\tilde\theta)
=
\frac{-1}{2\pi}\int_{\partial\D^\alpha_{r_n}}(v_1\,{\rm d}y-v_2\,{\rm d}x).
\end{equation}
On obtient donc, en notant $M$ le maximum de la fonction $u^\alpha$ sur $\b U^\alpha$,
\begin{align}
|I_n^2|
&\leq
\frac{1}{2\pi\ma_\varphi(r_n)}\int_0^{r_n}\frac{{\rm d}t}{t}\int_{\partial\D_{t}^\alpha}\left\|\tilde u\,{\rm d}\tilde\theta\right\|\,{\rm d}\sigma_{t}\\
&\leq
\frac{M}{2\pi\ma_\varphi(r_n)}\int_0^{r_n}\frac{{\rm d}t}{t}\int_{\partial\D_{t}^\alpha}\ltrivert{\rm d}\theta(\varphi(z))\rtrivert_\kappa\|\varphi'(z)\|_\kappa\,{\rm d}\sigma_{t}\\
&=
\frac{M\ltrivert{\rm d}\theta\rtrivert_{\kappa,\infty}}{2\pi}\,\frac{\ml_\varphi(r_n)}{\ma_\varphi(r_n)}
\underset{n\to+\infty}{\longrightarrow}0.
\end{align}

D'autre part, comme
\begin{equation}
{\rm d^c}\tilde u=\frac{1}{2\pi}\left(\frac{\partial\tilde u}{\partial y}{\rm d}x-\frac{\partial\tilde u}{\partial x}{\rm d}y\right),
\end{equation}
on a la majoration suivante pour $I_n^1$ :
\begin{align}
\left|I_n^1\right|
&\leq
\frac{\ltrivert\theta\rtrivert_{\kappa,\infty}}{2\pi\ma_\varphi(r_n)}\int_{0}^{r_n}\frac{{\rm d}t}{t}\int_{\partial\D^\alpha_{t}}\|\nabla\tilde u\|\,{\rm d}\sigma_{t}.
\end{align}

\begin{lemm}\label{lemm-ds}
Il existe $C>0$ telle que pour tout $r\geq 1$, on ait
\begin{equation}\label{eq-lemm-ds}
\int_{\D_r^\alpha}\|\nabla\tilde u\|\,{\rm dvol}
\leq C\,r\,{\sf a}_\varphi(2r)^{1/2}.
\end{equation}
\end{lemm}

\begin{proof}[Démonstration du lemme]
On note $\Omega=\varphi^{-1}(V^\alpha)$ le domaine de définition de $\tilde u$. Comme $\tilde u$ est harmonique sur $\Omega$, l'inégalité de Harnack (voir par exemple \cite[p. 28]{garnett-marshall}) implique l'existence de $K>0$ tel que
\begin{equation}\label{ineg-harnack}
\forall z\in \varphi^{-1}(\b U^\alpha), \quad \|\nabla\tilde u(z)\|\leq \frac{K}{\dist(z,\partial\Omega)}.
\end{equation}
Il suffit donc de majorer l'intégrale
\begin{equation}
I(r)=\int_{\D^\alpha_r}\frac{{\rm dvol}(z)}{\dist(z,\partial\Omega)}.
\end{equation}

Pour $\delta>0$, posons
\begin{align}
&\D_r^\alpha(\delta)=\{z\in\D^\alpha_r\,|\,\dist(z,\partial\Omega)<\delta\}.
\end{align}
D'après le théorème de Fubini, on a, pour tout $\Delta>0$,
\begin{align}
\int_{0}^{\Delta}\aire(\D_r^\alpha(\delta))\,\frac{{\rm d}\delta}{\delta^2}
&=\int_{\D_r^\alpha(\Delta)}{\rm dvol}(z)\int_{\dist(z,\partial\Omega)}^\Delta\frac{{\rm d}\delta}{\delta^2}\\
&=\int_{\D_r^\alpha(\Delta)}\left(\frac{1}{\dist(z,\partial\Omega)}-\frac{1}{\Delta}\right)\,{\rm dvol}(z).
\end{align}
On en déduit que
\begin{align}
I(r)
&= \int_{\D_r^\alpha\backslash\D_r^\alpha(\Delta)}\frac{{\rm dvol}(z)}{\dist(z,\partial\Omega)}
+ \int_{\D_r^\alpha(\Delta)}\frac{{\rm dvol}(z)}{\dist(z,\partial\Omega)}\\
&\leq \int_{\D_r^\alpha\backslash\D_r^\alpha(\Delta)}\frac{{\rm dvol}}{\Delta}
+ \int_{\D_r^\alpha(\Delta)}\frac{{\rm dvol}}{\Delta}
+ \int_{0}^{\Delta}\aire(\D_r^\alpha(\delta))\,\frac{{\rm d}\delta}{\delta^2}\\
&= \frac{1}{\Delta}\,\aire(\D_r^\alpha)
+ \int_{0}^{\Delta}\aire(\D_r^\alpha(\delta))\,\frac{{\rm d}\delta}{\delta^2}\\
&\leq \frac{\pi r^2}{\Delta}
+ \int_{0}^{\Delta}\aire(\D_r^\alpha(\delta))\,\frac{{\rm d}\delta}{\delta^2}.
\label{maj-inter}
\end{align}

\begin{enonce}{Affirmation}\label{affir}
Il existe $M>0$ tel que
\begin{equation}
\frac{\aire(\D_r^\alpha(\delta))}{\delta^2}\leq M\,{\sf a}_\varphi(2r)
\quad\quad
\forall\delta>0,~\forall{r\geq 1}.
\end{equation}
\end{enonce}

\begin{proof}[Démonstration de l'affirmation]
Pour tout $r\geq 1$ et $\delta\geq r/4$, on a
\begin{equation}
\frac{\aire(\D_r^\alpha(\delta))}{\delta^2}
\leq\frac{\pi r^2}{\delta^2}
\leq 16\pi
\leq \frac{16\pi}{{\sf a}_\varphi(2)}\,{\sf a}_\varphi(2r).
\end{equation}

Dans la suite, on suppose donc $0<\delta<r/4$. Pour tout $r\geq 1$, on peut trouver un recouvrement de $\D_r$ par un nombre fini de disques $\D(z_i,\delta)$ tel que tout point soit dans au plus 36 disques $\D(z_i,3\delta)$. Comme chaque $z_i$ est dans~$\D_{r+\delta}\subset\D_{5\delta/4}$, on a donc
\begin{equation}
\D(z_i,3\delta)\subset\D_{2r}.
\end{equation}
On applique alors le théorème \ref{aire-diam} sur la comparaison aire-diamètre : il existe~${\eta>0}$ tel que
\begin{equation}
\aire_\kappa(\varphi(\D(z_i,3\delta)))<\eta \implies \diam_\kappa(\varphi(\D(z_i,2\delta)))<\dist_\kappa(\b U^\alpha, \partial V^\alpha).
\end{equation}

On pose $\{z_i\}=\{a_i\}\cup\{b_i\}$, avec
\begin{equation}
\aire_\kappa(\varphi(\D(a_i,3\delta)))\geq\eta
\quad \text{et} \quad
\aire_\kappa(\varphi(\D(b_i,3\delta)))<\eta.
\end{equation}
Montrons que les disques $\D(b_i,\delta)$ n'intersectent pas $\D_r^\alpha(\delta)$. En effet, soit~${z\in\D_r^\alpha\cap\D(b_i,\delta)}$. Comme $\varphi(z)\in \b U^\alpha\cap\varphi(\D(b_i,2\delta))$ et
\begin{equation}
{\diam_\kappa(\varphi(\D(b_i,2\delta)))<\dist_\kappa(\b U^\alpha, \partial V^\alpha)},
\end{equation}
on a alors $\varphi(\D(b_i,2\delta))\cap\partial V^\alpha=\emptyset$, et donc $\partial\D(b_i,2\delta)\cap\partial\Omega=\emptyset$. On en déduit que
\begin{equation}
\dist(z,\partial\Omega)
> \dist(z,\partial\D(b_i,2\delta))
> \delta,
\end{equation}
et donc $z\notin\D_r^\alpha(\delta)$.

Par conséquent, $\D_r^\alpha(\delta)$ est recouvert par les disques $\D(a_i,\delta)$. Soit $N$ le cardinal des $a_i$. Comme $\D(a_i,3\delta)\subset\D_{2r}$ et que chaque point de $\D_{2r}$ est dans au plus 36 de ces disques, on a
\begin{equation}
N\eta
\leq \sum_i\aire_\kappa(\varphi(\D(a_i,3\delta)))
\leq 36\,{\sf a}_\varphi(2r).
\end{equation}
On en déduit que
\begin{equation}
\aire(\D_r^\alpha(\delta))
\leq \sum_i\aire(\D(a_i,\delta))
= N\pi\delta^2
\leq \frac{36\pi}{\eta}\,{\sf a}_\varphi(2r)\,\delta^2.
\end{equation}
L'affirmation est ainsi démontrée, en prenant $M=\max(16\pi/{\sf a}_\varphi(2),36\pi/\eta)$.
\end{proof}

Revenons à la preuve du lemme \ref{lemm-ds}. On déduit de l'affirmation la majoration suivante :
\begin{align}
\int_{0}^{\Delta}\aire(\D_r^\alpha(\delta))\,\frac{{\rm d}\delta}{\delta^2}
\leq M\,{\sf a}_\varphi(2r)\,\Delta.
\end{align}
En reportant dans $(\ref{maj-inter})$, on obtient
\begin{equation}
I(r) \leq \frac{\pi r^2}{\Delta}+M\,{\sf a}_\varphi(2r)\,\Delta,
\end{equation}
qui est valable pour tout $\Delta>0$. En prenant $\Delta=r/({\sf a}_\varphi(2r))^{1/2}$, cette inégalité donne
\begin{equation}
I(r) \leq (\pi+M)\,r\,{\sf a}_\varphi(2r).
\end{equation}
Ceci achève la preuve du lemme \ref{lemm-ds}.
\end{proof}

Reprenons le cours de la démonstration du théorème \ref{thm-ds}. On voulait montrer que $\lim_{n\to+\infty}I_n^1=0$, avec
\begin{equation}\label{in1}
\left|I_n^1\right|
\leq \frac{\ltrivert\theta\rtrivert_{\kappa,\infty}}{2\pi\ma_\varphi(r_n)}\int_{0}^{r_n}\frac{{\rm d}t}{t}\int_{\partial\D^\alpha_{t}}\|\nabla\tilde u\|\,{\rm d}\sigma_{t}.
\end{equation}

Posons $K_n=\mathbb{E}(\log_2(r_n))$. Grâce au lemme \ref{lemm-ds}, on obtient la majoration suivante :
\begin{align}
\int_1^{r_n}\frac{{\rm d}t}{t}\int_{\partial\D_t^\alpha}\|\nabla\tilde u\|\,{\rm d}\sigma_t
&\leq \sum_{k=0}^{K_n}\int_{r_n/2^{k+1}}^{r_n/2^k}\frac{{\rm d}t}{r_n/2^{k+1}}\int_{\partial\D_t^\alpha}\|\nabla\tilde u\|\,{\rm d}\sigma_t\\
&\leq 2 \sum_{k=0}^{K_n}\frac{1}{r_n/2^k}\int_{\D_{r_n/2^k}^\alpha}\|\nabla\tilde u\|\,{\rm dvol}\\
&\leq 2C\sum_{k=0}^{K_n}{\sf a}_\varphi(r_n/2^{k-1})^{1/2}\\
&\leq 2C\sum_{k=0}^{+\infty}\frac{2^{k-1}}{r_n}\int_{r_n/2^{k-1}}^{r_n/2^{k-2}}{\sf a}_\varphi(t)^{1/2}\,{\rm d}t\\
&\leq 4C\int_0^{4r_n}\frac{{\sf a}_\varphi(t)^{1/2}}{t}\,{\rm d}t\\
&= 4C{\sf mra}_\varphi(4r_n),
\end{align}
où l'on a posé 
\begin{equation}
\begin{split}
{\sf mra}_\varphi : \R^+ &\to \R^+\\
r &\to \int_0^{r}\frac{{\sf a}_\varphi(t)^{1/2}}{t}\,{\rm d}t.
\end{split}
\end{equation}

Comme $\varphi$ est une courbe de Brody, sa vitesse est majorée, disons par~${V>0}$, et on a alors ${\sf a}_\varphi(r)\leq V^2\pi r^2$ pour tout $r>0$. On en déduit que~${{\sf mra}_\varphi(r)\leq V\pi^{1/2}r}$, et donc
\begin{equation}
\liminf_{r\to+\infty}\frac{{\sf mra}_\varphi(4r)}{{\sf mra}_\varphi(r)}\leq 4.
\end{equation}
Quitte à extraire une suite de $(r_n)$, on peut donc supposer que
\begin{equation}
{\sf mra}_\varphi(4r_n)\leq 5\,{\sf mra}_\varphi(r_n).
\end{equation}
On en déduit que
\begin{equation}
\int_1^{r_n}\frac{{\rm d}t}{t}\int_{\partial\D_t^\alpha}\|\nabla\tilde u\|\,{\rm d}\sigma_t
\leq 20C\,{\sf mra}_\varphi(r_n).
\end{equation}
Il existe donc une constante $C'>0$ telle que
\begin{equation}
\int_0^{r_n}\frac{{\rm d}t}{t}\int_{\partial\D_t^\alpha}\|\nabla\tilde u\|\,{\rm d}\sigma_t
\leq C'\,{\sf mra}_\varphi(r_n),
\end{equation}
et donc d'après $(\ref{in1})$, on obtient :
\begin{equation}\label{quotientintegrales}
|I_n^1| \leq \frac{C'\ltrivert\theta\rtrivert_{\kappa,\infty}}{2\pi}\frac{\int_0^{r_n}t^{-1}{\sf a}_\varphi(t)^{1/2}\,{\rm d}t}{\int_0^{r_n}t^{-1}{\sf a}_\varphi(t)\,{\rm d}t}.
\end{equation}
Comme
\begin{equation}
\frac{t^{-1}{\sf a}_\varphi(t)^{1/2}}{t^{-1}{\sf a}_\varphi(t)} \mathop{\longrightarrow}\limits_{t\to+\infty} 0
\quad\text{et}\quad
\int_0^{r}t^{-1}{\sf a}_\varphi(t)\,{\rm d}t \mathop{\longrightarrow}\limits_{r\to+\infty} +\infty,
\end{equation}
le quotient des deux intégrales dans $(\ref{quotientintegrales})$ converge vers $0$. On en déduit que
\begin{equation}
\langle T\wedge S,\theta\rangle
= \lim_{n\to+\infty} (I_n^1+I_n^2) = 0.
\end{equation}
Ceci montre que $T\wedge S=0$ et achève ainsi la preuve du théorème \ref{thm-ds}.
\end{proof}

%%%%%%%%%%%%%%%%%%%%%%%%%%%%%%%%%%
% Le Fatou est hyperbolique
%%%%%%%%%%%%%%%%%%%%%%%%%%%%%%%%%%

\section[L'ensemble de Fatou est hyperbolique]{L'ensemble de Fatou est hyperbolique modulo les courbes périodiques}\label{sec-fat-hyp}

Nous sommes maintenant en mesure de démontrer le résultat principal de ce chapitre :

\begin{theo}\label{fat-hyp}
Soit $f$ un automorphisme loxodromique d'une surface kählérien\-ne compacte $X$. Alors l'ensemble de Fatou est hyperbolique modulo les courbes périodiques, {i.e.} pour tous $x$ et $y$ dans l'ensemble de Fatou :
$$
\kob_{\Fat(f)}(x,y)=0
\quad\implies\quad
\left\{\begin{array}{l}\text{$x$ et $y$ sont reliés par une}\\\text{courbe périodique connexe.}\end{array}\right.
$$
En particulier, $\Fat(f)\backslash{\sf CP}(f)$ est hyperbolique, où ${\sf CP}(f)$ désigne l'union des courbes périodiques.
\end{theo}

\begin{proof}
On note $T^+=T^+_f$, $T^-=T^-_f$ et $\lambda=\lambda(f)$. On considère le courant positif fermé 
\begin{equation}
T=T^++T^-,
\end{equation}
qui est à potentiels continus (cf. chapitre \ref{chap-dilat}). Soit $\Omega$ l'ouvert $f$-invariant défini par
\begin{equation}
\Omega=X\backslash\Supp(T).
\end{equation}

\begin{lemm}\label{incl-fatou}
On a l'inclusion $\Fat(f)\subset\Omega$.
\end{lemm}

\begin{proof}[Démonstration du lemme \ref{incl-fatou}]
Soit $x\in\Fat(f)$, et soit $U\subset\Fat(f)$ un voisinage ouvert relativement compact de $x$. On veut montrer que~${T^\pm_{|U}=0}$. Faisons-le pour $T^+$ (la preuve pour $T^-$ est la même, en remplaçant $f$ par~$f^{-1}$). Comme la famille~${{(f^n)}_{n\in\N}}$ est normale sur $\Fat(f)$, il existe une sous-suite~${{(f^{n_k})}_{k\in\N}}$ qui converge uniformément sur $\b U$ vers une application~$h$. Quitte à restreindre l'ouvert $U$ et à réextraire une sous-suite, on peut supposer qu'il existe un ouvert $V\supset h(U)$ tel que :
\begin{enumerate}
\item
$f^{n_k}(U)\subset V$ pour tout $k\in\N$ ;
\item
$T^+=\ddc  u$ sur $V$, avec $u$ continue ($u$ localement bornée suffit).
\end{enumerate}
Soit $\theta$ une $(1,1)$-forme $\mathcal{C}^\infty$ à support compact  dans $U$. Pour tout $k\in\N$, on~a :
\begin{align}
\langle T^+,\theta \rangle
&=
\frac{1}{\lambda^{n_k}} \langle {f^*}^{n_k}T^+,\theta \rangle\\
&=
\frac{1}{\lambda^{n_k}} \langle T^+,{f^*}^{-n_k}\theta \rangle\\
&=
\frac{1}{\lambda^{n_k}} \langle \ddc u,{f^*}^{-n_k}\theta \rangle \quad\quad\quad \text{car $\Supp({f^*}^{-n_k}\theta)\subset f^{n_k}(U)\subset V$}\\
&=
\frac{1}{\lambda^{n_k}} \int_K (u\circ f^{n_k}) \ddc \theta.
\end{align}
Par convergence dominée, l'intégrale converge vers $\int_U (u\circ h)\ddc \theta$, et on en déduit que $\langle T^+,\theta\rangle=0$ car $\lambda^{n_k}$ tend vers $+\infty$. Ceci achève la preuve du lemme \ref{incl-fatou}.
\end{proof}

Revenons à la démonstration du théorème \ref{fat-hyp}. Nous allons maintenant montrer que l'ouvert $\Omega$ est hyperbolique modulo les courbes périodiques, ce qui impliquera le même énoncé sur $\Fat(f)$, par décroissance de la pseudo-distance de Kobayashi.\\

\emph{Premier cas : $f$ ne possède aucune courbe périodique.} Supposons que $\Omega$ ne soit pas hyperboliquement plongé. Cela signifie qu'il existe des disques holomorphes $\psi_n:\D\to\Omega$ tels que $\|\psi_n'(0)\|_\kappa\to+\infty$, où $\kappa$ est une forme de Kähler sur $X$. La reparamétrisation de Brody permet alors de construire une courbe de Brody $\varphi:\C\to\b\Omega$, limite uniforme d'applications $\varphi_n:\D_{R_n}\to\Omega$. De plus, cette courbe de Brody vérifie
\begin{equation}
\varphi^*T=0.
\end{equation}
En effet, soit $U$ un ouvert sur lequel $T=\ddc u$, avec $u$ continue, et soit $\theta$ une fonction à support compact $K\subset\C$ dans $\varphi^{-1}(U)$. Pour $n$ assez grand, $\varphi_n$ est définie sur $K$ et $\varphi_n(K)\subset U$. Le théorème de convergence dominée donne alors
\begin{equation}
\langle\varphi_n^*T,\theta\rangle
=
\int_K(u\circ\varphi_n)\ddc\theta
\mathop{\longrightarrow}\limits_{n\to+\infty}
\int_K(u\circ\varphi)\ddc\theta
=
\langle \varphi^*T,\theta\rangle.
\end{equation}
Le membre de gauche est nul, car $\varphi_n$ est à valeurs dans $\Omega=X\backslash\Supp(T)$, et on en déduit que $\langle \varphi^*T,\theta\rangle=0$. Comme $\C$ est recouvert par de tels ouverts $\varphi^{-1}(U)$, ceci montre que $\varphi^*T=0$.

D'après le théorème de Dinh et Sibony \ref{thm-ds}, il existe alors un courant d'Ahlfors $S$ associé à $\varphi$ tel que $T\wedge S=0$, et donc $[T]\cdot[S]=0$. Comme~${[T]^2>0}$, le théorème de l'indice de Hodge implique que $[S]^2<0$, donc $[S]$ n'est pas une classe nef. D'après le théorème de Nevanlinna \ref{ahlfors-nef}, l'adhérence de l'image de $\varphi$ est donc une courbe compacte $C$, et on a $S=\{C\}$ à une constante près. Comme $[C]\cdot(\theta^+_f +\theta^-_f)=[S]\cdot[T]=0$ et $[C]\cdot\theta^\pm_f\geq 0$, on a $[C]\cdot\theta^\pm_f=0$. On en déduit que $C$ est une courbe périodique, d'après la proposition \ref{prop-per} : une contradiction. L'ouvert $\Omega$ est donc hyperboliquement plongé.\\

\emph{Cas général.} On considère le morphisme de contraction des courbes périodiques $\pi:X\to X_0$ (cf. théorème \ref{theo-contr}). Nous allons montrer, par un argument similaire au permier cas, que $\Omega_0:=\pi(\Omega)$ est hyperboliquement plongé dans $X_0$. Notons que ceci implique le résultat, car par décroissance de la pseudo-distance de Kobayashi, on a pour tous $x$ et $y$ dans~$\Omega$ :
\begin{equation}
\kob_\Omega(x,y)=0
\implies \kob_{\Omega_0}(\pi(x),\pi(y))=0
\implies \pi(x)=\pi(y).
\end{equation}

Supposons donc que $\Omega_0$ ne soit pas hyperboliquement plongé dans~$X_0$. Comme dans le premier cas, on peut construire une courbe de Brody ${\varphi_0:\C\to X_0}$, limite uniforme d'applications $\varphi_n:\D_{R_n}\to \Omega_0$. Cette courbe de Brody se relève en une courbe entière $\varphi:\C\to X$ telle que $\pi\circ\varphi=\varphi_0$. Mais $\varphi$ n'est pas nécessairement une courbe de Brody, et on ne peut donc pas appliquer directement le théorème \ref{thm-ds}. Cependant, on peut reprendre pas à pas la démonstration de ce théorème pour montrer que tout courant d'Ahlfors $S$ associé à $\varphi$ est tel que $T\wedge S=0$.

Le cas où l'aire est bornée n'utilise pas que $\varphi$ est de Brody, et peut donc être recopié tel quel. Pour le second cas, on utilise le théorème~\ref{dilat+per} démontré au chapitre \ref{chap-dilat}. D'après ce théorème, il est possible de choisir un recouvrement relativement compact~${(U^\alpha,V^\alpha,\psi^\alpha)}$ tel que 
\begin{enumerate}
\item $T_{|V^\alpha}=\ddc u^\alpha$, où $u^\alpha$ est un fonction continue sur $V^\alpha$ ;
\item le potentiel $u^\alpha$ est constant sur ${\sf CP}(f)\cap V^\alpha$, où ${\sf CP}(f)$ désigne l'union des courbes périodiques de $f$.
\end{enumerate}
Par conséquent, le potentiel $u^\alpha$ provient d'une fonction continue $u_0^\alpha$ sur~${V_0^\alpha:=\pi(V^\alpha)}$ telle que $u^\alpha=u_0^\alpha\circ\pi$. On a donc le diagramme commutatif suivant :

\begin{equation}
\xymatrix{
& V^\alpha \ar[d]_\pi \ar[r]^{u^\alpha} & \R \\
\C \supset \varphi^{-1}(V^\alpha) \ar[r]_-{\varphi_0} \ar@/^/[ru]^\varphi & V_0^\alpha \ar@/_/[ru]_{u_0^\alpha}
}
\end{equation}

Ainsi, la fonction $\tilde u=u^\alpha\circ\varphi$, utilisée dans la démonstration du théorème~\ref{thm-ds}, peut être vue aussi comme la composition
\begin{equation}
\tilde u = u_0^\alpha\circ\varphi_0.
\end{equation}
Dans la majoration du lemme \ref{lemm-ds}, qui fait intervenir uniquement cette fonction $\tilde u$, on peut donc utiliser la fonction $\varphi_0$ au lieu de $\varphi$. En effet, la seule chose à modifier est la démonstration de l'affirmation \ref{affir}. Pour ceci, au lieu d'appliquer le théorème \ref{aire-diam} de comparaison aire--diamètre sur $X$, on l'applique sur $X_0$ muni d'une métrique hermitienne, avec $\epsilon=\dist(\b U_0^\alpha,\partial V_0^\alpha)$. Le reste de la preuve est identique, et on obtient donc la majoration
\begin{equation}
\int_{\D_r^\alpha}\|\nabla\tilde u\|\,{\rm dvol}
\leq C\,r\,{\sf a}_{\varphi_0}(2r)^{1/2}
\end{equation}
qui remplace $(\ref{eq-lemm-ds})$. Puis en utilisant que $\varphi_0$ est à vitesse majorée, on obtient, au lieu de $(\ref{quotientintegrales})$, la majoration suivante :
\begin{equation}
|I_n^1| \leq \frac{C'\ltrivert\theta\rtrivert_{\kappa,\infty}}{2\pi}\frac{\int_0^{r_n}t^{-1}{\sf a}_{\varphi_{0}}(t)^{1/2}\,{\rm d}t}{\int_0^{r_n}t^{-1}{\sf a}_\varphi(t)\,{\rm d}t}.
\end{equation}
Or la métrique hermitienne sur $X_0$ peut être choisie de manière à ce que 
\begin{equation}
{\sf a}_\varphi\geq{\sf a}_{\varphi_0}.
\end{equation}
On réobtient donc ainsi l'inégalité $(\ref{quotientintegrales})$, et la fin de la preuve est la même.

Il est donc possible de construire un courant d'Ahlfors $S$ sur $X$, associé à la courbe entière $\varphi:\C\to X$, et tel que $T\wedge S=0$. Les mêmes arguments que dans le premier cas montrent alors que $S$ est, à une constante près, le courant d'intégration sur une courbe périodique irréductible $C$, et que $\varphi$ est à valeurs dans cette courbe. Mais alors $\varphi_0=\pi\circ\varphi$ serait constante : une contradiction.
\end{proof}

%%%%%%%%%%%%%%%%%%%%%%%%%%%%%%%%%%
% Une caractérisation de l'ensemble de Fatou
%%%%%%%%%%%%%%%%%%%%%%%%%%%%%%%%%%

\section{Une caractérisation de l'ensemble de Fatou}

\begin{prop}\label{prop-caract}
Soit $X$ une surface kählérienne compacte, et soit $f$ un automorphisme loxodromique de $X$. On a alors l'égalité suivante
\begin{equation}
\Fat(f) = X\backslash \Supp(T^+_f+T^-_f)
\end{equation}
modulo les courbes périodiques de $f$, ce qui signifie que la différence symétrique\footnote{En fait, on a toujours l'inclusion de gauche à droite (cf. lemme \ref{incl-fatou}), donc la différence symétrique est une différence.} de ces deux ensembles est contenue dans l'union des courbes périodiques.
\end{prop}

\begin{lemm}\label{fatou-contient-hyp}
Soit $f:X\to X$ un automorphisme d'un espace analytique complexe compact, et soit $\Omega\subset X_{\textit{reg}}$ un ouvert $f$-invariant hyperboliquement plongé dans $X$. Alors
\begin{equation}
\Omega \subset \Fat(f).
\end{equation}
\end{lemm}

\begin{proof}[Démonstration du lemme]
Fixons une métrique hermitienne $h$ sur $X$. Supposons qu'il existe $x\in\Omega\backslash\Fat(f)$. Comme $x$ n'est pas dans l'ensemble de Fatou, la famille $(f^n)_{n\in\Z}$ n'est pas équicontinue au voisinage de $x$, d'après le théorème d'Ascoli. Ceci donne l'existence d'un vecteur $v$ tangent à $x$ tel que
\begin{equation}
\sup_{n\in\Z} \left\|{\rm d}f^n(x)\cdot v\right\|_h = +\infty.
\end{equation}

Soit $\varphi:\D\to\Omega$ une application holomorphe telle que $\varphi(0)=x$ et $\varphi'(0)=v$. Pour $n\in\Z$, on pose 
\begin{equation}
\varphi_n=f^n\circ\varphi:\D\to\Omega.
\end{equation}
On a alors $\varphi_n'(0)={\rm d}f^n(x)\cdot v$, donc la suite $\left(\left\|\varphi_n'(0)\right\|_h\right)_{n\in\Z}$ n'est pas bornée. Ceci contredit le fait que $\Omega$ soit hyperboliquement plongé.
\end{proof}

\begin{proof}[Démonstration de la proposition \ref{prop-caract}]
Comme au paragraphe précédent, on note $\Omega = X \backslash \Supp(T^+_f+T^-_f)$, ainsi que $\Omega^* = \Omega \backslash {\sf CP}(f)$, où ${\sf CP}(f)$ désigne l'union des courbes périodiques de $f$. Nous allons montre la double inclusion
\begin{equation}
\Omega^* \subset \Fat(f) \subset \Omega.
\end{equation}
La seconde a déjà été démontrée au paragraphe précédent (cf. lemme \ref{incl-fatou}).

Pour montrer l'inclusion de gauche, on considère le morphisme ${\pi:X\to X_0}$ de contraction des courbes périodiques, et on note $\Omega_0=\pi(\Omega)$, $\Omega_0^*=\pi(\Omega^*)$. On a vu lors de la preuve du théorème \ref{fat-hyp} que $\Omega_0$ est hyperboliquement plongé, ce qui implique que $\Omega_0^*$ l'est aussi. On en déduit, d'après le lemme \ref{fatou-contient-hyp}, que
\begin{equation}
\Omega_0^*\subset\Fat(f_0),
\end{equation}
où $f_0$ est l'automorphisme induit sur $X_0$. En prenant l'image réciproque par $\pi$, on obtient $\Omega^* \subset \Fat(f)$, car $\Fat(f_0)$ et $\pi(\Fat(f))$ coïncident en dehors des valeurs critiques de $\pi$. Ceci achève la démonstration.
\end{proof}

\begin{rema}
On ne sait pas si ${\sf CP}(f)\cap\Omega$ est contenu dans l'ensemble de Fatou. Il faudrait pour cela mener une étude plus fine sur les courbes périodiques qui intersectent $\Omega$.
\end{rema}

%%%%%%%%%%%%%%%%%%%%%%%%%%%%%%%%%%
% Chapitre
%%%%%%%%%%%%%%%%%%%%%%%%%%%%%%%%%%
% Composantes récurrentes du Fatou
%%%%%%%%%%%%%%%%%%%%%%%%%%%%%%%%%%

\chapter[Composantes récurrentes de l'ensemble de Fatou]{\'Etude des composantes récurrentes de l'ensemble de Fatou}\label{chap-rec}

Dans ce chapitre, on montre que les composantes récurrentes de l'ensemble de Fatou sont des domaines de rotation, grâce aux propriétés d'hyperbolicité établies au chapitre précédent. Ce chapitre s'inspire très fortement de \cite{bs2,ueda,fs-recurrent,loeb-vigue,bedford-kim-rotation}.

%%%%%%%%%%%%%%%%%%%%%%%%%%%%%%%%%%
% Domaines de rotation
%%%%%%%%%%%%%%%%%%%%%%%%%%%%%%%%%%

\section{Généralités sur les domaines de rotation}

Dans ce paragraphe, $X$ désigne une variété complexe compacte de dimension quelconque, et $f$ un automorphisme de $X$.

\begin{defi}[\cite{ bs2, fs-higher1}]
Soit $\Omega$ un ouvert connexe de $X$ tel que ${\Omega\subset\Fat(f)}$ et $f(\Omega)=\Omega$. On dit que $\Omega$ est un \emph{domaine de rotation}\footnote{On dit aussi \emph{domaine de Siegel}.} lorsqu'il existe une suite $m_k\to\pm\infty$ telle que
\begin{equation}
f^{m_k}\mathop{\longrightarrow}\limits_{k\to+\infty}\id_\Omega
\end{equation}
uniformément sur les compacts de $\Omega$.
\end{defi}

On a la caractérisation suivante (voir aussi \cite{bedford-kim-rotation}) :

\begin{prop}\label{carac-rot}
Soit $\Omega\subset\Fat(f)$ un ouvert connexe fixe par $f$. On note $\mathcal{G}(\Omega)$ l'adhérence du sous-groupe engendré par $f$ dans $\Aut(\Omega)$, pour la topologie de la convergence uniforme sur les compacts. Les énoncés suivants sont équivalents :
\begin{enumerate}
\item $\Omega$ est un domaine de rotation pour $f$ ;
\item le groupe $\mathcal{G}(\Omega)$ est compact.
\end{enumerate}
Si ces conditions sont vérifiées, $\mathcal{G}(\Omega)$ est alors un groupe de Lie abélien compact, dont la composante connexe de l'identité~$\mathcal{G}(\Omega)^0$ est un tore réel $\T^d$. L'entier $d$ est appelé \emph{rang} du domaine de rotation.
\end{prop}

La démonstration de cette proposition utilise le théorème suivant.

\begin{theo}[Cartan ; Bochner -- Montgomery]\label{hcartan}
Soit $\Omega$ un ouvert connexe d'une variété complexe $X$. Le groupe $\Aut(\Omega)$ est muni de la topologie de la convergence uniforme sur les compacts, qui en fait un groupe topologique. Soit $\mathcal{G}$ un sous-groupe fermé de $\Aut(\Omega)$, dont les éléments forment une famille normale en tant qu'applications holomorphes de~$\Omega$ dans~$X$. Alors $\mathcal{G}$ est un groupe de Lie.
\end{theo}

\begin{proof}
On montre dans un premier temps que $\mathcal{G}$ est localement compact. Pour ceci, soit $K\subset\Omega$ un compact d'intérieur non vide, et soit~$x$ un point intérieur de $K$. Alors l'ensemble $\mathcal{G}_{x,K}:=\{g\in\mathcal{G}\,|\,g(x)\in K\}$ est un voisinage de l'identité dans $\mathcal{G}$. Montrons qu'il est compact. Soit $(g_n)$ une suite d'éléments de $\mathcal{G}_{x,K}$, dont on extrait une suite $\left(g_{n_k}\right)$ qui converge uniformément sur les compacts vers une application holomorphe $g:\Omega\to X$. Par compacité de $K$, on peut supposer qu'il existe $y\in K$ tel que $g_{n_k}(x)\to y$. On a alors~${g(x)=y\in\Omega}$. Une adaptation d'un théorème de H. Cartan (voir par exemple \cite[chapitre 5, théorème 4]{narasimhan}) montre alors que $g\in\Aut(\Omega)$, et donc $g\in\mathcal{G}_{x,K}$ car $\mathcal{G}$ est fermé.

Le groupe topologique $\mathcal{G}$ est donc localement compact. De plus, il agit par difféomorphismes de classe $\mathcal{C}^2$ sur $\Omega$, et tout élément $g\in\mathcal{G}$ fixant un ouvert est l'identité (par prolongement analytique). D'après un théorème de Bochner et Montgomery \cite{bm1}, on en déduit que $\mathcal{G}$ est un groupe de Lie réel.
\end{proof}

\begin{proof}[Démonstration de la proposition \ref{carac-rot}]
Seul $(1)\implies(2)$ est non trivial. D'après le théorème \ref{hcartan}, le groupe $\mathcal{G}(\Omega)$ est un groupe de Lie abélien. Il est donc isomorphe à $F\times\T^d\times\R^k$, où $F$ est un groupe abélien fini. Si $\Omega$ est un domaine de rotation, alors nécessairement $k=0$.
\end{proof}

En dimension $1$, les domaines de rotation sont les disques de Siegel et les anneaux de Herman, et le groupe $\mathcal{G}(\Omega)^0$ est un cercle (cf. \cite{milnor}).

En dimension $2$, on peut montrer que le rang du domaine de rotation est $1$ ou~$2$ pour les automorphismes loxodromiques (voir \cite[théorème 1.6]{bedford-kim-rotation}). Dans \cite{mcmullen-k3-siegel}, McMullen donne des exemples d'automorphismes loxodromiques sur des surfaces K3 non algébriques qui admettent un domaine de rotation de rang $2$ (voir aussi \cite{oguiso}). Sur les surfaces rationnelles, il existe des exemples de domaines de rotation de rang $1$ et $2$ (voir \cite{bedford-kim-maxent} et \cite{mcmullen-rat}).

%%%%%%%%%%%%%%%%%%%%%%%%%%%%%%%%%%
% composantes récurentes = domaines de rotation
%%%%%%%%%%%%%%%%%%%%%%%%%%%%%%%%%%

\section[Les composantes récurrentes sont des domaines de rotation]{Les composantes de Fatou récurrentes sont des domaines de rotation}

À partir de maintenant, $X$ désigne une surface kählérienne compacte, et $f$ un automorphisme loxodromique de $X$. On appelle \emph{composante de Fatou} une composante connexe de $\Fat(f)$.

Lorsque $X$ est une surface de dimension de Kodaira nulle, $X$ possède une forme volume canonique (cf. remarque \ref{forme-volume}) qui est préservée par $f$, ainsi que par les limites normales de $f^{n_k}_{|\Omega}$, qui sont donc des automorphismes de~$\Omega$. Toute composante de Fatou est donc un domaine de rotation, d'après la caractérisation \ref{carac-rot} (voir aussi \cite[Proposition 1.1]{bedford-kim-rotation}).

En revanche, sur les surfaces rationnelles, il existe des exemples où l'ensemble de Fatou est un bassin d'attraction pour $f$ (voir \cite{mcmullen-rat}) ; mais dans ce cas $\Fat(f)$ n'est pas récurrent, au sens de la définition suivante.

\begin{defi}
On dit qu'un ouvert $\Omega\subset X$ est \emph{récurrent} s'il existe des points $x_0,y_0\in\Omega$ et une suite $n_k\to\pm\infty$ tels que
\begin{equation}\label{cond-recur}
f^{n_k}(x_0)\mathop{\longrightarrow}\limits_{k\to+\infty}y_0.
\end{equation}
\end{defi}

Il est clair d'après la définition que les composantes de Fatou récurrentes sont périodiques (mais la réciproque n'est pas nécessairement vraie). Quitte à prendre un itéré de $f$, on supposera qu'elles sont fixes.

Bien évidemment, les domaines de rotations sont des composantes de Fatou récurrentes. Grâce à l'hyperbolicité de $\Fat(f)$, on montre que l'hypothèse de récurrence est suffisante pour garantir qu'une composante de Fatou est un domaine de rotation.

\begin{theo}\label{rec=rot}
Soit $f$ un automorphisme loxodromique d'une surface compacte kählérienne $X$, et soit $\Omega$ une composante de Fatou récurrente, telle que $f(\Omega)=\Omega$. Alors $\Omega$ est un domaine de rotation.
\end{theo}

\begin{proof}
On note $\Omega^*=\Omega\backslash{\sf CP(f)}$, où ${\sf CP}(f)$ désigne l'union des courbes périodiques de $f$. L'ensemble $\Omega^*$ est alors un autre ouvert fixé par~$f$, qui est dense dans $\Omega$. Montrons qu'il est récurrent également. Pour cela, soit~${x_0\in\Omega}$ vérifiant la condition $(\ref{cond-recur})$. Quitte à extraire une sous-suite, on suppose que $f^{n_k}$ converge uniformément sur tout compact de $\Omega$ vers une fonction $h:\Omega\to\b\Omega$, qui est continue. L'ouvert $h^{-1}(\Omega)$ est non vide (il contient~$x_0$), donc il intersecte $\Omega^*$ en un point $x$. Montrons que $y=h(x)\in\Omega^*$. Dans le cas contraire, $y$ est dans ${\sf CP}(f)$, qui est un fermé invariant. Quitte à réextraire une sous-suite, $f^{-n_k}$ converge uniformément sur les compacts de $\Omega$, et~${f^{-n_k}(y)\to z\in{\sf CP(f)}}$. Or comme $y\in\Omega$, on a 
\begin{equation}
\lim f^{-n_k}\circ f^{n_k}(x)=\lim f^{-n_k}(y)=z,
\end{equation}
d'où $x=z\in{\sf CP}(f)$, ce qui est une contradiction. On a donc
\begin{equation}
f^{n_k}(x) \mathop{\longrightarrow}\limits_{k\to+\infty} y\in\Omega^*
\end{equation}
avec $x$ et $y$ dans $\Omega^*$, ce qui montre que $\Omega^*$ est récurrent.
 
De plus, $\Omega^*$ est hyperbolique, d'après le théorème $\ref{fat-hyp}$. On montre alors, en suivant \cite[Theorem 3.1]{ueda}, que $\Omega$ est un domaine de rotation. Pour cela, on pose
\begin{align}
V &= \{x'\in\Omega^*\,|\,\kob_{\Omega^*}(x,x')<\epsilon\},\\
W &= \{y'\in\Omega^*\,|\,\kob_{\Omega^*}(y,y')<\epsilon/2\},
\end{align}
où $\epsilon>0$ est choisi de telle sorte que $V$ et $W$ soient relativement compacts dans~$\Omega^*$, et tel que $V\subset h^{-1}(\Omega^*)$. Quitte à extraire une sous-suite, on suppose que~${f^{n_k}(x)\in W}$ pour tout $k$. On a alors $W\subset f^{n_k}(V)$. En effet, pour tout~${y'\in W}$, on a, en utilisant que $f$ est une isométrie pour $\kob_{\Omega^*}$,
\begin{align}
\kob_{\Omega^*}(x,f^{-n_k}(y'))
&= \kob_{\Omega^*}(f^{n_k}(x),y')\\
&\leq \kob_{\Omega^*}(f^{n_k}(x),y)+\kob_{\Omega^*}(y,y')\\
&\leq \epsilon/2+\epsilon/2 = \epsilon.
\end{align}
Quitte à réextraire, on suppose que $m_k:=n_{k+1}-n_k\to\pm\infty$ et que $f^{m_k}\to g$ uniformément sur les compacts de $\Omega$. On a alors, à $k$ fixé,
\begin{align}
\sup_{y'\in W}\kob_{\Omega^*}(f^{m_k}(y'),y') 
&\leq \sup_{x'\in V}\kob_{\Omega^*}(f^{m_k}\circ f^{n_k}(x'),f^{n_k}(x'))\\
&= \sup_{x'\in V}\kob_{\Omega^*}(f^{n_{k+1}}(x'),f^{n_k}(x')).
\end{align}
Comme la suite $f^{n_k}$ converge uniformément sur $V$ vers $h:V\to\Omega^*$, cette dernière expression converge vers $0$ lorsque $k\to+\infty$. On en déduit que~${f^{m_k}\to\id_W}$ uniformément sur $W$, et par prolongement analytique $g=\id_\Omega$. Ceci montre que $\Omega$ est un domaine de rotation.
\end{proof}

%%%%%%%%%%%%%%%%%%%%%%%%%%%%%%%%%%
% Chapitre
%%%%%%%%%%%%%%%%%%%%%%%%%%%%%%%%%%
% Fatou et réel
%%%%%%%%%%%%%%%%%%%%%%%%%%%%%%%%%%

\chapter{Ensemble de Fatou et lieu réel}\label{chap-fatou-reel}

Dans ce chapitre, on applique les résultats des chapitres précédents au cas où $X$ possède une structure réelle qui est préservée par $f$. On obtient alors des restrictions topologiques sur les composantes connexes de $X(\R)$ qui sont contenues dans l'ensemble de Fatou :

\begin{theo}\label{car-neg}
Soit $X$ une surface kählérienne réelle, et soit $f$ un automorphisme loxodromique réel de $X$. Soit $S$ une composante connexe de~$X(\R)$ telle que $S\subset\Fat(f)$. Alors $f$ possède au plus $\chi(S)$ points fixes sur~$S$, où $\chi(S)$ désigne la caractéristique d'Euler topologique de $S$. En particulier,~$\chi(S)\geq 0$, et $S$ est l'une des surfaces topologiques suivantes :
\begin{itemize}
\item une sphère,
\item un tore,
\item un plan projectif,
\item une bouteille de Klein.
\end{itemize}
\end{theo}

De plus, nous verrons au paragraphe \ref{precisions} que $f$ est conjugué à une \og rotation\fg~sur $S$.

\begin{coro}
Soit $f$ un automorphisme loxodromique d'une surface kählérienne réelle $X$, et soit $S$ une composante connexe de $X(\R)$. On suppose que $f$ possède au moins trois points périodiques sur $S$. Alors $S$ intersecte l'ensemble de Julia.
\end{coro}

\begin{rema}
Après avoir terminé de rédigé ce manuscrit, je me suis rendu compte qu'il n'y avait en fait pas besoin de l'hyperbolicité de l'ensemble de Fatou pour démontrer les résultats contenus dans ce chapitre. Voir pour cela la remarque \ref{dom-rot}. Avec cette démonstration alternative, ces résultats se généralisent ainsi au cas où :
\begin{enumerate}
\item $f$ est d'ordre infini (il n'y pas besoin de supposer qu'il est de type loxodromique) ;
\item $f$ est un difféomorphisme birationnel (voir le chapitre suivant) ;
\item $X$ est une surface complexe compacte munie d'une structure réelle (le caractère kählérien n'intervient pas).
\end{enumerate}
\end{rema}

%%%%%%%%%%%%%%%%%%%%%%%%%%%%%%%%%%
% Formule de Lefschetz
%%%%%%%%%%%%%%%%%%%%%%%%%%%%%%%%%%

\section{Formule de Lefschetz}

Avant de démontrer le théorème \ref{car-neg}, faisons quelques rappels sur la formule des points fixes de Lefschetz \cite{lefschetz} (voir aussi \cite{griffiths-harris} pour une présentation plus moderne). Soit $g:M\to M$ un difféomorphisme d'une variété lisse. On suppose que les points fixes de $g$ sont isolés (il sont donc en nombre fini), et que ceux-ci sont non dégénérés, \emph{i.e.} pour tout point fixe $x$, l'endomorphisme~${({\rm d}g(x)-\id)}$ sur l'espace tangent en $x$ est inversible. On définit alors, pour tout point fixe $x$, l'indice de $g$ en $x$ :
\begin{equation}
i(g;x)={\sf signe}\left(\det({\rm d}g(x)-\id)\right).
\end{equation}
On définit aussi le nombre de Lefschetz :
\begin{equation}
{\sf Lef}(g) = \sum_{k=0}^{\dim M}(-1)^k\mathrm{tr}\left(g^*_{|H^k(M;\R)}\right),
\end{equation}
En exprimant de deux manières différentes le nombre d'intersection, dans~${M\times M}$, de la diagonale avec le graphe de $g$, on obtient la formule de Lefschetz :
\begin{equation}
{\sf Lef}(g)=\sum_{g(x)=x}i(g;x).
\end{equation}
Un cas particulier est celui où le difféomorphisme $g$ est isotope à l'identité. Dans ce cas, $g$ agit trivialement sur la cohomologie, et le nombre de Lefschetz est égal à la caractéristique d'Euler topologique $\chi(M)$, ce qui donne
\begin{equation}\label{lef-isot}
\chi(M)=\sum_{g(x)=x}i(g;x).
\end{equation}

%%%%%%%%%%%%%%%%%%%%%%%%%%%%%%%%%%
% Linéarisation
%%%%%%%%%%%%%%%%%%%%%%%%%%%%%%%%%%

\section{Linéarisation au voisinage d'un point fixe}

Soit $f$ un automorphisme d'une surface complexe $X$. Supposons qu'il existe un point fixe $x$ dans l'ensemble de Fatou de $f$, et notons~$\Omega$ la composante de Fatou qui contient le point $x$ (celle-ci est récurrente et stable par $f$). On considère l'adhérence $\mathcal{G}_x$ des germes $\{f^n_x\,|\,n\in\Z\}$ dans le groupe des germes d'automorphismes au voisinage de $x$. Le fait que $x$ soit dans l'ensemble de Fatou implique :

\begin{lemm}
$\mathcal{G}_x$ est un groupe abélien compact.
\end{lemm}

\begin{proof}
Soit $n_k\to\pm\infty$ telle que $f^{n_k}\to g$ et $f^{-n_k}\to h$ uniformément sur les compacts de $\Omega$ (on peut toujours s'y ramener par extraction d'une sous-suite). Comme $g$ est continue, il existe un voisinage ouvert $U$ de $x$ tel que $g(U)\subset\Omega$. En prenant $y\in U$, on obtient
\begin{equation}
y = \lim_{k\to+\infty}f^{-n_k}\circ f^{n_k}(y) = \lim_{k\to+\infty}f^{-n_k} g(y) = h\circ g(y).
\end{equation}
De manière similaire, on obtient $g\circ h=\id$ sur un voisinage de $x$, ce qui montre que $g$ et $h$ définissent des germes d'automorphismes en $x$, inverses l'un de l'autre.
\end{proof}

On peut alors utiliser un argument de linéarisation dû à Henri Cartan :

\begin{prop}
Il existe un germe de difféomorphisme
\begin{equation}
\psi:(\C^2,0)\to (\Omega,x)
\end{equation}
tel que $\psi^{-1}\mathcal{G}_x\psi$ soit un sous-groupe abélien compact de $\GL_2(\C)$.
\end{prop}

\begin{proof}
Soit $\mu$ la mesure de Haar sur $\mathcal{G}_x$. On se place dans des coordonnées locales au voisinage de $x$, et on pose
\begin{equation}
\psi = \int_{\mathcal{G}_x}g^{-1}\circ{\rm d}g(x)\,{\rm d}\mu(g).
\end{equation}
Pour tout $h\in\mathcal{G}_x$, on a alors
\begin{align}
h\circ\psi &= \int_{\mathcal{G}_x}hg^{-1}\circ{\rm d}g(x)\,{\rm d}\mu(g)\\
&= \int_{\mathcal{G}_x}(gh^{-1})^{-1}\circ{\rm d}(gh^{-1})(x)\circ{\rm d}h(x)\,{\rm d}\mu(g)\\
&= \psi\circ{\rm d}h(x).
\end{align}
De plus, $\psi$ est un germe de difféomorphisme, car
\begin{align}
{\rm d}\psi &= \int_{\mathcal{G}_x}{\rm d}(g^{-1}\circ{\rm d}g)\,{\rm d}\mu(g)\\
&= \int_{\mathcal{G}_x}{\rm d}g^{-1}\circ{\rm d}g\,{\rm d}\mu(g)\\
&= \id.
\end{align}
On a donc $\psi^{-1}\mathcal{G}_x\psi = \{{\rm d}h(x)\,|\,h\in\mathcal{G}_x\}$, qui est un sous-groupe abélien compact de $\GL_2(\C)$.
\end{proof}

\begin{rema}
La proposition précédente reste vraie en dimension quelconque, avec la même démonstration.
\end{rema}

Supposons maintenant que $f$ préserve une structure réelle, et que $x$ est un point fixe réel dans l'ensemble de Fatou. Les différentielles ${\rm d}h(x)$ sont alors dans un sous-groupe compact de $\GL_2(\R)$, donc quitte à composer $\psi$ à droite par une matrice de $\GL_2(\R)$, on peut supposer que $\psi^{-1}\mathcal{G}_x\psi \subset {\sf O}(2)$. Comme de plus, ce groupe est abélien et infini, on a en fait
\begin{equation}
\psi^{-1}\mathcal{G}_x\psi \subset {\sf SO}(2).
\end{equation}
En particulier, on obtient le corollaire suivant :

\begin{coro}\label{conj-rot}
Soit $f$ un automorphisme d'ordre infini d'une surface complexe connexe $X$, qui préserve une structure réelle sur $X$. Alors~$f_\R$ est conjugué à une rotation irrationnelle au voisinage de chaque point fixe réel $x$ dans l'ensemble de Fatou. En particulier, les points fixes de $f$ sur~${X(\R)\cap\Fat(f)}$ sont isolés, non dégénérés, et leur indice est égal à $+1$.
\end{coro}

\begin{proof}
On a déjà vu que $f_\R$ est conjugué à une rotation dans un voisinage de $x$. Si celle-ci était d'ordre fini $k$, alors $f^k$ serait l'identité sur un voisinage de $x$, donc sur $X$ tout entier par prolongmement analytique, contredisant ainsi l'hypothèse sur l'ordre de $f$.

Comme une rotation irrationnelle n'a pas de point fixe autre que l'origine, ceci implique que les points fixes sur $X(\R)\cap\Fat(f)$ sont isolés. Au voisinage d'un tel point $x$, $({\rm d}f(x)-\id)$ est conjugué à une matrice de la forme
\begin{equation}
\begin{pmatrix}
\cos(\theta)-1 & -\sin(\theta)\\
\sin(\theta) & \cos(\theta)-1
\end{pmatrix}
\end{equation}
avec $\theta\in\R\backslash 2\pi\Q$, qui a pour déterminant $2-2\cos(\theta)>0$. Ceci montre que $x$ est un point fixe non dégénéré et que $i(f;x)=+1$.
\end{proof}

%%%%%%%%%%%%%%%%%%%%%%%%%%%%%%%%%%
% Démonstration du théorème
%%%%%%%%%%%%%%%%%%%%%%%%%%%%%%%%%%

\section{Démonstration du théorème \ref{car-neg}}

Quitte à prendre un itéré, on suppose que $f(S)=S$. On note $\Omega$ la composante de Fatou contenant $S$. Comme $f$ préserve le fermé $S$, $\Omega$ est récurrent, et le théorème \ref{rec=rot} montre qu'il s'agit d'un domaine de rotation. Le groupe $\mathcal{G}(\Omega)$ est ainsi un groupe de Lie compact, qui n'a qu'un nombre fini de composantes connexes. Quitte à reprendre un itéré de $f$, on peut donc supposer qu'il est connexe. Or les éléments de ce groupe préservent la structure réelle (par passage à la limite de la relation $f^{n_k}\circ\sigma=\sigma\circ f^{n_k}$), et définissent par restriction des difféomorphismes $\R$-analytiques de $S$. La connexité montre alors que $f_{|S}$ est isotope à l'identité.

D'après le corollaire \ref{conj-rot}, on peut appliquer la formule de Lefschetz $(\ref{lef-isot})$ à $f_{|S}$, qui s'écrit
\begin{equation}
\chi(S) = {\sf card}\{x\in S\,|\,f(x)=x\}.
\end{equation}
En particulier, on obtient $\chi(S)\geq 0$.
\qed

\begin{rema}\label{dom-rot}
On peut montrer que $\Omega$ est un domaine de rotation de la manière suivante, qui n'utilise pas l'hyperbolicité démontrée au chapitre 11.

Comme $\Omega$ est contenu dans l'ensemble de Fatou, il existe une suite $n_k\to+\infty$ telle que $f^{n_k}\to g$ uniformément sur les compacts de $\Omega$, avec $g:\Omega\to X(\C)$ holomorphe. Quitte à passer à une sous-suite, on peut également supposer que $f^{-n_k}\to h$, $f^{m_k}\to i$ avec $m_k:=n_{k+1}-n_k\to+\infty$. En restriction à $S$, la convergence est uniforme et les fonctions $g$, $h$ et $i$ sont à valeurs dans $S$. On peut donc composer les limites dans les expressions $\id_S=f^{n_k}\circ f^{-n_k}$ et $f^{m_k}=f^{n_{k+1}}\circ f^{-n_k}$, ce qui donne $\id_S=g_{|S}\circ h_{|S}$ et $i_{|S}=g_{|S}\circ h_{|S}$. En particulier, $i_{|S}=\id_S$, et par prolongement analytique on en déduit que $i=\id_\Omega$. Ainsi $\Omega$ est un domaine de rotation.
\end{rema}

%%%%%%%%%%%%%%%%%%%%%%%%%%%%%%%%%%
% Précisions
%%%%%%%%%%%%%%%%%%%%%%%%%%%%%%%%%%

\section[Dynamique réelle dans l'ensemble de Fatou]{Précisions sur la dynamique réelle dans l'ensemble de Fatou}\label{precisions}

Soit $S$ une composante connexe de $X(\R)$ qui satisfait les hypothèses du théorème \ref{car-neg}. On sait que la composante de Fatou $\Omega$ contenant $S$ est un domaine de rotation (voir ci-dessus), et (quitte à prendre un itéré de $f$) le groupe $\mathcal{G}(\Omega)$ est donc soit un cercle $\S^1=\left\{z\in\C\,\big|\,|z|=1\right\}$, soit un tore de dimension deux $\T^2=\S^1\times\S^1$. Ce groupe agit fidèlement par difféomorphismes sur $S$, ce qui permet de décrire de manière simple la dynamique de $f$ sur $S$.

\begin{theo}\label{conj-rot-global}
Soit $f$ un automorphisme loxodromique réel d'une surface kählérienne compacte réelle $X$, et soit $S$ une composante connexe de~$X(\R)$ qui est contenue dans l'ensemble de Fatou de $f$. Alors, quitte à prendre un itéré\footnote{On doit prendre un itéré pour avoir $f(S)=S$ et $f_{|S}$ isotope à l'identité via une isotopie dans le groupe $\mathcal{G}(\Omega)$.}, $f_{|S}$ est conjugué à une rotation, c'est-à-dire :
\begin{enumerate}
\item
Si $S$ est une sphère, il existe un difféomorphisme 
\begin{equation}
\psi:\S^2=\left\{(x,y,z)\in\R^3\,\big|\,x^2+y^2+z^2=1\right\} \longrightarrow S
\end{equation}
tel que $\psi^{-1}\circ f_{|S}\circ\psi$ soit donné par une matrice de rotation
\begin{equation}
\begin{pmatrix}\cos(\theta)&-\sin(\theta)&0\\\sin(\theta)&\cos(\theta)&0\\0&0&1\end{pmatrix},
\end{equation}
pour un certain angle $\theta\in\R\backslash 2\pi\Q$.
\item
Si $S$ est un tore, il existe un difféomorphisme 
\begin{equation}
\psi: \T^2=\left\{(z_1,z_2)\in\C^2\,\big|\,|z_1|=|z_2|=1\right\} \longrightarrow S
\end{equation}
tel que $\psi^{-1}\circ f_{|S}\circ\psi$ soit donné par le produit de deux rotations
\begin{equation}
(z_1,z_2)\longmapsto({\rm e}^{i\theta_1}z_1,{\rm e}^{i\theta_2}z_2),
\end{equation}
pour certains angles $(\theta_1,\theta_2)\in\R^2\backslash(2\pi\Q)^2$.
\item
Si $S$ est un plan projectif réel, il existe un difféomorphisme 
\begin{equation}
\psi: \P^2(\R) \longrightarrow S
\end{equation}
tel que $\psi^{-1}\circ f_{|S}\circ\psi$ soit donné dans la carte affine $\R^2\subset\P^2(\R)$ par une matrice de rotation
\begin{equation}
\begin{pmatrix}\cos(\theta)&-\sin(\theta)\\\sin(\theta)&\cos(\theta)\end{pmatrix},
\end{equation}
pour un certain angle $\theta\in\R\backslash 2\pi\Q$.
\item
Si $S$ est une bouteille de Klein, il existe un difféomorphisme 
\begin{equation}
\psi: \mathbb{K}^2=\frac{\T^2=\left\{(z_1,z_2)\in\C^2\,\big|\,|z_1|=|z_2|=1\right\}}{(z_1,z_2)\sim(z_1^{-1},{\rm e}^{i\pi}z_2)} \longrightarrow S
\end{equation}
tel que $\psi^{-1}\circ f_{|S}\circ\psi$ soit donné par une rotation
\begin{equation}
(z_1,z_2)\longmapsto(z_1,{\rm e}^{i\theta}z_2),
\end{equation}
pour un certain angle $\theta\in\R\backslash 2\pi\Q$.
\end{enumerate}
\end{theo}

\begin{rema}
Le rang du domaine de rotation $\Omega\supset S$ ne peut  être égal à $2$ que lorsque $S$ est un tore et que les angles $\theta_1$ et $\theta_2$ (avec les notations ci-dessus) sont $\Q$-linéairement indépendants.
\end{rema}

\begin{lemm}\label{lemm-action}
Soit une action de classe $\mathcal{C}^1$, fidèle et sans point fixe, du groupe $\S^1$ sur le cylindre $\mathcal{C}=[0,1]\times\S^1$ ou sur le ruban de Möbius ${\mathcal{M}=\mathcal{C}/s}$, où $s$ est l'involution $(t,z)\mapsto(1-t,{\e}^{i\pi}z)$. Alors cette action est conjuguée à l'action standard de $\S^1$ par rotation :
\begin{equation}
{\rm e}^{i\theta}\cdot(t,z)=(t,{\rm e}^{i\theta}z).
\end{equation}
\end{lemm}

\begin{proof}
\begin{enumerate}
\item
On traite d'abord le cas où $\S^1$ agit sur le cylindre. Cette action provient d'un champ de vecteurs $\overrightarrow V$ de classe $\mathcal{C}^1$ sur $\mathcal{C}$. Comme l'action est sans point fixe, $\overrightarrow V$ ne s'annule pas. En particulier, les orbites sous l'action de~$\S^1$, qui sont les trajectoires du champ de vecteurs, sont des cercles homotopes aux bords du cylindre (c'est une conséquence facile du théorème de Poincaré--Bendixson).

En faisant tourner ce champ de vecteurs d'un angle positif de $\pi/2$, on obtient un champ de vecteurs orthogonal~$\overrightarrow W$. Chaque trajectoire $\gamma$ de $\overrightarrow V$ découpe $\mathcal C$ en deux cylindres $\mathcal{C}_1$ et $\mathcal{C}_2$ (éventuellement un des deux est un cercle si $\gamma$ est un des deux bords), avec le champ de vecteurs $\overrightarrow W$ qui est rentrant le long de~$\gamma$ sur~$\mathcal{C}_1$, et sortant sur $\mathcal{C}_2$. Ainsi, toute trajectoire de $\overrightarrow W$ sur $\mathcal{C}_1$ (resp. $\mathcal{C}_2$) partant d'un point de $\gamma$ (resp. terminant sur un point de $\gamma$) ne peut pas revenir sur le bord $\gamma$ une seconde fois, à cause de l'orientation de $\overrightarrow W$. On en déduit le fait suivant :
\begin{center}
Les trajectoires des champs $\overrightarrow V$ et $\overrightarrow W$ se coupent en au plus un point.
\end{center}

Considérons une trajectoire $\gamma$ pour $\overrightarrow W$, que l'on fait partir de l'un des deux bords du cylindre (celui pour lequel le champ de vecteurs $\overrightarrow W$ est rentrant). Comme cette trajectoire ne peut pas revenir sur le premier bord, a priori deux cas sont possibles :
\begin{itemize}
\item
La trajectoire $\gamma$ atteint le deuxième bord du cylindre en un temps fini~$t_1>0$.
\item
La courbe $\gamma$ est définie pour tout $t\geq 0$, et reste dans l'intérieur du cylindre. En vertu du théorème de Poincaré--Bendixson, cette trajectoire admet un point limite ou un cycle limite $\Gamma$. Comme le champ $\overrightarrow W$ ne s'annule pas, il ne peut s'agir que d'un cycle limite $\Gamma$, qui ne peut être homotopiquement trivial, donc qui est homotope au bord du cylindre. Or l'orbite sous $\S^1$ d'un point de $\Gamma$ est également un cercle homotope à $\Gamma$, qui coupe $\Gamma$ transversalement en au plus un point, d'après le fait énoncé plus haut : ceci est impossible.
\end{itemize}

Quitte à renormaliser et à inverser la trajectoire, on a ainsi construit une courbe $\gamma:[0,1]\to \mathcal{C}$ telle que
\begin{enumerate}
\item $\gamma(0)\in\{0\}\times\S^1$ et $\gamma(1)\in\{1\}\times\S^1$ ;
\item toute orbite coupe $\gamma$ exactement une fois, et de manière transverse.
\end{enumerate}
Le difféomorphisme
\begin{equation}
\begin{split}
\psi : [0,1]\times\S^1 &\longrightarrow \mathcal{C}\\
(t,{\rm e}^{i\theta}) & \longmapsto {\rm e}^{i\theta}\cdot \gamma(t),
\end{split}
\end{equation}
conjugue alors l'action standard de $\S^1$ sur $[0,1]\times\S^1$ à notre action sur $\mathcal{C}$.

\item Considérons maintenant une action fidèle et sans point fixe de $\S^1$ sur le ruban de Möbius $\mathcal{M}$. On note $\pi:\mathcal{C}\to\mathcal{M}$ le revêtement double correspondant à l'involution $s:(t,z)\mapsto(1-t,{\e}^{i\pi}z)$ sur $\mathcal{C}$.

Pour des raisons topologiques, le bord $\partial\mathcal{M}$ est nécessairement une orbite sous l'action de $\S^1$. Toute orbite est homotope à ce bord, donc dans $\pi_*\left(\pi_1(\mathcal{C})\right)$. On en déduit que l'action de $\S^1$ sur $\mathcal{M}$ se relève en une action de $\S^1$ sur $\mathcal{C}$, nécessairement fidèle et sans point fixe (l'action est fidèle, car elle l'est en restriction à chacun des bords). D'après le premier point, l'espace des orbites pour cette dernière action s'identifie avec le segment $[0,1]$, et l'involution $s$ induit une involution (non triviale) sur les orbites : par continuité, il existe donc une orbite $\gamma$ qui est fixée par $s$.

En découpant $\mathcal{C}$ le long de cette orbite $\gamma$, on obtient deux cylindres identiques $\mathcal{C}_1$ et $\mathcal{C}_2$, et le ruban de Möbius $\mathcal{M}$ est réobtenu à partir d'un tel cylindre en recollant les deux moitiés de $\gamma$ (voir figure \ref{fig-mobius}). Il suffit maintenant d'appliquer le premier cas à l'un de ces cylindres $\mathcal{C}_i$, puis de recoller de la manière que nous venons d'indiquer pour obtenir la conjugaison voulue.
\end{enumerate}
\begin{figure}[h]
\begin{center}
\scalebox{0.5}{\input{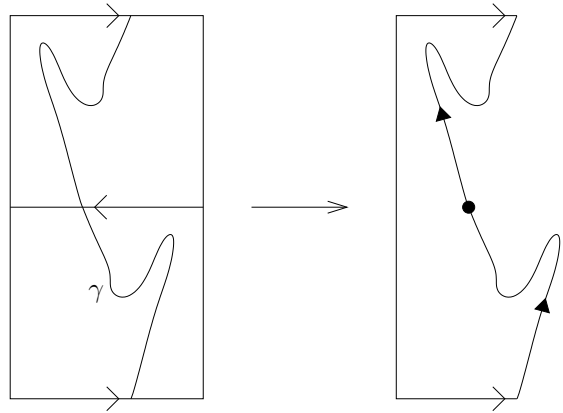}}
\end{center}
\caption{Le découpage et recollage du ruban de Möbius.}\label{fig-mobius}
\end{figure}
\end{proof}

\begin{proof}[Démonstration du théorème \ref{conj-rot-global}]
Soit $\Omega$ un domaine de rotation contenant $S$. Quitte à prendre un itéré de $f$, le groupe $\mathcal{G}(\Omega)$ est un tore de dimension $1$ ou $2$, qui agit par difféomorphismes $\R$-analytiques sur $S$. Cette action est fidèle, car si $g\in\mathcal{G}(\Omega)$ est l'identité sur $S$, alors $g$ est l'identité sur $\Omega$ par prolongement analytique.

Si $\mathcal{G}(\Omega)\simeq\T^2$, alors l'orbite d'un point générique $x_0\in S$ est un tore de dimension $2$, et le difféomorphisme
\begin{equation}
\begin{split}
\psi : \mathcal{G}(\Omega)\simeq\T^2 &\longrightarrow S\\
g & \longmapsto g(x_0),
\end{split}
\end{equation}
conjugue l'action par translations de $\T^2$ sur lui-même à l'action de $\mathcal{G}(\Omega)$ sur~$S$. En particulier, $f_{|S}$ est conjugué à une rotation.

À partir de maintenant, on suppose que $\mathcal{G}(\Omega)=\S^1$. Si $S$ est orientable, la formule de Lefschetz montre que $f$ possède $2$ points fixes sur $S$ si $S$ est une sphère, ou $0$ point fixe si $S$ est un tore. Dans le premier cas, $f$ est conjugué à une rotation dans des petits disques au voisinage de chaque point fixe ; en retirant ces disques, on obtient un cylindre $S'\subset S$ sur lequel $\S^1$ agit fidèlement et sans point fixe. Dans le second cas, en découpant le tore $S$ suivant une orbite de $\S^1$ (celle-ci est une courbe fermée simple non contractible\footnote{Sinon il y aurait un point fixe dans le disque bordé par cette courbe.}), on obtient aussi un cylindre $S'$ sur lequel $\S^1$ agit fidèlement et sans point fixe. On sait d'après le lemme \ref{lemm-action} que l'action de $\S^1$ sur $S'$ est conjuguée à une rotation, et en recollant, on obtient une action par rotation sur $S$.

Il reste à traiter le cas où $S$ est non orientable. Si $S$ est un plan projectif,~le groupe $\S^1$ possède $\chi(S)=1$ point fixe sur $S$, autour duquel l'action est linéarisable et conjuguée à une rotation. En enlevant un petit disque autour de ce point fixe, on obtient un ruban de Möbius, sur lequel $f$ agit par rotations d'après le lemme \ref{lemm-action}. Enfin, si $S$ est une bouteille de Klein, on obtient de même un ruban de Möbius en découpant le long d'une orbite, et $f$ agit par rotations sur ce ruban. Le résultat s'en déduit par recollement.
\end{proof}

%%%%%%%%%%%%%%%%%%%%%%%%%%%%%%%%%%
% Chapitre
%%%%%%%%%%%%%%%%%%%%%%%%%%%%%%%%%%
% Un exemple birationnel
%%%%%%%%%%%%%%%%%%%%%%%%%%%%%%%%%%

\chapter{Un exemple de difféomorphisme birationnel du tore}\label{chap-bir}

Au chapitre précédent, nous avons démontré que les seules composantes réelles qui pouvaient être incluses dans l'ensemble de Fatou sont des sphères, des tores, des plans projectifs ou des bouteilles de Klein. Il est maintenant naturel de chercher des exemples où l'ensemble de Fatou contient de telles composantes. Pour cela, on élargit l'étude aux applications birationnelles qui sont des difféomorphismes sur $X(\R)$ : on montre l'existence d'un tel difféomorphisme pour lequel le lieu réel est un tore inclus dans un domaine de rotation de rang~$2$.

%%%%%%%%%%%%%%%%%%%%%%%%%%%%%%%%%%
% BirDiff
%%%%%%%%%%%%%%%%%%%%%%%%%%%%%%%%%%

\section{Applications birationnelles et difféomorphismes birationnels}

Soit $X$ une surface complexe projective (ou kählérienne compacte), et soit~$f\in\Bir(X)$ une application birationnelle (ou biméromorphe) de $X$. On note~$\Ind(f)$ l'ensemble des points d'indétermination de $f$.

Dans ce cadre, il n'existe pas de définition universelle pour l'ensemble de Fatou. On décide d'adopter la définition suivante, qui est assez restrictive, ce qui n'est pas un problème compte tenu de ce que l'on veut montrer.

\begin{defi}
Un point $x\in X$ est dans l'ensemble de Fatou s'il existe un voisinage $U$ de $x$ tel que
\begin{itemize}
\item $U\cap\Ind(f^n)=\emptyset$ pour tout $n\in\Z$ ;
\item la famille ${(f^n_{|U})}_{n\in\Z}$ est une famille normale.
\end{itemize}
\end{defi}

En particulier, $f$ est un difféomorphisme analytique en restriction à l'ensemble de Fatou.

Par analogie avec le cas des automorphismes, on définit aussi le (premier) \emph{degré dynamique} grâce à l'action induite par les itérés de $f$ sur $H^{1,1}(X;\R)$ :
\begin{equation}
\lambda(f)=\lim_{n\to+\infty}\left\|{(f^n)}_*\right\|^{1/n}.
\end{equation}
Dans le cas des automorphismes, on a ${(f^n)}_*={(f_*)}^n$, donc $\lambda(f)$ est bien égal au rayon spectral de $f_*$, prolongeant ainsi la définition déjà donnée dans ce cadre. Les applications birationnelles dont le degré dynamique est strictement supérieur à $1$ apparaissent ainsi comme l'analogue des automorphismes de type loxodromique : ce sont celles qui ont une dynamique riche (voir \cite{diller-favre}). Pour cette raison, on les appelle \emph{applications birationnelles de type loxodromique}.

Lorsque $X$ possède une structure réelle, on note $\BD(X)$ l'ensemble des applications birationnelles qui préservent cette structure réelle, et telles que
\begin{equation}
\Ind(f)\cap X(\R)=\Ind(f^{-1})\cap X(\R)=\emptyset.
\end{equation}
Une telle application définit un difféomorphisme réel-analytique de $X(\R)$ :
\begin{align}
1\to\BD(X)&\longrightarrow\Diff^\omega(X(\R))\\
f&\longmapsto f_\R:=f_{|X(\R)}.
\end{align}
Pour cette raison, on dit que $f$ est un difféomorphisme birationnel\footnote{Cette terminologie n'est pas standard.}. Rappelons le résultat démontré par Koll\'ar et Mangolte :

\begin{theo}[\cite{kollar-mangolte,lukackii}]\label{theo-kml}
Soit $X$ une surface algébrique réelle qui est birationnelle à $\P^2_\R$. Alors le groupe $\BD(X)$ est dense dans~$\Diff(X(\R))$.
\end{theo}

Ceci s'applique par exemple lorsque $X$ est la surface réelle $\P^1_\R\times\P^1_\R$. On se propose de démontrer le théorème suivant :

\begin{theo}\label{exbir}
Soit $X$ la surface rationnelle $\P^1\times\P^1$, munie de sa structure réelle standard. Il existe $f\in\BD(X)$ tel que
\begin{enumerate}
\item $f$ est de type loxodromique, \emph{i.e.} $\lambda(f)>1$ ;
\item l'ensemble de Fatou de $f$ contient le lieu réel $X(\R)\simeq\T^2$.
\end{enumerate}
\end{theo}

Plus précisément, on montre que $X(\R)$ est contenu dans un domaine de rotation de rang~$2$.

%%%%%%%%%%%%%%%%%%%%%%%%%%%%%%%%%%
% Herman
%%%%%%%%%%%%%%%%%%%%%%%%%%%%%%%%%%

\section{Conjugaison à une rotation}

On désigne par $\S^1$ l'ensemble des nombres complexes de module $1$, et par~$\T^2$ le tore $\S^1\times\S^1$. Pour $\theta=(\theta_1,\theta_2)\in\R^2$, on note $\Rot_\theta$ la rotation d'angle $\theta$ sur le tore $\T^2$, donnée par la formule
\begin{equation}
\Rot_\theta(x,y)=(\e^{i\theta_1}x,\e^{i\theta_2}y).
\end{equation}

Le théorème suivant est démontré par Herman dans \cite{herman-tore}, en reprenant des idées de Arnol$'$d \cite{arnold} et Moser \cite{moser} (voir aussi {\cite[A.2.2]{herman}}) :

\begin{theo}[Arnold, Moser, Herman]\label{theo-herman}
Soit $\alpha=(\alpha_1,\alpha_2)\in\R^2$ qui vérifie une \emph{condition diophantienne}, c'est-à-dire qu'il existe $C>0$ et~${\beta>0}$ tels que pour tout $k=(k_1,k_2,k_3)\in\Z^3\backslash\{0\}$, on ait
\begin{equation}\label{diop}
\lvert k_1\alpha_1 + k_2\alpha_2 + 2\pi k_3 \rvert \geq \frac{C}{\left(\lVert k\rVert_\infty\right)^\beta}.
\end{equation}
Pour tout $\epsilon>0$, il existe alors un voisinage $\mathcal{U}_{\alpha,\epsilon}$ de $\Rot_\alpha$ dans le groupe $\Diff^\omega(\T^2)$ tel que
\begin{equation}\label{conj-herman}
\forall g\in\mathcal{U}_{\alpha,\epsilon},\,
\exists\theta\in\left]-\epsilon,\epsilon\right[^2,\,
\exists \psi\in\Diff^\omega(\T^2),\,
g = \Rot_\theta\circ\psi\circ \Rot_\alpha\circ\psi^{-1}.
\end{equation}
\end{theo}

À partir de maintenant, on se place sur la surface rationnelle ${X=\P^1\times\P^1}$, avec la structure réelle standard. Le lieu réel de cette surface 
\begin{equation}
{X(\R)=\P^1(\R)\times\P^1(\R)}
\end{equation}
s'identifie au tore $\T^2$ via la transformation de Cayley
\begin{align}
\Psi : X(\R) & \rightarrow \T^2 \\
(x,y) & \mapsto \left( \frac{x-i}{x+i} , \frac{y-i}{y+i} \right).
\end{align}
Via cette transformation, la rotation $\Rot_\theta$ correspond à un automorphisme réel de $X$ que l'on note $R_\theta$, et qui est donné par la formule
\begin{equation}
R_\theta(x,y) = 
\left(
\frac{x+\tan(\theta_1/2)}{-x\tan(\theta_1/2)+1}
,
\frac{y+\tan(\theta_2/2)}{-y\tan(\theta_2/2)+1}
\right).
\end{equation}

Fixons un nombre $\alpha\in\left[0,2\pi\right]^2$ qui vérifie une condition diophantienne du type (\ref{diop}), et soit $\mathcal{U}_{\alpha,\epsilon}$ le voisinage de $\Rot_\alpha$ dans $\Diff^\omega(\T^2)$ donné par le théorème de Herman-Arnold-Moser \ref{theo-herman}, où
\begin{equation}\label{condepsilon}
\epsilon = \max(|\pi-\alpha_1|,|\pi-\alpha_2|)>0.
\end{equation}
On note $\mathcal{V}_{\alpha}=\Psi^{-1}\mathcal{U}_{\alpha,\epsilon}\Psi$ le voisinage correspondant de $R_\alpha$ dans $\Diff^\omega(X(\R))$.

\begin{prop}\label{fatoucontientreel}
Soit $f\in\BD(X)$ tel que $f_\R\in\mathcal{V}_{\alpha}$. Il existe alors ${\theta\in\left]-\epsilon,\epsilon\right[^2}$ tel que la transformation $R_\theta f$ soit analytiquement conjuguée, sur un voisinage $\Omega$ de $X(\R)$, à la rotation d'angle $\alpha$ sur une bicouronne de $\C^2$. Autrement dit, il existe un difféomorphisme analytique complexe~${\psi:\mathcal{C}_1\times\mathcal{C}_2\to\Omega}$, où $\mathcal{C}_1$ et $\mathcal{C}_2$ sont des couronnes dans $\C$, tel que
\begin{equation}
\psi^{-1}\circ(R_\theta f)\circ\psi(z_1,z_2)=({\rm e}^{i\alpha_1}z_1,{\rm e}^{i\alpha_2}z_2).
\end{equation}
En particulier $\Omega$ est un domaine de rotation de rang $2$ pour $R_\theta f$, et
\begin{equation}
X(\R)\subset\Omega\subset\Fat(R_\theta f).
\end{equation}
\end{prop}

\begin{proof}
Par construction du voisinage $\mathcal{V}_{\alpha}$, il existe~$\theta\in\left]-\epsilon,\epsilon\right[^2$ et un difféomorphisme réel-analytique $\psi:\T^2\to X(\R)$ tels que 
\begin{equation}
f_{|X(\R)}=R_{-\theta}\circ\psi\circ \Rot_\alpha\circ\psi^{-1}.
\end{equation}
Comme $\psi$ est analytique, il se prolonge en un difféomorphisme analytique~${\widetilde\psi : \mathcal{C}_\eta^2\to\Omega}$ pour un certain $\eta>0$, où $\mathcal{C}_\eta$ désigne la couronne
\begin{equation}
\mathcal{C}_\eta = \left\{z\in\C\,\big|\,-\eta<\log|z|<\eta\right\}.
\end{equation}
On peut supposer, quitte à réduire $\eta$, que le voisinage $\Omega$ de $X(\R)$ ne contient pas de point d'indétermination pour $f$. On a alors, par prolongement analytique :
\begin{equation}
\left(R_\theta f\right)_{|\Omega} = \widetilde\psi\circ \Rot_\alpha\circ\widetilde\psi^{-1}.
\end{equation}
Comme $\alpha$ est diophantien, la rotation $\Rot_\alpha$ est irrationnelle (au sens où ses orbites sont denses dans $\T^2$), donc $\Omega$ est un domaine de rotation de rang $2$ pour $R_\theta f$.
\end{proof}

Comme les difféomorphismes birationnels sont denses dans $\Diff^\infty(X(\R))$ (cf. théorème \ref{theo-kml}), on peut s'attendre à ce qu'il existe de telles transformations birationnelles $R_\theta f$ qui soient de type loxodromique. C'est ce que nous allons montrer maintenant.

\begin{rema}
Cette construction est analogue à celle des anneaux de Herman à l'aide de produits de Blaschke pour des endomorphismes de $\P^1(\C)$, telle que présentée par exemple dans \cite[\textsection 15]{milnor}.
\end{rema}

%%%%%%%%%%%%%%%%%%%%%%%%%%%%%%%%%%
% Construction explicite
%%%%%%%%%%%%%%%%%%%%%%%%%%%%%%%%%%

\section[Construction explicite d'un difféomorphisme birationnel]{Construction explicite d'un difféomorphisme birationnel de grand degré dynamique}

Pour s'assurer que le degré dynamique est strictement supérieur à $1$, on va utiliser le lemme suivant, que l'on ne redémontre pas ici :

\begin{lemm}[{\cite[Theorem 3.1]{xie}}]\label{theo-xie}
Soit $X$ une surface complexe rationnelle, et soit $f:X\dashrightarrow X$ une application birationnelle de $X$. On fixe~${L\in\NS(X;\R)}$ la classe d'un $\R$-diviseur ample sur $X$, et on suppose que
\begin{equation}
q=\frac{\deg_L(f^2)}{\deg_L(f)} \geq 3^{18}\sqrt{2},
\end{equation}
où $\deg_L(g)$ désigne le nombre d'intersection $g_*L\cdot L$. Alors 
\begin{equation}
\lambda(f) > \frac{2\times 3^{36}(4\times 3^{36}-1) q^2 + 1}{2^{5/2}\times 3^{54}\,q} \geq 1.
\end{equation}
\end{lemm}

\begin{rema}
Lorsque $X=\P^2$ et $L$ est une droite, $\deg_L$ correspond au degré usuel d'une application rationnelle, c'est-à-dire le degré des polynômes homogènes $P$, $Q$ et $R$ tels que $f=[P:Q:R]$. Dans \cite{xie}, le lemme \ref{theo-xie} est démontré uniquement dans ce cadre, mais la démonstration qui en est faite se transpose naturellement au cas plus général donné ci-dessus. Nous allons l'appliquer avec $X=\P^1\times\P^1$ et $L$ un vecteur propre pour $f_*$ sur $\NS(X;\R)$\footnote{Contrairement à ce qui se passe pour les automorphismes, $f_*$ peut avoir des vecteurs propres qui sont des classes amples.}.
\end{rema}

Fixons $d$ un entier tel que
\begin{equation}\label{condsurd}
d \geq 2^{-3/4} \times 3^9 \approx 11703,6.
\end{equation}
Pour $n\in\N^*$, soit $F_n$ la fraction rationnelle suivante :
\begin{equation}
F_n(x)=\frac{x^{2d}+\frac{2}{n}x^d+1}{x^{2d}+1}.
\end{equation}
Ses zéros et ses pôles sont simples, et sont donnés par les ensembles
\begin{gather}
Z_n = \left\{
\e^{i(\pm\arccos(\frac{1}{n})+2k\pi)/d} \,\big|\, k\in\{0,\cdots,d-1\}
\right\},\\
P_n = \left\{
\e^{i(\pm\frac{\pi}{2}+2k\pi)/d} \,\big|\, k\in\{0,\cdots,d-1\}
\right\}.
\end{gather}
En particulier, $Z_n\cap P_n=\emptyset$ et $(Z_n\cup P_n)\subset\C\backslash\R$.
Soit $g_n:X\dashrightarrow X$ l'application rationnelle définie par
\begin{equation}
g_n(x,y)=(F_n(x)y,x).
\end{equation}
Cette application est birationnelle réelle, d'inverse $(x,y)\to(y,x/F_n(y))$. Les points d'indétermination de $g_n$ sont donnés par
\begin{equation}
\Ind(g_n) = \big(Z_n\times\{\infty\}\big) \cup \big(P_n\times\{0\}\big),
\end{equation}
et ceux de $g_n^{-1}$ sont donnés par
\begin{equation}
\Ind(g_n^{-1}) = (\{0\}\times Z_n) \cup (\{\infty\}\times P_n).
\end{equation}
En particulier, ces points d'indétermination ne sont pas réels, donc ${g_n}$ est un difféomorphisme birationnel de $X$. D'autre part, $\Ind(g_n)$ et $\Ind(g_n^{-1})$ ne s'intersectent pas, donc d'après \cite{diller-favre}
\begin{equation}
(g_n^2)_*=({g_n}_*)^2.
\end{equation}

\begin{theo}\label{fntheta}
Pour $n\in\N^*$ et $\theta=(\theta_1,\theta_2)\in\R^2$, on considère l'application
\begin{equation}
f_{n,\theta}=R_\theta \circ g_n^2 \in \BD(\P^1\times\P^1).
\end{equation}
On suppose que 
$\theta_j\neq \pi \mod 2\pi$ pour $j\in\{1,2\}$. Alors :
\begin{enumerate}
\item
$\Ind(f_{n,\theta})\cap\Ind(f_{n,\theta}^{-1})=\emptyset$, et donc $(f_{n,\theta}^2)_*=({f_{n,\theta}}_*)^2$.
\item
$\lambda(f_{n,\theta})>1$.
\end{enumerate}
\end{theo}

\begin{proof}
Les points d'indétermination de $f_{n,\theta}$ sont donnés par
\begin{align}
\Ind(f_{n,\theta})&=\Ind(g_n^2)\\
&=\Ind(g_n)\cup g_n^{-1}\left(\Ind(g_n)\right)\\
&=\big(Z_n\times\{\infty\}\big)\cup\big(P_n\times\{0\}\big)\cup\big(\{\infty\}\times Z_n\big)\cup\big(\{0\}\times P_n\big),
\end{align}
et ceux de $f_n^{-1}$ par
\begin{align}
\Ind(f_{n,\theta}^{-1})&=R_\theta(\Ind(g_n^{-2}))\\
&=R_\theta(\Ind(g_n^{-1})\cup g_n(\Ind(g_n^{-1}))\\
&=R_\theta\big((\{0\}\times Z_n)\cup(\{\infty\}\times P_n)\cup(Z_n\times\{0\})\cup(P_n\times\{\infty\})\big)\\
\begin{split}
&=\big(\{t_{\theta_1}\}\times Z_{n,\theta_2}\big)\cup\big(\{-t_{\theta_1}^{-1}\}\times P_{n,\theta_2}\big)\\
&\quad\quad \cup\big(Z_{n,\theta_1}\times\{t_{\theta_2}\}\big)\cup\big(P_{n,\theta_1}\times\{-t_{\theta_2}^{-1}\}\big),
\end{split}
\end{align}
où $t_{\theta_j}=\tan(\theta_j/2)$, $Z_{n,\theta_j}=R_{\theta_j}(Z_n)$ et $P_{n,\theta_j}=R_{\theta_j}(P_n)$, avec 
\begin{equation}
R_{\theta_j}(x)=\frac{x+t_{\theta_j}}{-t_{\theta_j}x+1}.
\end{equation}

La condition sur $\theta_j$ implique $t_{\theta_j}\neq\infty$. Comme de plus $t_{\theta_j}$ est réel, il n'est pas dans $Z_n\cup P_n$. Ainsi, la seule possibilité pour que $\Ind(f_{n,\theta})$ et $\Ind(f_{n,\theta}^{-1})$ s'intersectent est d'avoir
\begin{equation}
t_{\theta_1}=0 \quad \text{et} \quad \big(Z_n\cap P_{n,\theta_2}\big)\cup\big(P_n\cap Z_{n,\theta_2}\big)\neq\emptyset,
\end{equation}
ou la même condition en inversant $\theta_1$ et $\theta_2$.

On note $\mathcal{R}$ le groupe $\{R_\theta\,|\,\theta\in\R\}$. Ce groupe agit sur $\P^1(\C)$, et l'orbite d'un point $x\in\S^1\backslash\{i,-i\}$ est un cercle transverse à $\S^1$, qui n'intersecte $\S^1$ qu'aux points $x$ et $-1/x$. Les points $i$ et $-i$ sont quant à eux fixes par $\mathcal{R}$. Comme $Z_n$ et $P_n$ sont des sous-ensembles de $\S^1$ qui ne s'intersectent pas, et comme $P_n$ est stable par $x\mapsto-1/x$, on en déduit que $Z_n\cap(\mathcal{R}\cdot P_n)=\emptyset$. Par conséquent, les intersections ci-dessus sont toujours vides, et donc $\Ind(f_{n,\theta})\cap\Ind(f_{n,\theta}^{-1})=\emptyset$.

Notons $H=[\P^1\times\{0\}]$ et $V=[\{0\}\times\P^1]$ la base canonique de $\NS(X)$. Dans cette base, la matrice de ${g_n}_*$ s'écrit
\begin{equation}\label{matriceA}
A=\begin{pmatrix}2d&1\\1&0\end{pmatrix}.
\end{equation}
En effet, il suffit de calculer les nombres d'intersection ${g_n}_*H\cdot V$, ${g_n}_*V\cdot V$, etc. qui correpondent aux degrés des fonctions coordonnées par rapport à $x$ et $y$.
Cette matrice admet $\lambda=d+\sqrt{d^2+1}$ pour plus grande valeur propre, et $L=\lambda H + V$ pour vecteur propre associé à $\lambda$. Notons que $L$ est une classe ample.

Comme $(g_n^2)_*=({g_n}_*)^2$ et $\left(R_\theta\right)_*=\id$, la matrice de ${f_{n,\theta}}_*$ est la matrice $A^2$. De plus $(f_{n,\theta}^2)_*=({f_{n,\theta}}_*)^2$, donc la matrice de $(f_{n,\theta}^2)_*$ est $A^4$. On en déduit, avec les notations du lemme \ref{theo-xie}, que $\deg_L(f_{n,\theta})=2\lambda^3$ et $\deg_L(f_{n,\theta}^2)=2\lambda^5$ (on a~${L^2=2\lambda}$). Ainsi,
\begin{equation}
\frac{\deg_L(f_{n,\theta}^2)}{\deg_L(f_{n,\theta})} = \lambda^2 \geq 4d^2 \geq 3^{18}\sqrt{2}
\end{equation}
d'après la condition $(\ref{condsurd})$. Le lemme \ref{theo-xie} implique alors $\lambda(f_{n,\theta})>1$.
\end{proof}

\begin{proof}[Démonstration du théorème \ref{exbir}]
Comme $g_n$ converge vers l'application ${(x,y)\mapsto(y,x)}$ dans $\Diff^\omega(X(\R))$, la suite $\left(f_{n,\alpha}\right)_{n\in\N^*}$ converge vers $R_\alpha$. Pour~$n$ suffisamment grand, les applications $f_{n,\alpha}$ sont donc dans le voisinage~$\mathcal{V}_{\alpha}$ de~$R_\alpha$ (défini à la section précédente). D'après la proposition \ref{fatoucontientreel}, il existe~${\theta\in\left]-\epsilon,\epsilon\right[^2}$ tel que $R_\theta f_{n,\alpha}=f_{n,\alpha+\theta}$ soit conjugué à une rotation dans un voisinage de $X(\R)$. D'après le choix de $\epsilon$ (cf. $(\ref{condepsilon})$), la condition
\begin{equation}
\alpha_j+\theta_j\neq\pi \mod 2\pi
\end{equation}
du théorème \ref{fntheta} est satisfaite. On a donc :
\begin{align}
\lambda(f_{n,\alpha+\theta})>1
\quad \text{et} \quad
\Fat(f_{n,\alpha+\theta})\supset X(\R),
\end{align}
et le théorème \ref{exbir} est démontré.
\end{proof}

%%%%%%%%%%%%%%%%%%%%%%%%%%%%%%%%%%
% BIBLIO
%%%%%%%%%%%%%%%%%%%%%%%%%%%%%%%%%%

\backmatter

\bibliographystyle{smfalpha}
\bibliography{biblio}

\end{document}